\documentclass[12pt,oneside]{book}

\usepackage{amssymb, amsmath, amsthm, mathpazo, setspace }
\usepackage{titlesec}

\usepackage[left=1.5in, right=1in, top=1in, bottom=1in]{geometry}

\makeatletter
\renewcommand*\@makechapterhead[1]{%
  \vspace*{.75in}%
  {\parindent \z@ \raggedright \normalfont
    \ifnum \c@secnumdepth >\m@ne%controls chapter size
        \large\bfseries \@chapapp\space \thechapter
        \par\nobreak
        \vskip 20\p@
    \fi
    \interlinepenalty\@M %controls chapter title
    \large \bfseries #1\par\nobreak
    \vskip 40\p@
  }}
\renewcommand*\@makeschapterhead[1]{%
  \vspace*{.75in}%
  {\parindent \z@ \raggedright
    \normalfont
    \interlinepenalty\@M
    \large \bfseries  #1\par\nobreak
    \vskip 40\p@
  }}
\makeatother

\titleformat*{\section}{\large\bfseries}
\titleformat*{\subsection}{\large\bfseries}

\theoremstyle{plain}
\newtheorem{thm}{Theorem}[chapter]

\newtheorem{lem}[thm]{Lemma}
\newtheorem{prop}[thm]{Proposition}

\theoremstyle{definition}

\newtheorem{dfn}[thm]{Definition}

\newcommand \mylabel[1]{\label{#1}}

\def\bigskip{\vspace{30pt}}

\usepackage{amsmath,amssymb,amsfonts,graphicx}

\DeclareMathSymbol{\N}{\mathbin}{AMSb}{"4E} 
\DeclareMathSymbol{\Z}{\mathbin}{AMSb}{"5A} 
\DeclareMathSymbol{\R}{\mathbin}{AMSb}{"52} 
\DeclareMathSymbol{\Q}{\mathbin}{AMSb}{"51} 
\DeclareMathSymbol{\C}{\mathbin}{AMSb}{"43} 

\def\alp{\alpha}

\def\eps{\varepsilon}
\def\vph{\varphi}
\def\lam{\lambda}
\def\Lam{\Lambda}
\def\omg{\omega}
\def\Omg{\Omega}
\def\gam{\gamma}

\def\ds{\displaystyle}

\def\ni{\noindent}
\def\nn{\nonumber}
\def\F{{\mathcal F}}
\def\V{{\mathcal Vac}}

\def\fg{\mathfrak{g}}
\def\fgl{\mathfrak{gl}}
\def\fsl{\mathfrak{sl}}
\def\fh{\mathfrak{h}}
\def\fb{\mathfrak{b}}
\def\fa{\mathfrak{a}}
\def\fc{\mathfrak{c}}

\def\fp{\mathfrak{p}}
\def\fgh{\hat{\fg}}
\def\fgt{\tilde{\fg}}
\def\fhh{\hat{\fh}}
\def\fht{\tilde{\fh}}
\def\fpt{\tilde{\fp}}
\def\fat{\tilde{\fa}}
\def\fbt{\tilde{\fb}}
\def\fah{\hat{\fa}}
\def\fhhz{\fhh_{\Z}}
\def\PHI{\hat{\Phi}}
\def\DELH{\hat{\Delta}}
\def\cvir{c_{Vir}}
\def\sset{\subseteq}
\def\herfc{\langle \cdot \, , \cdot \rangle}
\def\bs{\backslash}
\def\veps{\varepsilon}
\def\hat{\widehat}

\def\bi{\bold{i}}

\def\Es{E_6^{(1)}}
\def\Ff{F_4^{(1)}}

\def\bo{\bold{1}}

\def\bold{\textbf}

\def\mylabel{\label}

\newcommand{\herf}[2]{\langle #1 , #2 \rangle}
\newcommand{\yvoz}[1]{Y(#1,z)}
\newcommand{\iyvoz}[3]{ \mathcal{Y}_{#3_{#2}}(#1,z)}
\newcommand{\ywvoz}[1]{Y_W(#1,z)}

\newcommand{\yvozr}[1]{E^-(-#1,z)E^+(-#1,z)e_{#1}z^{#1(0)}}

\newcommand{\yvow}[1]{Y(#1,w)}

\newcommand{\yvomn}[2]{Y_{#2}\left(#1\right)}

\newcommand{\ywtvomn}[3]{\sum_{#3 \in #1} \yvomn{#2}{#3}z^{-#3-wt(#2)}}

\newcommand{\bfor}[2]{( #1 , #2 )}
\newcommand{\otz}{1\otimes e^0}
\newcommand{\otea}[1]{1\otimes e^{#1}}

\begin{document}

% makes the page numbers roman numerals, doesn't count
% these pages in the table of contents
\frontmatter

\thispagestyle{empty}

\vbox to .875truein{}
%\begin{doublespacing} 
{\Large
\centerline{Branching Rule Decomposition of Irreducible Level-1 } 
\centerline{ $\Es$-modules with respect to $\Ff$}}
%\end{doublespacing}
\vskip 174.5pt

\centerline{BY}
\vskip 10pt

\centerline{CHRISTOPHER A MAURIELLO}
\vskip 10pt

\centerline{B.S., University of Texas at San Antonio, 2007}
\centerline{M.A., Binghamton University State University of New York, 2009}

\vskip 174.5pt

\centerline{DISSERTATION}
\vskip 10pt

\centerline{Submitted in partial fulfillment of the requirements for}
\centerline{the degree of Doctor of Philosophy in Mathematical Sciences}
\centerline{in the Graduate School of}
\centerline{Binghamton University}
\centerline{State University of New York}
\centerline{2013}

\newpage

\thispagestyle{empty}

\vbox to 8.30truein{}

\centerline{\copyright\ Copyright by Christopher Mauriello 2013}

\

\centerline{All Rights Reserved}

\newpage

{\baselineskip = 10pt

\vbox to 2.0truein{}

\centerline{Accepted in partial fulfillment of the requirements for}
\centerline{the degree of Doctor of Philosophy in Mathematical Sciences}
\centerline{in the Graduate School of}
\centerline{Binghamton University}
\centerline{State University of New York}
\centerline{2013}
\vskip 129.5pt

\centerline{April 24, 2013}
\vskip 129.5pt

\centerline{Alex J Feingold, Chair and Faculty Advisor}
\centerline{Department of Mathematical Sciences, Binghamton University}
\vskip 20pt

\centerline{Benjamin Brewster, Member}
\centerline{Department of Mathematical Sciences, Binghamton University}
\vskip 20pt

\centerline{Marcin Mazur, Member}
\centerline{Department of Mathematical Sciences, Binghamton University}
\vskip 20pt

\centerline{Antun Milas, Outside Examiner}
\centerline{Department of Mathematics and Statistics, University at Albany}

}

\newpage
\fontsize{11}{20pt} \selectfont  
\chapter*{Abstract}

It is well known that using the weight lattice of type $E_6$, $P$, and the lattice construction for vertex operator algebras one can obtain all three level 1 irreducible $\fgt$-modules with $V_P = V^{\Lam_0} \oplus V^{\Lam_1} \oplus V^{\Lam_6}$.  The Dynkin diagram of type $E_6$ has an order 2 automorphism, $\tau$, which can be lifted to $\tilde{\tau}$, a Lie algebra automorphism of $\fgt$ of type $\Es$.  The fixed points of $\tilde{\tau}$ are a subalgebra $\fat$ of type $\Ff$.  The automorphism $\tau$ lifts further to $\hat{\tau}$ a vertex operator algebra automorphism of $V_P$.

We investigate the branching rules, how these three modules for the affine Lie algebra $\fgt$ decompose as a direct sum of irreducible $\fat$-modules.  To complete the decomposition we use the Godard-Kent-Olive coset construction \cite{GKO} which gives a $c = \frac{4}{5}$ module for the Virasoro algebra on $V_P$ which commutes with $\fat$.  We use the irreducible modules for this coset Virasoro to give the space of highest weight vectors for $\fat$ in each $V^{\Lam_i}$.  The character theory related to this decomposition is examined and we make a connection to one of the famous Ramanujan identities.  This dissertation constructs coset Virasoro operators $\yvoz{\omg}$ by explicitly determining the generator $\omg$.  We also give the explicit highest weight vectors for each $Vir \otimes \Ff$-module in the decomposition of each $V^{\Lam_i}$.  This explicit work on determining highest weight vectors gives some insight into the relationship to the Zamolodchikov $\mathcal{W}_3$-algebra. 
\newpage

\newpage

\tableofcontents

% Changes page numbers to regular numbers, resets the counter
\mainmatter

% This gives 11pt font with 20pt spacing, text from here should be double spaced

% \include puts in the .tex file with the given name
% make sure that these files don't have any preamble material
%
% 
% 
% Introduction
% 
%

% make introduction a chapter without a #
\chapter*{Introduction}

% and include it on the table of contents
\addcontentsline{toc}{chapter}{Introduction}

This dissertation solves a problem in the representation theory of infinite dimensional affine Kac-Moody Lie algebras.  Given an irreducible unitary highest weight representation $V^{\Lam}$ of such an algebra $\fgt$, and a subalgebra $\fat$, branching rules give the decomposition of $V^{\Lam}$ into a direct sum of irreducible representations of $\fat$. The (infinite number of) summands in that decomposition correspond to highest weight vectors (HWVs) with respect to $\fat$, and the space of such HWVs forms a module for another Lie algebra, the Virasoro algebra, that commutes with $\fat$.   The decomposition can then be written as a finite sum of tensor products of the form $V(c,h) \otimes W$ where $V(c,h)$ is an irreducible highest weight Virasoro module and $W$ is a irreducible highest weight $\fat$-module. 

The specific decompositions investigated here are for the three level 1 irreducible representations, $V^{\Lam_0}, V^{\Lam_1}, V^{\Lam_6}$, of the affine algebra $\fgt$ of type $\Es$ with respect to its affine subalgebra $\fat$ of type $\Ff$.  We use the lattice (or bosonic) representations of these three modules given in \cite{FLM} (see also \cite{FK}) which gives a vertex operator algebra (VOA) structure on $V^{\Lam_0}$, and gives the $\fgt$-modules $V^{\Lam_1}$ and $V^{\Lam_6}$ a VOA module structure.  %Also, we give $V^{\Lam_0} \oplus V^{\Lam_1} \oplus V^{\Lam_6}$ a vertex operator para-algebra structure, also known as an abelian intertwining algebra.  
The VOA structure on $V^{\Lam_0}$ includes vertex operators $Y(v,z)$ for $v \in V^{\Lam_0}$ which represent the affine Lie algebras $\fgt$ and $\fat$.  For both $\fgt$ and $\fat$ there is a representation of the Virasoro algebra whose operators are provided by the vertex operators $Y(\omg_{E_6},z)$ and $Y(\omg_{F_4},z)$.  Using the Goddard-Kent-Olive coset construction \cite{GKO}, we get a representation of the Virasoro algebra with central charge $\frac{4}{5}$, given by the vertex operators $Y(\omg_{E_6}-\omg_{F_4},z)$ which commute with $\fat$.  Each irreducible $\fgt$-module, $V^{\Lam_i}$ for $i \in \{0,1,6\}$, decomposes into a direct sum of tensor products of the form $Vir\left(\frac{4}{5},h\right) \otimes W^{\Omg_j}$, where $Vir(c,h)$ is a irreducible highest weight Virasoro module with central charge $c$, $h$ is the eigenvalue of the $L(0)$ Virasoro operator on the highest weight vector, and $W^{\Omg_j}$ is an irreducible highest weight $\fat$-module with highest weight $\Omg_j$. 

Our first result, in Theorems \ref{thm br1} and \ref{thm br2}, expresses the principally graded dimension of each $V^{\Lam_i}$ as a sum of products, $gr_{princ}\left(Vir\left(\frac{4}{5},h\right)\right)gr_{princ}(W^{\Omg_j})$ of principally graded dimensions. This proof uses in an essential way one of the forty famous Ramanujan identities for the Rogers-Ramanujan series (See \cite{BCC}, Entry 3.6). Our second result gives the explicit formulas, in Theorems \ref{ch6thm1} and \ref{ch6thm2}, for the eight HWVs for the $Vir \otimes \fat$-modules that occur as summands in the decompositions of the $\fgt$-modules $V^{\Lam_i}$, $i \in \{0,1,6\}$.  Combining these two results proves that no other summands appear, and completes the branching decompositions.  The decompositions we obtained are
$$
V^{\Lam_0} = \left(Vir\left(\frac{4}{5},0\right) \oplus Vir\left(\frac{4}{5},3\right)\right) \otimes W^{\Omg_0} \bigoplus \left(Vir\left(\frac{4}{5},\frac{2}{5}\right) \oplus Vir\left(\frac{4}{5},\frac{7}{5}\right) \right) \otimes W^{\Omg_4} $$
$$
V^{\Lam_1} = Vir\left(\frac{4}{5},\frac{2}{3}\right) \otimes W^{\Omg_0} \bigoplus Vir\left(\frac{4}{5},\frac{1}{15}\right) \otimes W^{\Omg_4} $$
and
$$
V^{\Lam_6} = Vir\left(\frac{4}{5},\frac{2}{3}\right) \otimes W^{\Omg_0} \bigoplus Vir\left(\frac{4}{5},\frac{1}{15}\right) \otimes W^{\Omg_4} 
$$

It is known that $Vir\left(\frac{4}{5},0\right) \oplus Vir\left(\frac{4}{5},3\right) $ has the structure of a VOA, also known as the Zamolodchikov $\mathcal{W}_3$-algebra.  In fact, we can see that this is a sub-VOA of $V^{\Lam_0}$ which is the commutant of the sub-VOA $W^{\Omg_0}$.  So these decompositions can also be viewed as $\mathcal{W}_3 \otimes W^{\Omg_0}$-module decompositions.

%We prove these results by explicitly finding highest weight vectors corresponding to each direct summand above, and we use principally specialized graded dimensions to show the equalities above.

This dissertation is organized as follows.  First, necessary background material on semisimple finite and infinite dimensional Lie algebras is provided.  Next, we present the theory of characters and graded dimensions of modules for affine Lie algebras used in this work.  The principally specialized graded dimensions of $V^{\Lam_i}$ are expressed in terms of the principally specialized graded dimensions of $W^{\Omg_j}$ and $Vir\left(\frac{4}{5},h\right)$.  The theory of vertex operator algebras and their modules is introduced.  Next, the conformal vectors related to representations of the Virasoro algebras used in this work are given explicitly.  Finally, the verifications of the highest weight vectors are given explicitly for each module $Vir\left(\frac{4}{5},h\right) \otimes W^{\Omg_j}$ that occurs in each decomposition.

%
%  Chapter 1
%  
%  
%  
%  
%
\chapter{Finite Dimensional and Affine Lie Algebras}
% the label lets you refer to the chapter by number later.
\label{ch 1}
This chapter contains the necessary background material on finite dimensional Lie algebras, and some theory on infinite dimensional Lie algebras including affine Lie algebras, Virasoro algebras, and the vertex operator representation of certain affine Lie algebras.

\section{Finite Dimensional Semisimple Lie Algebras}
\mylabel{s 1.1}

This section is an introduction to finite dimensional semisimple Lie algebras.  This is a classic theory and it is not all covered here, but the essentials to be used later are presented for continuity. 

\begin{dfn}
\mylabel{Def: 1.1}
 A \bold{Lie algebra} is a finite dimensional vector space $\fg$ over a field $\mathbb{F}$, with a bilinear operation $[\cdot , \cdot ] : \fg \times \fg \to \fg$, denoted $(x,y) \mapsto [x,y]$,  which satisfies the following axioms:
\begin{description} \itemsep-5pt \parskip0pt \parsep0pt
\item{(L1)} \, $[x,x] = 0$, $\forall x \in \fg$ 
\item{(L2)} \, $[x,[y,z]] + [y,[z,x]] + [z,[x,y]] = 0$, $\forall x,y,z \in \fg$ 
\end{description}
\end{dfn}

The symbol $[\cdot , \cdot ]$ is commonly called the Lie bracket, and it is common to say $[x,y]$ as "$x$ bracket $y$".  Axiom $(L2)$ is commonly called the Jacobi identity, and the application of $(L1)$ and bilinearity to $[x + y, x+y]$ will imply $[x,y] = -[y,x]$.  Hence the Lie bracket is anticommutative.  Even though a Lie algebra can be defined over any field, unless noted, the field under consideration will be $\C$.

An example of a Lie algebra is $\fgl (V)$, which is the associative algebra $End(V)$ with the Lie bracket given by $[A, B] = AB - BA$, for any $A,B \in End(V)$.  The associativity of the product in $End(V)$ implies the axioms for this Lie bracket.

Define the dimension of a Lie algebra as its dimension as a vector space.  If for any two elements $x,y \in \fg$, $[x,y] = 0$, then we call $\fg$ a commutative Lie algebra or an abelian Lie algebra.

Let $A$ and $B$ be subsets of $\fg$ and denote by $[A,B]$, the vector subspace of $\fg$ spanned by $\{ [x,y] | x \in A, y \in B \}$.  There are two special cases of this notation.

\begin{dfn} \mylabel{Def: 1.2}
Let $\fg$ be a Lie algebra and let $\fa$ be a vector subspace of $\fg$.
\begin{enumerate} \itemsep-5pt \parskip0pt \parsep0pt
\item If $\fa$ satisfies $[\fa, \fa] \sset \fa$, then $\fa$ is called a \bold{Lie subalgebra} of $\fg$.
\item If $\fa$ satisfies $[\fa,\fg] \sset \fa$, then $\fa$ is called an \bold{ideal} of $\fg$.
\end{enumerate}
\end{dfn}

Two ideals that are always present are $\fg$ itself and the vector subspace $\{ 0 \}$.  These are called the \bold{trivial ideals} of $\fg$.  A Lie algebra that has no other ideals except the trivial ideals and is not abelian is called a \bold{simple} Lie algebra.  A Lie algebra that is a direct sum of simple Lie algebras is called a \bold{semi-simple} Lie algebra.  These definitions imply that a simple Lie algebra is also semi-simple.

For a Lie algebra $\fg$ define a derivation to be a linear map $\delta : \fg \to \fg$ that satisfies: $\delta([x,y]) = [x, \delta(y)] + [\delta(x), y]$.  In particular for any $x \in \fg$, $ad_x$, defined by $y \mapsto [x,y]$ is a derivation.  The Jacobi identity can be used to show that $ad_x$ is a derivation.

Given a linear transformation $\phi : \fg \to \fg'$, call $\phi$ a Lie algebra homomorphism if $\phi([x,y]) = [\phi(x), \phi(y)]$.  The notions of isomorphism and automorphism are defined accordingly.  Also, a representation of a Lie algebra on a vector space $V$ is defined as a Lie algebra homomorphism $\rho: \fg \to \fgl (V)$. 

One of the most important representations of a Lie algebra is called the adjoint representation.  The representation $ad: \fg \to \fgl (\fg)$, is given by $x \mapsto ad_x$.  The verification that $ad$ is a Lie algebra representation follows from the Jacobi identity as follows:
\begin{eqnarray*}
ad([x,y])(z) &=& ad_{[x,y]}(z) = [[x,y],z] \\
&=& [x,[y,z]] + [[x,z],y] \\
&=& [x,[y,z]] - [y,[x,z]] \\
&=& ad_x([y,z]) - ad_y([x,z]) \\
&=& ad_xad_y(z) - ad_y ad_x(z) \\
&=& [ad_x, ad_y](z) = [ad(x), ad(y)](z)
\end{eqnarray*}

It will be useful to also consider the equivalent language of modules along with the theory of representations. 

\begin{dfn} \mylabel{Def: 1.3}
A vector space $V$ along with the map from $\fg \times V$ to $V$, which is denoted $(x,v) \mapsto x \cdot v$, is called a \bold{$\fg$-module} if for $a,b \in \C$, $x,y \in \fg$, and $v,w \in V$, the following conditions are satisfied:
\begin{description} \itemsep-5pt \parskip0pt \parsep0pt
\item{(M1)} \, $(ax + by) \cdot v = a(x \cdot v) + b (y \cdot v)$,
\item{(M2)} \, $x \cdot (av + bw) = a(x \cdot v) = b (x \cdot w)$,
\item{(M3)} \, $[x,y] \cdot v = x \cdot (y \cdot v) - y \cdot (x \cdot v) $.
\end{description}
\end{dfn} 

A vector subspace $W$ of a $\fg$-module $V$ is called a \bold{$\fg$-submodule} if $W$ is a $\fg$-module.  If the only $\fg$-submodules of a $\fg$-module V are either 0 or $V$, call the module \bold{irreducible}.

\begin{dfn} \mylabel{Def: 1.4}
If $x,y \in \fg$, define the symmetric bilinear form called the \bold{Killing form}, as $\kappa(x,y) = Tr(ad_x \circ ad_y)$, the trace on $V$ of $ad_x \circ ad_y$ .
\end{dfn} 
The bilinearity of $\kappa$ follows from linearity of $ad$ and is symmetric since $Tr$ is symmetric.  $\kappa$ is also invariant, in the sense that $\kappa([x,y],z) = \kappa(x,[y,z])$.  This property is seen from the following identity, for $X,Y,Z \in \fgl(V) $, $Tr([X,Y]Z) = Tr(X[Y,Z])$.

\begin{thm}\mylabel{Thm: 1.5}(Cartan's Criterion)
For a finite dimensional Lie algebra $\fg$, $\fg$ is semisimple if and only if the Killing form is non-degenerate.
\end{thm}

\begin{thm} \mylabel{Thm: 1.6} If $\fg$ is a simple finite dimensional Lie algebra, then any symmetric invariant bilinear form is a scalar multiple of $\kappa$.
\end{thm} 

In fact since simple Lie algebras are semisimple both of these theorems are useful when studying simple Lie algebras.  If $\fg$ is thought of as just a vector space, denote the dual vector space of $\fg$ by $\fg^*$.  Since on a simple Lie algebra, $\fg$, $\kappa$ is non-degenerate, then $\fg$ is naturally isomorphic as a vector space to $\fg^*$.
  
\begin{dfn} \mylabel{Def: 1.7} A Lie subalgebra $\fh$ of a simple Lie algebra $\fg$ which satisfies the following properties is called a \bold{Cartan subalgebra} (CSA):
\begin{enumerate} \itemsep-5pt \parskip0pt \parsep0pt
\item $\fh$ is a maximal abelian subalgebra of $\fg$.
\item The linear transformation $ad(h)$ is diagonalizable, $\forall \,\, h \in \fh$.
\end{enumerate}
\end{dfn}

Note that Cartan subalgebras are in general not unique.  The next theorem says something about their conjugacy.
\begin{thm} \mylabel{Thm: 1.8} Let $\fg$ is a simple finite dimensional Lie algebra, then Cartan subalgebras exist and are conjugate under the group of inner automorphisms of $\fg$.
\end{thm} 

Hence, the dimension of a Cartan subalgebra is independent of the choice, and is therefore well defined.  This dimension of any Cartan subalgebra of a Lie algebra $\fg$, is called the \bold{rank} of $\fg$. 

Let $\fg$ be a finite dimensional semisimple complex Lie algebra of rank $l$, and fix a Cartan subalgebra $\fh$ of $\fg$.  For any $h \in \fh$, $\fg$ decomposes as a direct sum of eigenspaces of $ad(h)$, since this operator is diagonalizable.  Let $h_1, h_2, \dots, h_l$ be a basis for $\fh$, and since $\fh$ is abelian, $ad(h_i)$ commutes with $ad(h_j)$ for each $i,j \in \{1,2, \dots, l \}$.  Therefore $\fg$ decomposes into the direct sum of simultaneous eigenspaces of the $ad(h_i)$'s.  Let $x$ be a simultaneous eigenvector for each $ad(h_i), 1\leq i \leq l$, that is $ad(h_i)(x) = c_ix$, then for any linear combination $h = \sum_{i=1}^l a_ih_i$  
\begin{eqnarray*}
ad(h)(x) &=& \sum_{i = 1}^l a_i ad(h_i)(x) \\
&=& \left( \sum_{i = 1}^l a_ic_i\right)x \\
&=&\alp(h)x, \mbox{ for some } \alp \in \fh^*.  
\end{eqnarray*}

For $\alp \in \fh^*$ any linear functional, define the subspace $$\fg_{\alp} = \{ x \in \fg \mid [h,x] = \alp(h)x, \forall h \in \fh \}.$$
Note that $\fg_0 = \fh$, and we have the following definition.

\begin{dfn} \mylabel{Def: 1.9}
If $0 \neq \alp \in \fh^*$ satisfies $\fg_{\alp} \neq \{ 0 \}$, then $\alp$ is called a \bold{root} of $\fg$ with respect to $\fh$.  The set of roots of $\fg$ is denoted $\Phi$.  The subspace $\fg_{\alp}$ is called a \bold{root space} and a non-zero element of $\fg_{\alp}$ is called a \bold{root vector}.
\end{dfn}

This gives the following \bold{root space decomposition} of $\fg$ (as a vector space) with respect to $\fh$:
\begin{eqnarray}
\fg = \fh \oplus \bigoplus_{\alp \in \Phi} \fg_{\alp}.
\end{eqnarray}

\begin{prop} \mylabel{Prop: 1.10}For $\alp, \beta \in \Phi \cup \{ 0\}$ the following are true:
\begin{enumerate} \itemsep-5pt \parskip0pt \parsep0pt
\item $[\fg_{\alp}, \fg_{\beta}] \sset \fg_{\alp+\beta}$,
\item $[\fg_{\alp}, \fg_{-\alp}] \sset \fh$,
\item For $x \in \fg_{\alp}$, $y \in \fg_{\beta}$, $\kappa(x,y) = 0$ if $\alp+\beta \neq 0$.
\end{enumerate}
\end{prop}

The Lie algebra $\fsl(2,\C)$ is defined as the $2 \times 2$ complex matrices of trace zero.  The Lie bracket is given by $[A,B] = AB - BA$, and the "standard" basis below satisfies the following brackets, $[h,f] = 2e$, $[h,f] = -2f$, and $[e,f] = h$.  Where
$$\begin{array}{c c c}
h = \left[ \begin{array}{c c}
1 & 0 \\
0 & -1
\end{array} \right], &

e = \left[ \begin{array}{c c}
0 & 1 \\
0 & 0
\end{array} \right],&

f = \left[ \begin{array}{c c}
0 & 0 \\
1 & 0
\end{array} \right] 
\end{array}.
$$

\begin{lem} \mylabel{Lem: 1.11}The following are true.
\begin{enumerate} \itemsep-5pt \parskip0pt \parsep0pt
\item The restriction of $\kappa$ to $\fh$, $\kappa \mid_{\fh \times \fh} $, is non-degenerate.
\item For $\alp \in \Phi$, $\kappa$ gives a non-degenerate pairing between $\fg_{\alp}$ and $\fg_{-\alp}$.
\item If $\alp \in \Phi$, then $-\alp \in \Phi$ and $dim(\fg_{\alp}) = dim(\fg_{-\alp})$.
\end{enumerate}
\end{lem}

The non-degeneracy of $\kappa \mid_{\fh \times \fh}$, allows the identification of $\fh$ with $\fh^*$.  For $\alp \in \Phi$, the corresponding vector $h_{\alp} \in \fh$ is uniquely determined by the conditions $\kappa(h_{\alp}, h) = \alp(h)$, $\forall h \in \fh$.  Using this isomorphism, $\kappa$ induces on $\fh^*$ a bilinear form also denoted $\kappa$ and defined by $\kappa(\alp, \beta) := \kappa(h_{\alp}, h_{\beta}) = \alp(h_{\beta})$, for any $\alp, \beta \in \fh^*$. 

The next lemma contains some common properties for roots and root spaces.  This lemma is presented without a proof, since all of these properties can be found in most introductory books on Lie algebras. 

\begin{lem} \mylabel{Lem: 1.12}
For $\alp, \beta \in \Phi$ we have,
\begin{enumerate} \itemsep-5pt \parskip0pt \parsep0pt
\item If $x \in \fg_{\alp}$ and $y \in \fg_{-\alp}$, then $[x,y] = \kappa(x,y)h_{\alp} $,
\item $dim(\fg_{\alp}) = 1$,
\item $[\fg_{\alp}, \fg_{-\alp}] = \C h_{\alp}$,
\item If $\alp+\beta \in \Phi$, then $[\fg_{\alp}, \fg_{\beta}] = \fg_{\alp+\beta}$,
\item $2\alp \notin \Phi$.
\end{enumerate}
\end{lem}

Since $dim(\fg_{\alp}) = 1$, $\kappa(x,y) \neq 0$ and there exists a rescaling of $x \in \fg_{\alp}$ and $y \in \fg_{-\alp}$ so that the triple $\{x, y, h = [x,y]\}$, is a standard basis for $\fsl(2,\C)$. 

Define the real vector space $\fh_{\R}$ as the $\R$-span of $\{ h_{\alp} | \alp \in \Phi \}$, and also note $\fh^*_{\R}$ is the $\R$-span of $\Phi$.

Denote by $\sigma_{\alp}$ the reflection determined by $\alp$, so $\forall \beta \in E$, $\sigma_{\alp}(\beta) = \beta - \frac{2(\beta,\alp)}{(\alp,\alp)}\alp$.  Then $\sigma_{\alp}(\alp) = - \alp$ and $P_{\alp} = \{ \beta \in \mbox{E} | (\beta, \alp) = 0 \}$ is the hyperplane fixed by $\sigma_{\alp}$.

\begin{dfn} \mylabel{Def: 1.13} A subset $\Phi$ of E, is called a \bold{finite root system} in E if the following axioms are satisfied:
\begin{description} \itemsep-5pt \parskip0pt \parsep0pt
\item{(R1)} \, $\Phi$ is finite, spans E, and does not contain 0.
\item{(R2)} \, If $\alp \in \Phi$, then the only multiples of $\alp \in \Phi$ are $\pm \alp$.
\item{(R3)} \, If $\alp \in \Phi$, then $\sigma_{\alp}$ leaves $\Phi$ invariant. 
\item{(R4)} \, If $\alp, \beta \in \Phi$, then $\herf{\alp}{\beta} = \frac{2(\beta, \alp)}{(\alp,\alp)} \in \Z$. 
\end{description}
\end{dfn}

If $\Phi$ is a set of roots for a semi-simple Lie algebra $\fg$ as in Definition \ref{Def: 1.9}, it can be shown that $\Phi$ forms a root system in $\fh^*_{\R} = E$ with positive definite bilinear form $\kappa$, and $dim(E) = l$.

\begin{dfn} \mylabel{Def: 1.14} A subset $\Delta$ of $\Phi$, is called a \bold{base} if:
\begin{description} \itemsep-5pt \parskip0pt \parsep0pt
\item{(B1)}\, $\Delta$ is a basis of E.
\item{(B2)}\, Each root $\beta$ can be written $\sum_{\alp \in \Delta} c_{\alp} \alp$, with $c_{\alp} \in \Z$ and all $c_{\alp} \geq 0$ or all $c_{\alp} \leq 0$. 
\end{description}
\end{dfn}

Choose a base, $\Delta = \{\alp_1, \alp_2, \dots, \alp_l \}$, for a root system and call this set of roots simple.  Denote the positive roots $\Phi^+ = \{ \beta = \sum_{i=1}^l c_i \alp_i \in \Phi \mid c_i \geq 0 \}$ and the negative roots $\Phi^- = \{ \beta = \sum_{i=1}^l c_i \alp_i \in \Phi \mid c_i \leq 0 \}$.  By Definition \ref{Def: 1.14}, $|\Delta| = dim(E)$.  This number coincides with rank($\fg$) because both $\{\alp | \alp \in \Delta\}$ and $\{ h_{\alp} |  \alp \in \Delta\}$ form bases for $\fh^*_{\R}$ and $\fh_{\R}$, respectively. 

\begin{dfn} \mylabel{Def: 1.15} The integral span of the simple roots, $\sum_{i = 1}^l \Z \alp_i$, is called the \bold{root lattice}, and is denoted $Q_{\Phi}$.  Write $Q_{\Phi}^+ = \left\{ \sum_{i = 1}^l m_i \alp_i \mid 0 \leq m_i \in \Z \right\}$.  A partial order on $\fh_{\R}^*$ defined by $x \leq y$ when $y - x \in Q_{\Phi}^+$. 
\end{dfn}

For any positive root $\alp$, rescale the non-zero vectors $x \in \fg_{\alp}$, $y \in \fg_{-\alp}$, and $h_{\alp} = [x,y]$, associated to this root as follows: $e := x$, $f := \frac{2}{\alp(h_{\alp})}y$, and $h := \frac{2}{\alp(h_{\alp})}h_{\alp}$.  It can be checked that $\{e,f,h \}$ forms a standard basis for $sl(2,\C)$. 

Using the above rescaling of vectors, define the following: $$a_{ij} = \alp_j(h_i) = \frac{2 \alp_j(h_{\alp_i})}{\alp_i(h_{\alp_i})} = \frac{2 \kappa(h_{\alp_i}, h_{\alp_j})}{\kappa(h_{\alp_i}, h_{\alp_i})}, 1\leq i,j \leq l.$$

The following relations are satisfied for $1 \leq i,j \leq l$, $$[h_i, e_j] = a_{ij}e_j, \, \, \,  [h_i, f_j] = -a_{ij}f_j, \, \, \, [e_i, f_j] = \delta_{ij}h_i.$$

%page 8

\begin{dfn} \mylabel{Def: 1.16}The matrix $A := (a_{ij})$ for $i,j \in \{ 1,2, \dots, l\}$, is called a \bold{Cartan matrix} of a simple finite dimensional Lie algebra, $\fg$.  $A$ has the following properties: 
\begin{description} \itemsep-5pt \parskip0pt \parsep0pt
\item{(C1)} \, $a_{ij} \in \Z$. 
\item{(C2)} \, $a_{ii} = 2$.
\item{(C3)} \, $i \neq j \Rightarrow a_{ij} \leq 0$.
\item{(C4)} \, $a_{ij} = 0 \Leftrightarrow a_{ji} = 0$.
\item{(C5)} \, $A = DB$, where $D= (d_{ij})$ is a diagonal matrix, and $B = (b_{ij})$ is a positive definite symmetric matrix.
\item{(C6)} \, $A$ can not be written as a block diagonal matrix with more than one non-zero block, even if you are allowed to reorder rows and columns.
\end{description}
\end{dfn}

In particular, if we set $b_{ij} = \kappa(\alp_i, \alp_j)$, and $d_{ii} = \frac{2}{\kappa(\alp_i, \alp_i)}$, then $A = DB$.  Also, a convention is to rescale the Killing form so that the diagonal entries of $B$ have largest value 2.  For example the first equation below follows the convention,\\

 $\begin{bmatrix} 2 & -2 \\ -1 & 2 \end{bmatrix} = \begin{bmatrix} 2 & 0 \\ 0 & 1 \end{bmatrix}\begin{bmatrix} 1 & -1 \\ -1 & 2 \end{bmatrix}$,\\
 
\ni but this next equation does not. \\
 
 $\begin{bmatrix} 2 & -2 \\ -1 & 2 \end{bmatrix} = \begin{bmatrix} 1 & 0 \\ 0 & 1/2 \end{bmatrix}\begin{bmatrix} 2 & -2 \\ -2 & 4 \end{bmatrix}$. \\
 
Since the matrix $B$ represents the bilinear form with respect to $\Delta$ on $E$, this forces all square lengths of roots in this root system to be less than or equal to 2.  Now, drop the name $\kappa$ and just use $(\cdot, \cdot)$ for this new standardized bilinear form.

Given an $l \times l$ Cartan matrix satisfying the above properties, we can form a directed graph, called the Dynkin diagram, with vertices and multi-lines, as follows:

\begin{description} \itemsep-5pt \parskip0pt \parsep0pt
\item{(D1)} \, For each $1 \leq i \leq l$ there is a vertex $\bullet$ labeled by $i$.
\item{(D2)} \, For $1\leq i \neq j \leq l$ connect two vertices $i$ and $j$ by max$\{ |a_{ij}|,  |a_{ji}| \}$ lines.
\item{(D3)} \, If $i \neq j$ and $|a_{ij}| \geq 2$, put an arrow on the $\{i,j\}$-multi-line pointing from $j$ to $i$.
\end{description}

Using the matrix in the example above, the Dynkin diagram is:

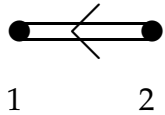
\begin{figure}[h!]
\centering
\begin{picture}(150,50)
\thicklines
\multiput(30,25)(50,0){2}{\circle*{8}}
%\put(10,25){\line(1,0){60}}
\put(30,28){\line(1,0){50}}
\put(30,22){\line(1,0){50}}
\put(50,25){\line(1,1){10}}
\put(50,25){\line(1,-1){10}}
\put(25,-5){$1$}
\put(75,-5){$2$}
\end{picture}
%\caption{\sl Dynkin diagram of $F_4^{(1)}$ with numbering of nodes.}
\end{figure}

Using the Cartan-Killing theorem of the classification of finite dimensional simple Lie algebras over $\C$, simple Lie algebras can be organized into four infinite families and five exceptional Lie algebras.  This classification is given either by Cartan matrices or Dynkin diagrams, which are in one-to-one correspondence.  Below are the Dynkin diagrams associated to the classification.

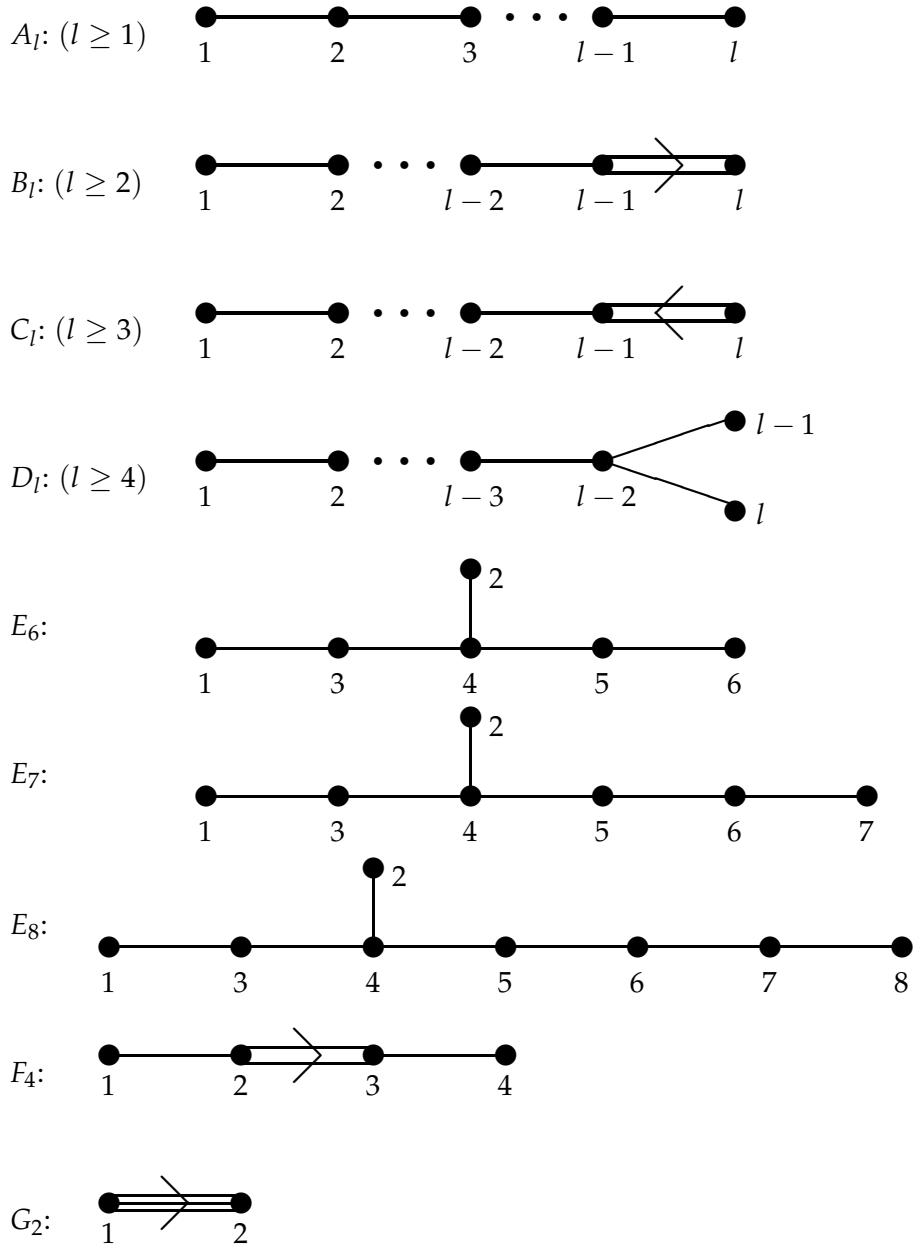
\begin{tabular}{l l}
$A_l$: $(l \geq 1)$ & \begin{picture}(250,50)
\thicklines
%\put(position){object {length or diam}}
\multiput(10,10)(50,0){5}{\circle*{8}}
\multiput(125,10)(10,0){3}{\circle*{3}}
\put(10,10){\line(1,0){100}}
\put(160,10){\line(1,0){50}}
\put(7,-7){$1$}
\put(57,-7){$2$}
\put(107,-7){$3$}
\put(150,-7){$l-1$}
\put(208,-7){$l$}
\end{picture}

 \\
$B_l$: $(l \geq 2)$ & \begin{picture}(250,50)
\thicklines
%\put(position){object {length or diam}}
\multiput(10,10)(50,0){5}{\circle*{8}}
\multiput(75,10)(10,0){3}{\circle*{3}}
\put(10,10){\line(1,0){50}}
\put(110,10){\line(1,0){50}}
\put(160,7){\line(1,0){50}}
\put(160,13){\line(1,0){50}}
\put(180,20){\line(1,-1){10}}
\put(180,0){\line(1,1){10}}
\put(7,-7){$1$}
\put(57,-7){$2$}
\put(100,-7){$l-2$}
\put(150,-7){$l-1$}
\put(210,-7){$l$}
\end{picture}
 \\
$C_l$: $(l \geq 3)$ & \begin{picture}(250,50)
\thicklines
%\put(position){object {length or diam}}
\multiput(10,10)(50,0){5}{\circle*{8}}
\multiput(75,10)(10,0){3}{\circle*{3}}
\put(10,10){\line(1,0){50}}
\put(110,10){\line(1,0){50}}
\put(160,7){\line(1,0){50}}
\put(160,13){\line(1,0){50}}
\put(180,10){\line(1,-1){10}}
\put(180,10){\line(1,1){10}}
\put(7,-7){$1$}
\put(57,-7){$2$}
\put(100,-7){$l-2$}
\put(150,-7){$l-1$}
\put(210,-7){$l$}
\end{picture}
 \\
$D_l$: $(l \geq 4)$ & \begin{picture}(250,50)
\thicklines
%\put(position){object {length or diam}}
\multiput(10,10)(50,0){4}{\circle*{8}}
\multiput(210,-9)(0,34){2}{\circle*{8}}
\multiput(75,10)(10,0){3}{\circle*{3}}
\put(10,10){\line(1,0){50}}
\put(110,10){\line(1,0){50}}
\put(160,10){\line(3,1){50}}
\put(160,10){\line(3,-1){50}}
\put(7,-7){$1$}
\put(57,-7){$2$}
\put(100,-7){$l-3$}
\put(150,-7){$l-2$}
\put(218,20){$l-1$}
\put(218,-14){$l$}
\end{picture}
 \\
$E_6$: & \begin{picture}(250,50)
\thicklines
%\put(position){object {length or diam}}
\put(110,25){\circle*{8}}
\multiput(10,-5)(50,0){5}{\circle*{8}}
\put(10,-5){\line(1,0){200}}
\put(110,-5){\line(0,1){30}}
\put(7,-22){$1$}
\put(117,18){$2$}
\put(57,-22){$3$}
\put(107,-22){$4$}
\put(157,-22){$5$}
\put(207,-22){$6$}
\end{picture}
 \\
$E_7$:  & \begin{picture}(300,50)
\thicklines
%\put(position){object {length or diam}}
\put(110,25){\circle*{8}}
\multiput(10,-5)(50,0){6}{\circle*{8}}
\put(10,-5){\line(1,0){250}}
\put(110,-5){\line(0,1){30}}
\put(7,-22){$1$}
\put(117,18){$2$}
\put(57,-22){$3$}
\put(107,-22){$4$}
\put(157,-22){$5$}
\put(207,-22){$6$}
\put(257,-22){$7$}
\end{picture} 
\end{tabular} 

\begin{tabular}{l l}
$E_8$:  & \begin{picture}(400,50)
\thicklines
%\put(position){object {length or diam}}
\put(110,25){\circle*{8}}
\multiput(10,-5)(50,0){7}{\circle*{8}}
\put(10,-5){\line(1,0){300}}
\put(110,-5){\line(0,1){30}}
\put(7,-22){$1$}
\put(117,18){$2$}
\put(57,-22){$3$}
\put(107,-22){$4$}
\put(157,-22){$5$}
\put(207,-22){$6$}
\put(257,-22){$7$}
\put(307,-22){$8$}
\end{picture} \\
$F_4$:  & \begin{picture}(150,50)
\thicklines
%\put(position){object {length or diam}}
\multiput(10,10)(50,0){4}{\circle*{8}}
\put(10,10){\line(1,0){50}}
\put(110,10){\line(1,0){50}}
\put(60,13){\line(1,0){50}}
\put(60,7){\line(1,0){50}}
\put(80,20){\line(1,-1){10}}
\put(80,0){\line(1,1){10}}
\put(7,-5){$1$}
\put(57,-5){$2$}
\put(107,-5){$3$}
\put(157,-5){$4$}
\end{picture} \\
$G_2$:  & \begin{picture}(150,50)
\thicklines
%\put(position){object {length or diam}}
\multiput(10,10)(50,0){2}{\circle*{8}}
\put(10,10){\line(1,0){50}}
\put(10,13){\line(1,0){50}}
\put(10,7){\line(1,0){50}}
\put(30,20){\line(1,-1){10}}
\put(30,0){\line(1,1){10}}
\put(7,-5){$1$}
\put(57,-5){$2$}
\end{picture} \\
\end{tabular} 
\\

Some of these diagrams have automorphisms which correspond to Lie algebra automorphisms.  We will need to consider a diagram automorphism for type $E_6$.  This diagram has an obvious order two symmetry that is given by interchanging the vertices labeled one and six, three and five, and leaving vertices two and four fixed.  Looking at the fixed points of the corresponding Lie algebra automorphism, $\tau$, will give a Lie subalgebra $\fa$ of type $F_4$ inside $\fg$ of type $E_6$.  The following roots of $\fg$ are fixed: $$\alp_1 + \alp_6, \quad \alp_3 + \alp_5, \quad \alp_4, \quad \alp_2,$$ and give a set of simple roots $\fa$, which we relabel as $$\beta_1 = \alp_2, \quad \beta_2 = \alp_4, \quad \beta_3 = \frac{\alp_3 + \alp_5}{2}, \quad\beta_4 = \frac{\alp_1 + \alp_6}{2} .$$  The Cartan matrix of type $F_4$ can be recovered here by using the Killing form from $E_6$.

Once this identification of simple roots is made, then the following vectors are identified inside $\fg$ to give a set of generators for $\fa$.  Notice the switch in notation from $e_i$ and $f_i$ to $x_{\alp_i}$ and $x_{-\alp_i}$. 
\begin{align*}
h_{\beta_1} &= h_{\alp_2}, &\quad  h_{\beta_2} &= h_{\alp_4}, &\quad h_{\beta_3} &= h_{\alp_3} + h_{\alp_5}, &\quad h_{\beta_4} &= h_{\alp_1} + h_{\alp_6} \\
x_{\beta_1} &= x_{\alp_2}, &\quad  x_{\beta_2} &= x_{\alp_4}, &\quad x_{\beta_3} &= x_{\alp_3} + x_{\alp_5}, &\quad x_{\beta_4} &= x_{\alp_1} + x_{\alp_6} \\
x_{-\beta_1} &= x_{-\alp_2}, &\quad x_{-\alp_2} &= x_{-\alp_4}, &\quad  x_{-\beta_3} &= x_{-\alp_3} + x_{-\alp_5}, &\quad x_{-\beta_4} &= x_{-\alp_1} + x_{-\alp_6}
\end{align*}

Let $\Phi$ be a root system for a finite dimensional semisimple Lie algebra $\fg$.  The \bold{Weyl group}, $\mathcal{W}$, is the group generated by the reflections associated to the simple roots, $\mathcal{W} = \langle \sigma_{\alp}| \alp \in \Delta \rangle$.

Let $V$ be a finite dimensional $\fg$-module such that the CSA $\fh$ acts simultaneously diagonalizably on $V$.  The simultaneous eigenspaces, $V_{\mu} = \left\{ v \in V | h \cdot v = \mu(h)v \right\}$ $\neq 0$ with $\mu \in \fh^*$, are called \bold{weight spaces}, and $\mu$ is called a \bold{weight} of $V$.   For a module $V$, $\Pi(V) = \left\{ \mu \in \fh^* | V_{\mu} \neq 0 \right\}$ denotes the set of weights.  If $0 \neq v \in V$ and $x \cdot v = 0$ for any $x \in \fg_{\alp}$, $\alp \in \Phi^+$, then call $v$ a highest weight vector (HWV) of $V$.  It is enough to check this property for $x \in \fg_{\alp}$, $\alp \in \Delta$.

If $V$ is an finite dimensional irreducible $\fg$-module, then $V$ contains, up to scalar multiples, exactly one HWV, and $V$ is uniquely determined by the weight of this HWV.  If the highest weight is $\lam \in \fh^*$, then label such a highest weight module $V^{\lam}$.  If $\mu \in \Pi(V^{\lam})$, then $\mu \leq \lam$.  $V$ has a \bold{weight space decomposition} as a direct sum of weight spaces, $\ds V^{\lam} = \bigoplus_{\mu \in \Pi(V^{\lam})} V^{\lam}_{\mu}$.  If $\fg$ is simple then $\fg$ is an irreducible $\fg$-module, $V^{\theta}$, whose weights are the roots, $\Phi$, so $\Pi(V^{\theta}) = \Phi$.  We call $\theta = \theta(\fg)$ the highest root of $\Phi$.

Given a finite dimensional irreducible highest weight $\fg$-module, $V^{\lam}$, the \bold{character} of $V^{\lam}$ is given as a specific element of the group ring 
$$\Z[\fh^*] = \left\{ \sum_{\mu \in S} m_{\mu} e^{\mu} | m_{\mu} \in \Z, S \sset \fh^*, |S| < \infty\right\},$$ 
namely, $$ch(V^{\lam}) = \sum_{\mu \in \Pi(V^{\lam})} dim(V^{\lam}_{\mu})e^{\mu}.$$

\begin{dfn} \mylabel{Def: 1.17} The \bold{weight lattice} associated to a root system $\Phi$ is defined as $P_{\Phi} = \left\{ \lam \in \fh^* | \herf{\lam}{\alp_i} \in \Z, \alp_i \in \Delta\right\}. $
\end{dfn}

It can be shown that $P_{\Phi} = \sum_{i = 1}^l \Z \lam_i$, where the $\lam_i$ are defined by $\herf{\lam_i}{\alp_j} = \delta_{i,j}$, $1 \leq i,j \leq l$.  Note that $Q_{\Phi} \sset P_{\Phi}$.  Call the set $P_{\Phi}^+ = \left\{ \lam \in \fh^* | 0 \leq \herf{\lam}{\alp_i} \in \Z, \alp_i \in \Delta \right\} $, the \bold{dominant integral weights}.  For any $\lam \in P_{\Phi}^+$, there exists a finite dimensional irreducible $\fg$-module $V^{\lam}$ with highest weight $\lam$ and set of weights $\Pi(V^{\lam})=$ \\$  \left\{ w(\mu) | \mu \in P_{\Phi}^+, \mu \leq \lam, w \in \mathcal{W} \right\}.$ 

\begin{thm} \mylabel{Thm: 1.18} For any $\lam \in P_{\Phi}^+$ and the associated irreducible finite dimensional highest weight module $V^{\lam}$,  the Weyl character formula gives the following form for the character of $V^{\lam}$, using $\rho = \sum_{i = 1}^l \lam_i$:
$$ch(V^{\lam}) = \frac{S_{\lam + \rho}}{S_{\rho}}, \mbox{  with  } S_{\mu} = \sum_{w \in \mathcal{W}} sgn(w)e^{w\mu - \rho}.$$
\end{thm}

The Weyl denominator formula says:
$$S_{\rho} = \prod_{\alp \in \Phi^+}(1-e^{-\alp})$$

Define the universal enveloping algebra of a Lie algebra, $\fg$, to be the associative algebra $U(\fg)$ (with 1), with a map $\mu:\fg \to U(\fg)$ satisfying $\mu([x,y]) = \mu(x)\mu(y) - \mu(y)\mu(x)$, with the property that if $V$ is any associative algebra (with 1), and $\eta:\fg \to V$ satisfies the same property as $\mu$, then there exists a unique map $\phi: U(\fg) \to V$, such that $\phi \, \circ \, \mu = \eta$.

At this point we fix some notation for the Lie algebras of type $E_6$ and $F_4$.  The fundamental weights for the Lie algebra of type $E_6$ are $\lam_1, \lam_2, \dots,  \lam_6$, corresponding to the roots $\alp_1, \alp_2, \dots, \alp_6 $.  Under our identification of $F_4$ inside of $E_6$ we get the following fundamental weights for $F_4$, $\omg_1 = \lam_2, \omg_2 = \lam_4, \omg_3 = \frac{\lam_3 + \lam_5}{2}$, and $\omg_4 = \frac{\lam_1 + \lam_6}{2}$.

If $\fg$ is of type $E_6$ and $\fa \sset \fg$ is of type $F_4$, then each of the three $\fg$-modules, $V^{\theta(\fg)} = V^{\lam_2}, V^{\lam_1}, V^{\lam_6}$ can be decomposed as a direct sum of irreducible highest weight $\fa$-modules, written as $W^{\omg}$ where $\omg$ is a weight of $\fa$.  First, $V^{\theta(\fg)} = \fg$ and therefore $\fg = U(\fg)\cdot x_{\theta(\fg)}$.  Also, $U(\fa) \cdot x_{\theta(\fg)} = \fa \subset \fg$ and therefore we have the branching $\fg = \fa \oplus \fc$, where $\fa = W^{\theta(\fa)} = W^{\omg_2}$ and $\fc = W^{\omg_4}$.  We know $dim(\fg) = 78$ and $dim(\fa) = 52$, therefore $dim(\fc) = 26$.   Letting $\fh$ and $\fb$ be the CSAs of $\fg$ and $\fa$ respectively, we know $\fb \sset \fh$ and $\fb^* \sset \fh^*$, therefore we have $Proj: \fh^* \to \fb^*$ where $Proj\left(\sum_{i = 1}^6 m_i \lam_i\right) = \sum_{i = 1}^4 n_i \omg_i$ determined by 
$$\lam_1 \mapsto \omg_4, \quad \lam_2 \mapsto \omg_1, \quad \lam_3 \mapsto \omg_3, \quad\lam_4 \mapsto \omg_2, \quad \lam_5 \mapsto \omg_3, \quad \lam_6 \mapsto \omg_4.$$
The projection allows us to describe the $\fa$-module $\fc = U(\fa) \cdot x_{\alp}$, where $x_{\alp} \in \fg$ is such that $Proj(\alp) = \omg_4$.  The branchings of  $V^{\lam_1}$ and $V^{\lam_6}$ are the same since $Proj(\lam_1) = Proj(\lam_6) = \omg_4$, therefore we focus on $V^{\lam_1}$.   A HWV $x_{\lam_1}$ for $V^{\lam_1}$ has weight $\omg_4$ for $\fa$ and therefore $U(\fa) \cdot x_{\lam_1} = \fc$.  Since $dim(V^{\lam_1}) = 27$ and $dim(\fc) = 26$ we must have $V^{\lam_1} = \fc \oplus W^0$, where $W^0$ is the one dimensional trivial $\fa$-module.

% if you want a subsection, but with no number
%\subsection*{Numberless Subsection} 

% you have to manually add it to the table of contents
%\addcontentsline{toc}{subsection}{Numberless Subsection} 
%\label{ss1.1.2}

\section{Affine Lie Algebras}
\mylabel{s 1.2}

This section contains the most important aspects from the theory of affine Lie algebras that are needed for this work.  It starts with the general theory of affine Lie algebras, then introduces the Virasoro algebra which plays a vital role in the decomposition in this dissertation.  Finally, the specifics of the affine Lie algebra $\Es$ are given along with an introduction to the specifics of its vertex operator representation.

\subsection{General Theory of Affine Lie Algebras}
\mylabel{ss 1.2.1}

Let $\fg$ be a finite dimensional Lie algebra with an invariant symmetric bilinear form $\bfor{\cdot}{\cdot}$, so $\bfor{[x,y]}{z} = \bfor{x}{[y,z]}$, for $x,y,z \in \fg$.
Let $\C[t,t^{-1}]$ be the algebra of Laurent polynomials in the variable $t$. Now consider the vector space, $ \ds \fgh = \fg \otimes_{\C} \C[t,t^{-1}] \oplus \C c$,   and the alternating bilinear map $ \ds [\cdot, \cdot] : \fgh \times \fgh \to \fgh$ determined by the following conditions: 
\begin{eqnarray}
&&[c, \fgh] = 0 \nonumber \\
&&[x \otimes t^m, y \otimes t^n] = [x,y] \otimes t^{m+n} + \bfor{x}{y}m \delta_{m, -n} c
\end{eqnarray} 
for all $x,y \in \fg$ and $m,n \in \Z$.   Then $\fgh$ with this Lie bracket is called  the \bold{untwisted affine Lie algebra associated to $\fg$} .  Write $x(n)$ for $x \otimes t^n \in \fgh$. 

Now consider the vector space $\fgt = \fgh \oplus \C d$ and extend the definition of brackets of $\fgh$ to $\fgt$ by also defining, $[d,x(n)] = nx(n)$ and $[d,c] = 0$, so that $\fgh$ is a subalgebra of $\fgt$.  Then $\fgt$ with these brackets is called the \bold{extended affine Lie algebra associated to $\fg$}.  Identify $\fg$ with $\fg \otimes t^0 \sset \fgh \sset \fgt$, hence $\fg$ is a subalgebra of both of these algebras.  

The Lie algebra, $\fgt$, has as a Cartan subalgebra (CSA), $\fht = \fh \oplus \C c \oplus \C d$, and therefore $\fgt$ has a root space decomposition.  Denote the affine root system of $\fgt$, by $\PHI$.  This root system is given by identifying $\Phi \sset \PHI$, then defining $\alp_i(c) = 0 = \alp_i(d)$, for $1 \leq i \leq l$.  Define the elements $\delta$ and $\Lambda_0$ of $\fht^*$ as follows.  If $h \in \fh$, then $\delta(h) = 0$ and  $\Lambda_0(h) = 0$.  Also, $\delta(c) = 0$, $\delta(d) = 1$, $\Lambda_0(c) = 1$, $\Lambda_0(d) = 0$, so $\left\{ \alp_1, \alp_2, \dots, \alp_l, \delta, \Lambda_0 \right\}$ is a basis for $\fht^*$.  If $h \in \fh$, $n\in \Z$ and $x \in \fg_{\alp}$, $\alp \in \Phi$, then we have:
\begin{eqnarray*}
[h, x(n)] = \alp(h) x(n), \quad [c, x(n)] = 0, \quad [d, x(n)] = nx(n). 
\end{eqnarray*}

So for $h \in \fht$, $[h, x(n)] = (n\delta+\alp)(h)x(n)$, hence $n\delta+\alp \in \PHI$.  Now consider $h_1(n)$, $n \neq 0$ and $h_1 \in \fh$, then $[h, h_1(n)] = (n\delta)(h)h_1(n)$, so for $n \neq 0$, $n\delta \in \PHI$.  Hence we have $\PHI = \left\{n\delta + \alp | n\in \Z, \alp \in \Phi \right\} \cup \left\{n\delta | 0 \neq n\in \Z \right\}$.  This gives a root space decomposition of $\fgt$.  Notice that the dimension of the $n\delta + \alp$ root space is one, but the dimension of the $n\delta$ root space is $l$, the rank $\fg$.  The roots of type $n\delta + \alp$ are called real roots and the roots of type $n\delta$ are called imaginary.  Recall, $\theta$ is the highest root in $\Phi$, set $\alp_0 := \delta - \theta$, then there is a base of $\PHI$, given by $\DELH = \left\{ \alp_0, \alp_1, \alp_2, \dots, \alp_l \right\}$.

Extend the non-degenerate invariant bilinear form on $\fg$ to such a form on $\fgt$ as follows:
\begin{eqnarray*}
&&\bfor{x(m)}{y(n)} = \delta_{m,-n} \bfor{x}{y}, \\
&&\bfor{x(m)}{c} = 0 = \bfor{x(m)}{d},\\
&&\bfor{c}{d} = 1, \\
&& \bfor{c}{c} = 0 = \bfor{d}{d}. 
\end{eqnarray*}

The bilinear form on $\fh^*$ also extends to a bilinear form on $\fht^*$ by defining, $\bfor{\fh^*}{c}$, $\bfor{\fh^*}{\delta}$, $\bfor{\Lambda_0}{\Lambda_0}$, and $\bfor{\delta}{\delta}$ all equal to 0, and $\bfor{\Lam_0}{\delta} = 1$.  The partial order on $\fh^*$ also extends to $\fht^*$ by defining $\mu \leq \lambda$ if $\lam - \mu = \sum_{i=0}^l a_i \alp_i$, with $0 \leq a_i \in \Z$.

The affine Lie algebra also has a Weyl group, $\hat{\mathcal{W}}= \, < \sigma_{\alp_i} \mid 0\leq i \leq l>$.  The fundamental weights for the affine Lie algebra are given by $\Lam_0$ and $\Lam_i = m_i\Lam_0 + \lam_i, 1\leq i \leq l$, where the coefficients $m_i$ are determined by $\herf{\Lam_i}{\alp_j} = \delta_{i,j}, 0 \leq i,j \leq l$.  The weight lattice is defined as in Definition \ref{Def: 1.17}, and is given by $\hat{P}_{\Phi} = \sum_{i = 0}^l \Z \Lam_i$ and $\hat{P}_{\Phi}^+ = \left\{ \sum_{i = 0}^l n_i \Lam_i | 0 \leq n_i \in \Z \right\}$. 

For each $\Lam = \sum_{i = 0}^l n_i \Lam_i \in \hat{P}_{\Phi}^+$ there is an irreducible highest weight $\fgt$-module denoted by $V^{\Lam}$.  The \bold{level} of $\Lam$, as well as $V^{\Lam}$, is $\Lam(c) = \sum_{i = 0}^l n_i \Lam_i(c) = n_0 + \sum_{i = 1}^l m_in_i$.  $V^{\Lam}$ has a weight space decomposition similar to the decomposition given for $\fg$-modules.  Hence, 
$$V^{\Lam} = \bigoplus_{\mu \in \Pi(V^{\Lam})} V^{\Lam}_{\mu},$$ 
with $$V^{\Lam}_{\mu} = \left\{ v \in V^{\Lam} | h \cdot v = \mu(h)v, \forall h \in \fht \right\}$$ and $\Pi(V^{\Lam}) = \left\{ \mu \in \hat{P}_{\Phi} | V_{\mu}^{\Lam} \neq 0 \right\}$. Define the vacuum space of $V^{\Lam}$ to be $\V(V^{\Lam}) =$ $\{ v \in V^{\Lam} \mid  h(m) \cdot v = 0, m > 0 \}$.  

There are three level 1 fundamental weights for the Lie algebra of type $\Es$, $\Lam_0, \Lam_1$, and $\Lam_6$.  There are two level 1 fundamental weights for the Lie algebra of type $\Ff$, $\Omg_0$ and $\Omg_4$.

There is a character of the module $V^{\Lam}$.  First define the graded dimension of $V^{\Lam}$ as the formal power series in the ring $Z[[ e^{-\alp_i} \mid 0\leq i \leq l ]]$,
$$gr(V^{\Lam}) = e^{-\Lam} \sum_{\mu \in \Pi(V^{\Lam})} dim(V_{\mu}^{\Lam})e^{\mu} .$$
Since $\Pi(V^{\Lam}) \sset \left\{  \mu \in \hat{P}_{\Phi} | \mu \leq \Lam \right\}$, we define the character $ch(V^{\Lam}) = e^{\Lam}gr(V^{\Lam})$.

\begin{thm} \mylabel{Thm: 1.19}For any $\Lam \in \hat{P}^+_{\Phi}$, the character of the irreducible highest weight module $V^{\Lam}$ is given by the Weyl-Kac character formula.  This formula is given as:
$$ch(V^{\Lam}) = \frac{1}{e^{\hat{\rho}}R}\sum_{w \in \hat{\mathcal{W}}} sgn(w)e^{w(\Lam + \hat{\rho})}$$
with $\hat{\rho} = \sum_{i = 0}^l \Lam_i$ and $R = \prod_{\alp \in \PHI^+} (1-e^{-\alp})^{dim(\fgt_{\alp})}$ is called the denominator.
\end{thm}

\subsection{The Virasoro Algebra}
\mylabel{ss 1.2.2}

The \bold{Witt algebra} is defined to be the infinite dimensional Lie algebra with basis  \\$\left\{ d_m | m \in \Z \right\}$ and brackets given by $[d_m, d_n] = (n-m)d_{m+n}$ for $m,n \in \Z$.  There is a representation of this algebra on the Laurent polynomials, with action of $d_m$ on a polynomial in the variable $t$ given by $t^{m+1}\frac{d}{dt}$.  Using this action on the Laurent polynomials extend to an action on $\fgh$ by $d_m \cdot x(n) = nx(m+n)$.  Hence the $d_0$ equals $d$ given in section \ref{ss 1.2.1}.  Define the Virasoro algebra, $Vir$, as the following central extension of the Witt algebra.  Let a basis for $Vir$ be $\left\{ L_m, c_{Vir} | m \in \Z \right\}$, where $c_{Vir}$ is a central element.  For all $m,n \in \Z$ the brackets 
$$[L_m, L_n] = (m-n)L_{m+n} + \frac{m^3-m}{12}\delta_{m,-n}c_{Vir} \quad \mbox{  and  } \quad [L_m, c_{Vir}] = 0.$$ 

\begin{thm} \mylabel{Thm: 1.20}
For each $(c, h) \in \C^2$ there exists an irreducible $Vir$-module, denoted by $Vir(c, h)$, such that $c_{Vir}$ acts as the scalar value $c$, called the central charge, and such that there is a highest weight vector $v$ satisfying $L_m \cdot v = 0$ for $m > 0$ and $L_0 \cdot v = hv$. For certain 
values of $c$ and $h$, $Vir(c, h)$ admits a positive definite Hermitian form such that $(L_m \cdot v_1, v_2) = (v_1, L_{-m}\cdot v_2)$ for any $v_1,v_2 \in Vir(c, h)$, and in these cases the module is called unitary.
\end{thm}

For $0 < c_{Vir} < 1$, there is a discrete series of minimal models for which the characters have special behavior.  For each $c_{Vir}$ in the discrete series there are only finitely many $h$ values parameterized as follows.  For $2 \leq s,t \in \Z$, $s,t$ relatively prime, and $1 \leq m < s$ and $1 \leq n < t$, set
$$c_{s,t} = 1 - \frac{6(s-t)^2}{st}, \quad h_{s,t}^{m,n} = \frac{(mt-ns)^2 - (s-t)^2}{4st}.$$
Given the four parameters as above, $m,n,s,t$, the characters $\chi_{s,t}^{m,n}(q)$ for the modules \\ $Vir(c_{s,t},h_{s,t}^{m,n})$ were given in Fe{\u\i}gin-Fuchs\cite{FF} to be
$$\chi_{s,t}^{m,n}(q) = \frac{q^{h_{s,t}^{m,n} - c_{s,t}/24}}{\vph(q)} \sum_{k\in \Z} q^{stk^2}\left( q^{k(mt-ns)} - q^{(mt+ns)k + mn}\right).$$

These Virasoro modules are unitary when $t = s+1$ so $c_{Vir} = 1 - \frac{6}{s( s+1)}$ for $3 \leq s \in \Z$, and $h = h^{m,n}_s = \frac{(( s +1)m - sn)^2 - 1}{4s(s+1)}$ , where $1 \leq m \leq n < s + 1$.  Later we will only need the case when $(s,t) = (5,6)$ and $c_{5,6} = \frac{4}{5}$.  

Let $V^{\Lam}$ be an irreducible highest weight $\fgh$-module.  Let $\{ u_i | 1 \leq i \leq dim(\fg)\}$ be a basis of $\fg$, and let $\{ u^i | 1 \leq i \leq dim(\fg) \}$ be the dual basis with respect to the bilinear form on $\fg$.  Let $k^{\vee}(\fg)$ be the dual Coxeter number of $\fg$.  The following Sugawara operators give a representation of $Vir$ on $V^{\Lam}$ denoted $Vir_{\fg}(V^{\Lam})$:
\begin{eqnarray}
L_m = \frac{1}{2(k^{\vee}(\fg) + \Lam(c))} \sum_{n \in \Z} \sum_{i = 1}^{dim(\fg)}\mbox{:}u_i(-n)u^i(m+n)\mbox{:} \quad \mbox{ for } m \in \Z
\end{eqnarray}
with $:$ $:$ indicating the bosonic normal ordering of the operators.  The bosonic normal ordering of two operators is given by:

$$\mbox{:}u_i(-n)u^i(m+n)\mbox{:} = \left\{ \begin{array}{r l}  
    u_i(-n)u^i(m+n) & \mbox{ if } -n \leq m+n \\
    u^i(m+n)u_i(-n) & \mbox{ if } m+n < -n 
    \end{array} \right.$$ 
        
\ni and $c_{Vir} = c = \frac{dim(\fg)\Lam(c)}{k^{\vee}(\fg) + \Lam(c)}$.  Note that $L_0$ represents $-d$ on $V^{\Lam}$.

The next theorem plays an important role in the decomposition of the Lie algebra $\Es$ with respect to $\Ff$.  
\begin{thm} \mylabel{thm: coset}Let $\fp$ be a simple Lie subalgebra of a finite dimensional simple Lie algbera $\fg$, and let $V^{\Lam}$ be an irreducible representation of $\fgt$ with $Vir_{\fp}(V^{\Lam})$ and $Vir_{\fg}(V^{\Lam})$ two representations of the Virasoro algebra on $V^{\Lam}$ which are provided by the Sugawara operators for $\fpt$ and $\fgt$ respectively.  Denote the operators of these representations by $\{ L_n^{\fp}, c^{\fp} \mid n\in \Z\}$  and $\{ L_n^{\fg}, c^{\fg} \mid n\in\Z\}$.  Then the differences $L_n^{\fg}-L_n^{\fp}$ provide a representation of the Virasoro algebra, $Vir_{\fg-\fp}(V^{\Lam})$, with central charge, $c = c^{\fg} - c^{\fp}$.  Furthermore, this representation commutes with $\fpt$ and with $Vir_{\fp}$, meaning $[L_n^{\fg}-L_n^{\fp}, \fpt] = 0$ and $[L_n^{\fg}-L_n^{\fp},L_m^{\fp}] = 0$, $\forall m,n \in \Z$.
\end{thm}

\subsection{The Affine Lie Algebra $\Es$}
\mylabel{ss 1.2.3}
The purpose of this section is to construct the adjoint representation of the affine Lie algebra $\fgt$ of type $\Es$ from the root lattice $Q = Q_{\Phi}$ of the finite dimensional simple Lie algebra $\fg$ of type $E_6$.  We discuss the construction of $\fg$ using a 2-cocyle on its root lattice which works for any lattice of type A,D,E.  We begin with a more general situation which will lead to the lattice construction of vertex operators.

A \bold{lattice of rank $r$} is a rank $r$ free abelian group $L$, with a rational valued symmetric $\Z$-bilinear form, denoted $\herfc$.  Call a lattice non-degenerate if the form is non-degenerate. Assume that $L$ is positive definite and even, that is the bilinear form is positive definite and $\herf{x}{x} \in 2\Z$, for $x \in L$, which implies that $\herf{x}{y} \in \Z$, for $x,y \in L$.

The construction of $\fg$ begins with its CSA defined to be $\fh = L \otimes_{\Z} \C$ with trivial Lie brackets making it an Abelian Lie algebra.  Identify $L$ with $L \otimes 1 \subset \fh$ and extend $\herfc$ linearly from $L$ to $\fh$.  

Consider the alternating $\Z$-bilinear map:
$$c_0 : L \times L \to \Z/2\Z$$
$$(\alp,\beta) \mapsto \, \herf{\alp}{\beta} + \, \,  2\Z.$$ 

Let $\hat{L}$ be the central extension of $L$ by the the cyclic group $<\kappa> = \{ \pm1\}$, 
$$ 1 \rightarrow \{ \pm1\} \rightarrow \hat{L} \bar{\rightarrow} L \rightarrow 0$$
determined by $c_0$ as follows.  For $a,b \in \hat{L}$ the commutator is
$$aba^{-1}b^{-1} = (-1)^{c_0(\bar{a},\bar{b})} = (-1)^{\veps_0(\bar{a},\bar{b})-\veps_0(\bar{b},\bar{a})}$$

\ni where $\, \,  \bar{}\, \,$  is the projection from $\hat{L}$ to $L$, and $\veps_{0}: L \times L \to \Z/2\Z$ denotes a corresponding bilinear 2-cocycle.

Choose the following section, $e:L \to \hat{L}$, such that $\alp \mapsto e_{\alp}$ and additionally $e_0 = 1$.
Then we have the following properties:
$$e_{\alp}e_{\beta} = e_{\alp + \beta}(-1)^{\veps_0(\alp,\beta)}$$ 
$$\veps_0(\alp,\beta) + \veps_0(\alp + \beta, \gamma) = \veps_0(\beta, \gamma) + \veps_0(\alp, \beta + \gamma)$$
$$\veps_0(\alp, \beta) - \veps_0(\beta, \alp) = \, \, \herf{\alp}{\beta} + \, \, 2\Z = c_0(\alp,\beta)$$
and because of the extra condition on the section,
$$\veps_0(\alp,0) = \veps_0(0,\alp) = 0$$
and since $\veps_0$ is bilinear,
$$\veps_0(\alp,\beta) = \veps_0(-\alp,\beta) = \veps_0(\alp,-\beta) = \veps_0(-\alp,-\beta).$$ 

To change the $\veps_0$ notation to a multiplicative view, define $\veps(\alp,\beta) = (-1)^{\veps_0(\alp,\beta)}$.  This gives the properties above as:
$$e_{\alp}e_{\beta} = e_{\alp + \beta}\veps(\alp,\beta)$$ 
$$\veps(\alp,\beta)\veps(\alp + \beta, \gamma) = \veps(\beta, \gamma)\veps(\alp, \beta + \gamma)$$
$$\veps(\alp, \beta)/\veps(\beta, \alp) = (-1)^{\herf{\alp}{\beta}}$$
$$\veps(\alp,0) = \veps(0,\alp) = 1$$ 
Denote $L_2 =\{ \alp \in L \mid \herf{\alp}{\alp} = 2\}$ and consider the finite dimensional vector space
$$\fg = \fh \oplus \coprod_{\alp \in L_2} \C x_{\alp}$$ 
where $\{x_{\alp} \mid \alp \in L_2\}$ is linearly independent.  Define a Lie algebra bracket on $\fg$ by:
$$[\fh,\fh] = 0$$
$$[h, x_{\alp}] = -[x_{\alp}, h] = \,\,\herf{h}{\alp}x_{\alp} \quad \mbox{ for } \quad h \in \fh, \alp \in L_2.$$
and for $\alp, \beta \in L_2$ we have
\[ [x_{\alp}, x_\beta] = \left\{ \begin{array}{lcll }
		\veps(\alp,-\alp)\alp & \mbox{if} & \alp + \beta = 0 & (\herf{\alp}{\beta} = -2) \\
		\veps(\alp,\beta)x_{\alp + \beta} & \mbox{if} & \alp + \beta \in L_2 & (\herf{\alp}{\beta} = -1)  \\
		0 & \mbox{if} & \alp + \beta \notin L_2 \cup \{ 0\} & (\herf{\alp}{\beta} \geq 0).  \end{array} \right. \] 
		
Also, we define a non-degenerate invariant symmetric bilinear form $\herf{\cdot}{\cdot}$ on $\fg$ extending $\herfc$ on $\fh$ as follows:
$$\herf{\fh}{x_{\alp}} \, = \, \herf{x_{\alp}}{\fh} \, = \, 0, \mbox{ for } \alp \in L_2$$
\[ \herf{x_{\alp}}{x_{\beta}} = \left\{ \begin{array}{lcl}
				\veps(\alp,-\alp) & \mbox{if} & \alp + \beta = 0 \\
				0 & \mbox{if} & \alp + \beta \neq 0.	 \end{array} \right.		
\] 

The diagram automorphism $\tau$ as a permutation of the simple roots determines an automorphism of $\fh^*$ by linear extension, and through the identification of $\fh^*$ with $\fh$ we have $\tau(h_{\alp}) = h_{\tau\alp}$.  We further extend $\tau$ from $\fh$ to $\fg$ by defining $\tau(x_{\alp}) = x_{\tau\alp}$ for $\alp \in L_2$, and choosing $\veps$ to be $\tau$ invariant in the sense that $\veps(\alp,\beta) = \veps(\tau\alp,\tau\beta)$ gives that this extension is a Lie algebra automorphism.  We have

\[ \tau([x_{\alp}, x_\beta]) = \left\{ \begin{array}{lcll }
		\veps(\alp,-\alp)\tau\alp & \mbox{if} & \alp + \beta = 0 & (\herf{\alp}{\beta} = -2) \\
		\veps(\alp,\beta)x_{\tau\alp + \tau\beta} & \mbox{if} & \alp + \beta \in L_2 & (\herf{\alp}{\beta} = -1)  \\
		0 & \mbox{if} & \alp + \beta \notin L_2 \cup \{ 0\} & (\herf{\alp}{\beta} \geq 0)  \end{array} \right. \]
and this is clearly equal to $[x_{\tau\alp}, x_{\tau\beta}]$.

One can check that the generators for the Lie subalgebra $\fa$ of type $F_4$, given in Section 1.1.1, are fixed under $\tau$.

\begin{thm} \mylabel{thm: 1.21} The nonassociative algebra $\fg$ above is a Lie algebra with root system $\Phi = L_2$, and $\herfc$ on $\fg$ is a non-degenerate invariant symmetric bilinear form.  Also, $\tau$ is an automorphism of $\fg$ and $\fa$ is the Lie subalgebra of fixed points under $\tau$.
\end{thm}

The construction above gives an explicit realization of the Lie algebra of type $E_6$ when a lattice of rank 6 is used and the positive definite bilinear form is determined  by the Cartan matrix of type $E_6$.  In fact, for any even root lattice of type $A_n$, $D_n$, $E_n$, the above construction of $\fg$ follows from the Frenkel-Kac \cite{FK} construction of the basic representation of the affine Lie algebra $\fgt$ of type $A_n^{(1)}$, $D_n^{(1)}$, $E_n^{(1)}$.  

We now focus on the affine Lie algebra $\fgt$ of type $\Es$ constructed as above from $\fg$, expressing its brackets in terms of formal variables. The algebra $\fgt$ contains the following subalgebras: 
$$\fht = \coprod_{n \in \Z} \fh \otimes t^n \oplus \C c \oplus \C d$$ 
and the derived algebra of $\fht$ is the following Heisenberg subalgebra
$$\fht'= \fhhz = \coprod_{\substack{ n \in \Z  \\ n \neq 0}} \fh \otimes t^n \oplus \C c.$$ 

The next step is to express the bracket relations of the Lie algebra, $\fgt$ using the technique of formal variables.  Given any Lie algebra $\fg$, we have the bilinear map
$$\fg[[z,z^{-1}]] \times \fg[[w,w^{-1}]] \to \fg[[z,z^{-1},w,w^{-1}]]$$
$$(x(z),y(w)) \mapsto [x(z),y(w)]$$
defined by
$$\left[\,  \sum_{m\in \Z} x_mz^m, \sum_{n \in \Z} y_n w^n\right] = \sum_{m,n \in \Z} [x_m,y_n]z^m w^n.$$ 
For $x \in \fg$, set 
$$x(z) = \sum_{n \in \Z} x(n)z^{-n} \in \tilde{\fg}[[z,z^{-1}]].$$
We fix some notation by writing
$$\delta(z) = \sum_{n \in \Z} z^n, \quad (D\delta)(z) = \sum_{n \in \Z} nz^n, \quad Dx(z) = \sum_{n\in \Z} -nx(n)z^{-n}.$$
We see that formally $D$ is the application of $z\frac{d}{dz}$. 

\begin{lem} \mylabel{lem: 1.22}
For $x,y \in \fg$, the brackets in $\fgt$ are expressed as
$$[x(z),y(w)] = [x,y](w)\delta(z/w) \, - \herf{x}{y}(D\delta)(z/w)c,$$
$$[c,x(z)] = 0, \quad [d,x(z)] = -Dx(z), \quad [c,d] = 0.$$ 
\end{lem}

We extend the automorphism $\tau$ of $\fg$ to an automorphism $\tilde{\tau}$ on the Lie algebra $\fgt$ by defining for any $x \in \fg$, $n\in \Z$,
$$\tilde{\tau}(x(n)) = (\tau x)(n), \quad \tilde{\tau}(c) = c, \quad \tilde{\tau}(d) = d.$$
If $x,y \in \fg$ and $m,n \in \Z$, then 
\begin{eqnarray*}
\tilde{\tau}[x(m), y(n)] &=& (\tau[x,y])(m+n) - \herf{\tau x}{\tau y}c \\
&=& [\tau x, \tau y](m+n) - \herf{\tau x}{\tau y}c
\end{eqnarray*}
which is readily seen to be $[\tilde{\tau}(x(m)), \tilde{\tau}(y(n))] = [(\tau x)(m), (\tau y)(n)]$.  The fixed points of $\tilde{\tau}$ give the Lie algebra $\fat$ of type $\Ff$ as a Lie subalgebra of $\fgt$ of type $\Es$.

Let us now express the the bracket relations of $\fgt$ in terms of the 2-cocycle, $\veps$, and the section chosen earlier. 
$$[ \, h(m), x_{\alp}(z) \,] = \herf{h}{\alp}z^mx_{\alp}(z),$$
\[ [x_{\alp}(z), x_\beta(w)] = \left\{ \begin{array}{lcl}
		\veps(\alp,-\alp)(\alp(w)\delta(z/w) - (D\delta)(z/w)c) &\mbox{if} & \alp + \beta = 0 \\
		\veps(\alp,\beta)x_{\alp + \beta}(w)\delta(z/w) & \mbox{if} & \alp + \beta \in L_2 \\
		0 & \mbox{if} & \alp + \beta \notin L_2 \cup \{ 0\}, \end{array} \right. \]
$$[d,x_{\alp}(z)] = -Dx_{\alp}(z),$$
$$[c,x_{\alp}(z)] = 0.$$ 

%\begin{thm} \mylabel{thm: 1.23} %Choosing $\herfc$ to be the Cartan matrix of type $E_6$, the construction above proves the existence of an affine Lie algebra of type $\Es$.
%The construction above gives a Lie algebra of type $\Es$ 
%\end{thm}

\subsection{The Vertex Operators Representing $\Es$} \mylabel{ss 1.2.4}

Let $\fht$ and $\fhhz$ be as in Section \ref{ss 1.2.3}, and also define 
$$\ds \fhhz^+ = \coprod_{n > 0} \fh \otimes t^n, \quad  \ds \fhhz^- = \coprod_{n < 0} \fh \otimes t^n \quad \mbox{ and } \quad \fb = \fhhz^+ \oplus \C c.$$ 
 Note that $\fb$ is a maximal abelian subalgebra of $\fhhz$.  
Denote by $\C_1$ the one dimensional space $\C$, with basis $\bo$ viewed as a $\fb$-module by: $\ds c\cdot \bo = \bo$ and  $\fhhz^+\cdot\bo = 0$.  These actions define the $\fhhz$-irreducible $\fht$-module,  $M(1) := U(\fhhz)\otimes_{U(\fb)}\C_1$, where $c$ acts as the identity, $d$ is determined by its brackets with $\fhhz$ and $d \cdot 1 \otimes \bo = 0$, and $\fh$ acts as 0.  $U(\fhhz)$ and $U(\fb)$ are the universal enveloping algebra of $\fhhz$ and $\fb$ defined in Section \ref{s 1.1}.  The Poincare-Birkoff-Witt theorem gives us the linear isomorphism $M(1) \cong S(\fhhz^-)$, where $S(\fhhz^-)$ is the symmetric algebra with commuting generators from $\fhhz^-$.   For $v = h_1(-n_1)h_2(-n_2) \dots h_k(-n_k)1 \in S(\fhhz^-)$ define  
$$wt(v)= n_1+ n_2 + \dots + n_k$$
so that, $d \cdot v = -wt(v)v$.  Then $S(\fhhz^-)$ is graded according to weight.

Because $\fh$ can be identified with $\fh^*$, given by $h_{\alp} \leftrightarrow \alp$, the operator $h_{\alp}(n)$ on $S(\fhhz^-)$ will be written as $\alp(n)$.  Since $\fh$ is Abelian and $c$ acts as 1, for $\alp, \beta \in \fh$ and $m,n \in \Z$, the operators acting on $S(\fhhz^-)$ satisfy the following
$$[\alp(m), \beta(n)] = \herf{\alp}{\beta} m \delta_{m,n} \quad \mbox{ and } \quad [d,\alp(m)] = m \alp(m).$$

Recall the definitions and conventions used in Section \ref{ss 1.2.3} for $\veps_0$, $\veps$ and the choice of section for the extension $\hat{L}$ of $L$, where $L$ is the root lattice of type $E_6$.  The property $e_{\alp}e_{\beta} = e_{\alp + \beta}\veps(\alp,\beta)$ gives motivation for the following action of $\hat{L}$ on $\C[L]$, the group algebra on $L$.  For $\alp, \beta \in L$, define the action as
$$e_{\alp} \cdot e^{\beta} = \veps(\alp, \beta) e^{\alp + \beta} \quad \mbox{  and  } \quad (-1) \cdot e^{\beta} = -e^{\beta}.$$

Set $V_L = S(\fhhz^-) \otimes \C[L]$, and regard $S(\fhhz^-)$ as a trivial $\hat{L}$-module and $\C[L]$ as a trivial $\fhhz$-module.  For $h \in \fh$ define the operator $h(0)$ on $\C[L]$, by $h(0) \cdot e^{\alp} = \herf{h}{\alp}e^{\alp}$.  For $h \in \fh$ define $z^{h(0)} \in (End\C[L])[[z, z^{-1}]]$ by $z^{h(0)} \cdot e^{\alp} = z^{\herf{h}{\alp}}e^{\alp}$.  As operators on $\C[L]$, for each $h,h_i \in \fh$ and $\alp \in L$, the following operator equalities are true:
$$[h(0), e_{\alp}] = \herf{h}{\alp}e_{\alp},$$
$$[h_1(0), z^{h_2(0)}] = 0,$$
$$[h_1(0), e_{\alp}z^{h_2(0)}] = \herf{h_1}{\alp}e_{\alp}z^{h_2(0)},$$
$$z^{h(0)}e_{\alp} = z^{\herf{h}{\alp}}e_{\alp}z^{h(0)} = e_{\alp}z^{h(0) + \herf{h}{\alp}}.$$  

\ni $\hat{L}$ has a grading given by: 
$wt(e^{\alp}) = \frac{1}{2}\herf{\alp}{\alp}$, and define the degree operator on $\C[L]$ by $d(e^{\alp}) = deg(e^{\alp})e^{\alp} = -\frac{1}{2}\herf{\alp}{\alp} e^{\alp}$.

For reference the actions of $\fht, \hat{L}$, and $z^{h(0)}$ (for $h \in \fh$) on $V_L$ are recorded as: 
\begin{align*}
c &\mapsto I_{V_L} &\\
d &\mapsto d = d \otimes 1 + 1 \otimes d  & \\
h = h \otimes t^0 &\mapsto h(0) = 1 \otimes h(0) &\mbox{ for } h \in \fh \\
h \otimes t^n &\mapsto h(n) = h(n) \otimes 1 &\mbox{ for } h \in \fh, n \in \Z\bs\{ 0\} \\
e_{\alp} &\mapsto 1 \otimes e_{\alp} &\mbox{ for } \alp \in L \\
z^{h(0)} &\mapsto 1 \otimes z^{h(0)} &\mbox{ for } h \in \fh
\end{align*}  

Before continuing this construction, as in the work of Dong-Lepowsky (\cite{DL}) to extend $\tau$ to $\hat{\tau}$ on $V_L$.  For $v = \alp_{i_1}(-n_{i_1}) 1\cdots\alp_{i_k}(-n_{i_k}) 1\otimes e^{\beta} \in V_L$, define 
$$\hat{\tau}(v) = (\tau\alp_{i_1})(-n_{i_1}) \cdots (\tau\alp_{i_k})(-n_{i_k}) 1\otimes e^{\tau\beta}.$$ 
If $x \in \fgt$, $v \in V_L$, then $\hat{\tau}(x \cdot v) = \tilde{\tau}(x) \cdot \hat{\tau}(v)$.  For $\alp, \beta \in L$ and $v$ as above,$$\hat{\tau}(e_{\alp} \cdot v) = e_{\tau\alp} \cdot \hat{\tau}(v)$$ since $\eps$ is $\tau$ invariant
and $$\hat{\tau}(z^{\alp(0)} \cdot v) = z^{(\tau\alp)(0)}\cdot \hat{\tau}(v).$$ 

For $\alp \in \fh$, define:
$$\ds E^{\pm}(\alp, z) = \exp \left(\sum_{n\in \pm \Z_+} \frac{\alp(n)}{n}z^{-n}\right) \in (End \, S(\fhhz^-))[[z^{\mp1}]].$$
\begin{lem} \mylabel{lemma 1.24}
The following properties for hold for $\alp, \beta \in \fh$:
\begin{eqnarray}
&E^{\pm}(0, z) = I_{S(\fhhz^-)} &\\
&E^{\pm}(\alp + \beta, z) = E^{\pm}(\alp, z)E^{\pm}(\beta, z) &\\
&[d,E^{\pm}(\alp, z)] = -DE^{\pm}(\alp, z) = \left( \sum_{n\in \pm \Z_+} \alp(n)z^{-n} \right)E^{\pm}(\alp, z) &\\
&[\beta(m), E^{+}(\alp, z)] = 0 &\mbox{ if } m \in \N \\
&[\beta(m), E^{-}(\alp, z)] = 0 & \mbox{ if } m \in -\N \\
&[\beta(m), E^{-}(\alp, z)] = -\herf{\beta}{\alp}z^mE^{-}(\alp, z) & \mbox{ if } m \in \Z_+\\
&[\beta(m), E^{+}(\alp, z)] = -\herf{\beta}{\alp}z^mE^{+}(\alp, z) & \mbox{ if } m \in -\Z_+ 
\end{eqnarray} 
\end{lem}

\begin{dfn} For $\alp \in L$ define an \bold{(untwisted) vertex operator} to be the following formal series  
$$\yvoz{\alp} = \yvozr{\alp} = \ywtvomn{\Z}{\alp}{n} \in (\mbox{End } V_L)[[z,z^{-1}]].$$ 
\end{dfn}

\begin{lem}The following equalities hold as operators in $ (\mbox{End } V_L)[[z,z^{-1}]]$.  If $h \in \fh$, $m, n \in \Z$, and $\alp \in L$, then:
\begin{eqnarray}
&[h(m), \yvoz{\alp}] = \herf{h}{\alp}z^m \yvoz{\alp}\\
&[h(m), \yvomn{\alp}{n}] =  \herf{h}{\alp} \yvomn{\alp}{m+n}\\
&\mbox{ deg } \yvomn{\alp}{n} = n. 
\end{eqnarray} 
\end{lem}

Also note, $\yvoz{0} = I_{V_{L}}$.  From (1.3), for $\alp = 0$, $z^{\alp(0)} = I_{\C[L]}$ and $e_0 = I_{\C[L]}$.  Hence, $\yvomn{0}{n} = \delta_{n,0}I_{V_L}$ for $n \in \Z$. 

Under the identification of $\fh$ with $\fh^*$ the Heisenberg operators $h_{\alp}(n)$ are written $\alp(n)$.

\begin{dfn} \mylabel{def 1.27} For $\alp_1,\alp_2 \in \fh$ and $n_1,n_2 \in \Z$, define the \bold{normal ordered product of two Heisenberg operators} as:
$$ \mbox{:}\alp_1(n_1)\alp_2(n_2)\mbox{:} = \left\{ \begin{array}{lcl}
		\alp_1(n_1)\alp_2(n_2) & \mbox{if} & n_1 \leq n_2 \\
		\alp_2(n_2)\alp_1(n_1) & \mbox{if} & n_1 > n_2. \end{array} \right. $$ 
\end{dfn}
\begin{dfn} \mylabel{def 1.28} Given $\alp_i \in \fh$, $n_i \in \Z$, $1\leq i \leq k$,  extend Definition \ref{def 1.27} to the \bold{normal ordered product of $k$ Heisenberg operators} as: $$ \mbox{:}\alp_1(n_1)\alp_2(n_2) \cdots \alp_k(n_k)\mbox{:} = \alp_{\pi(1)}(n_{\pi(1)})\alp_{\pi(2)}(n_{\pi(2)}) \cdots \alp_{\pi(k)}(n_{\pi(k)})$$
where $\pi$ is a permutation of $\{ 1,2, \dots, k\}$ so that $n_{\pi(1)} \leq n_{\pi(2)} \leq \cdots \leq n_{\pi(k)}$.
\end{dfn}

We further extend the definition of normal ordered product for more operators. 
%To ensure the equality of $\yvoz{\alp}$ and :$\yvoz{\alp}$:, define the following normal orders.

\begin{dfn}\mylabel{dfn 1.29} Define for $\alp, \beta \in \fh$,
\begin{align*}
&\mbox{:}z^{\alp(0)}e_{\alp}\mbox{:} = \mbox{:}e_\alp z^{\alp(0)}\mbox{:} = e_{\alp} z^{\alp(0)}, \\
&\mbox{:}\yvoz{\alp} \yvow{\beta}\mbox{:} =  E^-(-\alp, z)E^-(-\beta, w)E^{+}(-\alp, z)E^{+}(-\beta, w)e_{\alp}e_{\beta}z^{\alp(0)}w^{\beta(0)}.
\end{align*}
\end{dfn}
%\mbox{:}E^-(\alp, z)E^{+}(\alp, z)e_{\alp}z^{\alp(0)}\veps_{\alp}E^-(\beta, w)E^{+}(\beta, w)e_{\beta}w^{\beta(0)}\veps_{\beta}\mbox{:} \\
%&\, \, \, \, \, \, \,=

%For $\alp \in \fh$ set:

%$$\alp(z) = \sum_{n \in \Z} \alp(n) z^{-n}$$
%$$\alp(z)^{\pm} = \frac{1}{2}\alp(0) + \sum_{n \in \pm\Z} \alp(n) z^{-n}$$
%in (End $V$)$\{z\}$, so that $\alp(z) = \alp(z)^+ + \alp(z)^-$ \\

The theorem and lemmas below give the brackets showing that these vertex operators represent the affine Lie algebra. \\

\begin{thm} \mylabel{thm 1.30}
For $\alp, \beta \in \Phi_{E_6} = L_2$, since $\veps(\alp,\beta)/\veps(\beta,\alp) = (-1)^{\herf{\alp}{\beta}}$, we have

\[ \left[ \yvoz{\alp}, \yvow{\beta}\right] = \left\{ \begin{array}{lcl}
		0 & \mbox{if} & \herf{\alp}{\beta} \geq 0 \\
		\veps(\alp,\beta) \yvow{\alp + \beta}\delta(z/w)& \mbox{if} & \herf{\alp}{\beta} = -1\\ 
		\veps(\alp,\beta) (\alp(w)\delta(z/w) - (D\delta)(z/w)) & \mbox{if} & \herf{\alp}{\beta} = -2.  \end{array} \right. \] 
\end{thm}	

\begin{lem}\mylabel{lem 1.31} For $\alp, \beta \in \fh$, 
$$E^+(\alp,z)E^-(\beta,w) = E^-(\beta,w)E^+(\alp,z)\left(1- \frac{w}{z}\right)^{\herf{\alp}{\beta}} \in (End V_{L})[[z^{-1},w]].$$ 
In particular, $E^+(\alp,z)E^-(\beta,w) = E^-(\beta,w)E^+(\alp,z)$ if $\herf{\alp}{\beta} = 0$.
\end{lem}

\begin{lem} \mylabel{lem 1.32} For $\alp,\beta \in \fh$, 
\begin{align*}\mbox{:}\yvoz{\alp} \yvow{\beta}\mbox{:} &= \veps(\alp,\beta)/\veps(\beta,\alp)\mbox{:}\yvow{\beta} \yvoz{\alp}\mbox{:} \\
\yvoz{\alp} \yvow{\beta} &= \mbox{:}\yvoz{\alp} \yvow{\beta}\mbox{:}z^{\herf{\alp}{\beta}}\left(1 - \frac{w}{z}\right)^{\herf{\alp}{\beta}}.
\end{align*}
\end{lem}

\begin{lem} \mylabel{lem 1.33} Given $f(z,w)$ a function of $z$ and $w$ and if $D_z = z\frac{d}{dz}$, then the following equalities are true:
$$f(z,w)(D \delta)(w/z) = -f(w,w)(D\delta)(z/w) + (D_zf)(w,w)\delta(z/w)$$
and
$$f(z,w)(\delta)(z/w) = f(w,w)(\delta)(z/w).$$
\end{lem}

\begin{thm} \mylabel{thm 1.34} Let $\fg$ be the Lie algebra of type $E_6$ and $L$ its root lattice so $\fgt$ is the affine Lie algebra of type $\Es$.  Using the operators  above, define the linear map $\pi:\fgt \to \mbox{End } V_L$ by
\begin{align*}
c &\mapsto I_{V_L} &\mbox{ }&\mbox{ } \\
d &\mapsto d &\mbox{ }&\mbox{ } \\
h \otimes t^n &\mapsto h(n) & \mbox{ for } &h \in \fh, n \in \Z\\
x_{\alp} \otimes t^n &\mapsto \yvomn{\alp}{n} & \mbox{ for } &\alp \in \Phi_{E_6} = L_2, n \in \Z\\
\mbox{or equivalently in the last case, } \\
x_{\alp}(z) &\mapsto \yvoz{\alp} & \mbox{ for } &\alp \in \Phi_{E_6} = L_2. 
\end{align*}

Then $\pi$ is a representation of $\fgt$ on $V_L$ isomorphic to the irreducible $\fgt$-module $V^{\Lam_0}$ and $$\hat{\tau} \yvoz{\alp} \hat{\tau} = \yvoz{\tau\alp}$$.
\end{thm}

The appropriate linear combinations of these vertex operators will give a vertex operator representation of $\fat$ of type $\Ff$ acting on $V_L$, which will not be $\fat$-irreducible.  Note that all the vertex operator calculations for $\fat$ can be carried out using certain linear combinations of operators representing $\fgt$ given in Theorem \ref{thm 1.34}. 

For example, the vertex operator representing $x_{\beta_3}(n) = x_{\alp_3}(n)+x_{\alp_5}(n)$ is given by the corresponding sum, $\yvomn{\alp_3}{n} + \yvomn{\alp_5}{n}$.  Heisenberg operators are given in a similar manner, for example, $h_{\beta_3}(n) = h_{\alp_3}(n) + h_{\alp_5}(n) = (h_{\alp_3}+ h_{\alp_5})(n)$ which we write as $(\alp_3 + \alp_5)(n)$.  A choice of $\tau$ invariant bilinear 2-cocyle is now made and will be used throughout this text.  The chosen 2-cocycle is determined by its values on the simple roots given in the matrix 

$$
[\veps(\alp_i,\alp_j)] = \left[ \begin{array}{r  r  r  r  r  r } 
 1 & 1 & -1 & 1 & 1 & 1\\ 
 1 & 1 & 1 & -1 & 1 & 1\\  
 1 & 1 & 1 & -1 & 1 & 1\\  
 1 & 1 & 1 & 1 & 1 & 1\\ 
 1 & 1 & 1 & -1 & 1 & 1\\ 
 1 & 1 & 1 & 1 & -1 & 1 \\
\end{array} \right]. 
$$

%
%  Chapter 2
%  
%  
%  
%  
%
\chapter{Characters and Graded Dimensions of Affine Lie Algebra Modules}
\label{ch2}
% the label lets you refer to the chapter by number later.

This chapter makes use of the characters and graded dimension formulas of affine Lie algebras from Section 1.2.1. First, the concept of specialized characters is applied to the characters of irreducible highest weight modules for the Lie algebras of type $\Es$ and $\Ff$, and then the characters and graded dimension of the irreducible highest weight $c = \frac{4}{5}$ modules for the Virasoro algebra are given and related to one of the famous Ramanujan identities.

\section{Level 1 Affine Lie Algebra Modules}
\label{s2.1}

The homogeneous character of a highest weight $\fgt$-module with highest weight $\Lam$ is given in Theorem \ref{Thm: 1.19} as,
$$ch(V^{\Lam}) = \frac{1}{e^{\hat{\rho}}R}\sum_{w \in \hat{\mathcal{W}}} sgn(w)e^{w(\Lam + \hat{\rho})}$$
where $\hat{\rho} = \sum_{i = 0}^l \Lam_i$ and $R = \prod_{\alp \in \PHI^+} (1-e^{-\alp})^{dim(\fgt_{\alp})}$ is called the denominator.  It is useful to write the numerator, $N(\Lam) =\sum_{w \in \hat{\mathcal{W}}} sgn(w)e^{w(\Lam + \hat{\rho}) - \hat{\rho}}$, so that $ch(V^{\Lam}) = N(\Lam)/R$.

An $\bold{s} = (s_0, s_1, \dots, s_l)$ specialization of a character is given by substituting $v^{s_i}$ for $e^{-\alp_i}$, for each simple root $\alp_i$, $0 \leq i \leq l $.  For this to be performed on a character it must be first shifted so that it will be an element of the power series ring in the variables $e^{-\alp_i}, 0 \leq i \leq l$.  It is true that $R$ is already an element of this ring, but $N(\Lam)$ must be shifted by $e^{-\Lam}$, hence using $N(\Lam)e^{-\Lam}$ instead.  This element is known as the graded dimension, as in Section 1.2.1, 
\[ gr(V^{\Lam}) = ch(V^{\Lam})e^{-\Lam}.\]

The principal specialization of characters will be needed for this investigation and it is given by Lepowsky's \cite{L} relating the $(1,1,\dots, 1)$ principal specialization of the numerator, $N(\Lam)e^{-\Lam}$, to a $(s_0, s_1, \dots, s_l)$ specialization of the denominator of the dual Lie algebra, $\fgt^{\vee}$, where $s_i = (\Lam + \hat{\rho})(h_i)$.  Denote the denominator of $\fgt^{\vee}$ by $R^{\vee}$.  This relationship is $\bold{S}_{(1,1,\dots, 1)} \left(N(\Lam)e^{-\Lam}\right) = \bold{S}_{(s_0, s_1, \dots, s_l)} \left(R^{\vee}\right) $ giving the principally graded dimension
$$gr_{(1,1,\dots, 1)}\left( V^{\Lam}\right) = gr_{princ}\left( V^{\Lam}\right) = \frac{\bold{S}_{(s_0, s_1, \dots, s_l)} \left(R^{\vee}\right)}{\bold{S}_{(1, 1, \dots, 1)} \left(R\right) }.$$

Remember the vacuum space of $V^{\Lam}$ is denoted by $\V(V^{\Lam})$ and the principally graded character of this space is denoted $\chi(\V(V^{\Lam}))$.  The character of the Fock-space of the principal Heisenberg for $\fgt$ of type $X_n^{(1)}$ is denoted by $\F\left(X_n^{(1)}\right)$, and the following holds

$$gr_{princ}\left( V^{\Lam}\right) = \F\left(X_n^{(1)}\right)\cdot \chi(\V(V^{\Lam})).$$

\subsection{Graded Dimensions for $\Es$-modules}
\label{ss2.1.1}

In Section 1.2.1 the level of a $\fgt$-module, $V^{\Lam}$, was given as the number $\Lam(c)$.  For the Lie algebra of type $\Es$ there exist three such modules with $\Lam(c) = 1$. These three modules are denoted $V^{\Lam_0}$, $V^{\Lam_1}$, and $V^{\Lam_6}$.  The Lie algebra of type $\Es$ is self-dual and hence the principally graded dimension (see \cite{MAN}) is given by:

$$gr_{princ}\left( V^{\Lam}\right) = \frac{Spec_{(s_0, s_1, \dots, s_6)} \left(R\right)}{Spec_{(1, 1, \dots, 1)} \left(R\right) }.$$

For $i \in \{0,1,6\}$, $\chi(\V(V^{\Lam_i})) = 1$ and therefore 
\begin{eqnarray}gr_{princ}\left( V^{\Lam_i}\right) = \F\left(\Es\right) = \prod_{\substack{
   0 < n \\ n \equiv \pm1,\pm4,\pm5\bmod{12}}}\frac{1}{(1-v^{n})} = \frac{\vph(v^2)\vph(v^3)\vph(v^{12})}{\vph(v)\vph(v^4)\vph(v^6)},
\end{eqnarray}
where $\vph(v) = \prod_{i=1}^\infty (1-v^i)$.

\subsection{Graded Dimensions for $\Ff$-modules}
\label{ss2.1.2}  
The Lie algebra of type $\Ff$ has two level one modules which will be denoted $W^{\Omg_0}$ and $W^{\Omg_4}$.  Lepowsky's theorem can be used to compute (see \cite{MAN}) the principally specialized characters for these two irreducible level one $\Ff$-modules, which will be denoted $W^{\Omg_0}$ and $W^{\Omg_4}$.  Let $R$ be the denominator of type $\Ff$, so $R^{\vee}$ is the denominator of its dual $E^{(2)}_6$, and we have for $j \in \{0,4\}$
$$gr_{princ}\left( W^{\Omg_j}\right) = \frac{\bold{S}_{(s_0, s_1, \dots, s_4)} \left(R^{\vee}\right)}{\bold{S}_{(1, 1, \dots, 1)} \left(R\right) }.$$

From Mandia we have,
\begin{eqnarray}\F(\Ff) = \prod_{\substack{
   0 < n \\ n \equiv \pm1,\pm5 \bmod{12}}}\frac{1}{(1-v^{n})} = \frac{\vph(v^2)\vph(v^3)}{\vph(v)\vph(v^6)}.
\end{eqnarray}

The characters $\chi(\V(W^{\Omg_0}))$ and $\chi(\V(W^{\Omg_4}))$ are not trivial and are computed in Mandia (p. 137).  These two characters are related to the famous Rogers-Ramanujan series $a(q)$ and $b(q)$ defined as
\begin{eqnarray}
a(q) = \prod_{n\geq 1} \frac{1}{(1-q^{5n-2})(1-q^{5n-3})},  \\
b(q) = \prod_{n\geq 1} \frac{1}{(1-q^{5n-1})(1-q^{5n-4})}.
\end{eqnarray} 

These characters are given using the following product forms
\begin{eqnarray}
\chi(\V(W^{\Omg_0})) =  \prod_{n\geq 1} \frac{1}{(1-(v^4)^{5n-2})(1-(v^4)^{5n-3})} = a(v^4), \\
\chi(\V(W^{\Omg_4})) = \prod_{n\geq 1} \frac{1}{(1-(v^4)^{5n-1})(1-(v^4)^{5n-4})} = b(v^4).
\end{eqnarray}  

Hence we have
\begin{eqnarray}
gr_{princ}\left( W^{\Omg_0}\right) = \F\left(\Ff\right) \cdot \chi(\V(W^{\Omg_0})) = \frac{\vph(v^2)\vph(v^3)}{\vph(v)\vph(v^6)}a(v^4), \\
gr_{princ}\left( W^{\Omg_4}\right) = \F\left(\Ff\right) \cdot \chi(\V(W^{\Omg_4})) = \frac{\vph(v^2)\vph(v^3)}{\vph(v)\vph(v^6)}b(v^4).
\end{eqnarray}

\section{The Virasoro Modules with c = 4/5}
\label{s2.3}

In Section 1.2.2 the Virasoro algebra, $Vir$, was defined to be the infinite dimensional Lie algebra with basis $\left\{ L_n, \cvir | n \in \Z \right\}$, and its brackets were given as
\begin{eqnarray*}
&& [L_n,\cvir] = 0 \\
&& [L_m, L_n] = (m-n)L_{m+n} + \frac{1}{12}(m^3-m)\delta_{m,-n}\cvir \quad \mbox{ for } m,n\in \Z.
\end{eqnarray*}

In this section we study the unitary irreducible highest weight $Vir$-modules, \\$Vir(c,h)$, where $\cvir$ acts as the scalar $c = \frac{4}{5}$.  The operator $L_0$ acts diagonally on $Vir(c,h)$ and hence the module $Vir(c, h)$ decomposes into a direct sum of finite dimensional $L_0$-eigenspaces.  

Denote the $L_0$-eigenspace with eigenvalue $m$ by $Vir\left(c,h\right)_{m}$, and define the graded dimension of  $Vir\left(c,h\right)$ as 
$$gr(c,h) = \sum_{n\geq 0} dim\left(Vir\left(c,h\right)_{h - n}\right)q^n.$$

\ni The character, $\chi(c,h)$, for $Vir(c,h)$ is given by a shift of the graded dimension, 
$$\chi(c,h) = q^{h - c/24}gr(c,h).$$

Recall from Section 1.2.2 the Fe{\u\i}gin-Fuchs formulas for the characters $\chi_{s,t}^{m,n}(q)$ of the modules $Vir(c_{s,t},h_{s,t}^{m,n})$.  If $s = 5$ and $t = 6$, then $c_{5,6} = \frac{4}{5}$ and the 10 associated $h$ values are 
$$h_{5,6}^{1,1} = 0, \quad h_{5,6}^{1,2} = \frac{1}{8}, \quad  h_{5,6}^{1,3} = \frac{2}{3}, \quad h_{5,6}^{1,4}= \frac{13}{8}, \quad  h_{5,6}^{1,5} = 3,$$ 
$$h_{5,6}^{2,1} = \frac{2}{5}, \quad h_{5,6}^{2,2} = \frac{1}{40}, \quad h_{5,6}^{2,3} = \frac{1}{15}, \quad h_{5,6}^{2,4} = \frac{21}{40}, \quad h_{5,6}^{2,5} = \frac{7}{5}.$$

%$$c_{s,t} = 1 - \frac{6(s-t)^2}{st}, \quad h_{s,t}^{m,n} = \frac{(mt-ns)^2 - (s-t)^2}{4st}.$$

Only six of these modules will be needed in this work, and their characters are
\begin{eqnarray}
\chi_{5,6}^{1,1}(q) &=& \frac{q^{-1/30}}{\vph(q)}\sum_{k \in \Z} q^{30k^2}(q^k - q^{11k + 1}), \\
\chi_{5,6}^{1,5}(q) &=& \frac{q^{89/30}}{\vph(q)}\sum_{k \in \Z} q^{30k^2}(q^{-19k} - q^{31k + 5}),\\
\chi_{5,6}^{2,1}(q) &=& \frac{q^{11/30}}{\vph(q)}\sum_{k \in \Z} q^{30k^2}(q^{7k} - q^{17k + 2}), \\
\chi_{5,6}^{2,5}(q) &=& \frac{q^{41/30}}{\vph(q)}\sum_{k \in \Z} q^{30k^2}(q^{-13k} - q^{37k + 10}), \\
\chi_{5,6}^{1,3}(q) &=& \frac{q^{19/30}}{\vph(q)}\sum_{k \in \Z} q^{30k^2}(q^{9k} - q^{-21k + 3}),\\
\chi_{5,6}^{2,3}(q) &=& \frac{q^{1/30}}{\vph(q)}\sum_{k \in \Z} q^{30k^2}(q^{3k} - q^{27k + 6}).
\end{eqnarray}

Feingold-Milas \cite{FM} used all $h$ values for $c = \frac{4}{5}$ in their work on modules and twisted modules for the Zamolodchikov $\mathcal{W}_3$-algebra.

We will need sums (sometimes shifted by powers of $q$) of these characters. The next few manipulations will be necessary later in this work and rely heavily on the famous Jacobi triple product (JTP) identity,
$$\prod_{n \geq 1} (1- u^nv^n)(1- u^nv^{n-1})(1- u^{n-1}v^n) = \sum_{k \in \Z} (-1)^ku^{k(k+1)/2}v^{k(k-1)/2}.$$
The first sum to work with is  
\begin{eqnarray*}
\vph(q)q^{1/30}\left(\chi_{5,6}^{1,1}(q) + \chi_{5,6}^{1,5}(q)  \right) &=& \sum_{k \in \Z} q^{30k^2}(q^k - q^{11k + 1}) + q^3\sum_{k \in \Z} q^{30k^2}(q^{-19k} - q^{31k + 5}) \\
&=&  \sum_{k \in \Z} q^{30k^2}(q^k - q^{11k + 1}) + \sum_{k \in \Z} q^{30k^2}(q^{-19k + 3} - q^{31k + 8}). 
\end{eqnarray*}
The JTP with $u = q^8$ and $v = q^7$,
\begin{eqnarray*}
\prod_{n \geq 1} (1- q^{15n})(1- q^{15n -7})(1- q^{15n - 8}) &=& \sum_{k \in \Z} (-1)^kq^{8k(k+1)/2}q^{7k(k-1)/2} \\
&=& \sum_{k \in \Z} (-1)^kq^{(15k^2+k)/2} \\
&=& \sum_{m \in \Z} q^{(60m^2+2m)/2} - \sum_{m \in \Z} q^{(60m^2+62m+16)/2} \\
&=& \sum_{m \in \Z} q^{30m^2+m} - \sum_{m \in \Z} q^{30m^2+31m+8}.
\end{eqnarray*}
Using JTP again, but with $u = q^{13}$ and $v = q^{2}$,
\begin{eqnarray*}
\prod_{n \geq 1} (1- q^{15n})(1- q^{15n - 13})(1- q^{15n - 2}) &=& \sum_{k \in \Z} (-1)^kq^{13k(k+1)/2}q^{2k(k-1)/2} \\
&=& \sum_{k \in \Z} (-1)^kq^{(15k^2+11k)/2} \\
&=& \sum_{m \in \Z} q^{(60m^2+22m)/2} - \sum_{m \in \Z} q^{(60m^2-38m+4)/2} \\
&=& \sum_{m \in \Z} q^{30m^2+11m} - \sum_{m \in \Z} q^{30m^2-19m+2}. 
\end{eqnarray*}
Using these two equalities and shifting the second by $-q$, we obtain 
\begin{eqnarray}
\vph(q)q^{1/30}\left(\chi_{5,6}^{1,1}(q) + \chi_{5,6}^{1,5}(q)  \right) &=& \prod_{n \geq 1} (1- q^{15n})(1- q^{15n -7})(1- q^{15n - 8}) \\
&&- q \prod_{n \geq 1} (1- q^{15n})(1- q^{15n - 13})(1- q^{15n - 2}). \nonumber
\end{eqnarray}

Similarly for the next sum of characters
\begin{eqnarray*}
&&\vph(q)q^{-11/30}\left(\chi_{5,6}^{2,1}(q) + \chi_{5,6}^{2,5}(q)  \right) \\
&=& \sum_{k \in \Z} q^{30k^2}(q^{7k} - q^{17k + 2}) + q \sum_{k \in \Z} q^{30k^2}(q^{-13k} - q^{37k + 10})\\
&=&  \sum_{k \in \Z} q^{30k^2}(q^{7k} - q^{17k + 2}) + \sum_{k \in \Z} q^{30k^2}(q^{-13k+1} - q^{37k + 11}). 
\end{eqnarray*}
The JTP with $u = q^{11}$ and $v = q^4$,
\begin{eqnarray*}
\prod_{n \geq 1} (1- q^{15n})(1- q^{15n - 4})(1- q^{15n - 11}) &=& \sum_{k \in \Z} (-1)^kq^{11k(k+1)/2}q^{4k(k-1)/2} \\
&=& \sum_{k \in \Z} (-1)^kq^{(15k^2+7k)/2} \\
&=& \sum_{m \in \Z} q^{(60m^2+14m)/2} - \sum_{m \in \Z} q^{(60m^2+74m+22)/2} \\
&=& \sum_{m \in \Z} q^{30m^2+7m} - \sum_{m \in \Z} q^{30m^2+37m+11}. 
\end{eqnarray*}
Also JTP with $u = q^{1}$ and $v = q^{14}$,
\begin{eqnarray*}
\prod_{n \geq 1} (1- q^{15n})(1- q^{15n - 14})(1- q^{15n - 1}) &=& \sum_{k \in \Z} (-1)^kq^{k(k+1)/2}q^{14k(k-1)/2} \\
&=& \sum_{k \in \Z} (-1)^kq^{(15k^2-13k)/2} \\
&=& \sum_{m \in \Z} q^{(60m^2-26m)/2} - \sum_{m \in \Z} q^{(60m^2+34m+2)/2} \\
&=& \sum_{m \in \Z} q^{30m^2-13m} - \sum_{m \in \Z} q^{30m^2+17m+1}. 
\end{eqnarray*}
Using these two equalities and shifting the second by $q$, we obtain 
\begin{eqnarray}
\vph(q)q^{-11/30}\left(\chi_{5,6}^{2,1}(q) + \chi_{5,6}^{2,5}(q)  \right) &=& \prod_{n \geq 1} (1- q^{15n})(1- q^{15n - 4})(1- q^{15n - 11}) \\
&&+ q \prod_{n \geq 1} (1- q^{15n})(1- q^{15n - 14})(1- q^{15n - 1}). \nonumber
\end{eqnarray}

In order to express the last two characters in a useful way, we use two more special cases of (JTP), one with $u = q^4, v = q$ and the other with $u = q^3, v = q^2$, and  related to the Rogers-Ramanujan series:
\begin{eqnarray}\vph(q)a(q) = \prod_{n\geq 1} (1-q^{5n})(1-q^{5n-1})(1-q^{5n-4}) &=& \sum_{k\in \Z} (-1)^kq^{k(5k+3)/2} \mylabel{aq} \\
\vph(q)b(q) = \prod_{n\geq 1} (1-q^{5n})(1-q^{5n-2})(1-q^{5n-3}) &=& \sum_{k\in \Z} (-1)^kq^{k(5k+1)/2}\mylabel{bq}.
\end{eqnarray}

We will use the above to rewrite the series
$$\vph(q)q^{-19/30}\chi_{5,6}^{1,3}(q) = \sum_{k \in \Z} q^{30k^2}(q^{9k} - q^{-21k + 3}).$$
Use JTP with $u = q^{12}$ and $v = q^3$
\begin{eqnarray*}
\prod_{n \geq 1} (1- q^{15n})(1- q^{15n -3})(1- q^{15n - 12}) &=& \sum_{k \in \Z} (-1)^kq^{12k(k+1)/2}q^{3k(k-1)/2} \\
&=& \sum_{k \in \Z} (-1)^kq^{(15k^2+9k)/2} \\
&=& \sum_{m \in \Z} q^{(60m^2+18m)/2} - \sum_{m \in \Z} q^{(60m^2-42m+6)/2} \\
&=& \sum_{m \in \Z} q^{30m^2+9m} - \sum_{m \in \Z} q^{30m^2-21m+3}. 
\end{eqnarray*}
Hence, 
\begin{eqnarray}
\vph(q)q^{-19/30}\chi_{5,6}^{1,3}(q) &=& \prod_{n \geq 1} (1- (q^3)^{5n})(1- (q^3)^{5n - 1})(1- (q^3)^{5n - 4}) \nonumber \\
&=& \vph(q^3)a(q^3). 
\end{eqnarray}
Also the series 
$$\vph(q)q^{-1/30}\chi_{5,6}^{2,3}(q) = \sum_{k \in \Z} q^{30k^2}(q^{3k} - q^{27k + 6}).$$
The JTP with $u = q^{9}$ and $v = q^6$
\begin{eqnarray*}
\prod_{n \geq 1} (1- q^{15n})(1- q^{15n -6})(1- q^{15n - 9}) &=& \sum_{k \in \Z} (-1)^kq^{9k(k+1)/2}q^{6k(k-1)/2} \\
&=& \sum_{k \in \Z} (-1)^kq^{(15k^2+3k)/2} \\
&=& \sum_{m \in \Z} q^{(60m^2+6m)/2} - \sum_{m \in \Z} q^{(60m^2+54m+12)/2} \\
&=& \sum_{m \in \Z} q^{30m^2+3m} - \sum_{m \in \Z} q^{30m^2+27m+6}. 
\end{eqnarray*} 
Hence, 
\begin{eqnarray}
\vph(q)q^{-1/30}\chi_{5,6}^{2,3}(q) &=& \prod_{n \geq 1} (1- q^{15n})(1- q^{15n -6})(1- q^{15n - 9})  \nonumber \\
&=& \vph(q^3)b(q^3).
\end{eqnarray}

\chapter{Decomposition of $\Es$-modules as $Vir \otimes \Ff$-modules }
\label{ch3}
% the label lets you refer to the chapter by number later.

This chapter provides the branching decomposition for the irreducible level 1 $\fgt$-modules as $\fat$-modules.  First, some notation will be set up to aid this decomposition and then, building on the work from Chapter 2, the decomposition will be given for these three $\fgt$-modules. In \cite{BTM} Bernard and Thierry-Mieg state the branching decomposition of these three modules.  We give a rigorous proof of these results.

\section{Setting Up the Decomposition and Notation}
\label{s3.1}

Consider the Lie algebra $\fgt$ of type $\Es$ and its $\tilde{\tau}$-fixed subalgebra $\fat$ of type $\Ff$.  We wish to decompose the three level 1 $\fgt$-modules, $V^{\Lam_0}$, $V^{\Lam_1}$, $V^{\Lam_6}$, in terms of the level 1 $\fat$-modules, $W^{\Omg_0}, W^{\Omg_4}$.  For any $n \in \Z$, $j \in \{ 0,4\}$, the modules $W^{\Omg_j}$ and $W^{\Omg_j -n\delta}$ are isomorphic as $\fah$-modules, where $\fat = \fah \oplus \C d$, but not as $\fat$-modules.  The highest weight vectors in these two modules have $d$ eigenvalues that differ by n.  Inside each $\fgt$-module, $V^{\Lam_i}$, for $n \in \Z$,  define the subspace of $\fat$ HWV's of weight $\Omg_j - n\delta$ to be 
$$V^{\Lam_i}_{\fat}(\Omg_j)_n = \{ v \in V^{\Lam_i} \mid \fat^+ \cdot v = 0, \forall h \in \fbt, h \cdot v = (\Omg_j - n\delta)(h)v\}$$
where $\fat = \fat^- \oplus \fbt \oplus \fat^+$ is the triangular decomposition.
For $i \in \{0,1,6 \}$ let 
$$V_{\fat}^{\Lam_i}(\Omg_j) = \bigoplus_{0 \leq n \in \Z} V^{\Lam_i}_{\fat}(\Omg_j)_n. $$
As $\fat$-modules we have
$$ V^{\Lam_i} = V_{\fat}^{\Lam_i}(\Omg_0) \otimes_{\fah} W^{\Omg_0} \oplus V_{\fat}^{\Lam_i}(\Omg_4) \otimes_{\fah} W^{\Omg_4}.$$
As $\Z$-graded $\fbt \oplus \fat^+$-modules, letting $q = e^{-\delta}$, 
$$gr\left(V^{\Lam_i}_{\fat}(\Omg_j)\right) = e^{-Proj(\Lam_i)} \sum_{0 \leq n \in \Z} dim\left( V^{\Lam_i}_{\fat}(\Omg_j)_n\right) e^{\Omg_j}q^n$$
and for $i \in \{ 0,1,6 \}$ we define the coefficients $$c^i(n) = dim\left(  V^{\Lam_i}_{\fat}(\Omg_0)_n\right) \quad \mbox{ and } \quad d^i(n) = dim\left(  V^{\Lam_i}_{\fat}(\Omg_4)_n\right)$$
and define the power series $\ds c^i(q) = \sum_{n \geq 0} c^i(n) q^n$ and $\ds d^i(q) = \sum_{n \geq 0} d^i(n) q^n$.
Then we have the following graded dimension formula as $\fat$-modules,
\begin{eqnarray}
gr\left( V^{\Lam_i}\right) = e^{\Omg_0 - Proj(\Lam_i)}c^i(q) gr\left( W^{\Omg_0}\right) + e^{\Omg_4 - Proj(\Lam_i)}d^i(q) gr\left( W^{\Omg_4}\right).  
\end{eqnarray}

From Theorem \ref{thm: coset}, we have that for each $V^{\Lam_i}$ there exists $Vir_{\fg-\fa}(V^{\Lam_i})$ which commutes with $\fat$ and whose central charge is $c = \frac{4}{5}$.  It follows that for $i \in \{0,1,6 \}$, $j \in \{0,4\}$, $V_{\fat}^{\Lam_i}(\Omg_j)$ is a $Vir_{\fg-\fa}(V^{\Lam_i})$-module which we will show is a direct sum of irreducible $Vir$-modules, $Vir(\frac{4}{5},h)$, with $h$ chosen from the possible values listed in Section 2.2.    When dealing with the principally graded characters of these spaces, one needs to know that $$-\delta = -\theta(\fg) - \alp_0 = -\theta(\fa) - \alp_0 = -(\alp_1+2\alp_2+2\alp_3+3\alp_4+2\alp_5+\alp_6) - \alp_0,$$  hence $q$ is replaced in this principal specialization by $v^{12}$.  We also know that 
\begin{eqnarray*}
\Omg_4 - Proj(\Lam_0) &=& \Omg_4 - \Omg_0 \\
&=& \omg_4 \\
&=& \beta_1 + 2\beta_2 + 3\beta_3 + 2\beta_4,
\end{eqnarray*}
so the principal specialization of (3.1) when $i = 0$ is
\begin{eqnarray}
gr_{princ}\left( V^{\Lam_0}\right) &=& c^0(v^{12}) gr_{princ}\left( W^{\Omg_0}\right) + v^{-8}d^0(v^{12}) gr_{princ}\left( W^{\Omg_4}\right).
\end{eqnarray}
Since $d^0(0) = 0$ and $d^0(1) = 1$, temporarily define $k^0(m) = d^0(m+1)$ for $m \geq 0$, then 
\begin{eqnarray*}
v^{-8}\sum_{n \geq 0}d^0(n)(v^{12})^n &=& 0 + v^{-8}\sum_{n \geq 1}d^0(n)(v^{12})^n \\
&=& v^{12-8}\sum_{n \geq 1}d^0(n)(v^{12})^{n-1} \\
&=& v^{4}\sum_{m \geq 0}k^0(m)(v^{12})^m.
\end{eqnarray*}
For convenience we will keep the name $d^0(q)$ for the new series $k^0(q)$ so that in the new notation $d^0(0) = 1$.  Using the results from Sections 2.1.1 and 2.1.2 this principal specialization can be rewritten in the following ways (using $t = v^4$ in (\ref{eq 3.6})): 
\begin{eqnarray}
gr_{princ}\left( V^{\Lam_0}\right) &=& c^0(v^{12}) gr_{princ}\left( W^{\Omg_0}\right) + v^4d^0(v^{12}) gr_{princ}\left( W^{\Omg_4}\right) \\
\frac{\vph(v^2)\vph(v^3)\vph(v^{12})}{\vph(v)\vph(v^4)\vph(v^6)} &=& c^0(v^{12})\frac{\vph(v^2)\vph(v^3)}{\vph(v)\vph(v^6)}a(v^4)  + v^4d^0(v^{12}) \frac{\vph(v^2)\vph(v^3)}{\vph(v)\vph(v^6)}b(v^4) \quad  \\ 
\frac{\vph(v^{12})}{\vph(v^4)} &=& c^0(v^{12})a(v^4)  + v^4d^0(v^{12})b(v^4) \\   
\frac{\vph(t^3)}{\vph(t)} &=& c^0(t^3)a(t)  + td^0(t^3)b(t).  \mylabel{eq 3.6}
\end{eqnarray}

%In the module $V^{\Lam_1}$ the HWV $1 \otimes e^{\lam_1}$ has $\Ff$ weight $\Omg_4$ so we align the first $W^{\Omg_4}$ module at the top, and then empirically find the first $W^{\Omg_0}$ module.  Align this module by shifting by $v^8$.  Hence obtaining

We will see in the next chapter that $\tilde{\tau}$ determines an automorphism of $V^{\Lam_0} \oplus V^{\Lam_1} \oplus V^{\Lam_6}$ which fixes $V^{\Lam_0}$ and switches $V^{\Lam_1}$ and $V^{\Lam_6}$ in such a way that their principal graded dimensions are equal.  Since $c^1(0) \neq 0 \neq d^1(0)$, we do not shift either of these power series. Also note 
\begin{eqnarray*}
\Omg_0 - Proj(\Lam_1) &=& \Omg_0 - \Omg_4 \\
&=& -\omg_4 \\
&=& -\beta_1 - 2\beta_2 - 3\beta_3 - 2\beta_4,
\end{eqnarray*} 
and therefore the principal specialization of (3.1) when $i = 1$ can be written in the following ways: 
\begin{eqnarray}
gr_{princ}\left( V^{\Lam_1}\right) &=& v^8c^1(v^{12}) gr_{princ}\left( W^{\Omg_0}\right) + d^1(v^{12}) gr_{princ}\left( W^{\Omg_4}\right) \\
\frac{\vph(v^2)\vph(v^3)\vph(v^{12})}{\vph(v)\vph(v^4)\vph(v^6)} &=& v^8c^1(v^{12})\frac{\vph(v^2)\vph(v^3)}{\vph(v)\vph(v^6)}a(v^4)  + d^1(v^{12}) \frac{\vph(v^2)\vph(v^3)}{\vph(v)\vph(v^6)}b(v^4)  \quad \\
\frac{\vph(v^{12})}{\vph(v^4)} &=& v^8c^1(v^{12})a(v^4)  + d^1(v^{12})b(v^4) \\
\frac{\vph(t^3)}{\vph(t)} &=& t^2c^1(t^3)a(t)  + d^1(t^3)b(t). \mylabel{eq 3.10}   
\end{eqnarray}

Similarly, when $i = 6$ we have
\begin{eqnarray}
gr_{princ}\left( V^{\Lam_6}\right) &=& v^8c^6(v^{12}) gr_{princ}\left( W^{\Omg_0}\right) + d^6(v^{12}) gr_{princ}\left( W^{\Omg_4}\right) \\
\frac{\vph(t^3)}{\vph(t)} &=& t^2c^6(t^3)a(t)  + d^6(t^3)b(t) . \mylabel{eq 3.12}
\end{eqnarray}
 
We first show that there is a unique solution to the equations (\ref{eq 3.6}) and (\ref{eq 3.10}).  Let $\zeta \neq 1$ be a cube root of 1, and consider the equations
\begin{eqnarray*}
\frac{\vph(t^3)}{\vph(t)} &=& c^0(t^3)a(t) + td^0(t^3)b(t) \\
\frac{\vph(t^3)}{\vph(\zeta t)} &=& c^0(t^3)a(\zeta t) + \zeta td^0(t^3)b(\zeta t).
\end{eqnarray*} 
Setting $e(t) = \frac{\vph(t^3)}{\vph(t)}$ and $e(t\zeta) = \frac{\vph(t^3)}{\vph(\zeta t)}$ we have the following matrix equation
\begin{eqnarray*}
\left[\begin{array}{r r}
a(t) & b(t) \\
a(t\zeta) & \zeta b(t\zeta)
\end{array} \right] 
\left[
\begin{array}{r}
c^0(t^3) \\
td^0(t^3) 
\end{array} \right] =
\left[
\begin{array}{r}
e(t) \\
e(t\zeta) 
\end{array} \right].
\end{eqnarray*}

The constant term of $\zeta a(t)b(t\zeta) - a(t\zeta)b(t)$ is $\zeta - 1$ and  therefore $$det\left[\begin{array}{r r}
a(t) & b(t) \\
a(tz) & zb(tz)
\end{array} \right] $$ is an invertible power series.  This gives the uniqueness of a solution to the system and also to (\ref{eq 3.6}).  

Now consider the equations
\begin{eqnarray*}
\frac{\vph(t^3)}{\vph(t)} &=& t^2c^1(t^3)a(t) + d^1(t^3)b(t) \\
\frac{\vph(t^3)}{\vph(\zeta t)} &=& \zeta^2t^2c^1(t^3)a(\zeta t) + d^1(t^3)b(\zeta t)
\end{eqnarray*} 
and the matrix equation 
\begin{eqnarray*}
\left[\begin{array}{r r}
a(t) & b(t) \\
\zeta^2 a(t\zeta) & b(t\zeta)
\end{array} \right] 
\left[
\begin{array}{r}
t^2c^1(t^3) \\
d^1(t^3) 
\end{array} \right] =
\left[
\begin{array}{r}
e(t) \\
e(t\zeta) 
\end{array} \right].
\end{eqnarray*}

The constant term of $a(t)b(t\zeta) - \zeta^2 a(t\zeta)b(t)$ is $1 - \zeta^2$ and  therefore $$det\left[\begin{array}{r r}
a(t) & b(t) \\
\zeta^2 a(t\zeta) & b(t\zeta)
\end{array} \right]$$ is an invertible power series.  This gives the uniqueness of a solution to the system and also to (\ref{eq 3.10}) and (\ref{eq 3.12}).

\section{The Branching Rule Coefficients for $V^{\Lam_0}$}
\label{s3.2}
  
%Equation (\ref{eq 3.6}) is the equation associated to $V^{\Lam_0}$.  We now show that a substitution of $c^0(t^3)$ and $d^0(t^3)$ with appropriate characters of $Vir$-modules gives the solution  equation. 

%We now show that $c^0(t^3)$ and $d^0(t^3)$ satisfying (\ref{eq 3.6}) can be expressed using characters of certain $Vir$-modules.

\begin{thm} \label{thm br1} The branching rule coefficients for $V^{\Lam_0}$ satisfying (\ref{eq 3.6}) can be expressed as the Virasoro characters:
$$c^0(t) = t^{1/30}\left(\chi_{5,6}^{1,1}(t) + \chi_{5,6}^{1,5}(t)  \right) \quad \mbox{and} \quad d^0(t)= t^{-11/30}\left(\chi_{5,6}^{2,1}(t) + \chi_{5,6}^{2,5}(t)  \right),$$
so that we get
\begin{eqnarray*}
gr_{princ}\left( V^{\Lam_0}\right) &=& t^{1/10}\left(\chi_{5,6}^{1,1}(t^3) + \chi_{5,6}^{1,5}(t^3)  \right) gr_{princ}\left( W^{\Omg_0}\right) \\
&&+ t^{-1/10}\left(\chi_{5,6}^{2,1}(t^3) + \chi_{5,6}^{2,5}(t^3)  \right) gr_{princ}\left( W^{\Omg_4}\right).
\end{eqnarray*}
\end{thm}

\begin{proof}

Let the $G(t)$ and $H(t)$ be given by
$$G(t) = \vph(t)t^{1/30}\left(\chi_{5,6}^{1,1}(t) + \chi_{5,6}^{1,5}(t)  \right) \quad\mbox{ and } \quad H(t) = \vph(t)t^{-11/30}\left(\chi_{5,6}^{2,1}(t) + \chi_{5,6}^{2,5}(t)  \right),$$
so that
\begin{eqnarray}
G(t^3) &=& \vph(t^3)(t^3)^{1/30}\left(\chi_{5,6}^{1,1}(t^3) + \chi_{5,6}^{1,5}(t^3)  \right), \nn 
\end{eqnarray}
and
\begin{eqnarray}
H(t^3) &=& \vph(t^3)(t^3)^{-11/30}\left(\chi_{5,6}^{2,1}(t^3) + \chi_{5,6}^{2,5}(t^3) \right). \nn
\end{eqnarray}
According to (2.15) and (2.16) these equations can also be given by
\begin{eqnarray} 
G(t^3) &=& \prod_{n \geq 1} (1- (t^3)^{15n})(1- (t^3)^{15n -7})(1- (t^3)^{15n - 8}) \\
&&- t^3 \prod_{n \geq 1} (1- (t^3)^{15n})(1- (t^3)^{15n - 13})(1- (t^3)^{15n - 2}) ,\nonumber
\end{eqnarray}
and
\begin{eqnarray}
H(t^3) &=& \prod_{n \geq 1} (1- (t^3)^{15n})(1- (t^3)^{15n - 4})(1- (t^3)^{15n - 11}) \\
&&+ t^3 \prod_{n \geq 1} (1- (t^3)^{15n})(1- (t^3)^{15n - 14})(1- (t^3)^{15n - 1}) .\nonumber
\end{eqnarray}

Replacing $c^0(t^3)$ and $d^0(t^3)$ with $G(t^3)/\vph(t^3)$ and $H(t^3)/\vph(t^3)$, respectively, Theorem \ref{thm br1} is equivalent to 
$$\frac{\vph(t^3)}{\vph(t)} = G(t^3)a(t)/\vph(t^3) + tH(t^3)b(t)/\vph(t^3),$$
which is equivalent to
$$\vph(t^3)^2 = G(t^3)\vph(t)a(t) + tH(t^3)\vph(t)b(t).$$

The following equation, which is proven in Berndt et al \cite{BCC}, is one of the Ramanujan identities
\begin{eqnarray} \label{REq1}
\vph(t^3)^2 = t^2\vph(t)\vph(t^9)a(t)a(t^9) + \vph(t)\vph(t^9)b(t)b(t^9).
\end{eqnarray}

We write
\begin{eqnarray}
a(t)\vph(t) = \sum_{m\geq 0} x_mt^m &=& \sum_{m \geq 0}x_{3m}t^{3m}+\sum_{m \geq 0}x_{3m+1}t^{3m+1}+\sum_{m \geq 0}x_{3m+2}t^{3m+2} \nn \\
&=& [a(t)\vph(t)]_0 + [a(t)\vph(t)]_1 + [a(t)\vph(t)]_2,
\end{eqnarray}
and
\begin{eqnarray}
b(t)\vph(t) = \sum_{m\geq 0} y_mt^m &=& \sum_{m \geq 0}y_{3m}t^{3m}+\sum_{m \geq 0}y_{3m+1}t^{3m+1}+\sum_{m \geq 0}y_{3m+2}t^{3m+2} \nn \\
&=& [b(t)\vph(t)]_0 + [b(t)\vph(t)]_1 + [b(t)\vph(t)]_2
\end{eqnarray}
where $[a(t)\vph(t)]_i$ and $[b(t)\vph(t)]_i$ are the terms of the given series whose powers of $t$ are congruent to $i \bmod 3$ for $i \in \{ 0,1,2\}$.  Using this notation (\ref{REq1}) can be written as 
\begin{eqnarray} \label{REq2}
\vph(t^3)^2 &=& t^2 \left( [a(t)\vph(t)]_0 + [a(t)\vph(t)]_1 + [a(t)\vph(t)]_2 \right)\vph(t^9)a(t^9)  \\
&&+ \left( [b(t)\vph(t)]_0 + [b(t)\vph(t)]_1 + [b(t)\vph(t)]_2 \right)\vph(t^9)b(t^9). \nn  
\end{eqnarray}

\ni The next lemmas will provide simplifications to (\ref{REq2}).

\begin{lem} We show \label{lem3.2}$[a(t)\vph(t)]_2 = 0 = [b(t)\vph(t)]_1$. 
\end{lem}
\begin{proof} Equation (\ref{aq}) gives $a(t)\vph(t) = \sum_{k\in \Z} (-1)^kt^{k(5k+3)/2}$, then examining the powers of $t$:\\

\ni Case 1: If $k = 3m$, for $m \in \Z$, then $3m(15m+3)/2 \equiv 0 \bmod 3$ \\

\ni Case 2: If $k = 3m + 1$, for $m \in \Z$\\
And if $m$ even, then $(6n+1)(30n+8)/2 \equiv (6n+1)(15n+4)\equiv (1)(1)\equiv 1 \bmod 3$ \\
And if $m$ odd, then $(6n+4)(30n+23)/2 \equiv (3n+2)(30n+23)\equiv (2)(2)\equiv 1 \bmod 3$\\

\ni Case 3: If $k = 3m + 2$, for $m \in \Z$\\
And if $m$ even, then $(6n+2)(30n+13)/2 \equiv (3n+1)(30n+13)\equiv (1)(1)\equiv 1 \bmod 3$ \\
And if $m$ odd, then $(6n+5)(30n+28)/2 \equiv (6n+5)(15n+14)\equiv (2)(2)\equiv 1 \bmod 3$\\

\ni Hence $[a(t)\vph(t)]_2 = 0$. \\
Since (\ref{bq}) gives $b(t)\vph(t) = \sum_{k \in \Z} (-1)^k t^{k(5k+1)/2} $, then we also have the following argument for the resulting powers of $t$:\\

\ni Case 1: If $k = 3m$, for $m \in \Z$, then $3m(15m+1)/2 \equiv 0 \bmod 3$ \\

\ni Case 2: If $k = 3m + 1$, for $m \in \Z$\\
And if $m$ even, then $(6n+1)(30n+6)/2 \equiv (6n+1)(15n+3)\equiv (1)(0)\equiv 0 \bmod 3$ \\
And if $m$ odd, then $(6n+4)(30n+21)/2 \equiv (3n+2)(30n+21)\equiv (2)(0)\equiv 0 \bmod 3$ \\

\ni Case 3: If $k = 3m + 2$, for $m \in \Z$\\
And if $m$ even, then $(6n+2)(30n+11)/2 \equiv (3n+1)(30n+11)\equiv (1)(2)\equiv 2 \bmod 3$ \\
And if $m$ odd, then $(6n+5)(30n+26)/2 \equiv (6n+5)(15n+13)\equiv (2)(1)\equiv 2 \bmod 3$ \\

\ni Therefore $[b(t)\vph(t)]_1 = 0$. \\
\end{proof}

\begin{lem} \label{lem3.3}$[a(t)\vph(t)]_0 = \sum_{m \geq 0} x_{3m}t^{3m} = \sum_{k \in \Z} (-1)^kt^{3k(15k+3)/2} = b(t^9)\vph(t^9)$ 
\end{lem}
\begin{proof} Use the definition of $b(t)\vph(t)$ and substitute $t^9$ for $t$, along with the observation of the powers of $t$ from proof of Lemma \ref{lem3.2}. 
\end{proof}

\begin{lem} \label{lem3.4}$[a(t)\vph(t)]_1 =  \sum_{m\geq 0} x_{3m+1} t^{3m+1}  = -tH(t^3) $
\end{lem}
\begin{proof} Replacing $m$ by $-m-1$ to get from line 2 to line 3, 
\begin{eqnarray*}
\sum_{m\geq 0} x_{3m+1} t^{3m+1}  &=&  \sum_{k \in \Z, k \equiv 1,2 \bmod 3} (-1)^k t^{k(5k+3)/2} \\
 &=& -\sum_{m \in \Z} (-1)^m t^{(3m+1)(15m+8)/2} + \sum_{m \in \Z} (-1)^m t^{(3m+2)(15m+13)/2} \\
  &=& - \sum_{m \in \Z} (-1)^m t^{(45m^2+39m+8)/2} - \sum_{m \in \Z} (-1)^m t^{(45m^2 +21m+2)/2} \\
 &=& - t^4\sum_{m \in \Z} (-1)^m t^{(45m^2+39m)/2} - t \sum_{m \in \Z} (-1)^m t^{(45m^2 + 21m)/2} \\
 &=& -t^4 \sum_{m \in \Z} (-1)^m (t^3)^{(15m^2 + 13m)/2} - t \sum_{m \in \Z} (-1)^m (t^3)^{(15m^2+7m)/2} \\
 &=& -t^4 \sum_{m \in \Z} (-1)^m  (t^3)^{14m(m + 1)/2}(t^3)^{m(m-1)/2} \\ &&- t \sum_{m \in \Z} (-1)^m (t^3)^{11m(m +1)/2}(t^3)^{4m(m-1)/2} \\
 &=& -t^4 \prod_{n\geq 1} (1-(t^3)^{15n})(1-(t^3)^{15n - 1})(1-(t^3)^{15n - 14}) \\  &&+ -t\prod_{n\geq 1} (1-(t^3)^{15n})(1-(t^3)^{15n - 4})(1-(t^3)^{15n - 11})  \\
 &=& -tH(t^3)
\end{eqnarray*}
\end{proof}

\begin{lem} \label{lem3.5}$[b(t)\vph(t)]_0 =  \sum_{m\geq 0} y_{3m} t^{3m}  = G(t^3) $
\end{lem}
\begin{proof} 
\begin{eqnarray*}
 \sum_{m\geq 0} y_{3m} t^{3m}  &=&  \sum_{k \in \Z, k \equiv 0,1 \bmod 3} (-1)^k t^{k(5k+1)/2} \\
 &=& \sum_{m \in \Z} (-1)^m t^{3m(15m+1)/2} - \sum_{m \in \Z} (-1)^m t^{(3m+1)(15m+6)/2} \\
 &=& \sum_{m \in \Z} (-1)^m (t^3)^{m(15m+1)/2} - \sum_{m \in \Z} (-1)^m t^{(45m^2+33m+6)/2} \\
 &=& \sum_{m \in \Z} (-1)^m (t^3)^{(15m^2+m)/2} - t^3\sum_{m \in \Z} (-1)^m (t^3)^{(15m^2+11m)/2} \\
  &=& \sum_{m \in \Z} (-1)^m (t^3)^{(15m^2+m)/2} - t^3\sum_{m \in \Z} (-1)^m (t^3)^{(15m^2+11m)/2} \\
  &=& \sum_{m \in \Z} (-1)^m  (t^3)^{8m(m+1)/2}(t^3)^{7m(m-1)/2} \\ &&- t^3\sum_{m \in \Z} (-1)^m (t^3)^{13m(m+1)/2}(t^3)^{2m(m-1)/2} \\
  &=& \prod_{n\geq 1} (1-(t^3)^{15n})(1-(t^3)^{15n - 7})(1-(t^3)^{15n - 8}) \\
  &&- t^3\prod_{n\geq 1} (1-(t^3)^{15n})(1-(t^3)^{15n - 2})(1-(t^3)^{15n - 13}) \\
 &=& G(t^3)
\end{eqnarray*}

\end{proof}

\begin{lem} \label{lem3.6}$[b(t)\vph(t)]_2 =  \sum_{m\geq 0} y_{3m+2} t^{3m+2}  = -t^2a(t^9)b(t^9) $
\end{lem}
\begin{proof}
Replacing $m$ by $-m-1$ to get from line 6 to line 7,
\begin{eqnarray*}
 \sum_{m\geq 0} y_{3m+2} t^{3m+2}  &=& \sum_{k \in \Z, k \equiv 2 \bmod 3} (-1)^k t^{k(5k+1)/2} \\
 &=&  \sum_{m \in \Z} (-1)^m t^{(3m+2)(15m+11)/2} \\
 &=& \sum_{m \in \Z} (-1)^m t^{(45m^2 + 63m + 22)/2} \\
&=& \sum_{m \in \Z} (-1)^m t^{(45m^2 + 63m + 18 +  4)/2} \\
&=& t^2 \sum_{m \in \Z} (-1)^m t^{(45m^2 + 63m + 18)/2} \\
\end{eqnarray*}
\begin{eqnarray*}
&=& t^2 \sum_{m \in \Z} (-1)^m (t^9)^{(5m^2 + 7m + 2)/2} \\
&=& t^2 \sum_{m \in \Z} (-1)^m (t^9)^{(m+1)(5m+2)/2} \\
&=& t^2 \sum_{m \in \Z} (-1)^{m+1} (t^9)^{(m)(5m+3)/2} \\
&=& -t^2 \sum_{m \in \Z} (-1)^{m} (t^9)^{(m)(5m+3)/2} \\
&=& -t^2 a(t^9)\vph(t^9)
\end{eqnarray*}
\end{proof}

Using Lemmas 3.2-3.6, a reduction of (\ref{REq2}) can be made as follows 
\begin{eqnarray} \label{REq3}
\vph(t^3)^2 &=& t^2 \left( [a(t)\vph(t)]_0 + [a(t)\vph(t)]_1 + [a(t)\vph(t)]_2 \right)\vph(t^9)a(t^9) \nonumber \\
&&+ \left( [b(t)\vph(t)]_0 + [b(t)\vph(t)]_1 + [b(t)\vph(t)]_2 \right)\vph(t^9)b(t^9) \\
&=& t^2 \left( [a(t)\vph(t)]_0 + [a(t)\vph(t)]_1 \right)\vph(t^9)a(t^9) \nonumber \\
&&+ \left( [b(t)\vph(t)]_0 + [b(t)\vph(t)]_2 \right)\vph(t^9)b(t^9) \\
%  \end{eqnarray}
%  \begin{eqnarray}
&=& t^2 \left( \vph(t^9)b(t^9) -tH(t^3) \right)\vph(t^9)a(t^9) \nonumber \\
&&+ \left( G(t^3) - t^2 a(t^9)\vph(t^9)  \right)\vph(t^9)b(t^9) \\
&=&  tH(t^3)(-t^2)\vph(t^9)a(t^9) + G(t^3)\vph(t^9)b(t^9) \\
&=&  tH(t^3)[b(t)\vph(t)]_2 + G(t^3)[a(t)\vph(t)]_0 
\end{eqnarray}

Therefore, Theorem \ref{thm br1} holds if and only if $$G(t^3)[a(t)\vph(t)]_1 + tH(t^3) [b(t)\vph(t)]_0 = 0$$
This obviously holds by using Lemmas 3.4 and 3.5, so Theorem 3.1 has been proven.
\end{proof}

\section{The Branching Rule Coefficients for $V^{\Lam_1}$ and $V^{\Lam_6}$}
\label{s3.3}

\begin{thm} \label{thm br2} For $i \in \{1,6\}$, the branching rules of $V^{\Lam_i}$ satisfying (\ref{eq 3.10}) and (\ref{eq 3.12}) can be expressed as the Virasoro characters:
$$c^1(t) = c^6(t) = t^{-19/30}\chi_{5,6}^{1,3}(t) \quad \mbox{and} \quad  d^1(t) = d^6(t) =  t^{-1/30}\chi_{5,6}^{2,3}(t),$$
so we have
\begin{eqnarray*}
gr_{princ}\left( V^{\Lam_i}\right) &=& t^{-19/10}\chi_{5,6}^{1,3}(t^3) gr_{princ}\left( W^{\Omg_0}\right) + t^{-1/10}\chi_{5,6}^{2,3}(t^3)gr_{princ}\left( W^{\Omg_4}\right).
\end{eqnarray*}
\end{thm}

\begin{proof}
Equations (2.19) and (2.20), give $$t^{-19/30}\chi_{5,6}^{1,3}(t) = \vph(t^3)a(t^3)/\vph(t)$$ and $$t^{-1/30}\chi_{5,6}^{2,3}(t) = \vph(t^3)b(t^3)/\vph(t).$$ 

Evaluating $c^1(t)$ and $d^1(t)$ at $t = t^3$, Theorem \ref{thm br2} is equivalent to 
\begin{eqnarray}
\frac{\vph(t^3)}{\vph(t)} &=& t^2 a(t)\vph(t^9)a(t^9)/\vph(t^3) + b(t)\vph(t^9)b(t^9)/\vph(t^3)
\end{eqnarray}
or
\begin{eqnarray}
\vph(t^3)^2 &=& t^2 a(t)\vph(t)a(t^9)\vph(t^9)+ b(t)\vph(t)b(t^9)\vph(t^9)
\end{eqnarray}
This equation is just (\ref{REq1}), which is which is the same Ramanujan identity used in the last section.  
\end{proof}

\chapter{The Vertex Operator Algebra Associated to $V^{\Lam_0}$ and Associated Structures}
\label{ch4}
% the label lets you refer to the chapter by number later.

\section{The Vertex Operator Algebra Associated to $V^{\Lam_0}$}
\label{s4.1}

This section contains the well known lattice (bosonic) construction for the vertex operator algebra (VOA) associated to $V^{\Lam_0}$, which can be found in \cite{FLM}.  This construction works for any even lattice, but will be presented here for the root lattice of type $E_6$.  We continue to using some of the notation defined in Section 1.2.4.

\begin{dfn}
\mylabel{Def: 4.1}
A vertex operator algebra is a $\Z$-graded vector space
$$V = \coprod_{n \in \Z} V_{(n)}; \mbox{ for } v \in V_{(n)}, n = wt(v)$$ 
and such that
$$dim (V_{(n)}) < \infty \mbox{ for } n \in \Z,$$
$$V_{(n)} = 0 \mbox{ for } n \mbox{ sufficiently small}.$$
$V$ is equipped with a linear map
$$\yvoz{\cdot} : V \to (End V)[[z,z^{-1}]]$$
$$v \mapsto \yvoz{v} = \sum_{n \in \Z} \{v\}_n z^{-n-1} \mbox{ (where } \{v\}_n \in End V)$$
and with two distinguished homogeneous vectors $1, \omg \in V$, and satisfying the following conditions for $u,v \in V$:
$$\{u\}_nv = 0 \mbox{ for } n \mbox{ sufficiently large; } $$
$$\yvoz{1} = I_V, \mbox{ the right hand side is the identity operator};$$
$$\yvoz{v}1 \in V[[z]] \mbox{ and } lim_{z \to 0} \yvoz{v}1 = v;$$
and the "Jacobi Identity" 
\begin{eqnarray*}
&&z_0^{-1}\delta\left( \frac{z_1-z_2}{z_0}\right)Y(u,z_1)Y(v,z_2)-z_0^{-1}\delta\left( \frac{z_2-z_1}{z_0}\right)Y(v,z_2)Y(u,z_1) \\
&&\hspace{10pt}= z_2^{-1}\delta\left( \frac{z_1-z_0}{z_2}\right)Y(Y(u,z_0)\cdot v,z_2).
\end{eqnarray*}
The distinguished conformal vector $\omg$ of weight 2, with 
$$\yvoz{\omg} = \sum_{n \in \Z}L(n)z^{-n-2} = \sum_{n \in \Z}\{ \omg\}_nz^{-n-1} $$ 
so that $L(n) = L_n = \{ \omg\}_{n+1}$, and the operators $L(n)$ satisfy the following 
$$\left[ L(m), L(n)\right] = (m-n)L(m+n) + \frac{m^3-m}{12}\delta_{m,-n}(rank V)I_V$$
where $rankV \in \C$ is a scalar uniquely determined by V.  The following also hold,
$$L(0)\cdot v = nv = wt(v)v \mbox{ for } n \in \Z \mbox{ and } v \in V_{(n)},$$
$$\yvoz{L(-1)\cdot v} = \frac{d}{dz}\yvoz{v}, \quad (\mbox{ the } L(-1) \mbox{ derivative property}),$$
$$L(n)\cdot 1 = 0, \mbox{ for } n \geq -1,$$
$$L(-2) \cdot 1 = \omg,$$
$$L(0) \cdot \omg = 2 \omg.$$
\end{dfn}
An automorphism of a a vertex operator algebra $V$ is a linear automorphism $\rho$ such that 
$$\rho \yvoz{v} \rho^{-1} = \yvoz{\rho v} \mbox{ for } v \in V,$$
$$\rho \omg = \omg.$$

It is known that the vector space $V_L$ constructed in Section 1.2.4 can be given a vertex operator algebra structure.  The weight of $v = \alp_1(-n_1) \cdots \alp_k(-n_k) \otea{\alp} \in V_L$ is given by $wt(v) = n_1 + n_2 + \dots + n_k + \frac{\herf{\alp}{\alp}}{2}$, providing the required $\Z$-grading of $V_L$ and we denote the $n^{th}$ graded piece by $(V_L)_{(n)}$.  In order to define $\yvoz{v}$ for all $v \in V_L$ it is sufficient to define it for $v \in (V_L)_{(n)}$ in the form written above.  We do this in such a way as to extend the previous definitions from Chapter 1.  For $\alp \in L$ define $\yvoz{\otea{\alp}} = \yvoz{\alp}$ as given in Definition 1.25 and for $k \geq 0$ define $$\yvoz{\alp(-k-1)\otz} = \alp^{(k)}(z) = \frac{1}{k!}\left( \frac{d}{dz} \right)^k\sum_{n \in \Z} \alp(n)z^{-n-1}  .$$  For $v = \alp_1(-n_1-1) \cdots \alp_k(-n_k-1) \otea{\alp}$ we now define
\begin{eqnarray*}
\yvoz{v} = \quad :  \yvoz{\alp_1(-n_1-1)\otz}  \cdots \yvoz{\alp_k(-n_k-1)\otz} \yvoz{\otea{\alp}}:.
\end{eqnarray*}

For $v = \otea{\alp}$, the new component operators, $\{ v\}_n$,  given in Definition \ref{Def: 4.1} are related to the old operators $\yvomn{v}{n}$, by matching the coefficients of corresponding powers of $z$.  This relationship is given by $\{ v \}_n = \yvomn{v}{n - \frac{1}{2}\herf{\alp}{\alp} + 1}$, hence when $wt(\alp) = \frac{1}{2}\herf{\alp}{\alp} = 1$, then $\{ v \}_n = \yvomn{v}{n}$.  Depending on the context, either of these notations will be used.  For $\alp \in L$ and $v = \alp(-1)\otimes e^0$, we have $\alp(n) = \{ v\}_n = \yvomn{v}{n}$.  The distinguished vectors $1$ is given by the vector $\otz \in V_L$ and $\omg$ is given by the vector $\omg_{E_6}$, which will be given in Section 5.1.

In order to calculate with these operators, the notation of residues will be used heavily.  The $Res$ operator will mean to take the coefficient of $z^{-1}$, and this variable will be $z$ unless otherwise stated.  So for $v = \alp_1(-n_1) \cdots \alp_k(-n_k) \otea{\alp}$, 
$$\yvomn{v}{n} = Res\left( z^{n+wt(\alp)-1} \yvoz{v}\right) \quad \mbox{  and  } \quad \{v\}_n = Res \left( z^n \yvoz{v}\right),$$
and in particular the Heisenberg operators can be written
$$\alp(n) = Res \left( z^n \yvoz{\alp(-1)\otimes e^0}\right).$$

Define the degree of the operator $\{v\}_n$ as
$$deg(\{v\}_n) = n - wt(v) + 1 = n - n_1 -n_2 - \cdots - n_k - wt(\alp) + 1.$$
For homogeneous elements $v,w \in V_L$ and $n \in \Z$, 
\begin{eqnarray} \mylabel{wtformula}
wt(\{v \}_n \cdot w) = wt(w) - deg(\{v\}_n) =wt(v) + wt(w) - n - 1.
\end{eqnarray}

Using the "Jacobi Identity" for VOA's the following can be proven as in FLM Corollary 8.6.5. If $\herf{\alp}{L} \subset \Z$, $\herf{\beta}{L} \subset \Z$, $u', v' \in S(\fhhz^-)$ and we set $$u=u'\otimes e^{\alp} \quad \mbox{ and } \quad v=v'\otimes e^{\beta}$$
then for $m,n \in \Z$,
\begin{eqnarray} \mylabel{VOABracket}
[\{ u\}_m, \{v\}_n] = \sum_{k \geq 0} \binom{m}{k}\{ \{ u\}_k \cdot v\}_{m+n-k} 
\end{eqnarray}
is a finite sum since $\{ u\}_k \cdot v = 0$ for $n$ sufficiently large.

The $L(-1)$ derivative property will be used later and therefore, some expansion is useful.  First, for a product of two Heisenberg operators, $\alp_1(z)\alp_2(z)$, the derivative is given by the product rule,
$$\frac{d}{dz}\left(: \alp_1(z)\alp_2(z) :\right) = \quad :\alp_1^{(1)}(z)\alp_2(z) + \alp_1(z)\alp_2^{(1)}(z):.$$

Note that the derivative of $\yvoz{\otea{\alp}}$ has the form of the derivative of an
exponential function.  For any $\alp \in L$,  
$$\frac{d}{dz}\yvoz{\otea{\alp}} = \quad  :\alp(z)\yvoz{\otea{\alp}}:.$$

\section{The Vertex Operator Algebra Modules for $V^{\Lam_0}$}
\label{s4.2}

\begin{dfn}
\mylabel{Def: 4.2}
A module for a vertex operator algebra $V$ is a $\Q$-graded vector space
$$W = \coprod_{m \in \Q} W_{(m)}; \mbox{ for } w \in W_{(m)}, m = wt(w)$$ 
and such that
$$dim (W_{(m)}) < \infty \mbox{ for } m \in \Q,$$
$$W_{(m)} = 0 \mbox{ for } m \mbox{ sufficiently small}.$$
$V$ is equipped with a linear map
$$\yvoz{\cdot} : V \to (End W)[[z,z^{-1}]]$$
$$v \mapsto \ywvoz{v} = \sum_{n \in \Z} \{v\}_n z^{-n-1} \mbox{ (where } \{v\}_n \in End W)$$
satisfying the following conditions for $u,v \in V$ and $w \in W$:
$$\{u\}_nw = 0 \mbox{ for } n \mbox{ sufficiently large }, $$
$$\ywvoz{1} = I_W, \mbox{ the right hand side is the identity operator on } W;$$
and the "Jacobi Identity" 
\begin{eqnarray*}
&&z_0^{-1}\delta\left( \frac{z_1-z_2}{z_0}\right)Y_W(u,z_1)Y_W(v,z_2)-z_0^{-1}\delta\left( \frac{z_2-z_1}{z_0}\right)Y_W(v,z_2)Y_W(u,z_1) \\
&&\hspace{10pt}= z_2^{-1}\delta\left( \frac{z_1-z_0}{z_2}\right)Y_W(Y(u,z_0)\cdot v,z_2).
\end{eqnarray*} 
$W$ is a $Vir$-module with operators given by  
$$\ywvoz{\omg} = \sum_{n \in \Z}L(n)z^{-n-2} = \sum_{n \in \Z}\{ \omg\}_nz^{-n-1}. $$ 

Furthermore, the following hold,
$$L(0)\cdot w = nw = wt(w)w \mbox{ for } m \in \Q \mbox{ and } w \in W_{(m)},$$
$$\ywvoz{L(-1)\cdot v} = \frac{d}{dz}\ywvoz{v}, \quad (\mbox{ the } L(-1) \mbox{ derivative property}).$$
Note that a VOA $V$ is a module for $V$.
\end{dfn}

In this work we will use three modules for $V^{\Lam_0}$.  These three modules turn out to be the three irreducible level one modules for $\Es$ which can be constructed using the weight lattice of type $E_6$, $P = P_{\Phi} = \sum_{i=1}^6 \Z \lam_i$ where $\Phi$ is the root system of type $E_6$.  As we saw in Chapter 1, the root lattice of type $E_6$, $Q = \sum_{i=1}^6 \Z \alp_i \sset P$, and in fact $Q$ is a normal subgroup of $P$ with index 3 and the quotient $P/Q$ consists of the three cosets, $P_0 = Q$, $P_1 = Q+\lam_1$, and $P_2 = Q+\lam_6$.  The quotient group $\{P_0,P_1,P_2 \} \cong \Z/3\Z$ with addition given by $P_i + P_j = P_{k}$, where $k \equiv (i+j)(\bmod 3)$.

In order to modify the lattice construction in Section 1.2.4 so that $P$ replaces $Q$, we need to extend the 2-cocyle to a $\tau$-invariant function $\veps: P\times P \to \{ \pm1\}$ which will be needed for the construction of intertwining operators.  Note that $\fh = \C \otimes Q = \C \otimes P$ so that $S(\hat{\fh}^-_{\Z})$ is not modified, but $\C[Q]$ is replaced by $\C[P]$.  Then $V_P = S(\hat{\fh}^-_{\Z}) \otimes \C[P] = V_Q \bigoplus V_{Q+\lam_1} \bigoplus V_{Q+\lam_6}$ where the subspaces $V_{P_0}$, $V_{P_1}$, and $V_{P_2}$ are spanned by  
$S(\hat{\fh}^-_{\Z}) \otimes e^{\lam}$ for $\lam \in P_0$, $P_1$, and $P_2$, respectively.  For $\alp \in \Phi$, $\{\otea{\alp} \}_n$ and $\alp(n)$ represent $\Es$ on $V_P$ so that $V_{P_0} = V^{\Lam_0}$, $V_{P_1} = V^{\Lam_1}$, and $V_{P_2} = V^{\Lam_6}$ are the irreducible level 1 $\fgt$-modules ($\Es$-modules).  In fact, the VOA structure on $V_Q = V^{\Lam_0}$ extends to a VOA module structure on the other subspaces.  For $i,j \in \{0,1,2\}$, $\lam \in P_i$, we have intertwining operators $\{ 1 \otimes e^\lam\}_n$ such that $\{ 1 \otimes e^\lam\}_n \cdot V_{P_j} \sset V_{P_{i+j}}$ where $i + j$ is taken modulo 3.  Following \cite{DL}, for any $v \in V_{P_i}$ and $\mu_i \in P_i$ we define the intertwining operators on $V_P$ by $$\iyvoz{v}{i}{\mu} = \yvoz{v}e^{\bi\pi \mu_{i}(0)}c( \cdot, \mu_i),$$ with the operators $e^{\bi\pi \mu_{i}(0)}$ and $c( \cdot, \mu_i)$ on $V_P$ given by 
\begin{eqnarray*}
&& e^{\bi\pi \mu_{i}(0)} \cdot u \otimes e^{\lam} = e^{\bi\pi 
\herf{\mu_{i}}{\lam}}u \otimes e^{\lam}, \\
&& c( \cdot, \mu_i) \cdot u \otimes e^{\lam} = c( \lam, \mu_i) \cdot u \otimes e^{\lam}.
\end{eqnarray*}
This also gives a 
One can check using the $\eps$ defined on $P$ below that for $\alp \in P_0$, $c(\alp, \mu_i) = (-1)^{\herf{\alp}{\mu_i}}$, so if $v \in V_{P_0}$ and $\mu_i \in P_0$, then $\iyvoz{v}{i}{\mu} = \yvoz{v}$.  One can also use this property to show that the definition of the new operator $\iyvoz{v}{i}{\mu}$ is well defined.  We also have a special case of the "Jacobi Identity" which will be needed in Chapter 6.  If $u \in V_{P_0}$ and $v \in V_{P_i}$, then 
\begin{eqnarray*}
&&z_0^{-1}\delta\left( \frac{z_1-z_2}{z_0}\right)Y(u,z_1)\mathcal{Y}_{\mu_i}(v,z_2)-z_0^{-1}\delta\left( \frac{z_2-z_1}{z_0}\right)\mathcal{Y}_{\mu_i}(v,z_2)Y(u,z_1) \\
&&\hspace{10pt}= z_2^{-1}\delta\left( \frac{z_1-z_0}{z_2}\right)\mathcal{Y}_{\mu_i}(Y(u,z_0)\cdot v,z_2).
\end{eqnarray*}
This gives a similar formula to (4.2) but with $n \in \Q$ rather than $n \in \Z$.
We also extend the definition of $\hat{\tau}$ on $V_Q$ to $V_P$, by defining $\hat{\tau}(\otea{\lam}) = \otea{\tau \lam}$ for any $\lam \in P$.  In \cite{DL} it is shown that a lattice automorphism, which is extended in the same manner as we have extended $\tau$ to $\hat{\tau}$, is an VOA automorphism of $V_P$, in the sense that $\hat{\tau} \yvoz{v} \hat{\tau} = \yvoz{\hat{\tau}(v)} $.

We now give the $\tau$ invariant bilinear 2-cocycle on $\{ \lam_i \mid 1\leq i \leq 6\}$ the basis for $P$ as
$$
[\eps(\lam_i,\lam_j)] = \left[ \begin{array}{r  r  r  r  r  r } 
 1 & 1 & 1 & 1 & 1 & 1\\ 
 -1 & 1 & 1 & 1 & 1 & -1\\ 
 -1 & 1 & 1 & 1 & 1 & 1\\ 
 1 & -1 & 1 & 1 & 1 & 1\\ 
 1 & 1 & 1 & 1 & 1 & -1\\ 
 1 & 1 & 1 & 1 & 1 & 1 \\
\end{array} \right].
$$

\chapter{The Conformal Vectors $\omg_{E_6}$, $\omg_{F_4}$, and $\omg = \omg_{E_6}-\omg_{F_4}$}
\label{ch5}
% the label lets you refer to the chapter by number later.
This chapter gives the explicit conformal vectors for the three Virasoro representations used in this work.  The Virasoro vector $\omg_{E_6}$ is well known, but the other two vectors $\omg_{F_4}$ and $\omg_{E_6} - \omg_{F_4}$ have not been explicitly computed.  The Virasoro brackets for the operators generated by these two vectors will be given implicitly for $\omg_{F_4}$, but explicitly for $\omg_{E_6} - \omg_{F_4}$.

\section{The Conformal Vector $\omg_{E_6}$}
\label{s5.1}

From (1.3), we construct the conformal vector $\omg_{E_6} := L(-2) \cdot \left( \otz \right)$ for the vertex operator representation of $\fgt$ of type $\Es$. It is well known that when constructing level 1 $\fgt$-modules where $\fg$ is of type $A,D,E$, the Sugawara construction reduces to the construction on the  Heisenberg subalgebra $\fhhz$ that is
 $$ \omg_{E_6} = \frac{1}{2}\sum_{i = 1}^{6}\alp_i(-1) \lam_i(-1).$$ 
Since $\omg_{E_6}$ generates the operators representing $Vir$ on the VOA $V^{\Lam_0}$ are
 $$L(m) := \yvomn{\omg_{E_6}}{m} = \{ \omg_{E_6}\}_{m+1}$$
 and $c_{Vir}$ is represented by the scalar 6.

\section{The Conformal Vector $\omg_{F_4}$}
\label{s5.2}

Since the algebra $\fat$ of type $\Ff$ is a subalgebra of $\fgt$ of type $\Es$, it is possible to implement the Sugawara construction for $\fat$. 
We may describe all the roots $\beta \in \Phi_{F_4}$ in terms of $\Phi_{E_6}$.  For any root $\alp \in \Phi_{E_6}$ fixed by $\tau$, $\beta = \alp = \tau\alp \in \Phi_{F_4}$ and the corresponding root vector $x_{\beta} = x_{\alp} \in \fa_{\beta}$.  If $\alp \in \Phi_{E_6}$ not fixed by $\tau$ we have $\beta = \frac{\alp +\tau\alp}{2} \in \Phi_{F_4}$ and the corresponding root vector $x_{\beta} = x_{\alp} + x_{\tau\alp} \in \fa_{\beta}$.  This provides a complete description of $\Phi_{F_4}$.  A basis for $\fa$ is then $B = \{x_\beta, h_{\beta_i} \mid \beta \in \Phi_{F_4}, 1 \leq i \leq 4 \}$.  The fundamental weights of $\fa$ are given in terms of the fundamental weights of $\fg$ as follows:
$$\omg_1 = \lam_2, \quad \omg_2 = \lam_4, \quad \omg_3 = \frac{\lam_3+\lam_5}{2}, \quad \omg_4 = \frac{\lam_1+\lam_6}{2}.$$   A dual basis of the Cartan subalgebra $\fb$ of $\fa$ is $\{ h_{\omg_i} \mid 1\leq i\leq 4\}$. 
The vectors in the basis $B^*$ dual to $B$ can be described as follows: if $\beta = \alp = \tau\alp$, then the dual vector to $x_{\beta}$ is $\eps(\beta, -\beta)x_{-\beta}$ and if $\alp \neq \tau\alp$, then the dual vector to $x_{\alp} + x_{\tau\alp}$ is $\frac{1}{2}\eps(\alp,\tau\alp)(x_{-\alp} + x_{-\tau\alp})$.  
To carry out the Sugawara construction for $\fat$ on a $\fgt$-module, $V^{\Lam_j}$, $j =0,1,6$, note that $V^{\Lam_j}$ decomposes into a direct sum of irreducible level 1 $\fat$-modules of the form $W^{\Omg_0-n\delta}$ and $W^{\Omg_4-n\delta}$ for $0 \leq n \in \Z$.  We need to use that $\Lam_j(c) = \Omg_i(c) = 1$ and the dual Coxeter number $k^{\vee}(F_4) = 9$.  If we relabel these basis vectors $B = \{u_i \mid 1\leq i \leq dim(\fa)  \}$ and $B^*  = \{u^i \mid 1\leq i \leq dim(\fa)  \}$, then we have
$$\omg_{F_4} = \frac{1}{20} \sum_{n\in Z} :u_i(-n)u^i(-2+n):\cdot (\otz).$$

\begin{thm}The vector $\omg_{F_4}$ is given by 
\begin{eqnarray}
&&\frac{1}{5}  \left[4\lam_1(-1)^2 - 4\lam_1(-1)\lam_3(-1) -\lam_1(-1)\lam_5(-1) + 2\lam_1(-1)\lam_6(-1) \right.\nonumber \\
&&+ 5\lam_2(-1)^2 - 5\lam_2(-1)\lam_4(-1) +4\lam_3(-1)^2 - 5\lam_3(-1)\lam_4(-1)  \nonumber \\
&&+ 2\lam_3(-1)\lam_5(-1) - \lam_3(-1)\lam_6(-1) + 5\lam_4(-1)^2 - 5\lam_4(-1)\lam_5(-1)   \nonumber\\
&&\left. + 4\lam_5(-1)^2 - 4\lam_5(-1)\lam_6(-1) +4\lam_6(-1)^2 \right]\otimes e^0 \nn \\
&&+\frac{1}{5}\left(\otea{\alp_1-\alp_6}+\otea{\alp_3-\alp_5} - \otea{\alp_1+\alp_3-\alp_5-\alp_6}\right) \nonumber \\
&&+ \frac{1}{5}\left(\otea{-\alp_1+\alp_6}+\otea{-\alp_3+\alp_5} - \otea{-\alp_1-\alp_3+\alp_5+\alp_6}\right).\nonumber 
\end{eqnarray}
\end{thm}

There are essentially three parts to explicitly writing this vector.  First, for the vectors in the CSA, which are represented by Heisenberg operators, the contribution is given by:
\newpage
\begin{eqnarray}
&&\frac{1}{20} \left( \yvomn{\alp_2(-1)\lam_2(-1)}{-2} + \yvomn{\alp_4(-1)\lam_4(-1)}{-2}\right) \cdot \otz \nonumber \\
&&\hspace{10pt}+\frac{1}{20}\yvomn{\frac{1}{2}(\alp_3+\alp_5)(-1)(\lam_3 + \lam_5)(-1)}{-2} \cdot \otz \nn \\
&&\hspace{10pt}+\frac{1}{20}\yvomn{\frac{1}{2}(\alp_1+\alp_6)(-1)(\lam_1 + \lam_6)(-1)}{-2} \cdot \otz \nonumber \\ 
&=&\frac{1}{20}\left( \sum_{r\in \Z} :\alp_2(r)\lam_2(-2-r): + \sum_{r\in \Z} :\alp_4(r)\lam_4(-2-r): \right)\cdot \otz \nonumber \\
&&\hspace{10pt}+\frac{1}{20}\sum_{r\in \Z} :\frac{1}{2}(\alp_3+\alp_5)(r)(\lam_3 + \lam_5)(-2-r):\cdot \otz \nonumber \\ 
&&\hspace{10pt} + \frac{1}{20}\sum_{r\in \Z} :\frac{1}{2}(\alp_1+\alp_6)(r)(\lam_1 + \lam_6)(-2-r): \cdot \otz \nonumber \\
&=&\frac{1}{20} \left( \alp_2(-1)\lam_2(-1) + \alp_4(-1)\lam_4(-1)\right) \cdot \otz \\ 
&&\hspace{10pt} +\frac{1}{40} \left( (\alp_3+\alp_5)(-1)(\lam_3 + \lam_5)(-1) +(\alp_1+\alp_6)(-1)(\lam_1 + \lam_6)(-1)\right) \cdot \otz. \nonumber
\end{eqnarray}

Next, examine the root vectors with $x_{\beta} = x_{\alp}$. Combining the contributions of both $x_{\beta} = x_{\alp}$ and $x_{-\beta} = x_{-\alp}$ we have the following contribution for each such pair
\begin{eqnarray*}
&&\frac{1}{20} \left( \eps(\alp,\alp):\yvomn{\otea{\alp}}{-1}\yvomn{\otea{-\alp}}{-1}:\right) \cdot \otz \\
&&+\frac{1}{20} \left(  \eps(-\alp,-\alp):\yvomn{\otea{-\alp}}{-1}\yvomn{\otea{\alp}}{-1}:\right) \cdot \otz \\
&&=\frac{1}{20} \left( \eps(\alp,\alp)\yvomn{\otea{\alp}}{-1}\cdot \otea{-\alp} + \eps(-\alp,-\alp)\yvomn{\otea{-\alp}}{-1}\cdot\otea{\alp}\right)  \\
&&=\frac{1}{20} \left( \eps(\alp,0)\frac{1}{2}(\alp(-2) + \alp(-1)^2)\otimes e^0 + \eps(-\alp,0)\frac{1}{2}(-\alp(-2) + (-\alp)(-1)^2)\otimes e^0\right)  \\
&&=\frac{1}{20} \alp(-1)^2\otimes e^0. 
\end{eqnarray*}
Summing over all root vectors of this type we obtain
\begin{eqnarray} 
&&\frac{1}{20}\left[ \alp_2(-1)^2 + \alp_4(-1)^2 + (\alp_2 +\alp_4)(-1)^2 + (\alp_3 +\alp_4 + \alp_5)(-1)^2 \right.  \nonumber \\
&&\hspace{10pt} + (\alp_2 + \alp_3 +\alp_4 + \alp_5)(-1)^2 + (\alp_1 + \alp_3 +\alp_4 + \alp_5 + \alp_6)(-1)^2  \\
&& \hspace{10pt}+ (\alp_2 + \alp_3 +2\alp_4 + \alp_5)(-1)^2 + (\alp_1 + \alp_2 + \alp_3 + \alp_4 + \alp_5 + \alp_6)(-1)^2 \nonumber\\
&& \hspace{10pt} + (\alp_1 + \alp_2 + \alp_3 + 2\alp_4 + \alp_5 + \alp_6)(-1)^2 + (\alp_1 + \alp_2 + 2\alp_3 + 2\alp_4 + 2\alp_5 + \alp_6)(-1)^2\nonumber\\
&& \hspace{10pt} + (\alp_1 + \alp_2 + 2\alp_3 + 3\alp_4 + 2\alp_5 + \alp_6)(-1)^2 \nonumber \\
&&\hspace{10pt} \left.+ (\alp_1 + 2\alp_2 + 2\alp_3 + 2\alp_4 + 2\alp_5 + \alp_6)(-1)^2\right]\otimes e^0.\nonumber
\end{eqnarray}
The last contribution involves the vectors in the basis where the roots are not fixed by $\tau$, and their contribution is given by: 
\begin{eqnarray}
&&\frac{1}{40} \eps(\alp,\alp)\!:\!(\yvomn{\otea{\alp}}{-1}+\yvomn{\otea{\tau\alp}}{-1})(\yvomn{\otea{-\alp}}{-1}+\yvomn{\otea{-\tau\alp}}{-1})\!:\!\cdot \otz \nonumber\\
&&=\frac{1}{40} \eps(\alp,\alp) \left(:\yvomn{\otea{\alp}}{-1}\yvomn{\otea{-\alp}}{-1}: + :\yvomn{\otea{\tau\alp}}{-1}\yvomn{\otea{-\tau\alp}}{-1}: \right.\nonumber\\
&&\hspace{10pt}\left.+:\yvomn{\otea{\alp}}{-1}\yvomn{\otea{-\tau\alp}}{-1}: + :\yvomn{\otea{\tau\alp}}{-1}\yvomn{\otea{-\alp}}{-1}:\right) \cdot \otz \nonumber\\
&&=\frac{1}{40}  \left(\frac{1}{2}(\alp(-2)+\alp(-1)^2)\otimes e^0 + \frac{1}{2}(\tau\alp(-2)+\tau\alp(-1)^2)\otimes e^0 \right.\\
&&\hspace{10pt}\left.+\eps(\alp,\alp)\yvomn{\otea{\alp}}{-1}\cdot\otea{-\tau\alp} + \eps(\alp,\alp)\yvomn{\otea{\tau\alp}}{-1}\cdot \otea{-\alp} \right). \nonumber
\end{eqnarray}
Examining the second line of (5.3) and noting $\herf{\alp}{-\tau\alp} = 0$, we give the first calculation in the second line (the second calculation on this line being similar)
\begin{eqnarray*}
&&\eps(\alp,\alp)\yvomn{\otea{\alp}}{-1}\cdot\otea{-\tau\alp}\\
&=&\eps(\alp,\alp)Res\left( \eps(\alp,-\tau\alp)z^{-1} \exp\left( \sum_{k \geq 1}\frac{ \alp(-k)}{k}z^k\right) \cdot\otea{\alp-\tau\alp} \right) \\
&=&\eps(\alp,\alp-\tau\alp)\otimes e^{\alp-\tau\alp}. 
\end{eqnarray*}
Hence (5.3) becomes,
\begin{eqnarray}
&&\frac{1}{40}  \left(\frac{1}{2}(\alp(-2)+\alp(-1)^2)\otimes e^0 + \frac{1}{2}(\tau\alp(-2)+\tau\alp(-1)^2)\otimes e^0 \right. \nonumber\\
&&\hspace{10pt}\left.+\eps(\alp,\alp-\tau\alp)\otimes e^{\alp-\tau\alp} + \eps(\tau\alp,\tau\alp-\alp)\otimes e^{\tau\alp-\alp} \right) \nonumber \\
&=&\frac{1}{40}  \left(\frac{1}{2}(\alp(-2)+\alp(-1)^2)\otimes e^0 + \frac{1}{2}(\tau\alp(-2)+\tau\alp(-1)^2)\otimes e^0 \right. \nonumber\\
&&\hspace{10pt}\left.+\eps(\alp,\alp-\tau\alp)(\otea{\alp-\tau\alp} + \otea{\tau\alp-\alp}) \right). \nonumber
\end{eqnarray}
Since this calculation holds for $\alp, -\alp \in \Phi_{E_6}$ not $\tau$ fixed, for each quadruple of roots, $\pm\alp, \pm\tau\alp$, the total contribution is
\begin{eqnarray*}
&&\frac{1}{40}  \left(\frac{1}{2}(\alp(-2)+\alp(-1)^2)\otimes e^0 + \frac{1}{2}(\tau\alp(-2)+\tau\alp(-1)^2)\otimes e^0 \right. \\
&&\hspace{10pt}\left.+\eps(\alp,\alp-\tau\alp)(\otea{\alp-\tau\alp} + \otea{\tau\alp-\alp}) \right.  \\
&&\hspace{10pt} +\left.\frac{1}{2}(-\alp(-2)+\alp(-1)^2)\otimes e^0 + \frac{1}{2}(-\tau\alp(-2)+\tau\alp(-1)^2)\otimes e^0 \right. \\
&&\hspace{10pt}\left.+\eps(-\alp,-\alp+\tau\alp)(\otea{-\alp+\tau\alp} + \otea{-\tau\alp+\alp}) \right) \\
&=&\frac{1}{40}  \left(\alp(-1)^2\otimes e^0 + \tau\alp(-1)^2\otimes e^0 +2\eps(\alp,\alp-\tau\alp)(\otea{\alp-\tau\alp} + \otea{\tau\alp-\alp}) \right).  
\end{eqnarray*}
There are two parts to consider. First for each pair $\alp$ and $\tau\alp$, their difference will be one of these roots of $D= \{ \alp_1-\alp_6, \alp_3-\alp_5, \alp_1+\alp_3-\alp_5-\alp_6\}$. The twelve possible pairs, $\{\alp,\tau\alp \}$, divide into four of each of the roots in $D$.  The 2-cocyle $\veps(\alp,\alp-\tau\alp)$ will be 1 if $\alp-\tau\alp = \pm(\alp_1-\alp_6)$ or $\pm(\alp_3-\alp_5)$, and will be -1 if $\alp-\tau\alp =\pm(\alp_1+\alp_3-\alp_5-\alp_6)$.  Adding all possible vectors of this type gives 
\begin{eqnarray}
&&\frac{1}{5}\left(\otea{\alp_1-\alp_6}+\otea{\alp_3-\alp_5} - \otea{\alp_1+\alp_3-\alp_5-\alp_6}\right) \nonumber \\
&+& \frac{1}{5}\left(\otea{-\alp_1+\alp_6}+\otea{-\alp_3+\alp_5} - \otea{-\alp_1-\alp_3+\alp_5+\alp_6}\right).
\end{eqnarray}
Secondly, evaluating $\frac{1}{40}  \left(\alp(-1)^2\otimes e^0 + \tau\alp(-1)^2\otimes e^0 \right)$ for all possible pairs of roots $\{\alp,\tau\alp \}$ will give
\begin{eqnarray}
&&\frac{1}{40}  \left[\alp_1(-1)^2 + \alp_3(-1)^2 +\alp_5(-1)^2 + \alp_6(-1)^2 +(\alp_1+\alp_3)(-1)^2  \right. \nonumber\\
&&+ (\alp_5+\alp_6)(-1)^2 +(\alp_3+\alp_4)(-1)^2 + (\alp_4+\alp_5)(-1)^2 + (\alp_1 +\alp_3+\alp_4)(-1)^2 \nonumber\\
&&+ (\alp_4+\alp_5+\alp_6)(-1)^2 +(\alp_2+\alp_3+\alp_4)(-1)^2 + (\alp_2+\alp_4+\alp_5)(-1)^2 \nonumber\\
&&+(\alp_1+\alp_2+\alp_3+\alp_4)(-1)^2 + (\alp_2+\alp_4+\alp_5+\alp_6)(-1)^2 \nonumber\\
&&+(\alp_1+\alp_3+\alp_4+\alp_5)(-1)^2 + (\alp_3+\alp_4+\alp_5+\alp_6)(-1)^2 \\
&&+(\alp_1+\alp_2+\alp_3+\alp_4+\alp_5)(-1)^2 + (\alp_2+\alp_3+\alp_4+\alp_5+\alp_6)(-1)^2 \nonumber\\
&&+(\alp_1+\alp_2+\alp_3+2\alp_4+\alp_5)(-1)^2 + (\alp_2+\alp_3+2\alp_4+\alp_5+\alp_6)(-1)^2 \nonumber\\
&&+(\alp_1+\alp_2+2\alp_3+2\alp_4+\alp_5)(-1)^2 + (\alp_2+\alp_3+2\alp_4+2\alp_5+\alp_6)(-1)^2\nonumber\\
&&+(\alp_1+\alp_2+2\alp_3+2\alp_4+\alp_5+\alp_6)(-1)^2 \nonumber\\
&&\left.+ (\alp_1+\alp_2+\alp_3+2\alp_4+2\alp_5+\alp_6)(-1)^2\right]\otimes e^0. \nonumber
\end{eqnarray}
The sum of (5.1), (5.2), (5.4), and (5.5), gives the vector $\omg_{F_4}$ according to the Sugawara construction. This presentation is not the best for calculations in later chapters, and therefore it will be rewritten in a more useful form.  To obtain this form, rewrite the $\alp_i$'s in terms of the $\lam_i$'s, hence giving,
\begin{eqnarray}
&&\omg_{F_4} \nonumber\\
&=&\frac{1}{5}  \left[4\lam_1(-1)^2 - 4\lam_1(-1)\lam_3(-1) -\lam_1(-1)\lam_5(-1) + 2\lam_1(-1)\lam_6(-1) \right.\nonumber \\
&&\hspace{20pt}+ 5\lam_2(-1)^2 - 5\lam_2(-1)\lam_4(-1) +4\lam_3(-1)^2 - 5\lam_3(-1)\lam_4(-1)  \nonumber \\
&&\hspace{20pt}+ 2\lam_3(-1)\lam_5(-1) - \lam_3(-1)\lam_6(-1) + 5\lam_4(-1)^2 - 5\lam_4(-1)\lam_5(-1)   \nonumber\\
&&\hspace{20pt}\left. + 4\lam_5(-1)^2 - 4\lam_5(-1)\lam_6(-1) +4\lam_6(-1)^2 \right]\otimes e^0 \\
&&+\frac{1}{5}\left(\otea{\alp_1-\alp_6}+\otea{\alp_3-\alp_5} - \otea{\alp_1+\alp_3-\alp_5-\alp_6}\right) \nonumber \\
&&+ \frac{1}{5}\left(\otea{-\alp_1+\alp_6}+\otea{-\alp_3+\alp_5} - \otea{-\alp_1-\alp_3+\alp_5+\alp_6}\right).\nonumber 
\end{eqnarray}

Thus $\yvomn{\omg_{F_4}}{m}$ will give Virasoro operators and the central charge is $\frac{26}{5}$. Since the most important calculation for the Virasoro operators is for the coset Virasoro brackets coming in the next section, the verification that these operators represent $Vir$ will not be given.  After the next section, this verification of the brackets for this $c = \frac{26}{5}$ Virasoro representation will then be inferred from the verification of the coset Virasoro brackets.

\section{The Coset Conformal Vector $\omg$}
\label{s5.3}
\begin{thm}
The coset Virasoro generator $\omg = \omg_{E_6} - \omg_{F_4}$ is given by the vector 
\begin{eqnarray}
\omg&=&\frac{1}{10}  \left[(-\lam_1+\lam_6)(-1)^2 + (\lam_3-\lam_5)(-1)^2\right]\otimes e^0 \nonumber \\
 &&+ \frac{1}{10}(\lam_1-\lam_3+\lam_5-\lam_6)(-1)^2 \otimes e^0 \\
&&+\frac{1}{5}\left(-\otea{\alp_1-\alp_6}-\otea{\alp_3-\alp_5} +\otea{\alp_1+\alp_3-\alp_5-\alp_6}\right) \nonumber \\
&&+ \frac{1}{5}\left(-\otea{-\alp_1+\alp_6}-\otea{-\alp_3+\alp_5} +\otea{-\alp_1-\alp_3+\alp_5+\alp_6}\right).\nonumber 
\end{eqnarray}
\end{thm}
We will show that the operators $L(n) = \{\omg\}_{n+1}$ satisfy the Virasoro bracket relations.  According to the theorem of Goddard-Kent-Olive, the central charge of the coset Virasoro representation will be the scalar $\frac{4}{5}$.  Making use of the Jacobi identity for vertex operator algebras it is enough to verify: 
\begin{eqnarray*}
[L(m), L(n)] &=&[\{\omg\}_{m+1}, \{\omg\}_{n+1}]  \nonumber \\
 &=& \sum_{0\leq k} \binom{m+1}{k}\{\{ \omg\}_k\cdot \omg\}_{m+n+2-k} \nonumber \\
 &=& (m-n)\{ \omg\}_{m+n+1}+\frac{m^3-m}{12}\delta_{m,-n}\frac{4}{5}I_{V_P} \nonumber\\
  &=& (m-n)L(m+n)+\frac{m^3-m}{12}\delta_{m,-n}\frac{4}{5}I_{V_P}. \nonumber
\end{eqnarray*}
Using (\ref{wtformula}) and (\ref{VOABracket}), since $wt(\{ \omg\}_k \cdot \omg) = 3-k$, we only need to consider $0 \leq k \leq 3$.  

First, we consider $\{\omg\}_0\cdot\omg$ and  since each of the products of Heisenberg operators will act similarly, we only show the following computation:
\begin{eqnarray}
&&\frac{1}{10}\yvomn{(-\lam_1+\lam_6)(-1)^2}{-1} \cdot \omg \nn\\
&&=\frac{1}{10}\sum_{r\in\Z}:(-\lam_1+\lam_6)(r)(-\lam_1+\lam_6)(-1-r):\cdot \omg \nn\\
&&=\frac{1}{5}\left( (-\lam_1+\lam_6)(-1)(-\lam_1+\lam_6)(0)+(-\lam_1+\lam_6)(-2)(-\lam_1+\lam_6)(1)\right) \cdot \omg \nn 
\end{eqnarray}
\begin{eqnarray}
&&=\frac{1}{5}(-\lam_1+\lam_6)(-1)(-\lam_1+\lam_6)(0)\cdot \frac{1}{5}\left(-\otea{\alp_1-\alp_6}-\otea{-\alp_1+\alp_6}\right) \nn \\
&&\hspace{10pt}+\frac{1}{5}(-\lam_1+\lam_6)(-1)(-\lam_1+\lam_6)(0)\cdot \frac{1}{5}\left(-\otea{\alp_3-\alp_5}-\otea{-\alp_3+\alp_5}\right) \nn \\
&&\hspace{10pt}+\frac{1}{5}(-\lam_1+\lam_6)(-1)(-\lam_1+\lam_6)(0)\cdot \frac{1}{5}\left(\otea{\alp_1+\alp_3-\alp_5-\alp_6}+\otea{-\alp_1-\alp_3+\alp_5+\alp_6}\right) \nn \\
&&\hspace{10pt}+\frac{1}{5}(-\lam_1+\lam_6)(-2)(-\lam_1+\lam_6)(1)\cdot \frac{1}{10}(-\lam_1+\lam_6)(-1)^2 \otimes e^0\nn \\
&&\hspace{10pt}+\frac{1}{5}(-\lam_1+\lam_6)(-2)(-\lam_1+\lam_6)(1)\cdot \frac{1}{10}(\lam_3-\lam_5)(-1)^2 \otimes e^0\nn \\
&&\hspace{10pt}+\frac{1}{5}(-\lam_1+\lam_6)(-2)(-\lam_1+\lam_6)(1)\cdot \frac{1}{10}(\lam_1-\lam_3+\lam_5-\lam_6)(-1)^2 \otimes e^0 \nn \\
&&=\frac{2}{25}(-\lam_1+\lam_6)(-1)[\otea{\alp_1-\alp_6}-\otea{-\alp_1+\alp_6} ]\nn \\
&&\hspace{10pt}+\frac{2}{25}(-\lam_1+\lam_6)(-1)[-\otea{\alp_1+\alp_3-\alp_5-\alp_6}+\otea{-\alp_1-\alp_3+\alp_5+\alp_6} ]\nn \\
&&\hspace{10pt}+\frac{4}{75}(-\lam_1+\lam_6)(-2)(-\lam_1+\lam_6)(-1) \otimes e^0 \nn \\
&&\hspace{10pt} -\frac{2}{75}(-\lam_1+\lam_6)(-2)(\lam_3-\lam_5)(-1) \otimes e^0\nn \\
&&\hspace{10pt} -\frac{2}{75}(-\lam_1+\lam_6)(-2)(\lam_1-\lam_3+\lam_5-\lam_6)(-1) \otimes e^0 \nn \\
&&=\frac{2}{25}(-\lam_1+\lam_6)(-1)[\otea{\alp_1-\alp_6}-\otea{-\alp_1+\alp_6} ] \\
&&\hspace{10pt}+\frac{2}{25}(-\lam_1+\lam_6)(-1)[-\otea{\alp_1+\alp_3-\alp_5-\alp_6}+\otea{-\alp_1-\alp_3+\alp_5+\alp_6} ]\nn \\
&&\hspace{10pt}+\frac{2}{25}(-\lam_1+\lam_6)(-2)(-\lam_1+\lam_6)(-1) \otimes e^0. \nn
\end{eqnarray}
The actions of the other two operators are,
\begin{eqnarray}
&&\frac{1}{10}\yvomn{(\lam_3-\lam_5)(-1)^2}{-1} \cdot \omg \nn\\
&&=\frac{2}{25}(\lam_3-\lam_5)(-1)[-\otea{\alp_3-\alp_5}+\otea{-\alp_3+\alp_5} ] \\
&&\hspace{10pt}+\frac{2}{25}(\lam_3-\lam_5)(-1)[\otea{\alp_1+\alp_3-\alp_5-\alp_6}-\otea{-\alp_1-\alp_3+\alp_5+\alp_6} ]\nn \\
&&\hspace{10pt}+\frac{2}{25}(\lam_3-\lam_5)(-2)(\lam_3-\lam_5)(-1) \otimes e^0 \nn
\end{eqnarray}
and
\begin{eqnarray}
&&\frac{1}{10}\yvomn{(\lam_1-\lam_3+\lam_5-\lam_6)(-1)^2}{-1} \cdot \omg \nn\\
&&=\frac{2}{25}(\lam_1-\lam_3+\lam_5-\lam_6)(-1)[-\otea{\alp_1-\alp_6}+\otea{-\alp_1+\alp_6} ] \\
&&\hspace{10pt}+\frac{2}{25}(\lam_1-\lam_3+\lam_5-\lam_6)(-1)[\otea{\alp_3-\alp_5}-\otea{-\alp_3+\alp_5} ]\nn \\
&&\hspace{10pt}+\frac{2}{25}(\lam_1-\lam_3+\lam_5-\lam_6)(-2)(\lam_1-\lam_3+\lam_5-\lam_6)(-1) \otimes e^0. \nn
\end{eqnarray}
Next, $\yvomn{\otea{\alp_1-\alp_6}}{-1} \cdot \omg$ will be computed, and the other five operators follow in a similar manner.  If we consider $\yvomn{\otea{\gamma_i}}{-1} \cdot \otea{\gamma_j}$, then the result will be 0 if $\herf{\gamma_i}{\gamma_j} \geq 0$.  Therefore only the non-zero actions will be given, and writing
\begin{eqnarray*}
B &=& \frac{1}{5}\left(-\otea{\alp_1-\alp_6}-\otea{\alp_3-\alp_5} +\otea{\alp_1+\alp_3-\alp_5-\alp_6}\right) \\
&&+ \frac{1}{5}\left(-\otea{-\alp_1+\alp_6}-\otea{-\alp_3+\alp_5} +\otea{-\alp_1-\alp_3+\alp_5+\alp_6}\right),
\end{eqnarray*}
we have
\begin{eqnarray}
&&\frac{1}{5}\yvomn{\otea{\alp_1-\alp_6}}{-1} \cdot B \nn\\
&&=\frac{1}{25}Res\left[z^0 \exp\left( \sum_{k \geq 1}\frac{ (\alp_1-\alp_6)(-k)}{k}z^k\right) \exp\left( \sum_{k \geq 1}\frac{ (\alp_1-\alp_6)(k)}{-k}z^{-k}\right) \right. \nn \\
&&\hspace{80pt} \left.e_{\alp_1-\alp_6}z^{(\alp_1-\alp_6)(0)}  \cdot \otea{-\alp_1+\alp_6}\right] \nn \\
&& \hspace{10pt}+\frac{1}{25}Res\left[z^0 \exp\left( \sum_{k \geq 1}\frac{(\alp_1-\alp_6)(-k)}{k}z^k\right) \exp\left( \sum_{k \geq 1}\frac{ (\alp_1-\alp_6)(k)}{-k}z^{-k}\right) \right. \nn \\
&&\hspace{80pt} \left.e_{\alp_1-\alp_6}z^{(\alp_1-\alp_6)(0)}  \cdot \otea{\alp_3-\alp_5}\right] \nn \\
&& \hspace{10pt}+\frac{1}{25}Res\left[z^0 \exp\left( \sum_{k \geq 1}\frac{(\alp_1-\alp_6)(-k)}{k}z^k\right) \exp\left( \sum_{k \geq 1}\frac{ (\alp_1-\alp_6)(k)}{-k}z^{-k}\right) \right. \nn \\
&&\hspace{80pt} \left.e_{\alp_1-\alp_6}z^{(\alp_1-\alp_6)(0)} \cdot \otea{-\alp_1-\alp_3+\alp_5+\alp_6}\right] \nn \\
&&=\frac{1}{25}Res\left[z^{-4} \exp\left( \sum_{k \geq 1}\frac{(\alp_1-\alp_6)(-k)}{k}z^k\right) \cdot \otz \right] \nn \\
&& \hspace{10pt}+\frac{1}{25}Res\left[z^{-2} \exp\left( \sum_{k \geq 1}\frac{ (\alp_1-\alp_6)(-k)}{k}z^k\right) \cdot \otea{\alp_1+\alp_3-\alp_5-\alp_6}\right] \nn \\
&& \hspace{10pt}-\frac{1}{25}Res\left[z^{-2} \exp\left( \sum_{k \geq 1}\frac{(\alp_1-\alp_6)(-k)}{k}z^k\right) \cdot \otea{-\alp_3+\alp_5}\right] \nn \\
&&=\frac{1}{25}\left( \frac{(\alp_1-\alp_6)(-3)}{3}+\frac{(\alp_1-\alp_6)(-2)(\alp_1-\alp_6)(-1)}{2} + \frac{(\alp_1-\alp_6)(-1)^3}{6}\right) \otimes e^0 \nn \\
&& \hspace{10pt}+\frac{1}{25}(\alp_1-\alp_6)(-1)\otimes e^{\alp_1+\alp_3-\alp_5-\alp_6} -\frac{1}{25}(\alp_1-\alp_6)(-1) \otimes e^{-\alp_3+\alp_5}. \nn 
\end{eqnarray}
The results of the other five calculations are given as follows:
\begin{eqnarray}
&&\frac{1}{5}\yvomn{\otea{-\alp_1+\alp_6}}{-1} \cdot B \nn\\
&&=\frac{1}{25}\left( \frac{-(\alp_1-\alp_6)(-3)}{3}+\frac{(\alp_1-\alp_6)(-2)(\alp_1-\alp_6)(-1)}{2}\right. \nn \\
&&\hspace{40pt}\left.+ \frac{-(\alp_1-\alp_6)(-1)^3}{6}\right) \otimes e^0\nn \\
&& \hspace{10pt}+\frac{1}{25}(-\alp_1+\alp_6)(-1)\otimes e^{-\alp_1-\alp_3+\alp_5+\alp_6} -\frac{1}{25}(-\alp_1+\alp_6)(-1) \otimes e^{\alp_3-\alp_5}, \nn 
\end{eqnarray}
\begin{eqnarray}
&&\frac{1}{5}\yvomn{\otea{\alp_3-\alp_5}}{-1} \cdot B \nn\\
&&=\frac{1}{25}\left( \frac{(\alp_3-\alp_5)(-3)}{3}+\frac{(\alp_3-\alp_5)(-2)(\alp_3-\alp_5)(-1)}{2} \right. \nn \\
&&\hspace{40pt}\left.+ \frac{(\alp_3-\alp_5)(-1)^3}{6}\right) \otimes e^0\nn \\
&& \hspace{10pt}+\frac{1}{25}(\alp_3-\alp_5)(-1)\otimes e^{\alp_1+\alp_3-\alp_5-\alp_6} -\frac{1}{25}(\alp_3-\alp_5)(-1) \otimes e^{-\alp_1+\alp_6}, \nn 
\end{eqnarray}
\begin{eqnarray}
&&\frac{1}{5}\yvomn{\otea{-\alp_3+\alp_5}}{-1} \cdot B \nn\\
&&=\frac{1}{25}\left( \frac{-(\alp_3-\alp_5)(-3)}{3}+\frac{(\alp_3-\alp_5)(-2)(\alp_3-\alp_5)(-1)}{2} \right. \nn \\
&&\hspace{40pt}\left.+ \frac{-(\alp_3-\alp_5)(-1)^3}{6}\right) \otimes e^0\nn \\
&& \hspace{10pt}+\frac{1}{25}(-\alp_3+\alp_5)(-1)\otimes e^{-\alp_1-\alp_3+\alp_5+\alp_6} -\frac{1}{25}(-\alp_3+\alp_5)(-1) \otimes e^{\alp_1-\alp_6}, \nn 
\end{eqnarray}
\begin{eqnarray}
&&\frac{1}{5}\yvomn{\otea{\alp_1+\alp_3-\alp_5-\alp_6}}{-1} \cdot B \nn\\
&&=\frac{1}{25}\left( \frac{(\alp_1+\alp_3-\alp_5-\alp_6)(-3)}{3} + \frac{(\alp_1+\alp_3-\alp_5-\alp_6)(-1)^3}{6}\right.\nn\\
&&\hspace{40pt}\left. +\frac{(\alp_1+\alp_3-\alp_5-\alp_6)(-2)(\alp_1+\alp_3-\alp_5-\alp_6)(-1)}{2} \right) \otimes e^0\nn \\
&& \hspace{10pt}-\frac{1}{25}(\alp_1+\alp_3-\alp_5-\alp_6)(-1)\otimes e^{\alp_3-\alp_5} -\frac{1}{25}(\alp_1+\alp_3-\alp_5-\alp_6)(-1) \otimes e^{\alp_1-\alp_6} ,\nn 
\end{eqnarray}
and
\begin{eqnarray}
&&\frac{1}{5}\yvomn{\otea{-\alp_1-\alp_3+\alp_5+\alp_6}}{-1} \cdot B \nn\\
&&=\frac{1}{25}\left( \frac{-(\alp_1+\alp_3-\alp_5-\alp_6)(-3)}{3} + \frac{-(\alp_1+\alp_3-\alp_5-\alp_6)(-1)^3}{6}  \right.\nn\\
&&\hspace{40pt}\left. +\frac{(\alp_1+\alp_3-\alp_5-\alp_6)(-2)(\alp_1+\alp_3-\alp_5-\alp_6)(-1)}{2}\right) \otimes e^0\nn \\
&& \hspace{10pt}-\frac{1}{25}(-\alp_1-\alp_3+\alp_5+\alp_6)(-1)\otimes e^{-\alp_3+\alp_5} \nn \\
&& \hspace{10pt} -\frac{1}{25}(-\alp_1-\alp_3+\alp_5+\alp_6)(-1) \otimes e^{-\alp_1+\alp_6}.\nn 
\end{eqnarray}
Adding these six results together, the total is 
\newline
\begin{eqnarray}
&&\frac{1}{25}\left[(\alp_1-\alp_6)(-2)(\alp_1-\alp_6)(-1)+ (\alp_3-\alp_5)(-2)(\alp_3-\alp_5)(-1) \right. \nn\\
&&\hspace{40pt}\left.+(\alp_1-\alp_3-\alp_5-\alp_6)(-2)(\alp_1-\alp_3-\alp_5-\alp_6)(-1)\right]\otimes e^0 \nn\\
&&+\frac{1}{25}\left[(\alp_1-\alp_6)(-1)\otimes e^{\alp_1-\alp_6} + (-\alp_1+\alp_6)(-1)\otimes e^{-\alp_1+\alp_6} \right] \\
&&+\frac{1}{25}\left[(\alp_3-\alp_5)(-1)\otimes e^{\alp_3-\alp_5} + (-\alp_3+\alp_5)(-1)\otimes e^{-\alp_3+\alp_5} \right] \nn\\
&&+\frac{1}{25}\left[(\alp_1+\alp_3-\alp_5-\alp_6)(-1)\otimes e^{\alp_1+\alp_3-\alp_5-\alp_6} \right] \nn\\
&&+\frac{1}{25}\left[(-\alp_1-\alp_3+\alp_5+\alp_6)(-1)\otimes e^{-\alp_1-\alp_3+\alp_5+\alp_6} \right].\nn
\end{eqnarray}
Only the calculation of $\yvomn{\otea{\alp_1-\alp_6}}{-1}$ on the products of Heisenberg operators in $\omg$ will be shown since the other five calculations are similar.  Writing
\begin{eqnarray*}
A&=&\frac{1}{10}  \left[(-\lam_1+\lam_6)(-1)^2 + (\lam_3-\lam_5)(-1)^2\right]\otimes e^0 \nonumber \\
 &&+ \frac{1}{10}(\lam_1-\lam_3+\lam_5-\lam_6)(-1)^2 \otimes e^0,
\end{eqnarray*}
we have for one particular part of $A$,
\begin{eqnarray*}
&&-\frac{1}{5}\yvomn{\otea{\alp_1-\alp_6}}{-1} \cdot (-\lam_1+\lam_6)(-1)^2 \otimes e^0 \\
&&=-\frac{1}{50}Res\left[ z^0 \exp\left( \sum_{k \geq 1}\frac{ (\alp_1-\alp_6)(-k)}{k}z^k\right) \exp\left( \sum_{k \geq 1}\frac{ (\alp_1-\alp_6)(k)}{-k}z^{-k}\right) \right. \nn \\
&&\hspace{80pt} \left.e_{\alp_1-\alp_6}z^{(\alp_1-\alp_6)(0)}  \cdot (-\lam_1+\lam_6)(-1)^2 \otimes e^0\right] \nn \\
&&= -\frac{1}{50}Res\left[ z^0 \exp\left( \sum_{k \geq 1}\frac{ (\alp_1-\alp_6)(-k)}{k}z^k\right) \cdot (-\lam_1+\lam_6)(-1)^2 \otimes e^{\alp_1-\alp_6}\right] \nn \\
&&= -\frac{1}{50}(2)(-2)(-\lam_1+\lam_6)(-1) \otimes e^{\alp_1-\alp_6} \nn \\
&&= \frac{2}{25}(-\lam_1+\lam_6)(-1) \otimes e^{\alp_1-\alp_6}. \nn
\end{eqnarray*}
So, determining the action on all three parts of $A$ in a similar manner, we have
\begin{eqnarray}
&&-\frac{1}{5}\yvomn{\otea{\alp_1-\alp_6}}{-1} \cdot A \nn\\
&&= -\frac{2}{25}\left[-(-\lam_1+\lam_6)(-1) \otimes e^{\alp_1-\alp_6} +(\lam_1-\lam_3+\lam_5-\lam_6)(-1)\otimes e^{\alp_1-\alp_6} \right]\nn \\
&&= -\frac{2}{25}(\alp_1-\alp_6)(-1)\otimes e^{\alp_1-\alp_6}.
\end{eqnarray}
Hence, we obtain the following similar actions on $A$,
\begin{eqnarray}
&&-\frac{1}{5}\yvomn{\otea{-\alp_1+\alp_6}}{-1} \cdot A \nn\\
&&= -\frac{2}{25}\left[(-\lam_1+\lam_6)(-1) \otimes e^{-\alp_1+\alp_6} -(\lam_1-\lam_3+\lam_5-\lam_6)(-1)\otimes e^{-\alp_1+\alp_6} \right]\nn \\
&&= -\frac{2}{25}(-\alp_1+\alp_6)(-1)\otimes e^{-\alp_1+\alp_6},
\end{eqnarray}
\begin{eqnarray}
&&-\frac{1}{5}\yvomn{\otea{\alp_3+\alp_5}}{-1} \cdot A \nn\\
&&= -\frac{2}{25}\left[-(\lam_3-\lam_5)(-1) \otimes e^{\alp_3-\alp_5} +(\lam_1-\lam_3+\lam_5-\lam_6)(-1)\otimes e^{\alp_3-\alp_5} \right]\nn \\
&&= -\frac{2}{25}(\alp_3-\alp_5)(-1)\otimes e^{\alp_3-\alp_5},
\end{eqnarray}
\begin{eqnarray}
&&-\frac{1}{5}\yvomn{\otea{-\alp_3+\alp_5}}{-1} \cdot A \nn\\
&&= -\frac{2}{25}\left[(\lam_3-\lam_5)(-1) \otimes e^{-\alp_3+\alp_5} -(\lam_1-\lam_3+\lam_5-\lam_6)(-1)\otimes e^{-\alp_3+\alp_5} \right]\nn \\
&&= -\frac{2}{25}(-\alp_3+\alp_5)(-1)\otimes e^{-\alp_3+\alp_5},
\end{eqnarray}
\begin{eqnarray}
&&\frac{1}{5}\yvomn{\otea{\alp_1+\alp_3-\alp_5-\alp_6}}{-1} \cdot A \nn\\
&&= \frac{2}{25}\left[-(-\lam_1+\lam_6)(-1) \otimes e^{\alp_1+\alp_3-\alp_5-\alp_6} + (\lam_3-\lam_5)(-1)\otimes e^{\alp_1+\alp_3-\alp_5-\alp_6} \right]\nn \\
&&= \frac{2}{25}(\alp_1+\alp_3-\alp_5-\alp_6)(-1)\otimes e^{\alp_1+\alp_3-\alp_5-\alp_6},
\end{eqnarray}
and
\begin{eqnarray}
&&\frac{1}{5}\yvomn{\otea{-\alp_1-\alp_3+\alp_5+\alp_6}}{-1} \cdot A \nn\\
&&= \frac{2}{25}\left[(-\lam_1+\lam_6)(-1) \otimes e^{-\alp_1-\alp_3+\alp_5+\alp_6} -(\lam_3-\lam_5)(-1)\otimes e^{-\alp_1-\alp_3+\alp_5+\alp_6} \right]\nn \\
&&= \frac{2}{25}(-\alp_1-\alp_3+\alp_5+\alp_6)(-1)\otimes e^{-\alp_1-\alp_3+\alp_5+\alp_6}.
\end{eqnarray}
Summing over the right hand sides of (5.8) - (5.17), we have \newpage
\begin{eqnarray*}
&&\yvomn{\omg}{-1} \cdot \omg \\
&&= -\frac{1}{5}(\alp_1-\alp_6)(-1)\otimes e^{\alp_1-\alp_6}-\frac{1}{5}(-\alp_1+\alp_6)(-1)\otimes e^{-\alp_1+\alp_6} \\
&&\hspace{10pt} -\frac{1}{5}(\alp_3-\alp_5)(-1)\otimes e^{\alp_3-\alp_5}-\frac{1}{5}(-\alp_3+\alp_5)(-1)\otimes e^{-\alp_3+\alp_5} \\
&&\hspace{10pt}+ \frac{1}{5}(\alp_1+\alp_3-\alp_5-\alp_6)(-1)\otimes e^{\alp_1+\alp_3-\alp_5-\alp_6}\\
&&\hspace{10pt}+\frac{1}{5}(-\alp_1-\alp_3+\alp_5+\alp_6)(-1)\otimes e^{-\alp_1-\alp_3+\alp_5+\alp_6} \\
&&\hspace{10pt}+\frac{1}{10}\bigg( 2(-\lam_1+\lam_6)(-2)(-\lam_1+\lam_6)(-1) + 2(\lam_3-\lam_5)(-2)(\lam_3-\lam_5)(-1) \\
&&\hspace{40pt}+ 2(\lam_1-\lam_3+\lam_5-\lam_6)(-2)(\lam_1-\lam_3+\lam_5-\lam_6)(-1)\bigg)\otimes e^0.
\end{eqnarray*}
Using the derivatives of vertex operators at the end of Section 4.1,
\begin{eqnarray*}
&&\yvoz{\yvomn{\omg}{-1}\cdot \omg} \\
&&= :\frac{1}{10}(-\lam_1+\lam_6)^{(1)}(z)(-\lam_1+\lam_6)(z)+\frac{1}{10}(-\lam_1+\lam_6)^{(1)}(z)(-\lam_1+\lam_6)(z):\\
&&\hspace{10pt}+:\frac{1}{10}(\lam_3-\lam_5)^{(1)}(z)(\lam_3-\lam_5)(z)+\frac{1}{10}(\lam_3-\lam_5)^{(1)}(z)(\lam_3-\lam_5)(z):\\
&&\hspace{10pt}+:\frac{1}{10}(\lam_1-\lam_3+\lam_5-\lam_6)^{(1)}(z)(\lam_1-\lam_3+\lam_5-\lam_6)(z):\\
&&\hspace{10pt}+:\frac{1}{10}(\lam_1-\lam_3+\lam_5-\lam_6)^{(1)}(z)(\lam_1-\lam_3+\lam_5-\lam_6)(z): \\
&&\hspace{10pt}-\frac{1}{5}\yvoz{(\alp_1-\alp_6)(-1)\otimes e^{\alp_1-\alp_6}} -\frac{1}{5}\yvoz{(-\alp_1+\alp_6)(-1)\otimes e^{-\alp_1+\alp_6}} \\
&&\hspace{10pt}-\frac{1}{5}\yvoz{(\alp_3-\alp_5)(-1)\otimes e^{\alp_3-\alp_5}} -\frac{1}{5}\yvoz{(-\alp_3+\alp_5)(-1)\otimes e^{-\alp_3+\alp_5}} \\
&&\hspace{10pt}+\frac{1}{5}\yvoz{(\alp_1+\alp_3-\alp_5-\alp_6)(-1)\otimes e^{\alp_1+\alp_3-\alp_5-\alp_6}} \\
&&\hspace{10pt}+\frac{1}{5}\yvoz{(-\alp_1-\alp_3+\alp_5+\alp_6)(-1)\otimes e^{-\alp_1-\alp_3+\alp_5+\alp_6}} \\
&&=\frac{d}{dz}\yvoz{\omg}.
\end{eqnarray*}
Since 
\begin{eqnarray}
\yvoz{\omg} = \sum_{p\in \Z} \{ \omg\}_pz^{-p-1}
\end{eqnarray}
we have
\begin{eqnarray} \mylabel{dery}
\frac{d}{dz}\yvoz{\omg} = \sum_{p\in \Z} (-p-1)\{ \omg\}_pz^{-p-2}.
\end{eqnarray}
To complete the calculation, the operator $\{ \{ \omg\}_0\cdot \omg\}_{m+n+2}$ is given by the coefficient of $z^{-m-n-3}$ in (\ref{dery}).  Hence $p = m+n+1$ and 
\begin{eqnarray}
\{ \{ \omg\}_0\cdot \omg\}_{m+n+2}  = -(m+n+2)\{ \omg \}_{m+n+1}.
\end{eqnarray}

The next contribution for the calculation of the bracket comes from the calculation of $\{ \omg\}_1\cdot \omg$.  Noting that the rest of the actions of Heisenberg operators follow similarly, we compute
\begin{eqnarray}
&&\frac{1}{10}\yvomn{(-\lam_1+\lam_6)(-1)^2}{0} \cdot \omg \nn\\
&&=\frac{1}{10}\sum_{r\in\Z}:(-\lam_1+\lam_6)(r)(-\lam_1+\lam_6)(-r):\cdot \omg \nn\\
&&=\frac{1}{10}\left( (-\lam_1+\lam_6)(0)(-\lam_1+\lam_6)(0)+2(-\lam_1+\lam_6)(-1)(-\lam_1+\lam_6)(1)\right) \cdot \omg \nn \\
&&=\frac{1}{10}(-\lam_1+\lam_6)(0)(-\lam_1+\lam_6)(0)\cdot \frac{1}{5}\left(-\otea{\alp_1-\alp_6}-\otea{-\alp_1+\alp_6}\right) \nn \\
&&\hspace{10pt}+\frac{1}{10}(-\lam_1+\lam_6)(0)(-\lam_1+\lam_6)(0)\cdot \frac{1}{5}\left(-\otea{\alp_3-\alp_5}-\otea{-\alp_3+\alp_5}\right) \nn \\
&&\hspace{10pt}+\frac{1}{10}(-\lam_1+\lam_6)(0)(-\lam_1+\lam_6)(0)\cdot \frac{1}{5}\left(\otea{\alp_1+\alp_3-\alp_5-\alp_6}+\otea{-\alp_1-\alp_3+\alp_5+\alp_6}\right) \nn \\
&&\hspace{10pt}+\frac{1}{5}(-\lam_1+\lam_6)(-1)(-\lam_1+\lam_6)(1)\cdot \frac{1}{10}(-\lam_1+\lam_6)(-1)^2 \otimes e^0\nn \\
&&\hspace{10pt}+\frac{1}{5}(-\lam_1+\lam_6)(-1)(-\lam_1+\lam_6)(1)\cdot \frac{1}{10}(\lam_3-\lam_5)(-1)^2 \otimes e^0\nn \\
&&\hspace{10pt}+\frac{1}{5}(-\lam_1+\lam_6)(-1)(-\lam_1+\lam_6)(1)\cdot \frac{1}{10}(\lam_1-\lam_3+\lam_5-\lam_6)(-1)^2 \otimes e^0 \nn \\
&&=-\frac{2}{25}[\otea{\alp_1-\alp_6}+\otea{-\alp_1+\alp_6} ]+\frac{2}{25}[\otea{\alp_1+\alp_3-\alp_5-\alp_6}+\otea{-\alp_1-\alp_3+\alp_5+\alp_6} ]\nn \\
&&\hspace{10pt}+\frac{4}{75}(-\lam_1+\lam_6)(-1)^2\otimes e^0 -\frac{2}{75}(-\lam_1+\lam_6)(-1)(\lam_3-\lam_5)(-1) \otimes e^0\nn \\
&&\hspace{10pt} -\frac{2}{75}(-\lam_1+\lam_6)(-1)(\lam_1-\lam_3+\lam_5-\lam_6)(-1) \otimes e^0 \nn \\
&&=-\frac{2}{25}[\otea{\alp_1-\alp_6}+\otea{-\alp_1+\alp_6} ]+\frac{2}{25}[\otea{\alp_1+\alp_3-\alp_5-\alp_6}+\otea{-\alp_1-\alp_3+\alp_5+\alp_6} ]\nn \\
&&\hspace{10pt}+\frac{2}{25}(-\lam_1+\lam_6)(-1)^2\otimes e^0. \mylabel{5.21}
\end{eqnarray}
The actions of the other two operators are given by
\begin{eqnarray}
&&\frac{1}{10}\yvomn{(\lam_3-\lam_5)(-1)^2}{0} \cdot \omg \nn\\
&&=-\frac{2}{25}[\otea{\alp_3-\alp_5}+\otea{-\alp_3+\alp_5} ] +\frac{2}{25}[\otea{\alp_1+\alp_3-\alp_5-\alp_6}+\otea{-\alp_1-\alp_3+\alp_5+\alp_6} ]\nn \\
&&\hspace{10pt}+\frac{2}{25}(\lam_3-\lam_5)(-1)^2 \otimes e^0, 
\end{eqnarray}
and
\begin{eqnarray}
&&\frac{1}{10}\yvomn{(\lam_1-\lam_3+\lam_5-\lam_6)(-1)^2}{0} \cdot \omg \nn\\
&&=-\frac{2}{25}[\otea{\alp_1-\alp_6}+\otea{-\alp_1+\alp_6} ] -\frac{2}{25}[\otea{\alp_3-\alp_5}+\otea{-\alp_3+\alp_5} ]\nn \\
&&\hspace{10pt}+\frac{2}{25}(\lam_1-\lam_3+\lam_5-\lam_6)(-1)^2 \otimes e^0. 
\end{eqnarray}

Only the calculation $\yvomn{\otea{\alp_1-\alp_6}}{0} \cdot \omg$ will be shown since the other five calculations are similar.  If we consider $\yvomn{\otea{\gamma_i}}{0}\cdot \otea{\gamma_j}$, then the result will be 0 if $(\gamma_i,\gamma_j) \geq -1$.  We only give the non-zero actions and therefore we have
\begin{eqnarray}
&&\frac{1}{5}\yvomn{\otea{\alp_1-\alp_6}}{0} \cdot B \nn\\
&&=\frac{1}{25}Res\left[z^1 \exp\left( \sum_{k \geq 1}\frac{ (\alp_1-\alp_6)(-k)}{k}z^k\right) \exp\left( \sum_{k \geq 1}\frac{ (\alp_1-\alp_6)(k)}{-k}z^{-k}\right) \right. \nn \\
&&\hspace{80pt} \left.e_{\alp_1-\alp_6}z^{(\alp_1-\alp_6)(0)} \cdot \otea{-\alp_1+\alp_6}\right] \nn \\
&& \hspace{10pt}+\frac{1}{25}Res\left[z^1 \exp\left( \sum_{k \geq 1}\frac{ (\alp_1-\alp_6)(-k)}{k}z^k\right) \exp\left( \sum_{k \geq 1}\frac{ (\alp_1-\alp_6)(k)}{-k}z^{-k}\right) \right. \nn \\
&&\hspace{80pt} \left.e_{\alp_1-\alp_6}z^{(\alp_1-\alp_6)(0)} \cdot \otea{\alp_3-\alp_5}\right] \nn \\
&& \hspace{10pt}+\frac{1}{25}Res\left[z^1 \exp\left( \sum_{k \geq 1}\frac{ (\alp_1-\alp_6)(-k)}{k}z^k\right) \exp\left( \sum_{k \geq 1}\frac{ (\alp_1-\alp_6)(k)}{-k}z^{-k}\right) \right. \nn \\
&&\hspace{80pt} \left.e_{\alp_1-\alp_6}z^{(\alp_1-\alp_6)(0)} \cdot \otea{-\alp_1-\alp_3+\alp_5+\alp_6}\right] \nn \\
&&=\frac{1}{25}Res\left[z^{-3} \exp\left( \sum_{k \geq 1}\frac{ (\alp_1-\alp_6)(-k)}{k}z^k\right) \cdot \otz \right] \nn \\
&& \hspace{10pt}+\frac{1}{25}Res\left[z^{-1} \exp\left( \sum_{k \geq 1}\frac{(\alp_1-\alp_6)(-k)}{k}z^k\right) \cdot \otea{\alp_1+\alp_3-\alp_5-\alp_6}\right] \nn \\
&& \hspace{10pt}-\frac{1}{25}Res\left[z^{-1} \exp\left( \sum_{k \geq 1}\frac{ (\alp_1-\alp_6)(-k)}{k}z^k\right) \cdot \otea{-\alp_3+\alp_5}\right] \nn \\
&&=\frac{1}{25}\left( \frac{(\alp_1-\alp_6)(-2)}{2}+ \frac{(\alp_1-\alp_6)(-1)^2}{2}\right) \otimes e^0 \nn \\
&& \hspace{10pt}+\frac{1}{25}\otimes e^{\alp_1+\alp_3-\alp_5-\alp_6} -\frac{1}{25}\otimes e^{-\alp_3+\alp_5}. \nn 
\end{eqnarray}
The results of the other five calculations are given by:
\begin{eqnarray}
&&\frac{1}{5}\yvomn{\otea{-\alp_1+\alp_6}}{0} \cdot B \nn\\
&&=\frac{1}{25}\left( -\frac{(\alp_1-\alp_6)(-2)}{2}+ \frac{(\alp_1-\alp_6)(-1)^2}{2}\right) \otimes e^0 +\frac{1}{25}\otimes e^{-\alp_1-\alp_3+\alp_5+\alp_6}  \nn \\
&& \hspace{10pt}-\frac{1}{25}\otimes e^{\alp_3-\alp_5}, \nn 
\end{eqnarray}
\begin{eqnarray}
&&\frac{1}{5}\yvomn{\otea{\alp_3-\alp_5}}{0} \cdot B \nn\\
&&=\frac{1}{25}\left( \frac{(\alp_3-\alp_5)(-2)}{2} + \frac{(\alp_3-\alp_5)(-1)^2}{2}\right) \otimes e^0 +\frac{1}{25}\otimes e^{\alp_1+\alp_3-\alp_5-\alp_6} \nn \\
&& \hspace{10pt}-\frac{1}{25} \otimes e^{-\alp_1+\alp_6}, \nn 
\end{eqnarray}
\begin{eqnarray}
&&\frac{1}{5}\yvomn{\otea{-\alp_3+\alp_5}}{0} \cdot B \nn\\
&&=\frac{1}{25}\left( -\frac{(\alp_3-\alp_5)(-2)}{2} + \frac{(\alp_3-\alp_5)(-1)^2}{2}\right) \otimes e^0 +\frac{1}{25}\otimes e^{-\alp_1-\alp_3+\alp_5+\alp_6} \nn \\
&& \hspace{10pt}-\frac{1}{25} \otimes e^{\alp_1-\alp_6}, \nn 
\end{eqnarray}
\begin{eqnarray}
&&\frac{1}{5}\yvomn{\otea{\alp_1+\alp_3-\alp_5-\alp_6}}{0} \cdot B \nn\\
&&=\frac{1}{25}\left( \frac{(\alp_1+\alp_3-\alp_5-\alp_6)(-2)}{2} + \frac{(\alp_1+\alp_3-\alp_5-\alp_6)(-1)^2}{2}\right) \otimes e^0\nn \\
&& \hspace{10pt}-\frac{1}{25}\otimes e^{\alp_3-\alp_5} -\frac{1}{25}\otimes e^{\alp_1-\alp_6}, \nn 
\end{eqnarray}
and
\begin{eqnarray}
&&\frac{1}{5}\yvomn{\otea{-\alp_1-\alp_3+\alp_5+\alp_6}}{0} \cdot B \nn\\
&&=\frac{1}{25}\left( -\frac{(\alp_1+\alp_3-\alp_5-\alp_6)(-2)}{2} + \frac{(\alp_1+\alp_3-\alp_5-\alp_6)(-1)^2}{2}\right) \otimes e^0\nn \\
&& \hspace{10pt}-\frac{1}{25}\otimes e^{-\alp_3+\alp_5} -\frac{1}{25} \otimes e^{-\alp_1+\alp_6}. \nn 
\end{eqnarray}
Adding these six calculations together, the resulting vector is
\begin{eqnarray}
&&\frac{1}{25}\left[(\alp_1-\alp_6)(-1)^2 + (\alp_3-\alp_5)(-1)^2 +(\alp_1-\alp_3-\alp_5-\alp_6)(-1)^2\right] \otimes e^0\nn\\
&&-\frac{1}{25}\left[\otea{\alp_1-\alp_6} + \otea{-\alp_1+\alp_6} + \otea{\alp_3-\alp_5} + \otea{-\alp_3+\alp_5}\right] \mylabel{5.24}\\
&&+\frac{1}{25}\left[\otea{\alp_1+\alp_3-\alp_5-\alp_6} + \otea{-\alp_1-\alp_3+\alp_5+\alp_6} \right]. \nn
\end{eqnarray}

Only the calculation $\yvomn{\otea{\alp_1-\alp_6}}{0}$ on the products of Heisenberg operators in $\omg$ will be shown since the other five calculations are similar.  We have
\begin{eqnarray*}
&&-\frac{1}{5}\yvomn{\otea{\alp_1-\alp_6}}{0} \cdot (-\lam_1+\lam_6)(-1)^2 \otimes e^0 \\
&&=-\frac{1}{50}Res\left[ z^1 \exp\left( \sum_{k \geq 1}\frac{ (\alp_1-\alp_6)(-k)}{k}z^k\right) \exp\left( \sum_{k \geq 1}\frac{ (\alp_1-\alp_6)(k)}{-k}z^{-k}\right) \right. \nn \\
&&\hspace{80pt} \left.e_{\alp_1-\alp_6}z^{(\alp_1-\alp_6)(0)}  \cdot (-\lam_1+\lam_6)(-1)^2 \otimes e^0\right] \nn 
\end{eqnarray*}
\begin{eqnarray*}
&&= -\frac{1}{50}Res\left[ z^1 \exp\left( \sum_{k \geq 1}\frac{ (\alp_1-\alp_6)(-k)}{k}z^k\right) \exp\left( \sum_{k \geq 1}\frac{ (\alp_1-\alp_6)(k)}{-k}z^{-k}\right)\right. \\
&&\hspace{80pt}\left. \cdot (-\lam_1+\lam_6)(-1)^2 \otimes e^{\alp_1-\alp_6}\right] \nn \\
&&= -\frac{1}{50}Res\left[ z^{1-2} \frac{(\alp_1-\alp_6)(1)^2}{2}\cdot (-\lam_1+\lam_6)(-1)^2 \otimes e^{\alp_1-\alp_6}\right] \nn \\
&&= -\frac{1}{50}(\alp_1-\alp_6, -\lam_1+\lam_6)^2 \otimes e^{\alp_1-\alp_6} \nn \\
&&= -\frac{2}{25}\otea{\alp_1-\alp_6}. \nn
\end{eqnarray*}
Determining the action on all three parts of $A$ in a similar manner, we have 
\begin{eqnarray*}
-\frac{1}{5}\yvomn{\otea{-\alp_1+\alp_6}}{-1} \cdot A &=& -\frac{4}{25}\otea{\alp_1-\alp_6}. 
\end{eqnarray*}
Hence, the following are the results of the action on the Heisenberg operators:
\begin{eqnarray*}
-\frac{1}{5}\yvomn{\otea{-\alp_1-\alp_6}}{0} \cdot A &=& -\frac{4}{25}\otea{-\alp_1+\alp_6}, \\
-\frac{1}{5}\yvomn{\otea{\alp_3+\alp_5}}{0} \cdot A &=& -\frac{4}{25}\otea{\alp_3-\alp_5}, \\
-\frac{1}{5}\yvomn{\otea{-\alp_3-\alp_5}}{0} \cdot A &=& -\frac{4}{25}\otea{-\alp_3+\alp_5}, \\
\frac{1}{5}\yvomn{\otea{\alp_1+\alp_3-\alp_5-\alp_6}}{0} \cdot A &=& \frac{4}{25}\otea{\alp_1+\alp_3-\alp_5-\alp_6}, \\
\frac{1}{5}\yvomn{\otea{-\alp_1-\alp_3+\alp_5+\alp_6}}{0} \cdot A &=& \frac{4}{25}\otea{-\alp_1-\alp_3+\alp_5+\alp_6}.
\end{eqnarray*}
Summing all these calculations gives
\begin{eqnarray*}
\yvomn{\omg}{0} \cdot \omg &=&\frac{2}{10}  \left[(-\lam_1+\lam_6)(-1)^2 + (\lam_3-\lam_5)(-1)^2\right]\otimes e^0 \nonumber \\
 &&+ \frac{2}{10}(\lam_1-\lam_3+\lam_5-\lam_6)(-1)^2 \otimes e^0 \\
&&+\frac{2}{5}\left(-\otea{\alp_1-\alp_6}-\otea{\alp_3-\alp_5} +\otea{\alp_1+\alp_3-\alp_5-\alp_6}\right) \nn \\
&&+ \frac{2}{5}\left(-\otea{-\alp_1+\alp_6}-\otea{-\alp_3+\alp_5} +\otea{-\alp_1-\alp_3+\alp_5+\alp_6}\right) \nn \\
&=& 2\omg,
\end{eqnarray*}
so that  
\begin{eqnarray}
\yvoz{\yvomn{\omg}{0}\cdot\omg } &=& 2\yvoz{\omg } \nn \\
&=& 2\sum_{p\in \Z} \{ \omg\}_p z^{-p-1}. \mylabel{5.25}
\end{eqnarray}
Hence the operator $\{\{ \omg\}_1\cdot\omg \}_{m+n+1}$ is given by the coefficient of $z^{-m-n-2}$ in (\ref{5.25}).  Therefore, $p = m+n+1$ and  
\begin{eqnarray}
\{\{ \omg\}_1\cdot \omg\}_{m+n+1}  = 2\{ \omg\}_{m+n+1} .\end{eqnarray}

The next contribution for the calculation of the bracket comes from the calculation of $\{ \omg\}_2\cdot \omg$.  We first compute
\begin{eqnarray}
&&\frac{1}{10}\yvomn{(-\lam_1+\lam_6)(-1)^2}{1} \cdot \omg \nn \\
&&=\frac{1}{10}\sum_{r\in\Z}:(-\lam_1+\lam_6)(r)(-\lam_1+\lam_6)(1-r):\cdot \omg \nn \\
&&=\frac{1}{10}\left(\dots+(-\lam_1+\lam_6)(-1)(-\lam_1+\lam_6)(2)+ 2(-\lam_1+\lam_6)(0)(-\lam_1+\lam_6)(1) \right.\nn \\
&&\hspace{60pt}\left.+ (-\lam_1+\lam_6)(-1)(-\lam_1+\lam_6)(2) +\dots\right) \cdot \omg \nn \\
&&= 0
\end{eqnarray}
since every term in this sum kills $\omg$.
The argument for the other two operators acting on $\omg$ are exactly the same, hence
\begin{eqnarray}
\frac{1}{10}\yvomn{(\lam_3-\lam_5)(-1)^2}{1} \cdot \omg &=& 0, \nn \\
\frac{1}{10}\yvomn{(\lam_1-\lam_3+\lam_5-\lam_6)(-1)^2}{1} \cdot \omg &=& 0. \nn 
\end{eqnarray}

The action of $\yvomn{\otea{\alp_1-\alp_6}}{1}$ will be computed and the other five operators follow in a similar manner. We have that 
\begin{eqnarray}
&&-\frac{1}{5}\yvomn{\otea{\alp_1-\alp_6}}{1} \cdot \omg \nn\\
&&=-\frac{1}{5}Res\left[z^2 \exp\left( \sum_{k \geq 1}\frac{ (\alp_1-\alp_6)(-k)}{k}z^k\right) \exp\left( \sum_{k \geq 1}\frac{ (\alp_1-\alp_6)(k)}{-k}z^{-k}\right) \right. \nn \\
&&\hspace{80pt} \left.e_{\alp_1-\alp_6}z^{(\alp_1-\alp_6)(0)} \cdot \omg \right]. \nn \\
&&=\frac{1}{25}\eps(\alp_1-\alp_6,-\alp_1+\alp_6)(\alp_1-\alp_6)(-1)\otimes e^0 \nn
\end{eqnarray}
since the only non-zero contribution is from the term $\otea{-\alp_1+\alp_6}$ in $\omg$ because \\ $\herf{\alp_1-\alp_6}{\gam} \geq -2$ for all other terms $\otea{\gam}$ and $\herf{\alp_1-\alp_6}{-\alp_1+\alp_6} = -4$.

Similarly, the other five calculations are 
\begin{eqnarray}
-\frac{1}{5}\yvomn{\otea{-\alp_1+\alp_6}}{1} \cdot \omg &=& \frac{1}{25}\eps(-\alp_1+\alp_6,\alp_1-\alp_6)(-\alp_1+\alp_6)(-1)\otimes e^0, \nn \\
-\frac{1}{5}\yvomn{\otea{\alp_3-\alp_5}}{1} \cdot \omg &=& \frac{1}{25}\eps(\alp_3-\alp_5,-\alp_3+\alp_5)(\alp_3-\alp_5)(-1)\otimes e^0, \nn \\
-\frac{1}{5}\yvomn{\otea{-\alp_3+\alp_5}}{1} \cdot \omg &=& \frac{1}{25}\eps(-\alp_3+\alp_5,\alp_3-\alp_5)(-\alp_3+\alp_5)(-1)\otimes e^0, \nn 
\end{eqnarray}
\begin{eqnarray}
&&-\frac{1}{5}\yvomn{\otea{\alp_1+\alp_3-\alp_5-\alp_6}}{1} \cdot \omg \nn \\
&&\hspace{10pt}= \frac{1}{25}\eps(\alp_1+\alp_3-\alp_5-\alp_6,-\alp_1-\alp_3+\alp_5+\alp_6)(\alp_1+\alp_3-\alp_5-\alp_6)(-1)\otimes e^0, \nn \\
&&-\frac{1}{5}\yvomn{\otea{-\alp_1-\alp_3+\alp_5+\alp_6}}{1} \cdot \omg \nn\\ 
&&\hspace{10pt}= \frac{1}{25}\eps(-\alp_1-\alp_3+\alp_5+\alp_6,\alp_1+\alp_3-\alp_5-\alp_6)(-\alp_1-\alp_3+\alp_5+\alp_6)(-1)\otimes e^0. \nn
\end{eqnarray}
Adding these six calculations together, the resulting action gives the 0 vector.  The sum gives
$$\yvomn{\omg}{1} \cdot \omg = 0. $$
So $\{ \{\omg\}_2\cdot \omg\}_{m+n} = 0 I_{V_P}$.

%%%%%%%%%%%%%%%
%%%%%%%%%%%%%%%%%%
%%%%%%%%%%%%%%%%%%%%%%

The last contribution for the calculation of the bracket comes from $\{ \omg\}_3\cdot \omg$.  In the following computation only the r=1 term contributes,
\begin{eqnarray}
&&\frac{1}{10}\yvomn{(-\lam_1+\lam_6)(-1)^2}{2} \cdot \omg \nn \\
&&=\frac{1}{10}\sum_{r\in\Z}:(-\lam_1+\lam_6)(r)(-\lam_1+\lam_6)(2-r):\cdot \omg \nn \\
&&=\frac{1}{100}(-\lam_1+\lam_6)(1)^2\cdot \left[(-\lam_1+\lam_6)(-1)^2 \otimes e^0 +(\lam_3-\lam_5)(-1)^2 \otimes e^0\right. \nn \\
&&\hspace{140pt}+\left.(\lam_1-\lam_3+\lam_5-\lam_6)(-1)^2 \otimes e^0 \right] \nn \\
&&=\frac{1}{100}\left[2\left(\frac{4}{3}\right)^2 + 2\left(-\frac{2}{3}\right)^2 + 2\left(-\frac{2}{3}\right)^2  \right]\otimes e^0 \nn \\
&&= \frac{4}{75}\otimes e^0.\nn
\end{eqnarray}
We have similarly
\begin{eqnarray}
\frac{1}{10}\yvomn{(\lam_3-\lam_5)(-1)^2}{2} \cdot \omg &=& \frac{4}{75}\otimes e^0 \nn \\
\frac{1}{10}\yvomn{(\lam_1-\lam_3+\lam_5-\lam_6)(-1)^2}{2} \cdot \omg &=& \frac{4}{75}\otimes e^0. \nn 
\end{eqnarray}

Only the calculation $\yvomn{\otea{\alp_1-\alp_6}}{2}$ on $\omg$ will be shown since the five other calculations are similar.  We have
\begin{eqnarray}
&&-\frac{1}{5}\yvomn{\otea{\alp_1-\alp_6}}{2} \cdot \omg \nn\\
&&=-\frac{1}{5}Res\left[z^3 \exp\left( \sum_{k \geq 1}\frac{ (\alp_1-\alp_6)(-k)}{k}z^k\right) \exp\left( \sum_{k \geq 1} \frac{(\alp_1-\alp_6)(k)}{-k}z^{-k}\right) \right. \nn \\
&&\hspace{80pt} \left.e_{\alp_1-\alp_6}z^{(\alp_1-\alp_6)(0)} \cdot \omg \right]. \nn 
\end{eqnarray}
\begin{eqnarray}
&&=\frac{1}{25}\eps(\alp_1-\alp_6,-\alp_1+\alp_6)\otz \nn\\
&&=\frac{1}{25}\otimes e^0, \nn
\end{eqnarray}
since the only non-zero contribution is from the term $\otea{-\alp_1+\alp_6}$ in $\omg$ because $\herf{\alp_1-\alp_6}{\gam} \geq -2$ for all other terms $\otea{\gam}$ and $\herf{\alp_1-\alp_6}{-\alp_1+\alp_6} = -4$.

Similarly, the other five calculations are now given by
\begin{eqnarray}
\frac{1}{5}\yvomn{\otea{-\alp_1+\alp_6}}{2} \cdot \omg &=& \frac{1}{25}\otimes e^0,\nn \\
\frac{1}{5}\yvomn{\otea{\alp_3-\alp_5}}{2} \cdot \omg &=& \frac{1}{25}\otimes e^0, \nn\\
\frac{1}{5}\yvomn{\otea{-\alp_3+\alp_5}}{2} \cdot \omg &=& \frac{1}{25}\otimes e^0, \nn \\ 
\frac{1}{5}\yvomn{\otea{\alp_1+\alp_3-\alp_5-\alp_6}}{2} \cdot \omg &=& \frac{1}{25}\otimes e^0, \nn \\
\frac{1}{5}\yvomn{\otea{-\alp_1-\alp_3+\alp_5+\alp_6}}{2} \cdot \omg &=& \frac{1}{25}\otimes e^0. \nn 
\end{eqnarray}
Therefore $\yvomn{\omg}{2}\cdot \omg = \frac{2}{5}\otimes e^0$ and 
\begin{eqnarray}
\yvoz{\yvomn{\omg}{2}\cdot \omg} &=& Y\left(\frac{2}{5}\otimes e^0,z\right) \nn \\
&=&\sum_{p\in\Z} \left\{\frac{2}{5}\otimes e^0\right\}_p z^{-p-1} .
\end{eqnarray}
Since the operator $\{\{\omg\}_2 \cdot \omg \}_{m+n-1}$ is given by the coefficient of $z^{-m-n}$, $p = m+n-1$.  The operators $\{\frac{2}{5}\otimes e^0\}_p = 0$, unless $p = -1$, and therefore $-1=m+n-1$ implies $m = -n$.  We have the following
\begin{eqnarray}
\{\{\omg \}_3 \cdot \omg\}_{m+n-1} &=&  \left\{\frac{2}{5}\otimes e^0\right\}_{m+n-1} \nn \\
&=& \delta_{m,-n} \frac{2}{5}I_{V_P}
\end{eqnarray}

Now putting together all the previous calculations, we have 
\begin{eqnarray*}
[L(m), L(n)]  &=&[\{\omg\}_{m+1}, \{\omg\}_{n+1}]   \nonumber \\
&=& \sum_{0\leq k} \binom{m+1}{k}\{\{ \omg\}_k\cdot \omg\}_{m+n+2-k} \nonumber \\
&=&\binom{m+1}{0}(-m-n-2) \{\omg\}_{m+n+1}  +  \binom{m+1}{1} 2\{\omg\}_{m+n+1}   \\
&& +  \binom{m+1}{2} 0I_{V_P} +  \binom{m+1}{3} \delta_{m,-n} \frac{2}{5} I_{V_P}\\
&=& \left[(-m-n-2)+   2(m+1)\right]\{\omg\}_{m+n+1} +  \frac{m^3-m}{6} \delta_{m,-n} \frac{2}{5} I_{V_P} 
\end{eqnarray*}
\begin{eqnarray*}
&=& (m-n)\{ \omg\}_{m+n+1}+\frac{m^3-m}{12}\delta_{m,-n}\frac{4}{5}I_{V_P}\nonumber\\
&=&  (m-n)L(m+n)+\frac{m^3-m}{12}\delta_{m,-n}\frac{4}{5}I_{V_P}. \nonumber
\end{eqnarray*}

Therefore the vector $\omg$ generates a representation of the Virasoro algebra with central charge $\frac{4}{5}$.  Note that one can now recover the Virasoro brackets for the representation generated by $\omg_{F_4}$, by using the brackets the Virasoro representation generated by $\omg_{E_6}- \omg$.

\chapter{Verification of $Vir \otimes \Ff$-module Highest Weight Vectors }
\label{ch6}
% the label lets you refer to the chapter by number later.
This chapter gives the calculations which determine the summands in the decomposition for each $V^{\Lam_i}$ given in Chapter 3.  Each of these summands $Vir\left(\frac{4}{5},h\right) \otimes W^{\Omg_j}$ is determined by a HWV with respect to both $Vir$ and $\fat$.  In such a summand, the conditions for $v \in V_P$ to be a HWV for these two algebras are:
\begin{enumerate}
\item $\yvomn{\otea{-\theta}}{1} \cdot v = 0$.
\item $\yvomn{\beta_i}{0} \cdot v = 0, \, \, \, \forall \beta_i \in \Delta_{F_4}$.
\item $\yvomn{\omg}{1} \cdot v = 0$ and $\yvomn{\omg}{2} \cdot v = 0$.
\item $\yvomn{\omg}{0} \cdot v = hv$.
\item $\yvomn{\beta_i(-1)\otz}{0} \cdot v = \beta_i(0)\cdot v = \herf{\omg_j}{\beta_i} v, \, \, \, \forall \beta_i \in \Delta_{F_4}$.
\end{enumerate} 
When $\ds v \in \oplus_{k = 1}^{l} S(\fhhz) \otimes e^{\nu_k}$, condition 5 is equivalent to $Proj(\nu_k) = \omg_j$, for $1 \leq k \leq l$.  Condition 5 is simple to check in each case and is not discussed below.  Also we define
$$
\yvomn{\beta_i}{n} = \{\beta_i\}_n = \left\{ \begin{array}{ccl} 
\yvomn{\otea{\beta_i}}{n} & \mbox{for} & i = 1,2 \\ 
\yvomn{\otea{\alp_3}}{n} + \yvomn{\otea{\alp_5}}{n} & \mbox{for} & i = 3  \\ 
\yvomn{\otea{\alp_1}}{n} + \yvomn{\otea{\alp_6}}{n}& \mbox{for} & i = 4.  
\end{array}\right.
$$

\section{The Decomposition of $V^{\Lam_0}$}
\label{s6.1}
We now find four highest weight vectors for the summands in the decomposition of $V^{\Lam_0}$.

\begin{lem} $\otz$ is the HWV in $Vir\left( \frac{4}{5}, 0 \right) \otimes W^{\Omg_0}$.
\end{lem}

\begin{proof} 

1) $\yvomn{\otea{-\theta}}{1} \cdot \otz =0 $

Using (\ref{wtformula}), \, $wt(\{ \otea{-\theta}\}_1 \cdot \otz) = -1$ and therefore this action is 0. \\
%\begin{eqnarray*}
%&&\yvomn{-\theta}{1} \cdot \otz \\
%&&= Res \left[ z^1 \exp\left( \frac{\sum_{k \geq 1} -\theta(-k)}{k}z^k\right) \exp\left( \frac{\sum_{k \geq 1} -\theta(k)}{-k}z^{-k}\right) e^{-\theta}z^{-\theta(0)}\eps_{-\theta} \cdot \otz\right] \\
%&&= Res \left[ z^{1} \exp\left( \frac{\sum_{k \geq 1} -\theta(-k)}{k}z^k\right) \exp\left( \frac{\sum_{k \geq 1} -\theta(k)}{-k}z^{-k}\right) \cdot 1 \otimes e^{-\theta}\right] \\
%&&= Res \left[ z^{1} \left( I + \frac{-\theta(-1)}{1}z^1 \dots \right) \cdot 1 \otimes e^{-\theta}\right] \\
%&&= 0
%\end{eqnarray*}

2) $\yvomn{\beta}{0} \cdot \otz = 0, \forall \beta \in \Delta_{F_4}$ \\
We have for each $\alp \in \Delta_{E_6}$,
\begin{eqnarray*}
&&\yvomn{\otea{\alp}}{0} \cdot 1\ \otimes e^0 \\
&&= Res \left[ \exp\left( \sum_{k \geq 1}\frac{ \alp(-k)}{k}z^k\right) \exp\left( \sum_{k \geq 1}\frac{\alp(k)}{-k}z^{-k}\right) e_{\alp}z^{\alp(0)} \cdot \otz\right] \\
&&= Res \left[ \exp\left( \sum_{k \geq 1}\frac{ \alp(-k)}{k}z^k\right) \exp\left( \sum_{k \geq 1}\frac{ \alp(k)}{-k}z^{-k}\right) \cdot 1 \otimes e^{\alp}\right] \\
&&= Res \left[ \left( I + \frac{\alp(-1)}{1}z^1 \dots \right) \cdot 1 \otimes e^{\alp}\right] \\
&&= 0.
\end{eqnarray*}
Therefore $\yvomn{\beta}{0} \cdot \otz = 0$ for each $\beta \in \Delta_{F_4}$. \\

3) Remember the $\omg$, which generates the coset Virasoro, is 
\begin{eqnarray*} \omg &=&  \frac{1}{10}\left[ (-\lam_1 + \lam_6)(-1)^2 + (\lam_3 - \lam_5)(-1)^2 +  (\lam_1 - \lam_3 + \lam_5 - \lam_6)(-1)^2\right]\otimes e^0 \\
&&+ \frac{1}{5}\left[ -1 \otimes e^{\pm\gamma_1} - 1 \otimes e^{\pm\gamma_2} + 1 \otimes e^{\pm\gamma_3} \right],
\end{eqnarray*} 
where $\gamma_1 = \alp_1 - \alp_6$, $\gamma_2 = \alp_3 - \alp_5$, and $\gamma_3 = \gamma_1 + \gamma_2$.  The action of the Virasoro operators will be broken into a sum of two parts with $\omg = A + B$, where $A$ is the first line of $\omg$ above and $B$ is the second line.  

To give $\yvomn{\omg}{1} \cdot \otz$, use (\ref{wtformula}), $wt(\{ \omg\}_2 \cdot \otz) = -1$ and hence the action is 0.  Using (\ref{wtformula}) to check $\yvomn{\omg}{2} \cdot \otz$, we have $ wt(\{ \omg\}_3 \cdot \otz) = -2$ and therefore the action is 0. Hence all positive Virasoro operators will kill $\otz$.

4) $\otz$ has eigenvalue 0 for $\yvomn{\omg}{0}$.  The action of $\yvomn{(\lam_1+\lam_6)(-1)^2}{0}$ will be computed and the other three operators, from $A$, follow in a similar manner.  Only the $r = 0$ term could contribute in this calculation, so we have  
\begin{eqnarray*}
&&\yvomn{(\lam_1+\lam_6)(-1)^2}{0} \cdot \otz \\
&&= \sum_{r \in \Z} : (\lam_1+\lam_6)(r)(\lam_1+\lam_6)(-r): \cdot \otz \\
&&= 0(\otz).
\end{eqnarray*}
Hence the action of part A on $\otz$ will give eigenvalue 0 for $\otz$. \\
To check the action of part B, for any $\gamma_i$,
\begin{eqnarray*}
&&\yvomn{\otea{\gamma_i}}{0} \cdot  \otz \\
&&= Res \left[z \exp\left( \sum_{k \geq 1}\frac{ \gamma_i(-k)}{k}z^k\right) \exp\left( \sum_{k \geq 1}\frac{\gamma_i(k)}{-k}z^{-k}\right) e_{\gamma_i}z^{\gamma_i(0)} \cdot \otz\right] \\
&&= Res \left[ z\exp\left( \sum_{k \geq 1}\frac{\gamma_i(-k)}{k}z^k\right) \exp\left( \sum_{k \geq 1}\frac{ \gamma_i(k)}{-k}z^{-k}\right) \cdot 1 \otimes e^{\gamma_i}\right] \\
&&= Res \left[ z \left( I + \frac{\gamma_i(-1)}{1}z^1 \dots \right) \cdot 1 \otimes e^{\gamma_i}\right] \\
&&= 0.
\end{eqnarray*} 
Hence all of B will give eigenvalue 0 for $\otz$ and $\otz$ has eigenvalue 0 for $\yvomn{\omg}{0}$.

\end{proof}

%%%%%%%%%%%%%%% Second vector

We set $\mu = (1,1,2,2,1,1) , \tau\mu = (1,1,1,2,2,1) \in \Phi_{E_6}$, where the ordered sextuples are the coefficients in the basis of simple roots.  These are useful because they satisfy $Proj(\mu) = Proj(\tau\mu) = \omg_4$.  We have two facts to help complete the calculations below:
\begin{description}
\item{a)} $\alp_3 + \tau\mu = \alp_5 + \mu$ 
\item{b)} A table of dot products which make the calculations much quicker.

$
 \begin{array}{r | r | r | r | r | r | r | r | r | r | r | r | r | r | r | r | r}
  & \alp_1 & \alp_2 & \alp_3 & \alp_4 & \alp_5 & \alp_6 & -\theta & \gamma_1 & \gamma_2 & \gamma_3 & \lam_1 & \lam_2 & \lam_3 & \lam_4 & \lam_5 & \lam_6\\ \hline 
 \mu & 0 & 0 & 1 & 0 & -1 & 1 & -1 & -1 & 2 & 1 & 1 & 1 & 2 & 2 & 1 & 1\\ \hline 
 \tau\mu & 1 & 0 & -1 & 0 & 1 & 0 & -1 & 1 & -2 & -1 & 1 & 1 & 1 & 2 & 2 & 1\\  
\end{array}$
\end{description}

\begin{lem}$\otea{\mu} - \otea{\tau\mu}$ is a HWV in $Vir\left(\frac{4}{5}, \frac{2}{5}\right) \otimes W^{\Omg_4}$.
\end{lem}

\begin{proof}
1)$\yvomn{\otea{-\theta}}{1} \cdot (\otea{\mu} - \otea{\tau\mu}) =0 $ \\
Since the action on $\otea{\tau\mu}$ is similar, we only give the action on $\otea{\mu}$.  We have 
\begin{eqnarray*}
&&\yvomn{\otea{-\theta}}{1} \cdot \otea{\mu} \\
&&= Res \left[ z^1 \exp\left( \sum_{k \geq 1}\frac{ -\theta(-k)}{k}z^k\right) \exp\left( \sum_{k \geq 1}\frac{ -\theta(k)}{-k}z^{-k}\right) e_{-\theta}z^{-\theta(0)}\cdot \otea{\mu}\right] \\
&&= Res \left[ \eps(-\theta, \mu)\exp\left( \sum_{k \geq 1}\frac{ -\theta(-k)}{k}z^k\right) \exp\left( \sum_{k \geq 1}\frac{ -\theta(k)}{-k}z^{-k}\right) \cdot 1 \otimes e^{\mu-\theta}\right] \\
&&= Res \left[ \eps(-\theta, \mu)\left( I + \frac{-\theta(-1)}{1}z^1 \dots \right) \cdot 1 \otimes e^{\mu-\theta}\right] \\
&&= 0. \\
\end{eqnarray*}
Therefore using both calculations, $\yvomn{\otea{-\theta}}{1} \cdot  (\otea{\mu} - \otea{\tau\mu}) = 0$.

2) $\yvomn{\beta}{0} \cdot (\otea{\mu} - \otea{\tau\mu}) = 0$, $\forall \beta \in \Delta_{F_4}$.

For any $\alp \in \Phi_{E_6}$, we have
\begin{eqnarray*}
&&\yvomn{\otea{\alp}}{0} \cdot (\otea{\mu} - \otea{\tau\mu}) \\
&&= Res \left[ \exp\left( \sum_{k \geq 1}\frac{ \alp(-k)}{k}z^k\right) \exp\left( \sum_{k \geq 1}\frac{\alp(k)}{-k}z^{-k}\right) e_{\alp}z^{\alp(0)} \cdot (\otea{\mu} - \otea{\tau\mu})\right] \\
&&= Res \bigg[ \eps(\alp,\mu)z^{\herf{\alp}{\mu}}\exp\left( \sum_{k \geq 1}\frac{ \alp(-k)}{k}z^k\right) \exp\left( \sum_{k \geq 1}\frac{ \alp(k)}{-k}z^{-k}\right) \cdot \otea{\alp +\mu} \\
&&\hspace{.45in}-  \eps(\alp,\tau\mu)z^{\herf{\alp}{\tau\mu}}\exp\left( \sum_{k \geq 1}\frac{ \alp(-k)}{k}z^k\right) \exp\left( \sum_{k \geq 1}\frac{ \alp(k)}{-k}z^{-k}\right) \cdot \otea{\alp+\tau\mu} \bigg] \\
&&= Res \bigg[ \eps(\alp,\mu)z^{\herf{\alp}{\mu}}\exp\left( \sum_{k \geq 1}\frac{ \alp(-k)}{k}z^k\right) \cdot \otea{\alp +\mu} \\
&&\hspace{.45in}-  \eps(\alp,\tau\mu)z^{\herf{\alp}{\tau\mu}}\exp\left( \sum_{k \geq 1}\frac{\alp(-k)}{k}z^k\right) \cdot \otea{\alp+\tau\mu} \bigg]. \\
%&=& Res \left[ \left( I + \frac{\beta(-1)}{1}z^1 \dots \right) \cdot 1 \otimes e^{\beta}\right] \\
%&=& 0
\end{eqnarray*}

Notice the table above does not contain any dot products for $\beta$'s, but we can get this information by taking appropriate linear combinations of $\alp$'s.  The action computed above could only be non-zero if $\herf{\alp}{\mu} < 0$ or $\herf{\alp}{\tau\mu} < 0$. It happens when $\alp = \alp_3$ because $\herf{\alp_3}{\tau\mu} = -1$ and when $\alp = \alp_5$ because $\herf{\alp_5}{\mu} = -1$.  The only $\beta \in \Delta_{F_4}$ where the action could be non-zero is when $\beta = \beta_3 = \frac{\alp_3 + \alp_5}{2}$, and hence $\yvomn{\beta_3}{0} = \yvomn{\otea{\alp_3}}{0} + \yvomn{\otea{\alp_5}}{0}$.  We have 
\begin{eqnarray*}
&&\yvomn{\otea{\alp_3}}{0} \cdot (\otea{\mu} - \otea{\tau\mu}) \\
&&= -Res \left[ \eps(\alp_3,\tau\mu)z^{-1}\exp\left( \sum_{k \geq 1}\frac{ \alp_3(-k)}{k}z^k\right) \cdot \otea{\alp_3+\tau\mu} \right] \\
&&= -Res \left[ \eps(\alp_3,\tau\mu)z^{-1} \left(I + \dots \right) \cdot \otea{\alp_3+\tau\mu}\right] \\
&&= -Res \left[ \eps(\alp_3,\tau\mu)z^{-1} \otea{\alp_3+\tau\mu}\right] \\
&&= -\eps(\alp_3,\tau\mu) \otea{\alp_3+\tau\mu},
\end{eqnarray*}
and
\begin{eqnarray*}
&&\yvomn{\otea{\alp_5}}{0} \cdot (\otea{\mu} - \otea{\tau\mu}) \\
&&= Res \left[ \eps(\alp_5,\mu)z^{-1}\exp\left( \sum_{k \geq 1}\frac{\alp_5(-k)}{k}z^k\right) \cdot \otea{\alp_5+\mu} \right] 
\end{eqnarray*}
\begin{eqnarray*}
&&= Res \left[ \eps(\alp_5,\mu)z^{-1} \left(I + \dots \right) \cdot \otea{\alp_5+\mu}\right] \\
&&= Res \left[ \eps(\alp_5,\mu)z^{-1} \otea{\alp_5+\mu}\right] \\
&&= \eps(\alp_5,\mu) \otea{\alp_5+\mu}.
\end{eqnarray*}

Even though both of these are non-zero, since $\eps(\alp_3,\tau\mu) = \eps(\alp_5, \mu)$ and $\alp_3 + \tau\mu = \alp_5 + \mu$, the sum of these two operators acting on $\otea{\mu} - \otea{\tau\mu}$, gives 0. \\

3) The calculations for the positive Virasoro operators acting on this vector. 

First check $\yvomn{A}{1} \cdot (\otea{\mu} - \otea{\tau\mu})$.  Since A is a sum of Heisenberg operators then it will suffice to check this for each product of Heisenberg operators. None of the terms will give a non-zero contribution in this calculation, so we have
\begin{eqnarray*}
&&\yvomn{(\lam_1+\lam_6)(-1)^2}{1} \cdot (\otea{\mu} - \otea{\tau\mu}) \\
&&= \sum_{r \in \Z} : (\lam_1+\lam_6)(r)(\lam_1+\lam_6)(1-r): \cdot (\otea{\mu} - \otea{\tau\mu}) \\
%&&= \left( \cdots+ (\lam_1+\lam_6)(0)(\lam_1+\lam_6)(1) + (\lam_1+\lam_6)(0)(\lam_1+\lam_6)(1) + \cdots \right) \cdot (\otea{\mu} - \otea{\tau\mu}) \\
&&= \left( \cdots+ 2(\lam_1+\lam_6)(0)(\lam_1+\lam_6)(1) +\cdots \right) \cdot (\otea{\mu} - \otea{\tau\mu}) \\
&&= 0.
\end{eqnarray*}
Checking in a similar manner the other two parts of A, then $\yvomn{A}{1} (\otea{\mu} - \otea{\tau\mu})$. \\
To check the part B, consider the following
\begin{eqnarray*}
&&\yvomn{\otea{\gamma_i}}{1} \cdot (\otea{\mu} - \otea{\tau\mu}) \\&=& Res \left[z^2 \exp\left( \sum_{k \geq 1}\frac{\gamma_i(-k)}{k}z^k\right) \exp\left( \sum_{k \geq 1}\frac{\gamma_i(k)}{-k}z^{-k}\right) e_{\gamma_i}z^{\gamma_i(0)}  \cdot (\otea{\mu} - \otea{\tau\mu})\right] \\
&=& Res \left[ \eps(\gamma_i,\mu)z^{2+\herf{\gamma_i}{\mu}}\exp\left( \sum_{k \geq 1}\frac{ \gamma_i(-k)}{k}z^k\right) \exp\left( \sum_{k \geq 1}\frac{ \gamma_i(k)}{-k}z^{-k}\right) \cdot \otea{\gamma_i + \mu}\right] \\
&&-Res \left[ \eps(\gamma_i,\tau\mu)z^{2+\herf{\gamma_i}{\tau\mu}}\exp\left( \sum_{k \geq 1}\frac{ \gamma_i(-k)}{k}z^k\right) \exp\left( \sum_{k \geq 1}\frac{ \gamma_i(k)}{-k}z^{-k}\right)  \otea{\gamma_i+\tau\mu}\right] \\
&=& Res \left[ \eps(\gamma_i,\mu)z^{2+\herf{\gamma_i}{\mu}}\exp\left( \sum_{k \geq 1}\frac{ \gamma_i(-k)}{k}z^k\right)  \cdot \otea{\gamma_i + \mu}\right] \\
&&-Res \left[ \eps(\gamma_i,\tau\mu)z^{2+\herf{\gamma_i}{\tau\mu}}\exp\left( \sum_{k \geq 1}\frac{ \gamma_i(-k)}{k}z^k\right) \cdot \otea{\gamma_i+\tau\mu}\right] \\
&=& 0
\end{eqnarray*} 
because the powers of $z$ in both terms are greater or equal to zero for any $\gamma_i$, and the same is true for $-\gam_i$.  Therefore both $\yvomn{A}{1}$ and $\yvomn{B}{1}$ will kill $\otea{\mu} - \otea{\tau\mu}$, and hence $\yvomn{\omg}{1}$ will also kill. 

%Since $\yvomn{\omg}{2}$, ensures there is still at least one mode number that is positive, then $\yvomn{A}{2}$ will kill the vector.  $\yvomn{B}{2}$, makes the power of $z$ in the calculation at least 3 and ensures with the given dot products that power of $z$ will be positive.  It follows that $\yvomn{\omg}{2}$ kills $\otea{\mu} - \otea{\tau\mu}$.
Because $wt(\{ \omg \}_3 \cdot (\otea{\mu} - \otea{\tau\mu})) = -1$, $\yvomn{\omg}{2}$ kills the vector.   Therefore all positive Virasoro operators kill $\otea{\mu} - \otea{\tau\mu}$. \\

4) To check that $\yvomn{\omg}{0}$ gives eigenvalue of $\frac{2}{5}$ for $ \otea{\mu} - \otea{\tau\mu}$, first note that $\yvomn{A}{0}$ will only be non-zero when both mode numbers are 0.  Using the table of dot products above, we have that
\begin{eqnarray*}
&&\yvomn{A}{0}\cdot (\otea{\mu} - \otea{\tau\mu}) \\
&=& \frac{1}{10}\left((-\lam_1+\lam_6)(0)^2 + (\lam_3-\lam_5)(0)^2 + (\lam_1-\lam_3+\lam_5-\lam_6)(0)^2 \right) \\
&& \hspace{100pt} \cdot (\otea{\mu} - \otea{\tau\mu}) \\
&=& \frac{1}{10} \left(\herf{-\lam_1+\lam_6}{\mu}^2 + \herf{\lam_3-\lam_5}{\mu}^2 + \herf{\lam_1-\lam_3+\lam_5-\lam_6}{\mu}^2 \right)\otea{\mu}\\
&&- \frac{1}{10} \left(\herf{-\lam_1+\lam_6}{\tau\mu}^2 + \herf{\lam_3-\lam_5}{\tau\mu}^2 + \herf{\lam_1-\lam_3+\lam_5-\lam_6}{\tau\mu}^2 \right)\otea{\tau\mu}\\
&=& \frac{1}{10} \left(2\otimes e^{\mu} - 2 \otimes e^{\tau\mu}\right) \\
&=& \frac{1}{5}(\otea{\mu} - \otea{\tau\mu}).
\end{eqnarray*} 

And for the action of part $B$, look at the action for an arbitrary $\gamma_i$ given by

\begin{eqnarray*}
&&\yvomn{\otea{\gamma_i}}{0}\cdot (\otea{\mu} - \otea{\tau\mu}) \\
&=& Res \left[z^1 \exp\left( \sum_{k \geq 1}\frac{\gamma_i(-k)}{k}z^k\right) \exp\left( \sum_{k \geq 1}\frac{\gamma_i(k)}{-k}z^{-k}\right) e_{\gamma_i}z^{\gamma_i(0)}  \cdot (\otea{\mu} - \otea{\tau\mu})\right] \\
%&=& Res \left[ \eps(\gamma_i,\mu)z^{1+(\gamma_i,\mu)}\exp\left( \frac{\sum_{k \geq 1} \gamma_i(-k)}{k}z^k\right) \exp\left( \frac{\sum_{k \geq 1} \gamma_i(k)}{-k}z^{-k}\right) \cdot \otea{\gamma_i + \mu}\right] \\
%&&-Res \left[ \eps(\gamma_i,\tau\mu)z^{1+(\gamma_i,\mu)}\exp\left( \frac{\sum_{k \geq 1} \gamma_i(-k)}{k}z^k\right) \exp\left( \frac{\sum_{k \geq 1} \gamma_i(k)}{-k}z^{-k}\right) \right. \\
%&& \hspace{100pt} \left. \cdot \otea{\gamma_i+\tau\mu}\right] \\
&=& Res \left[ \eps(\gamma_i,\mu)z^{1+\herf{\gamma_i}{\mu}}\exp\left( \sum_{k \geq 1}\frac{ \gamma_i(-k)}{k}z^k\right)  \cdot \otea{\gamma_i + \mu}\right] \\
&&-Res \left[ \eps(\gamma_i,\tau\mu)z^{1+\herf{\gamma_i}{\tau\mu}}\exp\left( \sum_{k \geq 1}\frac{ \gamma_i(-k)}{k}z^k\right) \cdot \otea{\gamma_i+\tau\mu}\right]. \\
\end{eqnarray*}  

If $\gamma_i$ is replaced by $-\gamma_i$, we have
\begin{eqnarray*}
&&\yvomn{\otea{-\gamma_i}}{0} \cdot (\otea{\mu} - \otea{\tau\mu}) \\
&=& Res \left[ \eps(-\gamma_i,\mu)z^{1+\herf{-\gamma_i}{\mu}}\exp\left( \sum_{k \geq 1}\frac{ -\gamma_i(-k)}{k}z^k\right)  \cdot \otea{-\gamma_i + \mu}\right] \\
&&-Res \left[ \eps(-\gamma_i,\tau\mu)z^{1+\herf{-\gamma_i}{\tau\mu}}\exp\left( \sum_{k \geq 1}\frac{-\gamma_i(-k)}{k}z^k\right) \cdot \otea{-\gamma_i+\tau\mu}\right] .\\
\end{eqnarray*} 

Using the table of dot products, the residues will only be non-zero in the cases where $\herf{\gamma_2}{\tau\mu} = -2$ and $\herf{-\gamma_2}{\mu} = -2$.  All other $\gamma$'s will give residues equal to 0.  Also, $\eps(\gamma_2,\mu) = \eps(\gamma_2,\tau\mu) =1$ and $\tau\gamma_2 = -\gamma_2$, hence these reductions give

\begin{eqnarray*}
&&\yvomn{B}{0} \cdot (\otea{\mu} - \otea{\tau\mu}) \\
&=& -\frac{1}{5}Res \left[ z^{-1}\exp\left( \frac{ -\gamma_2(-k)}{k}z^k\right)  \cdot \otea{-\gamma_2 + \mu}\right] \\
&&+\frac{1}{5}Res \left[ z^{-1}\exp\left( \sum_{k \geq 1}\frac{ \gamma_2(-k)}{k}z^k\right) \cdot \otea{\gamma_2+\tau\mu}\right] \\
&=& -\frac{1}{5}\otea{-\gamma_2 + \mu} +\frac{1}{5} \otea{\gamma_2+\tau\mu}\\
&=& - \frac{1}{5}\otea{\tau\mu} + \frac{1}{5}\otea{\mu}. 
\end{eqnarray*} 

Adding these two parts together we have
\begin{eqnarray*}
&&\yvomn{\omg}{0} \cdot (\otea{\mu} - \otea{\tau\mu}) \\
&=&  \frac{1}{5}\otea{\mu} - \frac{1}{5}\otea{\tau\mu} - \frac{1}{5}\otea{\tau\mu} + \frac{1}{5}\otea{\mu} \\
&=&  \frac{2}{5}\left( \otea{\mu} - \otea{\tau\mu}\right).
\end{eqnarray*} 

Hence, $ \otea{\mu} - \otea{\tau\mu}$ has eigenvalue $\frac{2}{5}$ for $\yvomn{\omg}{0}$.

\end{proof}

\begin{lem} The vector $P$ is a HWV for Vir$\left(\frac{4}{5}, \frac{7}{5}\right) \otimes W^{\Omg_4}$ where
\begin{eqnarray*}
P &=& (-2\alp_1(-1) - \alp_3(-1) + \alp_5(-1) + 2\alp_6(-1))\otimes e^{\mu} \\
 &&+ (2\alp_1(-1) + \alp_3(-1) - \alp_5(-1) - 2\alp_6(-1))\otimes e^{\tau\mu} \\
 &&+ 3 \otimes e^{\mu + \alp_1 - \alp_6} + 3 \otimes e^{\tau\mu - \alp_1 + \alp_6}.
\end{eqnarray*}
\end{lem}

\begin{proof}
1)$ \yvomn{\otea{-\theta}}{1}\cdot P =0 $ \\
This calculation will be completed in two parts, first for the action on \\ $\ds \sum_{k = 1,3,5,6} a_k \alp_k(-1) \otimes e^{\mu} + \sum_{k = 1,3,5,6} -a_k \alp_k(-1) \otimes e^{\tau\mu}$. \\

Because $(\theta, \alp_k) = 0$ for $k \in \{ 1,3,5,6\}$ we have
\begin{eqnarray*}
&&\yvomn{\otea{-\theta}}{1} \cdot \sum_{k = 1,3,5,6} a_k \alp_k(-1) \otimes e^{\mu} 
\end{eqnarray*}
\begin{eqnarray*}
&=& Res \left[ z^1 \exp\left( \sum_{k \geq 1}\frac{-\theta(-k)}{k}z^k\right) \exp\left( \sum_{k \geq 1}\frac{ -\theta(k)}{-k}z^{-k}\right) e_{-\theta}z^{-\theta(0)} \right. \\
&& \hspace{100pt} \left. \cdot \sum_{k = 1,3,5,6} a_k \alp_k(-1) \otimes e^{\mu} \right] \\
&=& Res \left[ \eps(-\theta, \mu)\exp\left( \sum_{k \geq 1}\frac{ -\theta(-k)}{k}z^k\right) \exp\left( \sum_{k \geq 1}\frac{ -\theta(k)}{-k}z^{-k}\right) \right. \\
&& \hspace{100pt} \left. \cdot \sum_{k = 1,3,5,6} a_k \alp_k(-1) \otimes e^{\mu - \theta}\right] \\
&=& Res \left[ \eps(-\theta, \mu)\exp\left( \sum_{k \geq 1}\frac{-\theta(-k)}{k}z^k\right) \left( I + \frac{-\theta(1)}{-1}z^{-1} \cdots\right) \right. \\
&& \hspace{100pt} \left. \cdot \sum_{k = 1,3,5,6} a_k \alp_k(-1) \otimes e^{\mu - \theta}\right] \\
&=& Res \left[ \eps(-\theta, \mu)\left( I + \frac{-\theta(-1)}{1}z^1 \dots \right) \cdot \left(\sum_{k = 1,3,5,6} a_k (\theta,\alp_k) \otimes e^{\mu - \theta}\right)z^{-1}\right] \\
&=& Res \left[ \eps(-\theta, \mu) \left(\sum_{k = 1,3,5,6} a_k (\theta,\alp_k) \otimes e^{\mu - \theta}\right)z^{-1}\right] \\
&=& \eps(-\theta, \mu) \left(\sum_{k = 1,3,5,6} a_k (\theta,\alp_k) \otimes e^{\mu - \theta}\right)\\
&=& 0
\end{eqnarray*}

%The last step because $(\theta, \alp_k) = 0$ for $k \in \{ 1,3,5,6\}$.

The second half of this calculation is similar and given by
\begin{eqnarray*}
&&\yvomn{\otea{-\theta}}{1} \cdot \sum_{k = 1,3,5,6} -a_k \alp_k(-1) \otimes e^{\tau\mu} \\
&=& Res \left[ \eps(-\theta, \tau\mu) \left(\sum_{k = 1,3,5,6} -a_k (\theta,\alp_k) \otimes e^{\tau\mu - \theta}\right)z^{-1}\right] \\
&=& \eps(-\theta, \tau\mu) \left(\sum_{k = 1,3,5,6} -a_k (\theta,\alp_k) \otimes e^{\tau\mu - \theta}\right)\\
&=& 0.
\end{eqnarray*}

Notice that $\herf{-\theta}{\mu + \alp_1 - \alp_6} = -1$ and $\herf{-\theta}{\tau\mu -\alp_1 + \alp_6} = -1$, hence
\begin{eqnarray*}
&&\yvomn{\otea{-\theta}}{1}\cdot (3 \otimes e^{\mu + \alp_1 - \alp_6} + 3 \otimes e^{\tau\mu - \alp_1 + \alp_6})\\
&=& Res \left[ z^1 \exp\left( \sum_{k \geq 1}\frac{ -\theta(-k)}{k}z^k\right) \exp\left( \sum_{k \geq 1}\frac{-\theta(k)}{-k}z^{-k}\right) e_{-\theta}z^{-\theta(0)}\cdot (3 \otimes e^{\mu + \alp_1 - \alp_6}) \right] \\
&&+ Res \left[ z^1 \exp\left( \sum_{k \geq 1}\frac{ -\theta(-k)}{k}z^k\right) \exp\left( \sum_{k \geq 1}\frac{-\theta(k)}{-k}z^{-k}\right) e_{-\theta}z^{-\theta(0)} \right. \\
&& \hspace{100pt} \left. \cdot (3 \otimes e^{\tau\mu - \alp_1 + \alp_6}) \right] 
\end{eqnarray*}
\begin{eqnarray*}
&=& Res \left[ \eps(-\theta, \tau\mu) \exp\left( \sum_{k \geq 1}\frac{ -\theta(-k)}{k}z^k\right) \cdot (3 \otimes e^{-\theta + \mu + \alp_1 - \alp_6} + 3 \otimes e^{-\theta + \tau\mu - \alp_1 + \alp_6}) \right] \\
&=& 0
\end{eqnarray*}

Therefore, $\yvomn{-\theta}{1}\cdot P = 0$.\\

2) $\yvomn{\beta}{0} \cdot P = 0$, $\forall \beta \in \Delta_{F_4}$.

First, for $\alp_i$, $i \in \{2,4\}$ we calculate $\yvomn{\otea{\alp_i}}{0} \cdot P$. 
\begin{eqnarray*}
&&\yvomn{\otea{\alp_i}}{0} \cdot \sum_{k = 1,3,5,6} a_k \alp_k(-1) \otimes e^{\mu} \\
&=& Res \left[ \exp\left( \sum_{k \geq 1}\frac{ -\alp_i(-k)}{k}z^k\right) \exp\left( \sum_{k \geq 1}\frac{\alp_i(k)}{-k}z^{-k}\right) e_{\alp_i}z^{\alp_i(0)} \right. \\
&& \hspace{100pt} \left. \cdot \sum_{k = 1,3,5,6} a_k \alp_k(-1) \otimes e^{\mu} \right] \\
&=& Res \left[ \eps(\alp_i, \mu)\exp\left( \sum_{k \geq 1}\frac{ \alp_i(-k)}{k}z^k\right) \exp\left( \sum_{k \geq 1}\frac{ \alp_i(k)}{-k}z^{-k}\right) \right. \\
&& \hspace{100pt} \left. \cdot \sum_{k = 1,3,5,6} a_k \alp_k(-1) \otimes e^{\mu + \alp_i}\right] \\
&=& Res \left[ \eps(\alp_i, \mu)\exp\left( \sum_{k \geq 1}\frac{ \alp_i(-k)}{k}z^k\right) \left( I + \frac{\alp_i(1)}{-1}z^{-1} \cdots\right) \right. \\
&& \hspace{100pt} \left. \cdot \sum_{k = 1,3,5,6} a_k \alp_k(-1) \otimes e^{\mu + \alp_i}\right] \\
&=& Res \left[ \eps(\alp_i, \mu)z^{-1}\left( I + \frac{\alp_i(-1)}{1}z^1 \dots \right) \cdot \left(\sum_{k = 1,3,5,6} -a_k (\alp_i,\alp_k) \otimes e^{\mu + \alp_i}\right)\right] \\
&=& Res \left[ \eps(\alp_i, \mu) z^{-1}\left(\sum_{k = 1,3,5,6} -a_k (\alp_i,\alp_k) \otimes e^{\mu + \alp_i}\right)\right] \\
&=& \eps(\alp_i, \mu) \left(\sum_{k = 1,3,5,6} -a_k (\alp_i,\alp_k) \otimes e^{\mu + \alp_i}\right)\\
&=& 0.
\end{eqnarray*}
Similarly we have 
\begin{eqnarray*}
&&\yvomn{\otea{\alp_i}}{0} \cdot \sum_{k = 1,3,5,6} -a_k \alp_k(-1) \otimes e^{\tau\mu} \\
&=& \eps(\alp_i, \tau\mu) \left(\sum_{k = 1,3,5,6} a_k (\alp_i,\alp_k) \otimes e^{\tau\mu + \alp_i}\right)\\
&=& 0.
\end{eqnarray*}
We also have the following
\begin{eqnarray*}
&&\yvomn{\otea{\alp_i}}{0} \cdot 3 \otimes e^{\mu+\alp_1-\alp_6} \\
&=& Res \left[ \exp\left( \sum_{k \geq 1}\frac{ -\alp_i(-k)}{k}z^k\right) \exp\left( \sum_{k \geq 1}\frac{ \alp_i(k)}{-k}z^{-k}\right) e_{\alp_i}z^{\alp_i(0)} \cdot 3 \otimes e^{\mu+\alp_1-\alp_6} \right] \\
&=& Res \left[ \eps(\alp_i, \mu+\alp_1-\alp_6)\exp\left( \sum_{k \geq 1}\frac{\alp_i(-k)}{k}z^k\right) \exp\left( \sum_{k \geq 1}\frac{ \alp_i(k)}{-k}z^{-k}\right) \right. \\
&& \hspace{100pt} \left. \cdot 3 \otimes e^{\mu+\alp_1-\alp_6-\alp_i}\right] \\
&=& Res \left[ \eps(\alp_i, \mu+\alp_1-\alp_6)\exp\left( \sum_{k \geq 1}\frac{\alp_i(-k)}{k}z^k\right) \left( I + \frac{\alp_i(1)}{-1}z^{-1} \cdots\right) \right. \\
&& \hspace{100pt} \left. \cdot 3 \otimes e^{\mu+\alp_1-\alp_6-\alp_i}\right] \\
&=& 0,
\end{eqnarray*}
and by a similar calculation  
$$\yvomn{\otea{\alp_i}}{0} \cdot 3 \otimes e^{\tau\mu-\alp_1+\alp_6} = 0.$$
For $\beta = \beta_4$, we have $\yvomn{\beta_4}{0} = \yvomn{\otea{\alp_1}}{0} + \yvomn{\otea{\alp_6}}{0}$ and we consider each of these calculations separately.  We first have
\begin{eqnarray*}
&&\yvomn{\otea{\alp_1}}{0} \cdot \sum_{k = 1,3,5,6} a_k \alp_k(-1) \otimes e^{\mu} \\
&=& Res \left[ \exp\left( \sum_{k \geq 1}\frac{-\alp_1(-k)}{k}z^k\right) \exp\left( \sum_{k \geq 1}\frac{ \alp_1(k)}{-k}z^{-k}\right) e_{\alp_1}z^{\alp_1(0)} \right. \\
&& \hspace{100pt} \left. \cdot \sum_{k = 1,3,5,6} a_k \alp_k(-1) \otimes e^{\mu} \right] \\
&=& Res \left[ \exp\left( \sum_{k \geq 1}\frac{\alp_1(-k)}{k}z^k\right) \exp\left( \sum_{k \geq 1}\frac{ \alp_1(k)}{-k}z^{-k}\right) \cdot \sum_{k = 1,3,5,6} a_k \alp_k(-1) \otimes e^{\mu + \alp_1}\right] \\
&=& Res \left[ \exp\left( \sum_{k \geq 1}\frac{ \alp_1(-k)}{k}z^k\right) \left( I + \frac{\alp_1(1)}{-1}z^{-1} \cdots\right) \cdot \sum_{k = 1,3,5,6} a_k \alp_k(-1) \otimes e^{\mu + \alp_1}\right] \\
&=& Res \left[ z^{-1}\left( I + \frac{\alp_1(-1)}{1}z^1 \dots \right) \cdot \left(\sum_{k = 1,3,5,6} -a_k (\alp_1,\alp_k) \otimes e^{\mu + \alp_1}\right)\right] \\
&=& Res \left[z^{-1} \left(\sum_{k = 1,3,5,6} -a_k (\alp_1,\alp_k) \otimes e^{\mu + \alp_1}\right)\right] \\
&=& \left(\sum_{k = 1,3,5,6} -a_k (\alp_1,\alp_k) \otimes e^{\mu + \alp_1}\right)\\
&=& 3 \otimes e^{\mu+\alp_1},
\end{eqnarray*}

%%%%%%%%%%%%%
and also we have 
\begin{eqnarray*}
&&\yvomn{\otea{\alp_1}}{0} \cdot \sum_{k = 1,3,5,6} -a_k \alp_k(-1) \otimes e^{\tau\mu} \\
&=& Res \left[ \exp\left( \sum_{k \geq 1}\frac{ -\alp_1(-k)}{k}z^k\right) \exp\left( \sum_{k \geq 1}\frac{ \alp_1(k)}{-k}z^{-k}\right) e_{\alp_1}z^{\alp_1(0)}\right. \\
&& \hspace{100pt} \left. \cdot \sum_{k = 1,3,5,6} -a_k \alp_k(-1) \otimes e^{\tau\mu} \right] \\
&=& Res \left[ z^1\exp\left( \sum_{k \geq 1}\frac{ \alp_1(-k)}{k}z^k\right) \exp\left( \sum_{k \geq 1}\frac{ \alp_1(k)}{-k}z^{-k}\right) \right. \\
&& \hspace{100pt} \left. \cdot \sum_{k = 1,3,5,6} -a_k \alp_k(-1) \otimes e^{\tau\mu + \alp_1}\right] \\
&=& Res \left[z^1 \exp\left( \sum_{k \geq 1}\frac{ \alp_1(-k)}{k}z^k\right) \left( I + \frac{\alp_1(1)}{-1}z^{-1} \cdots\right) \right. \\
&& \hspace{100pt} \left. \cdot \sum_{k = 1,3,5,6} -a_k \alp_k(-1) \otimes e^{\tau\mu + \alp_1}\right] \\
&=& 0.
\end{eqnarray*}
Similarly for $\alp_6$, we get 
$$\yvomn{\otea{\alp_6}}{0} \cdot \sum_{k = 1,3,5,6} a_k \alp_k(-1) \otimes e^{\mu} = 0 $$
and 
$$\yvomn{\otea{\alp_6}}{0} \cdot \sum_{k = 1,3,5,6} -a_k \alp_k(-1) \otimes e^{\tau\mu} = -3\otimes e^{\tau\mu+\alp_6} .$$
We also have 
\begin{eqnarray*}
&&\yvomn{\otea{\alp_1}}{0} \cdot 3 \otimes e^{\mu+\alp_1-\alp_6} \\
&=& Res \left[ \exp\left( \sum_{k \geq 1}\frac{ -\alp_1(-k)}{k}z^k\right) \exp\left( \sum_{k \geq 1}\frac{ \alp_1(k)}{-k}z^{-k}\right) e_{\alp_1}z^{\alp_1(0)} \cdot 3 \otimes e^{\mu+\alp_1-\alp_6}  \right] \\
&=& Res \left[ z^2\exp\left( \sum_{k \geq 1}\frac{\alp_1(-k)}{k}z^k\right) \exp\left( \sum_{k \geq 1}\frac{\alp_1(k)}{-k}z^{-k}\right) \cdot 3 \otimes e^{\mu+2\alp_1-\alp_6} \right] \\
&=& 0.
\end{eqnarray*}
and 
\begin{eqnarray*}
&&\yvomn{\otea{\alp_1}}{0} \cdot 3 \otimes e^{\tau\mu-\alp_1+\alp_6} \\
&=& Res \left[ \exp\left( \sum_{k \geq 1}\frac{-\alp_1(-k)}{k}z^k\right) \exp\left( \sum_{k \geq 1}\frac{ \alp_1(k)}{-k}z^{-k}\right) e_{\alp_1}z^{\alp_1(0)} \cdot 3 \otimes e^{\tau\mu-\alp_1+\alp_6}   \right] \\
\end{eqnarray*}
\begin{eqnarray*}
&=& Res \left[ z^{-1}\exp\left( \sum_{k \geq 1}\frac{\alp_1(-k)}{k}z^k\right) \exp\left( \sum_{k \geq 1}\frac{\alp_1(k)}{-k}z^{-k}\right) \cdot 3 \otimes e^{\tau\mu+\alp_6}  \right] \\
&=& Res \left[ z^{-1}\exp\left( \sum_{k \geq 1}\frac{ \alp_1(-k)}{k}z^k\right) \cdot 3 \otimes e^{\tau\mu+\alp_6}  \right] \\
&=& 3 \otimes e^{\tau\mu+\alp_6}
\end{eqnarray*}
Similarly for $\alp_6$, we have 
$$ \yvomn{\otea{\alp_6}}{0} \cdot 3 \otimes e^{\mu+\alp_1-\alp_6}= -3 \otimes e^{\mu+\alp_1}$$
and
$$\yvomn{\otea{\alp_6}}{0} \cdot 3 \otimes e^{\tau\mu-\alp_1+\alp_6} = 0.$$
By adding all these actions together for $\alp_1$ and $\alp_6$, we have 
$$\yvomn{\beta_4}{0} \cdot P = 0.$$

For $\beta = \beta_3$, we have $\yvomn{\beta_3}{0} = \yvomn{\otea{\alp_3}}{0} + \yvomn{\otea{\alp_5}}{0}$ we consider each of these calculations separately. We have 

\begin{eqnarray*}
&&\yvomn{\otea{\alp_3}}{0} \cdot \sum_{k = 1,3,5,6} a_k \alp_k(-1) \otimes e^{\mu} \\
&=& Res \left[ \exp\left( \sum_{k \geq 1}\frac{ -\alp_3(-k)}{k}z^k\right) \exp\left( \sum_{k \geq 1}\frac{ \alp_3(k)}{-k}z^{-k}\right) e_{\alp_3}z^{\alp_3(0)} \right. \\
&& \hspace{100pt} \left. \cdot \sum_{k = 1,3,5,6} a_k \alp_k(-1) \otimes e^{\mu} \right] \\
&=& Res \left[ z^1\exp\left( \sum_{k \geq 1}\frac{ \alp_3(-k)}{k}z^k\right) \exp\left( \sum_{k \geq 1}\frac{ \alp_3(k)}{-k}z^{-k}\right) \right. \\
&& \hspace{100pt} \left. \cdot \sum_{k = 1,3,5,6} a_k \alp_k(-1) \otimes e^{\mu + \alp_3}\right] \\
&=& 0,
\end{eqnarray*}
and also
\begin{eqnarray*}
&&\yvomn{\otea{\alp_3}}{0} \cdot \sum_{k = 1,3,5,6} -a_k \alp_k(-1) \otimes e^{\tau\mu} \\
&=& Res \left[ \exp\left( \sum_{k \geq 1}\frac{ -\alp_3(-k)}{k}z^k\right) \exp\left( \sum_{k \geq 1}\frac{\alp_3(k)}{-k}z^{-k}\right) e_{\alp_3}z^{\alp_3(0)} \right. \\
&& \hspace{100pt} \left. \cdot \sum_{k = 1,3,5,6} -a_k \alp_k(-1) \otimes e^{\tau\mu} \right] 
\end{eqnarray*}
\begin{eqnarray*}
&=& Res \left[ z^{-1}\exp\left( \sum_{k \geq 1}\frac{\alp_3(-k)}{k}z^k\right) \exp\left( \sum_{k \geq 1}\frac{\alp_3(k)}{-k}z^{-k}\right) \right. \\
&& \hspace{100pt} \left. \cdot \sum_{k = 1,3,5,6} -a_k \alp_k(-1) \otimes e^{\tau\mu + \alp_3}\right] \\
&=& \sum_{k = 1,3,5,6} -a_k \alp_k(-1) \otimes e^{\tau\mu + \alp_3}.
\end{eqnarray*}
Similarly for $\alp_5$, we get 
$$\yvomn{\otea{\alp_5}}{0} \cdot \sum_{k = 1,3,5,6} a_k \alp_k(-1) \otimes e^{\mu} = \sum_{k = 1,3,5,6} a_k \alp_k(-1) \otimes e^{\mu + \alp_5} $$
and 
$$\yvomn{\otea{\alp_5}}{0} \cdot \sum_{k = 1,3,5,6} -a_k \alp_k(-1) \otimes e^{\tau\mu} = 0 .$$
We have 
\begin{eqnarray*}
&&\yvomn{\otea{\alp_3}}{0} \cdot 3 \otimes e^{\mu+\alp_1-\alp_6} \\
&=& Res \left[ \exp\left( \sum_{k \geq 1}\frac{-\alp_3(-k)}{k}z^k\right) \exp\left( \sum_{k \geq 1}\frac{\alp_3(k)}{-k}z^{-k}\right) e_{\alp_3}z^{\alp_3(0)} \cdot 3 \otimes e^{\mu+\alp_1-\alp_6}  \right] \\
&=& Res \left[ \exp\left( \sum_{k \geq 1}\frac{ \alp_3(-k)}{k}z^k\right) \exp\left( \sum_{k \geq 1}\frac{ \alp_3(k)}{-k}z^{-k}\right) \cdot 3 \otimes e^{\mu+\alp_1+\alp_3-\alp_6} \right] \\
&=& 0,
\end{eqnarray*}
and 
\begin{eqnarray*}
&&\yvomn{\otea{\alp_3}}{0} \cdot 3 \otimes e^{\tau\mu-\alp_1+\alp_6} \\
&=& Res \left[ \exp\left( \sum_{k \geq 1}\frac{ -\alp_3(-k)}{k}z^k\right) \exp\left( \sum_{k \geq 1}\frac{ \alp_3(k)}{-k}z^{-k}\right) e_{\alp_3}z^{\alp_3(0)} \cdot 3 \otimes e^{\tau\mu-\alp_3+\alp_6}   \right] \\
&=& Res \left[\exp\left( \sum_{k \geq 1}\frac{ \alp_3(-k)}{k}z^k\right) \exp\left( \sum_{k \geq 1}\frac{ \alp_3(k)}{-k}z^{-k}\right) \cdot 3 \otimes e^{\tau\mu-\alp_1+\alp_3\alp_6}  \right] \\
&=& 0.
\end{eqnarray*}
For $\alp_5$, we get 
$$ \yvomn{\otea{\alp_5}}{0} \cdot 3 \otimes e^{\mu+\alp_1-\alp_6}= 0$$
and
$$\yvomn{\otea{\alp_5}}{0} \cdot 3 \otimes e^{\tau\mu-\alp_1+\alp_6} = 0.$$
By adding all these actions together for $\alp_1$ and $\alp_6$, we have 
$$\yvomn{\beta_3}{0} \cdot P = 0.$$
Hence for $i \in \{1,2,3,4\}$, $\yvomn{\beta_i}{0} \cdot P = 0$.  \\

3)$ \yvomn{\omg}{1} \cdot P = 0$.  This calculation will be completed in four parts. 

First we will calculate the action of $\yvomn{A}{1} $ on $\ds \sum_{k = 1,3,5,6} a_k \alp_k(-1) \otimes e^{\mu} +$ \\$ \sum_{k = 1,3,5,6} -a_k \alp_k(-1) \otimes e^{\tau\mu}$

\begin{eqnarray*}
&&\yvomn{A}{1} \cdot \left( \sum_{k = 1,3,5,6} a_k \alp_k(-1) \otimes e^{\mu} \right) \\
&=& \frac{1}{10}: (-\lam_1 + \lam_6)(r)(-\lam_1 + \lam_6)(1-r): \cdot \left( \sum_{k = 1,3,5,6} a_k \alp_k(-1) \otimes e^{\mu}\right)\\
&&+ \frac{1}{10}: (\lam_3 - \lam_5)(r)(-\lam_1 + \lam_6)(1-r): \cdot \left( \sum_{k = 1,3,5,6} a_k \alp_k(-1) \otimes e^{\mu} \right)\\
&&+ \frac{1}{10}: (\lam_1 -\lam_3 + \lam_5 - \lam_6)(r)(\lam_1 -\lam_3 + \lam_5 - \lam_6)(1-r):  \\
&& \hspace{100pt}  \cdot \left( \sum_{k = 1,3,5,6} a_k \alp_k(-1) \otimes e^{\mu}\right)\\
&=& \frac{1}{5}(-\lam_1 + \lam_6)(0)(-\lam_1 + \lam_6)(1) \cdot \left( \sum_{k = 1,3,5,6} a_k \alp_k(-1) \otimes e^{\mu}\right)\\
&&+ \frac{1}{5}(\lam_3 - \lam_5)(0)(-\lam_1 + \lam_6)(1) \cdot \left( \sum_{k = 1,3,5,6} a_k \alp_k(-1) \otimes e^{\mu} \right)\\
&&+ \frac{1}{5} (\lam_1 -\lam_3 + \lam_5 - \lam_6)(0)(\lam_1 -\lam_3 + \lam_5 - \lam_6)(1) \cdot \left( \sum_{k = 1,3,5,6} a_k \alp_k(-1) \otimes e^{\mu}\right)\\
&=& \frac{1}{5}\herf{-\lam_1 + \lam_6}{\mu}\herf{-\lam_1 + \lam_6}{ \sum_{k = 1,3,5,6} a_k \alp_k} \otimes e^{\mu}\\
&&+ \frac{1}{5}\herf{\lam_3 - \lam_5}{\mu}\herf{-\lam_1 + \lam_6}{ \sum_{k = 1,3,5,6} a_k \alp_k} \otimes e^{\mu}\\
&&+ \frac{1}{5} \herf{\lam_1 -\lam_3 + \lam_5 - \lam_6}{\mu}\herf{\lam_1 -\lam_3 + \lam_5 - \lam_6}{ \sum_{k = 1,3,5,6} a_k \alp_k} \otimes e^{\mu}\\
&=& 0.
\end{eqnarray*}
This is zero because in the expression above the three dot products containing $\mu$ are each zero, hence the sum will be also.  Similarly the dot products with $\mu$ replaced by $\tau\mu$ give:

\begin{eqnarray*}
&&\yvomn{A}{1} \cdot \left(\sum_{k = 1,3,5,6} -a_k \alp_k(-1) \otimes e^{\tau\mu}\right) \\
&=& \frac{1}{5}\herf{-\lam_1 + \lam_6}{\tau\mu}\herf{-\lam_1 + \lam_6}{ \sum_{k = 1,3,5,6} -a_k \alp_k} \otimes e^{\tau\mu}\\
&&+ \frac{1}{5}\herf{\lam_3 - \lam_5}{\tau\mu}\herf{-\lam_1 + \lam_6}{ \sum_{k = 1,3,5,6} -a_k \alp_k} \otimes e^{\tau\mu}\\
&&+ \frac{1}{5} \herf{\lam_1 -\lam_3 + \lam_5 - \lam_6}{\tau\mu}\herf{\lam_1 -\lam_3 + \lam_5 - \lam_6}{ \sum_{k = 1,3,5,6} -a_k \alp_k} \otimes e^{\tau\mu}\\
&=& 0.
\end{eqnarray*}

Hence the operator $\yvomn{A}{1}$ kills this first part of the vector $P$.

It is simple to see that $\yvomn{A}{1}$ kills $3 \otimes e^{\mu + \alp_1 - \alp_6} + 3 \otimes e^{\tau\mu - \alp_1 + \alp_6}$. This is true since at least one of the mode numbers will be positive, hence killing each term.

The third part will be when we look at $\yvomn{B}{1}$ acting on $$\sum_{k = 1,3,5,6} a_k \alp_k(-1) \otimes e^{\mu} + \sum_{k = 1,3,5,6} -a_k \alp_k(-1) \otimes e^{\tau\mu}.$$  Note that $(\gamma_2, \tau\mu) = -2 = (-\gamma_2, \mu)$, are the only relevant dot products and hence simplify the calculations.  We have
\begin{eqnarray*}
&&\yvomn{\otea{\gamma_2}}{1} \cdot \left( \sum_{k = 1,3,5,6} -a_k \alp_k(-1) \otimes e^{\tau\mu} \right) \\
&&= Res \left[z^2 \exp\left( \sum_{k \geq 1}\frac{ \gamma_2(-k)}{k}z^k\right) \exp\left( \sum_{k \geq 1}\frac{ \gamma_2(k)}{-k}z^{-k}\right) e_{\gamma_2}z^{\gamma_2(0)} \right. \\
&&\hspace{1in}\cdot \left. \left( \sum_{k = 1,3,5,6} -a_k \alp_k(-1) \otimes e^{\tau\mu} \right)\right] \\
&&= Res \left[\eps(\gamma_2,\tau\mu)z^{2-2} \exp\left( \sum_{k \geq 1}\frac{ \gamma_2(-k)}{k}z^k\right) \left( I + \frac{\gamma_2(1)}{-1}z^{-1}\right) \right. \\
&&\hspace{1in} \left.\cdot \left( \sum_{k = 1,3,5,6} -a_k \alp_k(-1) \otimes e^{\gamma_2 + \tau\mu} \right)\right] \\
&&= Res \left[ \eps(\gamma_2,\tau\mu) \left(I +  \frac{\gamma_2(-1)}{1}z^{1}\right) \cdot \left(z^{-1}\sum_{k = 1,3,5,6} a_k \herf{\gamma_2}{\alp_k} \otimes e^{\gamma_2 + \tau\mu} \right) \right] \\
&&= Res \left[ \eps(\gamma_2,\tau\mu) z^{-1}\sum_{k = 1,3,5,6} a_k \herf{\gamma_2}{\alp_k} \otimes e^{\gamma_2 + \tau\mu}  \right] \\
&&= \eps(\gamma_2,\tau\mu) \sum_{k = 1,3,5,6} a_k \herf{\gamma_2}{\alp_k} \otimes e^{\gamma_2 + \tau\mu}\\
&&= \eps(\gamma_2,\tau\mu)(-2 + 2 + 2 - 2) \otimes e^{\gamma_2 + \tau\mu}  = 0.
\end{eqnarray*}

Similarly we have
\begin{eqnarray*}
&&\yvomn{\otea{-\gamma_2}}{1} \cdot \left( \sum_{k = 1,3,5,6} a_k \alp_k(-1) \otimes e^{\mu} \right) \\
&&= Res \left[z^2 \exp\left( \sum_{k \geq 1}\frac{ -\gamma_2(-k)}{k}z^k\right) \exp\left( \sum_{k \geq 1}\frac{ -\gamma_2(k)}{-k}z^{-k}\right) e_{-\gamma_2}z^{-\gamma_2(0)} \right. \\
&&\hspace{1in}\cdot \left. \left( \sum_{k = 1,3,5,6} a_k \alp_k(-1) \otimes e^{\mu} \right)\right] \\
&&= Res \left[\eps(-\gamma_2,\mu)z^{2-2} \exp\left( \sum_{k \geq 1}\frac{ -\gamma_2(-k)}{k}z^k\right) \left( I + \frac{-\gamma_2(1)}{-1}z^{-1}\right) \right. \\
&&\hspace{1in} \left.\cdot \left( \sum_{k = 1,3,5,6} a_k \alp_k(-1) \otimes e^{-\gamma_2 + \mu} \right)\right] \\
&&= Res \left[ \eps(-\gamma_2,\mu) \left(I +  \frac{-\gamma_2(-1)}{1}z^{1}\right) \cdot \left(z^{-1}\sum_{k = 1,3,5,6} a_k \herf{-\gamma_2}{\alp_k}  \otimes e^{-\gamma_2 + \mu} \right) \right] \\
&&= Res \left[ \eps(\gamma_2,\tau\mu) z^{-1}\sum_{k = 1,3,5,6} a_k \herf{-\gamma_2}{\alp_k}  \otimes e^{\gamma_2 + \tau\mu}  \right] \\
&&= \eps(-\gamma_2,\mu) \sum_{k = 1,3,5,6} a_k \herf{-\gamma_2}{\alp_k} \otimes e^{-\gamma_2 + \mu}\\
&&= \eps(-\gamma_2,\mu)(2 - 2 - 2 + 2) \otimes e^{-\gamma_2 + \mu} \\
&&= 0.
\end{eqnarray*}

Hence, $\yvomn{\otea{\gamma_i} + \otea{-\gamma_i}}{1} \cdot \sum_{k = 1,3,5,6}\left( a_k \alp_k(-1) \otimes e^{\mu} + -a_k \alp_k(-1) \otimes e^{\tau\mu}\right) - 0$.\\

The fourth part will calculate the action of $\yvomn{B}{1}$ on $(3 \otimes e^{\mu + \alp_1 - \alp_6} + 3 \otimes e^{\tau\mu - \alp_1 + \alp_6})$.  To decrease the number of calculations notice that $(\gamma_1, \tau\mu-\alp_1+\alp_6) = (-\gamma_1, \mu+\alp_1-\alp_6) = -3$, and $(\gamma_3, \tau\mu-\alp_1+\alp_6) = (-\gamma_3, \mu +\alp_1-\alp_6) = -3$ will be the only relevant dot products. 

For $i = \{1,3 \}$, we have:
\begin{eqnarray*}
&&\yvomn{\otea{\gamma_i}}{1} \cdot \left( 3 \otimes e^{\tau\mu - \alp_1 + \alp_6} \right) \\
&&= Res \left[z^2 \exp\left( \sum_{k \geq 1}\frac{ \gamma_i(-k)}{k}z^k\right) \exp\left( \sum_{k \geq 1}\frac{ \gamma_i(k)}{-k}z^{-k}\right) e_{\gamma_i}z^{\gamma_i(0)}\right. \\
&& \hspace{100pt} \left. \cdot \left( 3 \otimes e^{\tau\mu - \alp_1 + \alp_6} \right)\right] \\
%&&= Res \left[\eps(\gamma_i,\tau\mu)z^{2-3} \exp\left( \frac{\sum_{k \geq 1} \gamma_i(-k)}{k}z^k\right) \cdot \left( 3 \otimes e^{\gamma_i + \tau\mu - \alp_1 + \alp_6}  \right)\right] \\
&&= Res \left[ z^{-1} \eps(\gamma_i,\tau\mu-\alp_1+\alp_6) \left(I +  \frac{\gamma_i(-1)}{1}z^{1}\right) \cdot \left( 3 \otimes e^{\gamma_i + \tau\mu - \alp_1 + \alp_6} \right) \right] \\
&&= Res \left[ z^{-1} \eps(\gamma_i,\tau\mu-\alp_1+\alp_6) \left( 3 \otimes e^{\gamma_i + \tau\mu - \alp_1 + \alp_6} \right) \right] \\
&&= \eps(\gamma_i,\tau\mu-\alp_1+\alp_6) \left( 3 \otimes e^{\gamma_i + \tau\mu - \alp_1 + \alp_6} \right)
\end{eqnarray*}
and 
\begin{eqnarray*}
&&\yvomn{\otea{-\gamma_i}}{1} \cdot \left( 3 \otimes e^{\mu + \alp_1 - \alp_6} \right) \\
&&= Res \left[z^2 \exp\left( \sum_{k \geq 1}\frac{ \gamma_i(-k)}{k}z^k\right) \exp\left( \sum_{k \geq 1}\frac{ -\gamma_i(k)}{-k}z^{-k}\right) e_{-\gamma_i}z^{-\gamma_i(0)} \right. \\
&& \hspace{100pt} \left. \cdot \left( 3 \otimes e^{\mu + \alp_1 - \alp_6} \right)\right] \\
%&&= Res \left[\eps(\gamma_i,\tau\mu)z^{2-3} \exp\left( \sum_{k \geq 1}\frac{ \gamma_i(-k)}{k}z^k\right) \cdot \left( 3 \otimes e^{\gamma_i + \tau\mu - \alp_1 + \alp_6}  \right)\right] \\
&&= Res \left[ z^{-1} \eps(-\gamma_i,\mu+ \alp_1 - \alp_6) \left(I +  \frac{-\gamma_i(-1)}{1}z^{1}\right) \cdot \left( 3 \otimes e^{-\gamma_i + \mu + \alp_1 - \alp_6} \right) \right] \\
&&= Res \left[ z^{-1} \eps(-\gamma_i,\mu+ \alp_1 - \alp_6) \left( 3 \otimes e^{-\gamma_i + \mu + \alp_1 - \alp_6} \right) \right] \\
&&= \eps(-\gamma_i,\mu+ \alp_1 - \alp_6) \left( 3 \otimes e^{-\gamma_i + \mu + \alp_1 - \alp_6} \right). 
\end{eqnarray*}

Therefore, 
\begin{eqnarray*}
&&-\frac{1}{5}\yvomn{\otea{\gamma_1} + \otea{-\gamma_1}}{1} \cdot \left( 3 \otimes e^{\mu + \alp_1 - \alp_6} + 3 \otimes e^{\tau\mu - \alp_1 + \alp_6} \right) \\
&&= -\frac{1}{5}\eps(\gamma_1,\tau\mu-\alp_1+\alp_6)( 3 \otimes e^{\tau\mu}) -\frac{1}{5}\eps(-\gamma_1,\mu+\alp_1-\alp_6)  (3 \otimes e^{\mu}) 
\end{eqnarray*}
and
\begin{eqnarray*}
&&\frac{1}{5}\yvomn{\otea{\gamma_3} + \otea{-\gamma_3}}{1} \cdot \left( 3 \otimes e^{\mu + \alp_1 - \alp_6} + 3 \otimes e^{\tau\mu - \alp_1 + \alp_6} \right) \\
&&= \frac{1}{5}\eps(\gamma_3,\tau\mu-\alp_1+\alp_6)( 3 \otimes e^{\mu}) + \frac{1}{5}\eps(-\gamma_3,\mu+\alp_1-\alp_6)  (3 \otimes e^{\tau\mu}). 
\end{eqnarray*}

Adding both of these together, 
\begin{eqnarray*}
&&\yvomn{B}{1}\cdot \left( 3 \otimes e^{\mu + \alp_1 - \alp_6} + 3 \otimes e^{\tau\mu - \alp_1 + \alp_6} \right) \\
&&=\frac{1}{5}(\eps(\gamma_3,\tau\mu-\alp_1+\alp_6) - \eps(-\gamma_1,\mu+\alp_1-\alp_6))  (3 \otimes e^{\mu}) \\
&&\hspace{.5in}+\frac{1}{5}(\eps(-\gamma_3,\mu+\alp_1-\alp_6) - \eps(\gamma_1,\tau\mu-\alp_1+\alp_6))  (3 \otimes e^{\mu}).
\end{eqnarray*}
A calculation of these four 2-cocyles yields that $\eps(\gamma_3,\tau\mu-\alp_1+\alp_6) - \eps(-\gamma_1,\mu+\alp_1-\alp_6) = 0$, and $\eps(-\gamma_3,\mu+\alp_1-\alp_6) - \eps(\gamma_1,\tau\mu-\alp_1+\alp_6)=0$
Therefore giving  $\yvomn{B}{1} \cdot P = 0$.
Putting all four of these parts together gives $\yvomn{\omg}{1} \cdot P = 0$.

We break the calculation $\yvomn{\omg}{2} \cdot P = 0$ into two parts.  First for the part A, notice that $\yvomn{A}{2}$ will always have at least one mode number positive, and therefore it is easily seen to kill the vectors $3 \otimes e^{\mu+\alp_1-\alp_6}$ and  $3 \otimes e^{\tau\mu\alp_1+\alp_6}$.  It is also easy to see that they will kill the other vectors of the form $$\sum_{k = 1,3,5,6} a_k \alp_k(-1) \otimes e^{\mu}$$ and  $$\sum_{k = 1,3,5,6} -a_k \alp_k(-1) \otimes e^{\tau\mu}. $$ The operators with at least one mode number greater than one are easily seen to kill the vectors, and the only other option is for both mode numbers to be 1.  In this case, the first operator will not kill, but the second one will kill the remaining vector.

%For $\yvomn{B}{2}$, since each of the $\pm\gamma_i, i \in \{1,2,3\}$ has dot product no smaller than -2, then we have 

%%%%%%%%%%%%
%%%%%%%%%%%%%
We look at $\yvomn{B}{2} \cdot \left( \sum_{k = 1,3,5,6} a_k \alp_k(-1) \otimes e^{\mu} + \sum_{k = 1,3,5,6} -a_k \alp_k(-1) \otimes e^{\tau\mu}\right) $.  Note \\ that $\herf{\gamma_2}{ \tau\mu} = -2 = \herf{-\gamma_2}{ \mu}$, since they are the smallest, are the only relevant dot products and hence simplify the calculations to
\begin{eqnarray*}
&&\yvomn{\otea{\gamma_2}}{2} \cdot \left( \sum_{k = 1,3,5,6} -a_k \alp_k(-1) \otimes e^{\tau\mu} \right) \\
&&= Res \left[z^3 \exp\left( \sum_{k \geq 1}\frac{ \gamma_2(-k)}{k}z^k\right) \exp\left( \sum_{k \geq 1}\frac{ \gamma_2(k)}{-k}z^{-k}\right) e_{\gamma_2}z^{\gamma_2(0)} \right. \\
&&\hspace{1in}\cdot \left. \left( \sum_{k = 1,3,5,6} -a_k \alp_k(-1) \otimes e^{\tau\mu} \right)\right] \\
&&= Res \left[\eps(\gamma_2,\tau\mu)z^1 \exp\left( \sum_{k \geq 1}\frac{ \gamma_2(-k)}{k}z^k\right) \left( I + \frac{\gamma_2(1)}{-1}z^{-1}\right) \right. \\
&&\hspace{1in} \left.\cdot \left( \sum_{k = 1,3,5,6} -a_k \alp_k(-1) \otimes e^{\gamma_2 + \tau\mu} \right)\right] \\
&&= 0.
\end{eqnarray*} 

Similarly we have

\begin{eqnarray*}
&&\yvomn{\otea{-\gamma_2}}{2} \cdot \left( \sum_{k = 1,3,5,6} a_k \alp_k(-1) \otimes e^{\mu} \right) \\
&&= Res \left[z^3 \exp\left( \sum_{k \geq 1}\frac{ -\gamma_2(-k)}{k}z^k\right) \exp\left( \sum_{k \geq 1}\frac{ -\gamma_2(k)}{-k}z^{-k}\right) e^{-\gamma_2}z^{-\gamma_2(0)} \right. \\
&&\hspace{1in}\cdot \left. \left( \sum_{k = 1,3,5,6} a_k \alp_k(-1) \otimes e^{\mu} \right)\right] \\
&&= Res \left[\eps(-\gamma_2,\mu)z^1 \exp\left( \sum_{k \geq 1}\frac{ -\gamma_2(-k)}{k}z^k\right) \left( I + \frac{-\gamma_2(1)}{-1}z^{-1}\right) \right. \\
&&\hspace{1in} \left.\cdot \left( \sum_{k = 1,3,5,6} a_k \alp_k(-1) \otimes e^{-\gamma_2 + \mu} \right)\right] \\
&&= 0.
\end{eqnarray*}

Hence, $$\yvomn{\otea{\gamma_i} + \otea{-\gamma_i}}{2} \cdot \left( \sum_{k = 1,3,5,6} a_k \alp_k(-1) \otimes e^{\mu} + \sum_{k = 1,3,5,6} -a_k \alp_k(-1) \otimes e^{\tau\mu}\right) =0. $$

Now the action of $\yvomn{B}{2}$ on $(3 \otimes e^{\mu + \alp_1 - \alp_6} + 3 \otimes e^{\tau\mu - \alp_1 + \alp_6})$ is calculated.  To decrease the number of calculations notice that $\herf{\gamma_1}{ \tau\mu-\alp_1+\alp_6} = \herf{-\gamma_1}{ \mu+\alp_1-\alp_6} = -3$, and $\herf{\gamma_3}{ \tau\mu-\alp_1+\alp_6} = \herf{-\gamma_3}{ \mu +\alp_1-\alp_6} = -3$ will be the only relevant dot products. \\
For $i = \{1,3 \}$, we have:
\begin{eqnarray*}
&&\yvomn{\otea{\gamma_i}}{2} \cdot \left( 3 \otimes e^{\tau\mu - \alp_1 + \alp_6} \right) \\
&&= Res \left[z^3 \exp\left( \sum_{k \geq 1}\frac{ \gamma_i(-k)}{k}z^k\right) \exp\left( \sum_{k \geq 1}\frac{ \gamma_i(k)}{-k}z^{-k}\right) e_{\gamma_i}z^{\gamma_i(0)} \right. \\
&& \hspace{100pt} \left. \cdot \left( 3 \otimes e^{\tau\mu - \alp_1 + \alp_6} \right)\right] \\
&&= Res \left[ \eps(\gamma_i,\tau\mu-\alp_1+\alp_6) \left(I +  \frac{\gamma_i(-1)}{1}z^{1}\right) \cdot \left( 3 \otimes e^{\gamma_i + \tau\mu - \alp_1 + \alp_6} \right) \right] \\
&&= 0.
\end{eqnarray*}
And,
\begin{eqnarray*}
&&\yvomn{\otea{-\gamma_i}}{2} \cdot \left( 3 \otimes e^{\mu + \alp_1 - \alp_6} \right) \\
&&= Res \left[z^3 \exp\left( \sum_{k \geq 1}\frac{ \gamma_i(-k)}{k}z^k\right) \exp\left( \sum_{k \geq 1}\frac{ -\gamma_i(k)}{-k}z^{-k}\right) e_{-\gamma_i}z^{-\gamma_i(0)} \right. \\
&& \hspace{100pt} \left. \cdot \left( 3 \otimes e^{\mu + \alp_1 - \alp_6} \right)\right] \\
&&= Res \left[  \eps(-\gamma_i,\mu+ \alp_1 - \alp_6) \left(I +  \frac{-\gamma_i(-1)}{1}z^{1}\right) \cdot \left( 3 \otimes e^{-\gamma_i + \mu + \alp_1 - \alp_6} \right) \right] \\
&&= 0.
\end{eqnarray*}

Therefore, $\yvomn{B}{2} \cdot P = 0$. \\

%%%%%%%%%%%%%
%%%%%%%%%%%%%

4) We show $P$ has $\yvomn{\omg}{0}$ eigenvalue $\frac{7}{5}$.
This calculation will also be shown in four parts the first of which is
\begin{eqnarray*}
&&\yvomn{A}{0} \cdot a_k \alp_k(-1) \otimes e^{\mu} \\
&=& \frac{1}{10}\left( 2(-\lam_1+\lam_6)(-1)(-\lam_1+\lam_6)(1) + (-\lam_1+\lam_6)^2(0)\right) \cdot a_k \alp_k(-1) \otimes e^{\mu} \\
&&\hspace{.1pt}+ \frac{1}{10}\left( 2(\lam_3-\lam_5)(-1)(\lam_3-\lam_5)(1) + (\lam_3-\lam_5)(0)^2 \right) \cdot a_k \alp_k(-1) \otimes e^{\mu} \\
&&\hspace{.1pt}+ \frac{1}{10}\left(  2(\lam_1-\lam_3+\lam_5-\lam_6)(-1)(\lam_1-\lam_3+\lam_5-\lam_6)(1) \right. \\
&& \hspace{60pt}\left.+ (\lam_1-\lam_3+\lam_5-\lam_6)^2(0) \right) \cdot a_k \alp_k(-1) \otimes e^{\mu} 
\end{eqnarray*}
\begin{eqnarray*}
&=& \left( \frac{1}{5}a_k\herf{-\lam_1+\lam_6}{\alp_k}(-\lam_1+\lam_6)(-1) +  \frac{1}{10}\herf{-\lam_1+\lam_6}{\mu}^2 a_k \alp_k(-1) \right)\otimes e^{\mu}\\
&&\hspace{.1pt}+  \left(\frac{1}{5}a_k\herf{\lam_3-\lam_5}{\alp_k}(\lam_3-\lam_5)(-1) + \frac{1}{10}\herf{\lam_3-\lam_5}{\mu}^2 a_k \alp_k(-1)\right) \otimes e^{\mu} \\
&&\hspace{.1pt}+  \left(\frac{1}{5}a_k\herf{\lam_1-\lam_3+\lam_5-\lam_6}{\alp_k}(\lam_1-\lam_3+\lam_5-\lam_6)(-1) \right.\\
&& \hspace{60pt}+\left.\frac{1}{10} \herf{\lam_1-\lam_3+\lam_5-\lam_6}{\mu}^2 a_k \alp_k(-1)\right) \otimes e^{\mu} \\
&=& \left( \frac{1}{5}a_k\herf{-\lam_1+\lam_6}{\alp_k}(-\lam_1+\lam_6)(-1) +  \frac{1}{5}a_k\herf{\lam_3-\lam_5}{\alp_k}(\lam_3-\lam_5)(-1) \right. \\
&&\hspace{.1pt}\left.+  \frac{1}{5}a_k\herf{\lam_1-\lam_3+\lam_5-\lam_6}{\alp_k}(\lam_1-\lam_3+\lam_5-\lam_6)(-1)+  \frac{1}{5}a_k \alp_k(-1) \right)\otimes e^{\mu}.
\end{eqnarray*} 

Summing over $k\in \{ 1,3,5,6\}$ and evaluating the dot products, the sum becomes:
$$ \left( \frac{6}{5}(-\lam_1+\lam_6)(-1) +  \frac{1}{5}\sum_{k = 1,3,5,6}a_k \alp_k(-1) \right)\otimes e^{\mu}.$$
But, $-2\alp_1 - \alp_3 + \alp_5 + 2\alp_6 = -3\lam_1+3\lam_6$, hence,
$$\yvomn{A}{0} \cdot \sum_{k = 1,3,5,6}a_k \alp_k(-1) \otimes e^{\mu} = \frac{3}{5}\left(\sum_{k = 1,3,5,6}a_k \alp_k(-1) \otimes e^{\mu}\right).$$

The calculation for the action on $-a_k \alp_k(-1) \otimes e^{\tau\mu}$, is similar and is shortened below:
\begin{eqnarray*}
&&\yvomn{A}{0} \cdot -a_k \alp_k(-1) \otimes e^{\tau\mu} \\
&=& \left( -\frac{1}{5}a_k\herf{-\lam_1+\lam_6}{\alp_k}(-\lam_1+\lam_6)(-1) -  \frac{1}{10}\herf{-\lam_1+\lam_6}{\tau\mu}^2a_k \alp_k(-1) \right)\otimes e^{\tau\mu}\\
&&\hspace{.1pt}+  \left(-\frac{1}{5}a_k\herf{\lam_3-\lam_5}{\alp_k}(\lam_3-\lam_5)(-1) - \frac{1}{10}\herf{\lam_3-\lam_5}{\tau\mu}^2 a_k \alp_k(-1)\right) \otimes e^{\tau\mu} \\
&&\hspace{.1pt}+  \left(-\frac{1}{5}a_k\herf{\lam_1-\lam_3+\lam_5-\lam_6}{\alp_k}(\lam_1-\lam_3+\lam_5-\lam_6)(-1) \right.\\
&& \hspace{60pt}-\left.\frac{1}{10} \herf{\lam_1-\lam_3+\lam_5-\lam_6}{\tau\mu}^2 a_k \alp_k(-1)\right) \otimes e^{\tau\mu} \\
&=& \left( -\frac{1}{5}a_k\herf{-\lam_1+\lam_6}{\alp_k}(-\lam_1+\lam_6)(-1) - \frac{1}{5}a_k\herf{\lam_3-\lam_5}{\alp_k}(\lam_3-\lam_5)(-1) \right. \\
&&\hspace{.1pt}\left.-  \frac{1}{5}a_k\herf{\lam_1-\lam_3+\lam_5-\lam_6}{\alp_k}(\lam_1-\lam_3+\lam_5-\lam_6)(-1) -  \frac{1}{5}a_k \alp_k(-1) \right)\otimes e^{\tau\mu}.
\end{eqnarray*}

Once again, summing over $k\in \{ 1,3,5,6\}$ and evaluating the dot products, the sum becomes:
$$ \left( \frac{6}{5}(\lam_1-\lam_6)(-1) + \frac{1}{5}\sum_{k = 1,3,5,6}-a_k \alp_k(-1) \right)\otimes e^{\tau\mu}.$$
From above, $2\alp_1 + \alp_3  - \alp_5 - 2\alp_6 = 3\lam_1-3\lam_6$, hence,\\
$$\yvomn{A}{0} \cdot \sum_{k = 1,3,5,6}-a_k \alp_k(-1) \otimes e^{\tau\mu} = \frac{3}{5}\left(\sum_{k = 1,3,5,6}-a_k \alp_k(-1) \otimes e^{\tau\mu}\right).$$ 

We also have
\begin{eqnarray*}
&&\yvomn{A}{0} \cdot \left(3 \otimes e^{\mu + \alp_1 - \alp_6} + 3 \otimes e^{\tau\mu - \alp_1 + \alp_6} \right)\\
&=& \frac{1}{10}\left( \cdots + (-\lam_1+\lam_6)^2(0) + \cdots\right) \cdot \left(3 \otimes e^{\mu + \alp_1 - \alp_6} + 3 \otimes e^{\tau\mu - \alp_1 + \alp_6} \right) \\
&&+ \frac{1}{10}\left( \cdots + (\lam_3-\lam_5)^2(0) + \cdots\right) \cdot \left(3 \otimes e^{\mu + \alp_1 - \alp_6} + 3 \otimes e^{\tau\mu - \alp_1 + \alp_6} \right) \\
&&+ \frac{1}{10}\left( \cdots + (\lam_1-\lam_3+\lam_5-\lam_6)^2(0) + \cdots\right) \cdot \left(3 \otimes e^{\mu + \alp_1 - \alp_6} + 3 \otimes e^{\tau\mu - \alp_1 + \alp_6} \right) \\
&=& \frac{3}{10} \left(\herf{-\lam_1+\lam_6}{ \mu + \alp_1 - \alp_6}^2+\herf{\lam_3-\lam_5}{ \mu + \alp_1 - \alp_6}^2 \right.\\
&&\hspace{40pt} \left.+\ \herf{\lam_1-\lam_3+\lam_5-\lam_6}{\mu + \alp_1 - \alp_6}^2\right)\otimes e^{\mu + \alp_1 - \alp_6} \\
&&+ \frac{3}{10} \left(\herf{-\lam_1+\lam_6}{ \tau\mu - \alp_1 + \alp_6}^2+\herf{\lam_3-\lam_5}{ \tau\mu - \alp_1 + \alp_6}^2 \right.\\
&&\hspace{40pt} \left.+\ \herf{\lam_1-\lam_3+\lam_5-\lam_6}{\tau\mu - \alp_1 + \alp_6}^2\right)\otimes e^{\tau\mu - \alp_1 + \alp_6} \\
&=& \frac{3}{10}(4+1+1)\otimes e^{\mu + \alp_1 - \alp_6} + \frac{3}{10}(4+1+1)\otimes e^{\tau\mu - \alp_1 + \alp_6}\\
&=&\frac{9}{5}\otimes e^{\mu + \alp_1 - \alp_6} + \frac{9}{5}\otimes e^{\tau\mu - \alp_1 + \alp_6} \\
&=&\frac{3}{5}\left(3 \otimes e^{\mu + \alp_1 - \alp_6} + 3 \otimes e^{\tau\mu - \alp_1 + \alp_6} \right).
\end{eqnarray*} 

The next two calculations will be for $\yvomn{B}{0}$ acting on the two parts of $P$.  The first is given by
\begin{eqnarray*}
&&\yvomn{\otea{\gamma_i}}{0} \cdot a_k \alp_k(-1) \otimes e^{\mu} \\
&&= Res \!\left[ z^1 \exp\left( \sum_{k \geq 1}\frac{ \gamma_i(-k)}{k}z^k\right) \exp\left( \sum_{k \geq 1}\frac{ \gamma_i(k)}{-k}z^{-k}\right) e_{\gamma_i}z^{\gamma_i(0)} \right. \\
&& \hspace{100pt} \left. \cdot a_k \alp_k(-1) \otimes e^{\mu} \right] \\
&&= Res \!\left[ \eps(\gamma_i, \mu) z^{1+ \herf{\gamma_i}{\mu}} \exp\left( \sum_{k \geq 1}\frac{ \gamma_i(-k)}{k}z^k\right) \left( I + \frac{\gamma_i(1)}{-1}z^{-1}\right) \right. \\
&& \hspace{100pt} \left. \cdot a_k \alp_k(-1) \otimes e^{\gamma_i + \mu} \right] \\
&&= Res\! \left[ \eps(\gamma_i, \mu) z^{1+ \herf{\gamma_i}{\mu}} \left( I + \frac{\gamma_i(-1)}{1}z^1\right) \right. \\
&& \hspace{100pt} \left. \cdot \left( a_k \alp_k(-1) \otimes e^{\gamma_i + \mu} - z^{-1}a_k(\gamma_i,\alp_k) \otimes e^{\gamma_i + \mu}\right)\right].
\end{eqnarray*} 

It is easier to break into cases, first if $\herf{\gamma_i}{ \mu} = -1$, then we have: \newpage
\begin{eqnarray*}
&&Res \left[ \eps(\gamma_i, \mu) z^{1+ \herf{\gamma_i}{\mu}} \left( I + \frac{\gamma_i(-1)}{1}z^1\right) \right. \\
&& \hspace{100pt} \left. \cdot \left( a_k \alp_k(-1) \otimes e^{\gamma_i + \mu} - z^{-1}a_k(\gamma_i,\alp_k) \otimes e^{\gamma_i + \mu}\right)\right] \\
&&=Res \left[ \eps(\gamma_i, \mu) \left( I + \frac{\gamma_i(-1)}{1}z^1\right) \right. \\
&& \hspace{100pt} \left. \cdot \left( a_k \alp_k(-1) \otimes e^{\gamma_i + \mu} - z^{-1}a_k(\gamma_i,\alp_k) \otimes e^{\gamma_i + \mu}\right)\right] \\
&&=Res \left[ -\eps(\gamma_i, \mu)z^{-1}a_k(\gamma_i,\alp_k) \otimes e^{\gamma_i + \mu}\right] \\
&&=-\eps(\gamma_i, \mu)a_k(\gamma_i,\alp_k) \otimes e^{\gamma_i + \mu}.
\end{eqnarray*}

If $\herf{\gamma_i}{ \mu} = -2$, then we have: 	
\begin{eqnarray*}
&&Res \left[ \eps(\gamma_i, \mu) z^{1+ \herf{\gamma_i}{\mu}} \left( I + \frac{\gamma_i(-1)}{1}z^1\right) \right. \\
&& \hspace{100pt} \left. \cdot \left( a_k \alp_k(-1) \otimes e^{\gamma_i + \mu} - z^{-1}a_k(\gamma_i,\alp_k) \otimes e^{\gamma_i + \mu}\right)\right] \\
&&=Res \left[ \eps(\gamma_i, \mu) z^{-1} \left( I + \frac{\gamma_i(-1)}{1}z^1\right) \right. \\
&& \hspace{100pt} \left. \cdot \left( a_k \alp_k(-1) \otimes e^{\gamma_i + \mu} - z^{-1}a_k(\gamma_i,\alp_k) \otimes e^{\gamma_i + \mu}\right)\right] \\
&&=Res \left[ \eps(\gamma_i, \mu) z^{-1} \left( a_k \alp_k(-1) \otimes e^{\gamma_i + \mu} - a_k(\gamma_i,\alp_k)\gamma_i(-1) \otimes e^{\gamma_i + \mu}\right)\right] \\
&&= \eps(\gamma_i, \mu) \left( a_k \alp_k(-1) \otimes e^{\gamma_i + \mu} - a_k(\gamma_i,\alp_k)\gamma_i(-1) \otimes e^{\gamma_i + \mu}\right).
\end{eqnarray*}

Analogously, if $\herf{\gamma_i}{ \tau\mu} = -1$, then
\begin{eqnarray*}
&&\yvomn{\otea{\gamma_i}}{0} \cdot -a_k \alp_k(-1) \otimes e^{\tau\mu}\\
&&=\eps(\gamma_i, \tau\mu)a_k(\gamma_i,\alp_k) \otimes e^{\gamma_i + \tau\mu}
\end{eqnarray*}
and if $\herf{\gamma_i}{ \tau\mu} = -2$, then 
\begin{eqnarray*}
&&\yvomn{\otea{\gamma_i}}{0} \cdot -a_k \alp_k(-1) \otimes e^{\tau\mu}\\
&&=\eps(\gamma_i, \tau\mu) \left( -a_k \alp_k(-1) \otimes e^{\gamma_i + \tau\mu} + a_k(\gamma_i,\alp_k)\gamma_i(-1) \otimes e^{\gamma_i + \tau\mu}\right).
\end{eqnarray*}

Let $\phi$ be $(\mu+\alp_1-\alp_6)$ or $(\tau\mu-\alp_1+\alp_6)$, then
\begin{eqnarray*}
&&\yvomn{\otea{\gamma_i}}{0} \cdot 3 \otimes e^{\phi} \\
&=& Res \left[ z^1 \exp\left( \sum_{k \geq 1}\frac{ \gamma_i(-k)}{k}z^k\right) e_{\gamma_i}z^{\gamma_i(0)}\cdot 3 \otimes e^{\phi} \right] \\
&=& Res \left[ \eps(\gamma_i, \phi) z^{1+ \herf{\gamma_i}{\phi}} \left( I + \frac{\gamma_i(-1)}{1}z^1\right)\cdot 3 \otimes e^{\gamma_i + \phi} \right]. 
\end{eqnarray*} 

This calculation simplifies if you see that the only useful dot products here are: \\
$\herf{\gamma_1}{ \tau\mu - \alp_1 + \alp_6} = \herf{-\gamma_1}{ \mu + \alp_1 - \alp_6} = -3$ and \\ $\herf{\gamma_3}{ \tau\mu - \alp_1 + \alp_6} = \herf{-\gamma_3}{ \mu + \alp_1 - \alp_6} = -3.$ 

If $\herf{\gamma_i}{ \phi} = -3$, then  
\begin{eqnarray*}
&&Res \left[ \eps(\gamma_i, \phi) z^{1+ (\gamma_i,\phi)} \left( I + \frac{\gamma_i(-1)}{1}z^1\right)\cdot 3 \otimes e^{\gamma_i + \phi} \right]\\
&&= Res \left[ \eps(\gamma_i, \mu) z^{-2} \left( I + \frac{\gamma_i(-1)}{1}z^1\right)\cdot 3 \otimes e^{\gamma_i + \phi} \right] \\
&&= 3\eps(\gamma_i, \mu) \gamma_i(-1) \otimes e^{\gamma_i + \phi}. 
\end{eqnarray*} 

Breaking the action into parts and using the calculations from above we have
\begin{eqnarray*}
&& -\frac{1}{5}\yvomn{\otea{\gamma_1} + \otea{-\gamma_1}}{0} \cdot \left( \sum_{k \in 1,3,5,6}a_k \alp_k(-1) \otimes e^{\mu}  + \sum_{k \in 1,3,5,6}-a_k \alp_k(-1) \otimes e^{\tau\mu}\right)\\
&&=\frac{2}{5}\left( 3 \otimes e^{\mu + \alp_1 - \alp_6}  + 3 \otimes e^{\tau\mu - \alp_1 + \alp_6}\right),
\end{eqnarray*}
and 
\begin{eqnarray*}
&&-\frac{1}{5}\yvomn{\otea{\gamma_2} + \otea{-\gamma_2}}{0} \cdot \left( \sum_{k \in 1,3,5,6}a_k \alp_k(-1) \otimes e^{\mu}  + \sum_{k \in 1,3,5,6}-a_k \alp_k(-1) \otimes e^{\tau\mu}\right)\\
&&=\frac{1}{5} \left[  \left( -2 \alp_1 -\alp_3 + \alp_5 +2 \alp_6\right)(-1) \otimes e^{\mu} + \left( 2 \alp_1 +\alp_3 - \alp_5 -2 \alp_6\right)(-1) \otimes e^{\tau\mu},\right]\end{eqnarray*}
also,
\begin{eqnarray*}
&&\frac{1}{5}\yvomn{\otea{\gamma_3} + \otea{-\gamma_3}}{0} \cdot \left( \sum_{k \in 1,3,5,6}a_k \alp_k(-1) \otimes e^{\mu}  + \sum_{k \in 1,3,5,6}-a_k \alp_k(-1) \otimes e^{\tau\mu}\right)\\
&&=\frac{2}{5}\left( 3 \otimes e^{\mu + \alp_1 - \alp_6}  + 3 \otimes e^{\tau\mu - \alp_1 + \alp_6}\right).
\end{eqnarray*}
Now we give the results for $\yvomn{B}{0}$ acting on $\left( 3 \otimes e^{\mu + \alp_1 - \alp_6}  + 3 \otimes e^{\tau\mu - \alp_1 + \alp_6}\right)$, first we have
\begin{eqnarray*}
&& -\frac{1}{5}\yvomn{\otea{\gamma_1} + \otea{-\gamma_1}}{0} \cdot \left( 3 \otimes e^{\mu + \alp_1 - \alp_6}  + 3 \otimes e^{\tau\mu - \alp_1 + \alp_6}\right) \\
&&=\frac{3}{5}\left( (-\alp_1(-1) + \alp_6(-1)) \otimes e^{\mu}  + (\alp_1(-1) - \alp_6(-1)) \otimes e^{\tau\mu}\right),
\end{eqnarray*}
and 
\begin{eqnarray*}
&&-\frac{1}{5}\yvomn{\otea{\gamma_2} + \otea{-\gamma_2}}{0} \cdot \left( 3 \otimes e^{\mu + \alp_1 - \alp_6}  + 3 \otimes e^{\tau\mu - \alp_1 + \alp_6}\right)\\
&&=0.
\end{eqnarray*}
Lastly,
\begin{eqnarray*}
&&\frac{1}{5}\yvomn{\otea{\gamma_3} + \otea{-\gamma_3}}{0} \cdot \left( 3 \otimes e^{\mu + \alp_1 - \alp_6}  + 3 \otimes e^{\tau\mu - \alp_1 + \alp_6}\right)\\
&&=\frac{3}{5}(-\alp_1(-1) - \alp_3(-1) + \alp_5(-1) + \alp_6(-1)) \otimes e^{\mu}  \\
&&\hspace{10pt}+\frac{3}{5}(\alp_1(-1) + \alp_3(-1) - \alp_5(-1) - \alp_6(-1)) \otimes e^{\tau\mu}.
\end{eqnarray*}

Adding all these actions together gives  
$$\yvomn{\omg}{0} \cdot P = \frac{7}{5}P.$$ 
Hence $P$ is a HWV for Vir$\left(\frac{4}{5}, \frac{7}{5}\right) \otimes W^{\Omg_4}$. 
\end{proof}

To determine the HWV for the module $Vir\left(\frac{4}{5}, 3\right) \otimes W^{\Omg_0}$, the use of the intertwining operators on $V_P$ and two of the HWVs from $V^{\Lam_1}$ and $V^{\Lam_6}$ play a key role. We start with two vectors, to be found in Section 6.2, each of which is a HWV in a summand $Vir\left(\frac{4}{5}, \frac23\right) \otimes W^{\Omg_0}$ that occurs in $V^{\Lam_1}$ and $V^{\Lam_6}$.  From $V^{\Lam_1}$ we have
$R = \otea{-\lam_1 + \lam_6} + \otea{\lam_3 - \lam_5} - \otea{\lam_1 - \lam_3 + \lam_5 - \lam_6}$, and from $V^{\Lam_6}$, $\tau R = \otea{\lam_1 - \lam_6} + $\\$\otea{-\lam_3 + \lam_5} - \otea{-\lam_1 + \lam_3 - \lam_5 + \lam_6}$.

\begin{lem}
$U = \yvomn{R}{-\frac{7}{3}}\cdot \tau R -\frac{5}{2} \yvomn{\omg}{-3}\cdot \otea{0} = \left\{ R \right\}_{-\frac{8}{3}}\cdot \tau R -\frac{5}{2} \left\{ \omg \right\}_{-2} \cdot \otea{0}$ 

is the HWV for Vir$\left(\frac{4}{5}, 3\right) \otimes W^{\Omg_0}$.
\end{lem}

\begin{proof}
1)$\yvomn{\otea{-\theta}}{1}\cdot U = \{ \otea{-\theta} \}_1 \cdot U=  0$

Before calculating, note that $R$ is the HWV in a 1 dimensional $F_4$-module, hence in particular $\{-\theta\}_0\cdot R =\{ \otea{-\theta} \}_0 \cdot R = 0$. 
We now have
\begin{eqnarray*}
&&\{ \otea{-\theta} \}_1 \cdot \left( \left\{ R \right\}_{-\frac{8}{3}} \cdot \tau R \right) \\
&=& \left[ \{ \otea{-\theta} \}_1 , \left\{ R \right\}_{-\frac{8}{3}} \right] \cdot \tau R + \left\{ R \right\}_{-\frac{8}{3}} \left( \{ \otea{-\theta} \}_1 \cdot \tau R \right) \\
&=& \left[ \{ \otea{-\theta} \}_1 , \left\{ R \right\}_{-\frac{8}{3}} \right] \cdot \tau R + 0 \\
&=& \sum_{k \geq 0} \binom{1}{k} \left\{ \left\{ \otea{-\theta} \right\}_k \cdot R\right\}_{1-\frac{8}{3}-k} \cdot \tau R\\
&=& \left\{ \left\{ \otea{-\theta} \right\}_0 \cdot R \right\}_{-\frac{8}{3}} \cdot \tau R + \left\{ \{ \otea{-\theta} \}_1 \cdot R \right\}_{-\frac{5}{3}} \cdot \tau R  \\
&=& \left\{ 0 \right\}_{-\frac{8}{3}} \cdot \tau R + \left\{ 0 \right\}_{-\frac{5}{3}} \cdot \tau R \\
&=& 0.
\end{eqnarray*}

Since $\omg$ is the generator for coset Virasoro, then the operators generated by this vector commute with the $F_4^{(1)}$ operators.  This gives:
 $$\{ \otea{-\theta} \}_1 \cdot (\left\{ \omg \right\}_{-2} \cdot \otea{0}) =  \left\{ \omg \right\}_{-2} \cdot (\{ \otea{-\theta} \}_1 \cdot \otz) = 0.$$
Therefore $U$ is killed by $\yvomn{\otea{-\theta}}{1}$ \\

2)$\yvomn{\beta}{0} \cdot U  = \{\beta\}_0 \cdot U = 0$

\begin{eqnarray*}
\{\beta\}_0 \cdot \left( \left\{ R \right\}_{-\frac{8}{3}} \cdot \tau R \right) &=& \left[ \{\beta\}_0 , \left\{ R \right\}_{-\frac{8}{3}} \right] \cdot \tau R + \left\{ R \right\}_{-\frac{8}{3}} \left( \{\beta\}_0 \cdot \tau R \right) \\
&=& \left[ \{\beta\}_0 , \left\{ R \right\}_{-\frac{8}{3}} \right] \cdot \tau R + 0 \\
&=& \sum_{k \geq 0} \binom{0}{k} \left\{ \left\{ \beta \right\}_k \cdot R\right\}_{0-\frac{8}{3}-k} \cdot \tau R\\
&=& \left\{ \{\beta\}_0 \cdot R \right\}_{-\frac{8}{3}} \cdot \tau R \\
&=& \left\{ 0 \right\}_{-\frac{8}{3}} \cdot \tau R \\
&=& 0.
\end{eqnarray*}

Once again we use $\{ \omg\}_n$ operators commute with the $F_4^{(1)}$ operators, which gives:
 $$\{\otea{\beta}\}_0 \cdot (\left\{ \omg \right\}_{-2} \cdot \otea{0}) =  \left\{ \omg \right\}_{-2} \cdot (\{\otea{\beta}\}_0 \cdot \otea{0}) = 0.$$

Instead of separating the action of the three Virasoro operators, they are calculated together because they all need some of the same techniques to be completed.  \\

3) $\yvomn{\omg}{n}\cdot U = \{\omg\}_{n+1} \cdot U = 0$, for $n\in \{ 2,3\}$, and $\yvomn{\omg}{0}\cdot U = \{\omg\}_{1} \cdot U = 3U$.

Three preliminary calculations for all three of the actions.
\begin{eqnarray*}
\left\{ \omg \right\}_3 \cdot \left( \left\{ R \right\}_{-\frac{8}{3}} \cdot \tau R \right) &=& \left[ \left\{ \omg \right\}_3 , \left\{ R \right\}_{-\frac{8}{3}} \right] \cdot \tau R + \left\{ R \right\}_{-\frac{8}{3}} \left( \left\{ \omg \right\}_3 \cdot \tau R \right) \\
&=& \left[ \left\{ \omg \right\}_3 , \left\{ R \right\}_{-\frac{8}{3}} \right] \cdot \tau R + 0 \\
&=& \sum_{k \geq 0} \binom{3}{k} \left\{ \left\{ \omg \right\}_k \cdot R\right\}_{3-\frac{8}{3}-k} \cdot \tau R \\
&=& \left\{ \left\{ \omg \right\}_0 \cdot R \right\}_{\frac{1}{3}} \cdot \tau R + 3 \left\{ \left\{ \omg \right\}_1 \cdot R \right\}_{-\frac{2}{3}} \cdot \tau R \\
&&+ 3\left\{ \left\{ \omg \right\}_2 \cdot R \right\}_{-\frac{5}{3}} \cdot \tau R + \left\{ \left\{ \omg \right\}_3 \cdot R \right\}_{-\frac{8}{3}} \cdot \tau R\\
&=& \left\{ \left\{ \omg \right\}_0 \cdot R \right\}_{\frac{1}{3}} \cdot \tau R + 3 \left\{ 2/3 R \right\}_{-\frac{2}{3}} \cdot \tau R \\
&&+ \left\{ 0 \right\}_{-\frac{5}{3}} \cdot \tau R + \left\{ 0\right\}_{-\frac{8}{3}} \cdot \tau R\\
&=& \left\{ \left\{ \omg \right\}_0 \cdot R \right\}_{\frac{1}{3}} \cdot \tau R + 2 \left\{ R \right\}_{-\frac{2}{3}} \cdot \tau R 
\end{eqnarray*}

Also we have 
\begin{eqnarray*}
\left\{ \omg \right\}_2 \cdot \left( \left\{ R \right\}_{-\frac{8}{3}} \cdot \tau R \right) &=& \left[ \left\{ \omg \right\}_2 , \left\{ R \right\}_{-\frac{8}{3}} \right] \cdot \tau R + \left\{ R \right\}_{-\frac{8}{3}} \left( \left\{ \omg \right\}_2 \cdot \tau R \right) \\
&=& \left[ \left\{ \omg \right\}_2 , \left\{ R \right\}_{-\frac{8}{3}} \right] \cdot \tau R + 0 \\
&=& \sum_{k \geq 0} \binom{2}{k} \left\{ \left\{ \omg \right\}_k \cdot R\right\}_{2-\frac{8}{3}-k} \\
&=& \left\{ \left\{ \omg \right\}_0 \cdot R \right\}_{-\frac{2}{3}} \cdot \tau R + 2 \left\{ \left\{ \omg \right\}_1 \cdot R \right\}_{-\frac{5}{3}} \cdot \tau R \\
&&+ \left\{ \left\{ \omg \right\}_2 \cdot R \right\}_{-\frac{8}{3}} \cdot \tau R\\
&=& \left\{ \left\{ \omg \right\}_0 \cdot R \right\}_{-\frac{2}{3}} \cdot \tau R + 2 \left\{ 2/3 R \right\}_{-\frac{5}{3}} \cdot \tau R \\
&&+ \left\{ 0 \right\}_{-\frac{8}{3}} \cdot \tau R \\
&=& \left\{ \left\{ \omg \right\}_0 \cdot R \right\}_{-\frac{2}{3}} \cdot \tau R + \frac{4}{3} \left\{ R \right\}_{-\frac{5}{3}} \cdot \tau R
\end{eqnarray*}

and 
\begin{eqnarray*}
\left\{ \omg \right\}_1 \cdot \left( \left\{ R \right\}_{-\frac{8}{3}} \cdot \tau R \right) &=& \left[ \left\{ \omg \right\}_1 , \left\{ R \right\}_{-\frac{8}{3}} \right] \cdot \tau R + \left\{ R \right\}_{-\frac{8}{3}} \left( \left\{ \omg \right\}_1 \cdot \tau R \right) \\
&=& \left[ \left\{ \omg \right\}_1 , \left\{ R \right\}_{-\frac{8}{3}} \right] \cdot \tau R +  \left\{ R \right\}_{-\frac{8}{3}} \cdot \left( 2/3 \tau R \right) \\
&=& \sum_{k \geq 0} \binom{1}{k} \left\{ \left\{ \omg \right\}_k \cdot R\right\}_{1-\frac{8}{3}-k} + \frac{2}{3} \left( \left\{ R \right\}_{-\frac{8}{3}} \cdot \tau R \right) \\
&=& \left\{ \left\{ \omg \right\}_0 \cdot R \right\}_{-\frac{5}{3}} \cdot \tau R + \left\{ \left\{ \omg \right\}_1 \cdot R \right\}_{-\frac{8}{3}} \cdot \tau R + \frac{2}{3} \left( \left\{ R \right\}_{-\frac{8}{3}} \cdot \tau R \right)  \\
&=& \left\{ \left\{ \omg \right\}_0 \cdot R \right\}_{-\frac{5}{3}} \cdot \tau R +  \left\{ 2/3 R \right\}_{-\frac{8}{3}} \cdot \tau R + \frac{2}{3} \left( \left\{ R \right\}_{-\frac{8}{3}} \cdot \tau R \right)  \\
&=& \left\{ \left\{ \omg \right\}_0 \cdot R \right\}_{-\frac{5}{3}} \cdot \tau R +  \frac{2}{3} \left( \left\{ R \right\}_{-\frac{8}{3}} \cdot \tau R \right)  + \frac{2}{3} \left( \left\{ R \right\}_{-\frac{8}{3}} \cdot \tau R \right)  \\
&=& \left\{ \left\{ \omg \right\}_0 \cdot R \right\}_{-\frac{5}{3}} \cdot \tau R +  \frac{4}{3}  \left\{ R \right\}_{-\frac{8}{3}} \cdot \tau R. 
\end{eqnarray*}

A necessary vector for these calculations is $\{\omg\}_0 \cdot R$.  Record the following observation about $\{\omg\}_0 \cdot R$, so it can be used as necessary.  We know $\{ \omg_{E_6} \}_1 \cdot R = \frac{2}{3}R$ and also $\{ \omg \}_1 \cdot R = \frac{2}{3}R$, therefore $\{ \omg_{F_4} \}_1 \cdot R = 0 \cdot R$ and then we must have $\{ \omg_{F_4} \}_0 \cdot R = 0$. We now have $ \{ \omg \}_0 \cdot R =\left(\{ \omg_{E_6} \}_0 - \{ \omg_{F_4} \}_0 \right) \cdot R =  \{ \omg_{E_6} \}_0 \cdot R.$

By the derivative property for the VOA of $\Es$,  $\{ \{ \omg_{E_6} \}_0 \cdot R\}_n = -n\{ R\}_{n-1}$ and therefore $\{ \{ \omg \}_0 \cdot R\}_n = -n\{ R\}_{n-1}$.  We need a few lemmas before we finish these calculations. 

\begin{lem}If $\herf{\alp}{\alp} = \frac{4}{3}$, then $\{ \otea{\alp}\}_{-2/3} \cdot (\otea{-\alp}) = \alp(-1)\otimes e^0$.
\end{lem}

\begin{proof} 
\begin{eqnarray*}
&&\{ \otea{\alp}\}_{-2/3} \cdot (\otea{-\alp}) \\ 
&=& Res\left( z^{-2/3} \exp\left( \sum_{k \geq 1}\frac{ \alp(-k)}{k}z^k\right) e_{\alp}z^{\alp(0)}  \cdot \otea{-\alp}  \right) \\
&=& Res\left( z^{-2}\eps(\alp,-\alp) \exp\left( \sum_{k \geq 1}\frac{ \alp(-k)}{k}z^k\right)\cdot \otea{0} \right)  \\
&=& Res\left( z^{-1} \alp(-1) \otimes e^{0} \right) \\
&=& \alp(-1) \otimes e^0
\end{eqnarray*}
\end{proof}

\begin{lem}
If $\herf{\alp}{\beta} = \frac23$ , then $\{ \alp(-1)\otimes e^{\alp}\}_{-2/3} \cdot (\otea{\beta}) = 0$.
\end{lem}

\begin{proof} 
\begin{eqnarray*}
&&\{ \otea{\alp}\}_{-2/3} \cdot \otea{\beta} \\ 
&=& Res\left( z^{-2/3} \exp\left( \sum_{k \geq 1}\frac{ \alp(-k)}{k}z^k\right) e_{\alp}z^{\alp(0)}  \cdot \otea{\beta}  \right) \\
&=& Res\left( z^{0}\eps(\alp,\beta) \exp\left( \sum_{k \geq 1}\frac{ \alp(-k)}{k}z^k\right)\cdot \otea{\alp+\beta} \right)  \\
&=& 0
\end{eqnarray*}
\end{proof}

Since the exponents of $e$ in $R$ and $\tau R$ satisfy one of these lemmas, then we have
\begin{eqnarray*}
&&\{ R\}_{-2/3} \cdot \tau R \\
&&=((-\lam_1+\lam_6)(-1) + (\lam_3 - \lam_5)(-1) + (\lam_1  - \lam_3 +\lam_5 - \lam_6)(-1))\otimes e^0 = 0.
\end{eqnarray*}

\begin{lem}
If $\herf{\alp}{\beta} = \frac23$, then $\{ \otea{\alp}\}_{-5/3} \cdot (\otea{\beta}) = \eps(\alp,\beta)\otea{\alp+\beta}.$
\end{lem}
\begin{proof}
\begin{eqnarray*}
&&\{ \otea{\alp}\}_{-5/3} \cdot (\otea{\beta}) \\ 
&=& Res\left( z^{-5/3} \exp\left( \sum_{k \geq 1}\frac{ \alp(-k)}{k}z^k\right) e_{\alp}z^{\alp(0)} \cdot \otea{\beta}  \right) \\
&=& Res\left( \eps(\alp,\beta) z^{-1} \exp\left( \sum_{k \geq 1}\frac{ \alp(-k)}{k}z^k\right)\cdot \otea{\alp+\beta} \right)  \\
&=& \eps(\alp,\beta) \otea{\alp+\beta}
\end{eqnarray*}
\end{proof}

\begin{lem}
If $\herf{\alp}{\alp} = \frac{4}{3}$, then $\{ \otea{\alp}\}_{-5/3} \cdot (\otea{-\alp}) = \frac{1}{2}(\alp^2(-1)+\alp(-2))\otimes e^0.$
\end{lem}
 
\begin{proof}
\begin{eqnarray*}
&&\{ \otea{\alp}\}_{-5/3} \cdot (\otea{-\alp}) \\ 
&=& Res\left( z^{-5/3} \exp\left( \sum_{k \geq 1}\frac{ \alp(-k)}{k}z^k\right) e_{\alp}z^{\alp(0)}  \cdot \otea{-\alp}  \right) \\
&=& Res\left( z^{-3} \eps(\alp,-\alp)\exp\left( \sum_{k \geq 1}\frac{ \alp(-k)}{k}z^k\right)\cdot \otea{0} \right)  \\
&=& Res\left( z^{-1} \eps(\alp,-\alp)\left(\frac{1}{2}\alp(-2) + \frac{1}{2}\alp(-1)^2 \right) \otimes e^{0} \right) \\
&=& \frac{1}{2}(\alp(-1)^2+\alp(-2))\otimes e^0
\end{eqnarray*}
\end{proof}

Using these two lemmas we have that  
\begin{eqnarray*}
&&\{ R \}_{-5/3} \cdot \tau R \\ 
&=& \otea{-\lam_1 - \lam_3 + \lam_5 + \lam_6} - \otea{-2\lam_1 + \lam_3 - \lam_5 + 2\lam_6} \\
&&\hspace{10pt}+ \frac{1}{2}((-\lam_1 + \lam_6)(-2) + (-\lam_1 + \lam_6)(-1)^2) \otimes e^0 \\
&&\hspace{10pt}+\otea{\lam_1 + \lam_3 - \lam_5 - \lam_6} - \otea{-\lam_1 + 2\lam_3 - 2\lam_5 + \lam_6} \\
&&\hspace{10pt}+ \frac{1}{2}( (\lam_3 -\lam_5)(-2) + (\lam_3 - \lam_5)(-1)^2) \otimes e^0\\
&&\hspace{10pt}- \otea{2\lam_1 - \lam_3 + \lam_5 - 2\lam_6}- \otea{\lam_1 - 2\lam_3 + 2\lam_5 - \lam_6} \\
&&\hspace{10pt}+ \frac{1}{2}((\lam_1 - \lam_3 +\lam_5 - \lam_6)(-2) + (\lam_1 - \lam_3 + \lam_5 - \lam_6)(-1)^2) \otimes e^0\\
&=& \otea{-\alp_1 - \alp_3 + \alp_5 + \alp_6} - \otea{-\alp_1 + \alp_6} + \frac{1}{2}(-\lam_1 + \lam_6)(-1)^2 \otimes e^0 \\
&&\hspace{10pt}+ \otea{\alp_1 + \alp_3 - \alp_5 - \alp_6} - \otea{-\alp_3 + \alp_5} + \frac{1}{2}(\lam_3 - \lam_5)(-1)^2\otimes e^0 \\
&&\hspace{10pt}- \otea{\alp_1 - \alp_6}  - \otea{\alp_3 - \alp_5} + \frac{1}{2}(\lam_1 - \lam_3 + \lam_5 - \lam_6)(-1)^2 \otimes e^0 \\
&=& \frac{1}{2}((-\lam_1 + \lam_6)(-1)^2 + (\lam_3 - \lam_5)(-1)^2 + (\lam_1 - \lam_3 + \lam_5 - \lam_6)(-1)^2) \otimes e^0 \\
&&-(\otea{\gamma_1}+ \otea{-\gamma_1}) - (\otea{\gamma_2}+ \otea{-\gamma_2}) + (\otea{\gamma_3}+ \otea{-\gamma_3}) \\
&=& 5\omg.
\end{eqnarray*}

We also have the following observation 
\begin{eqnarray*}
\left\{ \omg \right\}_n \cdot (\left\{ \omg \right\}_{-2} \cdot \otea{0}) &=& [ \{\omg\}_n, \{ \omg\}_{-2}] \cdot \otz + \{ \omg \}_{-2} \cdot ( \{ \omg \}_{n} \cdot \otz)\\
 &=& (n+2)\{ \omg\}_{n-3} \cdot \otz + \{ \omg \}_{-2} \cdot ( \{ \omg \}_{n} \cdot \otz).
\end{eqnarray*}
Since $ \{ \omg \}_{n} \cdot \otz = 0$ for $n \geq 0$, then for $n \in \{ 0,1,2,3\}$
\begin{eqnarray*}
\left\{ \omg \right\}_n \cdot (\left\{ \omg \right\}_{-2} \cdot \otea{0}) &=& (n+2)\{ \omg\}_{n-3} \cdot \otz. 
\end{eqnarray*}

We have
\begin{eqnarray*}
\left\{ \omg \right\}_3 \cdot U &=& \left\{ \left\{ \omg \right\}_0 \cdot R \right\}_{\frac{1}{3}} \cdot \tau R + 2 \left\{ R \right\}_{-\frac{2}{3}} \cdot \tau R -\frac{5}{2} (5)\{ \omg\}_{0} \cdot \otz \\
&=& -\frac{1}{3}\left\{ R \right\}_{-\frac{2}{3}} \cdot \tau R + 2 \left\{ R \right\}_{-\frac{2}{3}} \cdot \tau R -\frac{5}{2} (5)\{ \omg\}_{0} \cdot \otz \\
&=& \frac{5}{3}\left\{ R \right\}_{-\frac{2}{3}} \cdot \tau R - 0 \\ 
&=& 0,
\end{eqnarray*}

\begin{eqnarray*}
\left\{ \omg \right\}_2 \cdot U &=& \left\{ \left\{ \omg \right\}_0 \cdot R \right\}_{-\frac{2}{3}} \cdot \tau R + 4/3 \left\{ R \right\}_{-\frac{5}{3}} \cdot \tau R -\frac{5}{2} (4)\{ \omg\}_{-1} \cdot \otz \\
&=& \frac{2}{3}\left\{ R \right\}_{-\frac{5}{3}} \cdot \tau R + 4/3 \left\{ R \right\}_{-\frac{5}{3}} \cdot \tau R -10\{ \omg\}_{-1} \cdot \otz \\
&=& 2 \left\{ R \right\}_{-\frac{5}{3}} \cdot \tau R -  2 \left\{ R \right\}_{-\frac{5}{3}} \cdot \tau R \\ 
&=& 0,
\end{eqnarray*}

and lastly
\begin{eqnarray*}
\left\{ \omg \right\}_1 \cdot U &=& \left\{ \left\{ \omg \right\}_0 \cdot R \right\}_{-\frac{5}{3}} \cdot \tau R +  \frac{4}{3} \left\{ R \right\}_{-\frac{8}{3}} \cdot \tau R -\frac{5}{2} (3)\{ \omg\}_{-2} \cdot \otz \\
&=& \frac{5}{3}\left\{  R \right\}_{-\frac{8}{3}} \cdot \tau R +  \frac{4}{3} \left\{ R \right\}_{-\frac{8}{3}} \cdot \tau R -3\left(\frac{5}{2} \{ \omg\}_{-2} \cdot \otz \right) \\
&=& 3\left\{ R \right\}_{-\frac{8}{3}} \cdot \tau R -3\left(\frac{5}{2} \{ \omg\}_{-2} \cdot \otz \right)\\
&=&3U.
\end{eqnarray*}

There are still two items to clarify. First, to determine that $U \neq 0$  and then to find the exact description of this vector in $V^{\Lam_0}$. 

The next step is to determine both the vector $\{R\}_{-8/3} \cdot \tau R$ and $\{ \omg \}_{-2} \cdot \otz$, hence answering both questions simultaneously.  We present two lemmas needed to determine $\{R\}_{-8/3} \cdot \tau R$.

\begin{lem}
If $\herf{\alp}{\beta} = \frac23$, then $\{ \otea{\alp}\}_{-8/3} \cdot (\otea{\beta}) = \eps(\alp,\beta)\alp(-1)\otimes e^{\alp+\beta}.$
\end{lem}
 
\begin{proof} 
\begin{eqnarray*}
&&\{ \otea{\alp}\}_{-8/3} \cdot (\otea{\beta}) \\ 
&=& Res\left( z^{-8/3} \exp\left( \sum_{k \geq 1}\frac{ \alp(-k)}{k}z^k\right) e_{\alp}z^{\alp(0)} \cdot \otea{\beta}  \right) \\
&=& Res\left( \eps(\alp,\beta) z^{-2} \exp\left( \sum_{k \geq 1}\frac{ \alp(-k)}{k}z^k\right)\cdot \otea{\alp+\beta} \right)  \\
&=& \eps(\alp,\beta) \alp(-1)\otimes e^{\alp+\beta}.
\end{eqnarray*}
\end{proof}

\begin{lem}
If $\herf{\alp}{\alp} = \frac{4}{3}$, then
 $$\{ \otea{\alp}\}_{-8/3} \cdot (\otea{-\alp}) = \left(\frac{\alp(-3)}{3} + \frac{\alp(-1)\alp(-2)}{2} + \frac{\alp(-1)^3}{6}\right) \otimes e^0.$$
\end{lem} 
 
\begin{proof}
\begin{eqnarray*}
&&\{ \otea{\alp}\}_{-8/3} \cdot (\otea{-\alp}) \\ 
&=& Res\left( z^{-8/3} \exp\left( \sum_{k \geq 1}\frac{ \alp(-k)}{k}z^k\right) e_{\alp}z^{\alp(0)}  \cdot \otea{-\alp}  \right) \\
&=& Res\left( z^{-4} \eps(\alp,-\alp)\exp\left( \sum_{k \geq 1}\frac{ \alp(-k)}{k}z^k\right)\cdot \otz \right)  \\
&=& Res\left( z^{-1} \left(\frac{\alp(-3)}{3} + \frac{\alp(-1)\alp(-2)}{2} + \frac{\alp(-1)^3}{6}\right) \otimes e^0 \right) \\
&=& \left(\frac{\alp(-3)}{3} + \frac{\alp(-1)\alp(-2)}{2} + \frac{\alp(-1)^3}{6}\right) \otimes e^0.
\end{eqnarray*}
\end{proof}

Using these two lemmas with the vectors $R$ and $\tau R$, then 
\begin{eqnarray*}
&&\{ R\}_{-8/3} \cdot \tau R \\
&=& \left(\frac{(-\lam_1 + \lam_6)(-3)}{3} + \frac{(-\lam_1 + \lam_6)(-1)(-\lam_1 + \lam_6)(-2)}{2}\right)\otimes e^0 \\
&& +\left( \frac{(-\lam_1 + \lam_6)(-1)^3}{6}\right) \otimes e^0 + (-\lam_1 + \lam_6)(-1) \otimes e^{-\lam_1 -\lam_3 + \lam_5 + \lam_6} \\
&& - (-\lam_1 + \lam_6)(-1) \otimes e^{-2\lam_1 + \lam_3 - \lam_5 + 2\lam_6}\\
&&+ \left(\frac{(\lam_3 - \lam_5)(-3)}{3} + \frac{(\lam_3 - \lam_5)(-1)(\lam_3 - \lam_5)(-2)}{2} + \frac{(\lam_3 - \lam_5)(-1)^3}{6}\right) \otimes e^0 \\
&&+ (\lam_3 - \lam_5)(-1) \otimes e^{\lam_1 +\lam_3 - \lam_5 - \lam_6} - (\lam_3 - \lam_5)(-1) \otimes e^{-\lam_1 + 2\lam_3 - 2\lam_5 + \lam_6} \\
&&+ \left(\frac{(\lam_1 - \lam_3 + \lam_5 - \lam_6)(-3)}{3} + \frac{(\lam_1 - \lam_3 + \lam_5 - \lam_6)^3(-1)}{6}\right) \otimes e^0\\
&& + \left(\frac{(\lam_1 - \lam_3 + \lam_5 - \lam_6)(-1)(\lam_1 - \lam_3 + \lam_5 - \lam_6)(-2)}{2}\right) \otimes e^0 \\
&& - (\lam_1 - \lam_3 + \lam_5 - \lam_6)(-1) \otimes e^{2\lam_1 -\lam_3 + \lam_5 - 2\lam_6} \\
&&- (\lam_1 - \lam_3 + \lam_5 - \lam_6)(-1) \otimes e^{\lam_1 - 2\lam_3 + 2\lam_5 -\lam_6} \\
&=& \left(\frac{(-\lam_1 + \lam_6)(-1)(-\lam_1 + \lam_6)(-2)}{2} + \frac{(-\lam_1 + \lam_6)(-1)^3}{6}\right) \otimes e^0 \\
&&+ \left(\frac{(\lam_3 - \lam_5)(-1)(\lam_3 - \lam_5)(-2)}{2} + \frac{(\lam_3 - \lam_5)(-1)^3}{6}\right) \otimes e^0 \\
&&+ \left(\frac{(\lam_1 - \lam_3 + \lam_5 - \lam_6)(-1)(\lam_1 - \lam_3 + \lam_5 - \lam_6)(-2)}{2}\right)\otimes e^0 \\
&& + \left(\frac{(\lam_1 - \lam_3 + \lam_5 - \lam_6)(-1)^3}{6}\right) \otimes e^0 \\
&&+ (-\lam_1 + \lam_6)(-1) \otimes e^{-\lam_1 -\lam_3 + \lam_5 + \lam_6} - (-\lam_1 + \lam_6)(-1) \otimes e^{-2\lam_1 + \lam_3 - \lam_5 + 2\lam_6}\\
&&+ (\lam_3 - \lam_5)(-1) \otimes e^{\lam_1 +\lam_3 - \lam_5 - \lam_6} - (\lam_3 - \lam_5)(-1) \otimes e^{-\lam_1 + 2\lam_3 - 2\lam_5 + \lam_6} \\
&&- (\lam_1 - \lam_3 + \lam_5 - \lam_6)(-1) \otimes e^{2\lam_1 -\lam_3 + \lam_5 - 2\lam_6} \\
&&- (\lam_1 - \lam_3 + \lam_5 - \lam_6)(-1) \otimes e^{\lam_1 - 2\lam_3 + 2\lam_5 -\lam_6}.
%&=& \left(\frac{(-\lam_1 + \lam_6)(-1)(-\lam_1 + \lam_6)(-2)}{2} + \frac{(-\lam_1 + \lam_6)^3(-3)}{6}\right) \otimes e^0 \\
%&&+ \left(\frac{(\lam_3 - \lam_5)(-1)(\lam_3 - \lam_5)(-2)}{2} + \frac{(\lam_3 - \lam_5)^3(-3)}{6}\right) \otimes e^0 \\
%&&+ \left(\frac{(\lam_1 - \lam_3 + \lam_5 - \lam_6)(-1)(\lam_1 - \lam_3 + \lam_5 - \lam_6)(-2)}{2}\right) \otimes e^0\\
%&&+ \left(\frac{(\lam_1 - \lam_3 + \lam_5 - \lam_6)^3(-3)}{6}\right) \otimes e^0 \\
%&&+ (-\lam_1 + \lam_6)(-1) \otimes e^{-\alp_1 - \alp_3 + \alp_5 + \alp_6} - (-\lam_1 + \lam_6)(-1) \otimes e^{-\alp_1 + \alp_6}\\
%&&+ (\lam_3 - \lam_5)(-1) \otimes e^{\alp_1 +\alp_3 - \alp_5 - \alp_6} - (\lam_3 - \lam_5)(-1) \otimes e^{\alp_3 - \alp_5} \\
%&&- (\lam_1 - \lam_3 + \lam_5 - \lam_6)(-1) \otimes e^{\alp_1 - \alp_6} - (\lam_1 - \lam_3 + \lam_5 - \lam_6)(-1) \otimes e^{- \alp_3 + \alp_5 }\\
\end{eqnarray*}
Two lemmas are also needed to compute $\{ \omg \}_{-2} \cdot \otz$.

\begin{lem}
If $\herf{\alp}{\alp} =4$, then $\{ \otea{\alp}\}_{-2} \cdot (\otz) = \alp(-1) \otimes e^{\alp}$.
\end{lem}
\begin{proof}
\begin{eqnarray*}
\{ \otea{\alp} \}_{-2} \cdot \otz 
&=& Res\left( z^{-2} \exp\left( \sum_{k \geq 1}\frac{ \alp(-k)}{k}z^k\right) e_{\alp}z^{\alp(0)}  \cdot (\otz)  \right) \\
&=& Res\left( z^{-2} \exp\left( \sum_{k \geq 1}\frac{ \alp(-k)}{k}z^k\right)\cdot (\otea{\alp}) \right)  \\
&=& \alp(-1) \otimes e^{\alp}
\end{eqnarray*}
\end{proof}

The next lemma helps calculate the action of the Heisenberg operators.
\begin{lem} We have $\{ \alp(-1)^2\otz\}_{-3} \cdot \otz = 2\alp(-1)\alp(-2) \otimes e^0.$
\end{lem}
\begin{proof}
\begin{eqnarray*}
\{ \alp(-1)^2\otz\}_{-3} \cdot \otz 
&=&  \left( \dots + \alp(-1)\alp(-2) + \alp(-2)\alp(-1) + \dots \right) \cdot \otz\\
&=&  (\alp(-1)\alp(-2) + \alp(-2)\alp(-1)) \cdot \otz\\
&=& 2\alp(-1)\alp(-2) \otimes e^0
\end{eqnarray*}
\end{proof}
Using these two lemmas gives 
\begin{eqnarray*}
&&\{ \omg \}_{-2} \cdot \otz \\
&=& \frac{1}{10}\left( 2(-\lam_1 + \lam_6)(-1)(-\lam_1 + \lam_6)(-2) + 2(\lam_3 - \lam_5)(-1)(\lam_3 - \lam_5)(-2)\right) \otimes e^0 \\
&&+ \frac{1}{10}\left(2(\lam_1 - \lam_3 + \lam_5 - \lam_6)(-1)(\lam_1 - \lam_3 + \lam_5 - \lam_6)(-2)\right) \otimes e^0 \\
&&-\frac{1}{5}((\alp_1 - \alp_6)(-1) \otimes e^{\alp_1 - \alp_6}+ (-\alp_1 + \alp_6)(-1)\otimes e^{-\alp_1 + \alp_6}) \\
&&-\frac{1}{5}((\alp_3 - \alp_5)(-1) \otimes e^{\alp_3 - \alp_5}+ (-\alp_3 + \alp_5)(-1)\otimes e^{-\alp_3 + \alp_5}) \\
&&+\frac{1}{5}(\alp_1 + \alp_3 - \alp_5 - \alp_6)(-1) \otimes e^{\alp_1 + \alp_3 - \alp_5 - \alp_6} \\
&&+ \frac{1}{5}(-\alp_1 - \alp_3 + \alp_5 + \alp_6)(-1)\otimes e^{-\alp_1 - \alp_3 + \alp_5 + \alp_6} \\
&=& \frac{1}{5}\left((-\lam_1 + \lam_6)(-1)(-\lam_1 + \lam_6)(-2) + (\lam_3 - \lam_5)(-1)(\lam_3 - \lam_5)(-2)\right) \otimes e^0 \\
&&+ \frac{1}{5}\left((\lam_1 - \lam_3 + \lam_5 - \lam_6)(-1)(\lam_1 - \lam_3 + \lam_5 - \lam_6)(-2)\right) \otimes e^0 \\
&&-\frac{1}{5}((\alp_1 - \alp_6)(-1) \otimes e^{\alp_1 - \alp_6}+ (-\alp_1 + \alp_6)(-1)\otimes e^{-\alp_1 + \alp_6}) \\
&&-\frac{1}{5}((\alp_3 - \alp_5)(-1) \otimes e^{\alp_3 - \alp_5}+ (-\alp_3 + \alp_5)(-1)\otimes e^{-\alp_3 + \alp_5}) \\
&&+\frac{1}{5}(\alp_1 + \alp_3 - \alp_5 - \alp_6)(-1) \otimes e^{\alp_1 + \alp_3 - \alp_5 - \alp_6} \\
&&+ \frac{1}{5}(-\alp_1 - \alp_3 + \alp_5 + \alp_6)(-1)\otimes e^{-\alp_1 - \alp_3 + \alp_5 + \alp_6}.
\end{eqnarray*}

It is now apparent that $\{ R\}_{-8/3} \cdot \tau R $ and $\{ \omg \}_{-2} \cdot \otz$ are linearly independent and therefore $U \neq 0$.  We may rewrite $\{ R\}_{-8/3} \cdot \tau R$ as follows by expressing the exponents of $e$ in terms of of simple roots.  We have
\begin{eqnarray*}
&&\{ R\}_{-8/3} \cdot \tau R \\
&&= \left(\frac{(-\lam_1 + \lam_6)(-1)(-\lam_1 + \lam_6)(-2)}{2} + \frac{(-\lam_1 + \lam_6)(-1)^3}{6}\right) \otimes e^0 \\
&&+ \left(\frac{(\lam_3 - \lam_5)(-1)(\lam_3 - \lam_5)(-2)}{2} + \frac{(\lam_3 - \lam_5)(-1)^3}{6}\right) \otimes e^0 \\
&&+ \left(\frac{(\lam_1 - \lam_3 + \lam_5 - \lam_6)(-1)(\lam_1 - \lam_3 + \lam_5 - \lam_6)(-2)}{2}\right)\otimes e^0 \\
&& + \left(\frac{(\lam_1 - \lam_3 + \lam_5 - \lam_6)(-1)^3}{6}\right) \otimes e^0 \\
&&+ (-\lam_1 + \lam_6)(-1) \otimes e^{-\alp_1 - \alp_3 + \alp_5 + \alp_6} - (-\lam_1 + \lam_6)(-1) \otimes e^{-\alp_1 + \alp_6}\\
&&+ (\lam_3 - \lam_5)(-1) \otimes e^{\alp_1 +\alp_3 - \alp_5 - \alp_6} - (\lam_3 - \lam_5)(-1) \otimes e^{\alp_3 - \alp_5} \\
&&- (\lam_1 - \lam_3 + \lam_5 - \lam_6)(-1) \otimes e^{\alp_1 - \alp_6} - (\lam_1 - \lam_3 + \lam_5 - \lam_6)(-1) \otimes e^{- \alp_3 + \alp_5 }.\\
\end{eqnarray*}
We may finally rewrite the HWV of $Vir\left(\frac{4}{5}, 3\right)\otimes W^{\Omg_0}$ as
\begin{eqnarray*}
U &=& \left\{ R \right\}_{-\frac{8}{3}}\cdot \tau R -\frac{5}{2} \left\{ \omg \right\}_{-2} \cdot \otea{0} \\
&=& \frac16 \left((-\lam_1 + \lam_6)(-1)^3 + (\lam_3 - \lam_5)(-1)^3\right) \otimes e^0 \\
&& + \frac16 (\lam_1 - \lam_3 + \lam_5 - \lam_6)(-1)^3  \otimes e^0 \\
&&+ \frac{1}{2}((\lam_3 - \lam_5)(-1) \otimes e^{\alp_1 - \alp_6} + (\lam_3 - \lam_5)(-1) \otimes e^{-\alp_1 + \alp_6}) \\
&&+ \frac{1}{2}((-\lam_1 + \lam_6)(-1) \otimes e^{\alp_3 - \alp_5} + (-\lam_1 + \lam_6)(-1) \otimes e^{-\alp_3 + \alp_5} ) \\
&&+ \frac{1}{2} (-\lam_1 + \lam_3 - \lam_5 + \lam_6)(-1) \otimes e^{\alp_1 +\alp_3 - \alp_5 - \alp_6}  \\
&&+ \frac{1}{2}(-\lam_1 + \lam_3 - \lam_5 + \lam_6)(-1) \otimes e^{-\alp_1 - \alp_3 + \alp_5 + \alp_6}.
\end{eqnarray*}

%\frac{1}{6}\left( (-\lam_1 + \lam_6)(-3) + (\lam_3 - \lam_5)(-3) + (\lam_1 - \lam_3 + \lam_5 - \lam_6)(-3)\right) \otimes e^0 \\ &&+

\end{proof}

\begin{thm} \label{ch6thm1}The decomposition of $V^{\Lam_0}$ as a direct sum of $Vir \otimes \Ff$-modules is as follows: 
$$V^{\Lam_0} = \left(Vir\left(\frac{4}{5},0\right) \oplus Vir\left(\frac{4}{5},3\right)\right) \otimes W^{\Omg_0} \space \space  \bigoplus \space \space \left(Vir\left(\frac{4}{5},\frac{2}{5}\right) \oplus Vir\left(\frac{4}{5},\frac{7}{5}\right)\right) \otimes W^{\Omg_4},$$
with the highest weight vectors of the summands given by
$$\begin{aligned}
&\otz, U (\mbox{as above}), \otea{\mu} - \otea{\tau\mu}, \mbox{ and } \\
&(-2\alp_1 - \alp_3 + \alp_5 + 2\alp_6)(-1)\otimes e^{\mu} 
 + (2\alp_1 + \alp_3 - \alp_5 - 2\alp_6)(-1)\otimes e^{\tau\mu} \\
 &\hspace{60pt}+ 3 \otimes (e^{\mu + \alp_1 - \alp_6} + e^{\tau\mu - \alp_1 + \alp_6}).
\end{aligned}$$

\end{thm}

%
%%
%%%
%%%%
%%%%%

\section{The Decomposition of $V^{\Lam_1}$ and $V^{\Lam_6}$}
\label{s6.2}
This section contains detailed calculations used to determine the HWVs in the decomposition of $V^{\Lam_1}$.  These calculations are followed by a brief explanation of the determination for the HWVs in the decomposition of $V^{\Lam_6}$.

\begin{lem} $\otea{\lam_1}$ is the HWV in $Vir\left(\frac{4}{5}, \frac{1}{15}\right) \otimes W^{\Omg_4}$ in $V^{\Lam_1}$. \end{lem}

\begin{proof} 1)$\yvomn{\otea{-\theta}}{1} \cdot \otea{\lam_1} =0. $

Noting that $\herf{-\theta}{\lam_1} = -1$, then we have
\begin{eqnarray*}
&&\yvomn{\otea{-\theta}}{1} \cdot \otea{\lam_1}  \\
&=& Res \left[ z^1 \exp\left( \sum_{k \geq 1}\frac{ -\theta(-k)}{k}z^k\right) \exp\left( \sum_{k \geq 1}\frac{ -\theta(k)}{-k}z^{-k}\right) e_{-\theta}z^{-\theta(0)} \cdot \otea{\lam_1} \right] \\
&=& Res \left[ \eps(-\theta, \lam_1) z^0 \exp\left( \sum_{k \geq 1}\frac{ -\theta(-k)}{k}z^k\right) \cdot 1 \otimes e^{-\theta + \lam_1}\right] \\
&=& 0.
\end{eqnarray*}

2) $\yvomn{\beta}{0} \cdot  \otea{\lam_1}= 0$ for all $\beta \in \Delta_{F_4}$. \\
Notice for all $\alp$, simple roots of $E_6$, $\herf{\alp}{\lam_1} \geq 0$, hence we have 
\begin{eqnarray*}
&&\yvomn{\otea{\alp}}{0} \cdot \otea{\lam_1} \\
 &=& Res \left[ \exp\left( \sum_{k \geq 1}\frac{ \alp(-k)}{k}z^k\right) \exp\left( \sum_{k \geq 1}\frac{ \alp(k)}{-k}z^{-k}\right) e_{\alp}z^{\alp(0)} \cdot \otea{\lam_1}\right] \\
&=& Res \left[ \eps(\alp, \lam_1) z^{\herf{\alp}{\lam_1}}\exp\left( \sum_{k \geq 1}\frac{ \alp(-k)}{k}z^k\right) \cdot 1 \otimes e^{\alp + \lam_1}\right] \\
&=& 0.
\end{eqnarray*}

Since each $\yvomn{\otea{\alp}}{0}$ then each $\yvomn{\beta}{0}$ will also kill.\\

3) The positive Virasoro operators kill $\otea{\lam_1}$. \\
Recall that the conformal vector generating the coset Virasoro is 
\begin{eqnarray*} \omg &=&  \frac{1}{10}\left[ (-\lam_1 + \lam_6)^2(-1) + (\lam_3 - \lam_5)^2(-1) +  (\lam_1 - \lam_3 + \lam_5 - \lam_6)^2(-1)\right]\otimes e^0 \\
&&+ \frac{1}{5}\left[ -1 \otimes e^{\pm\gamma_1} - 1 \otimes e^{\pm\gamma_2} + 1 \otimes e^{\pm\gamma_3} \right], 
\end{eqnarray*}
where $\gamma_1 = \alp_1 - \alp_6$, $\gamma_2 = \alp_3 - \alp_5$, and $\gamma_3 = \gamma_1 + \gamma_2$.

This action will be considered in two parts, first we check $\yvomn{A}{1} \cdot\otea{\lam_1} $.  Since A is a sum of products of Heisenberg operators then it will suffice to check this for each term.  Since the rest follow similarly, we show the first calculation as follows
\begin{eqnarray*}
&&\yvomn{(\lam_1+\lam_6)(-1)^2}{1} \cdot \otea{\lam_1} \\
&&= \sum_{r \in \Z} : (\lam_1+\lam_6)(r)(\lam_1+\lam_6)(1-r): \cdot \otea{\lam_1} \\
&&= \left( \cdots+ 2(\lam_1+\lam_6)(0)(\lam_1+\lam_6)(1) +\cdots \right) \cdot \otea{\lam_1} \\
&&= 0.
\end{eqnarray*}
It follows from carrying out the other 2 calculations that $\yvomn{A}{1} \cdot \otea{\lam_1} = 0$. 

To check part B, consider the following:
\begin{eqnarray*}
&&\yvomn{\otea{\gamma_i}}{1} \cdot \otea{\lam_1} \\
&=& Res \left[z^2 \exp\left( \sum_{k \geq 1}\frac{ \gamma_i(-k)}{k}z^k\right) \exp\left( \sum_{k \geq 1}\frac{ \gamma_i(k)}{-k}z^{-k}\right) e_{\gamma_i}z^{\gamma_i(0)} \cdot \otea{\lam_1}\right] \\
&=& Res \left[ \eps(\gamma_i, \lam_i)z^{2 + \herf{\gamma_i}{ \lam_1}}\exp\left( \sum_{k \geq 1}\frac{ \gamma_i(-k)}{k}z^k\right) \cdot 1 \otimes e^{\gamma_i + \lam_1}\right] \\
&=& 0.
\end{eqnarray*} 

With the last step happening since, for all $i \in \{1,2,3 \}$ we have $\herf{\pm\gamma_i}{ \lam_1} > -1$, hence the residue will be 0.  Since the operators $\yvomn{A}{1}$ and $\yvomn{B}{1}$ kill $\otea{\lam_1}$, then so does $\yvomn{\omg}{1}$.  

The calculation that $\yvomn{A}{2} \cdot \otea{\lam_1}  = 0$ is the same as the one above for $\yvomn{A}{1}$.  The calculation that $\yvomn{B}{2} \cdot \otea{\lam_1}  = 0$ only differs from the above by having a $z^3$ in place of $z^2$, therefore $\yvomn{\omg}{2} \cdot \otea{\lam_1}  = 0$.\\

4)The vector $\otea{\lam_1}$ has eigenvalue $\frac{1}{15}$ for $\yvomn{\omg}{0}$.
The first part of calculating $\yvomn{A}{0}$ is given by the following
\begin{eqnarray*}
&&-\frac{1}{10}\yvomn{(-\lam_1 + \lam_6)(-1)^2}{0} \cdot \otea{\lam_1} \\
&=& -\frac{1}{10}\sum_{r \in \Z} : (-\lam_1 + \lam_6)(r)(-\lam_1 + \lam_6)(-r): \cdot \otea{\lam_1} \\
&=& -\frac{1}{10}\left( \cdots + (-\lam_1 + \lam_6)^2(0) + \cdots \right) \cdot \otea{\lam_1} \\
&=& -\frac{1}{10}\herf{-\lam_1 + \lam_6}{\lam_1}^2 \otimes e^{\lam_1} \\
&=& \frac{1}{15} \otimes e^{\lam_1}.
\end{eqnarray*}
Results for the other two parts are
$$-\frac{1}{10}\yvomn{(\lam_3 - \lam_5)^2(-1)}{0} \cdot \otea{\lam_1} = -\frac{1}{30} \otimes e^{\lam_1},$$
\noindent and
$$\frac{1}{10}\yvomn{(\lam_1- \lam_3 +\lam_5-\lam_6)^2(-1)}{0} \cdot \otea{\lam_1} = \frac{1}{30} \otimes e^{\lam_1}.$$

\noindent Therefore adding these three together, $\yvomn{A}{0} \cdot \otea{\lam_1} = \frac{1}{15}\otea{\lam_1}$.

When calculating the action of $\yvomn{B}{0}$, it is useful to observe that  $\forall i \in \{ 1,2,3\}$, \\$\herf{\pm\gamma_i}{\lam_1} > -1$, therefore
\begin{eqnarray*}
&&\yvomn{\otea{\gamma_i}}{0} \cdot  \otea{\lam_1} \\
&=& Res \left[z \exp\left( \sum_{k \geq 1}\frac{ \gamma_i(-k)}{k}z^k\right) \exp\left( \sum_{k \geq 1}\frac{ \gamma_i(k)}{-k}z^{-k}\right) e_{\gamma_i}z^{\gamma_i(0)} \cdot \otea{\lam_1}\right] \\
&=& Res \left[ z^{1 + \herf{\gamma_i}{\lam_1}} \eps(\gam_i,\lam_1)\exp\left( \sum_{k \geq 1}\frac{ \gamma_i(-k)}{k}z^k\right) \cdot \otea{\gamma_1 + \lam_1}\right] \\
&=& 0.
\end{eqnarray*}

So $\yvomn{\omg}{0} \cdot \otea{\lam_1} = \frac{1}{15}(\otea{\lam_1})$. 
\end{proof}
%Determining the HWV for $Vir\left(\frac{4}{5}, \frac{2}{3}\right) \otimes W^{\Omg_0}$. \\
\begin{lem} $R = \otea{-\lam_1 + \lam_6} + \otea{\lam_3 - \lam_5} - \otea{\lam_1 - \lam_3 + \lam_5 - \lam_6}$ is the HWV for \\$Vir\left(\frac{4}{5}, \frac{2}{3}\right) \otimes W^{\Omg_0}$ in $V^{\Lam_1}$.
 \end{lem}
\begin{proof}
For this section let $\nu_1 = -\lam_1 + \lam_6, \nu_2 = \lam_3 - \lam_5 , \nu_3 = \lam_1 - \lam_3 + \lam_5 - \lam_6$

1)$\yvomn{\otea{-\theta}}{1} \cdot R =0 $

For any $i \in \{1,2,3\}$, $\herf{-\theta}{\nu_i} = 0$, so 
\begin{eqnarray*}
&&\yvomn{\otea{-\theta}}{1} \cdot \otea{\nu_i}  \\
&=& Res \left[ z^1 \exp\left( \sum_{k \geq 1}\frac{ -\theta(-k)}{k}z^k\right) \exp\left( \sum_{k \geq 1}\frac{ -\theta(k)}{-k}z^{-k}\right) e_{-\theta}z^{-\theta(0)} \cdot \otea{\nu_i} \right] \\
&=& Res \left[ \eps(-\theta, \nu_i) z^1 \exp\left( \sum_{k \geq 1}\frac{ -\theta(-k)}{k}z^k\right) \cdot 1 \otimes e^{-\theta + \nu_i}\right] \\
&=& 0.
\end{eqnarray*}

2) $\yvomn{\beta}{0} \cdot R= 0$ \\
Notice for all simple roots $\alp$ of $E_6$, $\herf{\alp}{\nu_i} \geq -1$ for all $i \in \{1,2,3 \}$, and we have:
\begin{eqnarray*}
&&\yvomn{\otea{\alp}}{0} \cdot \otea{\nu_i} \\
&=& Res \left[ \exp\left( \sum_{k \geq 1}\frac{ \alp(-k)}{k}z^k\right) \exp\left( \sum_{k \geq 1}\frac{ \alp(k)}{-k}z^{-k}\right) e_{\alp}z^{\alp(0)} \cdot \otea{\nu_i}\right] \\
&=& Res \left[ \eps(\alp, \nu_i) z^{\herf{\alp}{\nu_i}}\exp\left( \sum_{k \geq 1}\frac{ \alp(-k)}{k}z^k\right) \cdot 1 \otimes e^{\alp + \nu_i}\right] \\
&=& \left\{ \begin{array}{rcl} 
0 & \mbox{if} & \herf{\alp}{\nu_i} \geq 0 \\ 
\eps(\alp,\nu_i) \otea{\alp+ \nu_i} & \mbox{if} & \herf{\alp}{\nu_i} = -1. 
\end{array}\right.
\end{eqnarray*}

Using this result, we see that
\begin{eqnarray*}
&&\yvomn{\beta_1}{0} \cdot \left( \otea{-\lam_1 + \lam_6} + \otea{\lam_3 - \lam_5} - \otea{\lam_1 - \lam_3 + \lam_5 - \lam_6} \right) \\
&&=  \yvomn{\otea{\alp_2}}{0} \cdot \left( \otea{-\lam_1 + \lam_6} + \otea{\lam_3 - \lam_5} - \otea{\lam_1 - \lam_3 + \lam_5 - \lam_6} \right) \\
&&=0
\end{eqnarray*}
and
\begin{eqnarray*}
&&\yvomn{\beta_2}{0} \cdot \left( \otea{-\lam_1 + \lam_6} + \otea{\lam_3 - \lam_5} - \otea{\lam_1 - \lam_3 + \lam_5 - \lam_6} \right) \\
&&=  \yvomn{\otea{\alp_4}}{0} \cdot \left( \otea{-\lam_1 + \lam_6} + \otea{\lam_3 - \lam_5} - \otea{\lam_1 - \lam_3 + \lam_5 - \lam_6} \right) \\
&&=0.
\end{eqnarray*}
For the last two we have
\begin{eqnarray*}
&&\yvomn{\beta_3}{0} \cdot \left( \otea{-\lam_1 + \lam_6} + \otea{\lam_3 - \lam_5} - \otea{\lam_1 - \lam_3 + \lam_5 - \lam_6} \right) \\
&&=  \left( \yvomn{\otea{\alp_3}}{0} + \yvomn{\otea{\alp_5}}{0}\right) \cdot \left( \otea{-\lam_1 + \lam_6} + \otea{\lam_3 - \lam_5} - \otea{\lam_1 - \lam_3 + \lam_5 - \lam_6} \right) \\
&&= -\eps(\alp_3, \lam_1-\lam_3+\lam_5-\lam_6) \otea{\alp_3 + \lam_1-\lam_3+\lam_5-\lam_6} + \eps(\alp_5, \lam_3-\lam_5) \otea{\alp_5 + \lam_3-\lam_5} \\
&&= -\eps(\alp_3, -\lam_3+\lam_5)\eps(\alp_3, \lam_1-\lam_6) \otea{\alp_5 +\lam_3-\lam_5} + \eps(\alp_5, \lam_3-\lam_5) \otea{\alp_5 + \lam_3-\lam_5}\\
&&= 0
\end{eqnarray*}
and
\begin{eqnarray*}
&&\yvomn{\beta_4}{0} \cdot \left( \otea{-\lam_1 + \lam_6} + \otea{\lam_3 - \lam_5} - \otea{\lam_1 - \lam_3 + \lam_5 - \lam_6} \right) \\
&&=  \left( \yvomn{\otea{\alp_1}}{0} + \yvomn{\otea{\alp_6}}{0}\right) \cdot \left( \otea{-\lam_1 + \lam_6} + \otea{\lam_3 - \lam_5} - \otea{\lam_1 - \lam_3 + \lam_5 - \lam_6} \right) \\
&&= \eps(\alp_1, -\lam_1+\lam_6) \otea{\alp_1 - \lam_1+\lam_6} - \eps(\alp_6, \lam_1- \lam_3+\lam_5-\lam_6) \otea{\alp_5 + \lam_3-\lam_5} \\
&&= \eps(\alp_1, -\lam_1+\lam_6) \otea{\alp_1 -\lam_1+\lam_6} + -\eps(\alp_6, -\lam_3+\lam_5)\eps(\alp_6, \lam_1-\lam_6) \otea{\alp_1 - \lam_1+\lam_6}\\
&&= 0.
\end{eqnarray*}
Hence each $\yvomn{\beta_i}{0}$ kills $R$.\\

3)$\yvomn{\omg}{1}\cdot R = 0$ 

For $i \in \{ 1,2,3\}$, we first check $\yvomn{A}{1}\cdot \otea{\nu_i}$.  As usual we show one calculation, the rest being similar:
\begin{eqnarray*}
&&\yvomn{(-\lam_1+\lam_6)(-1)^2}{1} \cdot  \otea{\nu_i} \\
&=& \sum_{r \in \Z} : (-\lam_1+\lam_6)(r)(-\lam_1+\lam_6)(1-r): \cdot \otea{\nu_i} \\
&=& \left( \cdots + 2(-\lam_1+\lam_6)(0)(-\lam_1+\lam_6)(1) +\cdots \right) \otea{\nu_i} \\
&=& 0.
\end{eqnarray*}
Hence each component of A will kill each $\otea{\nu_i}$, and therefore kill the sum of these vectors. 

To check part B, for all $i,j \in \{1,2,3 \}$, consider the following
\begin{eqnarray*}
&&\yvomn{\otea{\gamma_i}}{1} \cdot \otea{\nu_j} \\
&=& Res \left[z^2 \exp\left( \sum_{k \geq 1}\frac{ \gamma_i(-k)}{k}z^k\right) \exp\left( \sum_{k \geq 1}\frac{ \gamma_i(k)}{-k}z^{-k}\right) e_{\gamma_i}z^{\gamma_i(0)} \cdot \otea{\nu_j}\right] \\
&=& Res \left[ \eps(\gamma_i, \nu_j)z^{2 + \herf{\gamma_i}{ \nu_j}}\exp\left( \sum_{k \geq 1}\frac{ \gamma_i(-k)}{k}z^k\right) \cdot 1 \otimes e^{\gamma_i + \nu_j}\right] \\
&=& 0
\end{eqnarray*} 
because $\herf{\pm\gamma_i}{ \nu_j} \geq -2$.  So, all the operators from components of $B$ will kill $R$, and therefore $\yvomn{\omg}{1}\cdot R = 0$.

Since each term of $\yvomn{A}{2}$ contains at least one Heisenberg annihilation operator, \\$\yvomn{A}{2} \cdot R = 0$. Since the $z^2$ in the above calculation is replaced by $z^3$ for the operator $\yvomn{B}{2}$, the residue will also be zero.  Hence $\yvomn{\omg}{2} \cdot R =0$.\\ 

4)$\yvomn{\omg}{0} \cdot R = \frac23 R$ \\
For the action of part A on the vector, all calculations are similar to the following:
\begin{eqnarray*}
&&\yvomn{(-\lam_1 + \lam_6)(-1)^2}{0}\cdot \otea{\nu_i} \\
&=& \sum_{r \in \Z} : (-\lam_1+\lam_6)(r)(-\lam_1+\lam_6)(-r): \cdot \otea{\nu_i} \\
&=& \left( \cdots + (-\lam_1+\lam_6)(0)^2 +\cdots \right) \otea{\nu_i} \\
&=& \herf{-\lam_1+\lam_6}{\nu_i}^2 \otea{\nu_i}.
\end{eqnarray*}
Hence calculating the action on the vector is as follows:

$\yvomn{A}{0} \cdot \otea{-\lam_1 + \lam_6}= \left( \frac{4}{15}\right) \otea{-\lam_1 + \lam_6}$ 

$\yvomn{A}{0}\cdot \otea{\lam_3 - \lam_5}= \left( \frac{4}{15}\right) \otea{\lam_3 - \lam_5}$ 

$\yvomn{A}{0}\cdot (-\otea{\lam_1 - \lam_3 + \lam_5 - \lam_6})= \left( \frac{4}{15}\right)\left(-\otea{\lam_1 - \lam_3 + \lam_5 - \lam_6} \right). $
So $\yvomn{A}{0}\cdot R = \frac{4}{15}R$.

Before proceeding, a few useful identities are recorded below:\\
$\begin{aligned}
-\lam_1 + \lam_6 + \alp_1 - \alp_6 &= \lam_1 - \lam_3 + \lam_5 - \lam_6 \\
\lam_3 - \lam_5 - \alp_3 + \alp_5 &= \lam_1 - \lam_3 + \lam_5 - \lam_6 \\
-\lam_1 + \lam_6 + \alp_1 + \alp_3 - \alp_5 - \alp_6 &= \lam_3 - \lam_5 \\
\lam_1 - \lam_3 + \lam_5 - \lam_6 - \alp_1 + \alp_6 &= -\lam_1 + \lam_6 \\
\lam_3 - \lam_5 - \alp_1 - \alp_3 + \alp_5 + \alp_6 &= -\lam_1 + \lam_6.
\end{aligned}$

Now, we give the action of part B on $R$.  Consider the following calculation 
\begin{eqnarray*}
&&\yvomn{\otea{\gamma_i}}{0} \cdot  \otea{\nu_j} \\
&=& Res \left[z \exp\left( \sum_{k \geq 1}\frac{ \gamma_i(-k)}{k}z^k\right) \exp\left( \sum_{k \geq 1}\frac{ \gamma_i(k)}{-k}z^{-k}\right) e_{\gamma_i}z^{\gamma_i(0)} \cdot \otea{\nu_j}\right] \\
&=& Res \left[ \eps(\gamma_i,\nu_j)z^{1 + \herf{\gamma_i}{ \nu_j}}\exp\left( \sum_{k \geq 1}\frac{ \gamma_i(-k)}{k}z^k\right) \cdot \otea{\gamma_i + \nu_j}\right] \\
&=&\left\{ \begin{array}{rcl} 
0 & \mbox{if} & \herf{\alp}{\nu_i} \geq -1 \\ 
\eps(\gamma_i,\nu_j) \otea{\gamma_i + \nu_j} & \mbox{if} & \herf{\alp}{\nu_i} = -2. 
\end{array}\right.
\end{eqnarray*}

The non-zero contributions of $\yvomn{B}{0} \cdot R$, are given by:
for $\gamma_1$ or $-\gamma_1$
\begin{eqnarray*}
&&\herf{\alp_1 - \alp_6}{ -\lam_1 + \lam_6} = -2 \\
&&\herf{-\alp_1 + \alp_6}{ \lam_1 -\lam_3 + \lam_5 - \lam_6} = -2
\end{eqnarray*}
for $\gamma_2$ or $-\gamma_2$
\begin{eqnarray*}
&&\herf{\alp_3 - \alp_5}{ \lam_1 - \lam_3 + \lam_5 - \lam_6} = -2 \\
&&\herf{-\alp_3 + \alp_5}{ \lam_3 - \lam_5} = -2
\end{eqnarray*}
for $\gamma_3$ or $-\gamma_3$
\begin{eqnarray*}
&&\herf{\alp_1 + \alp_3 - \alp_5 - \alp_6}{ -\lam_1 + \lam_6} = -2 \\
&&\herf{-\alp_1 - \alp_3 + \alp_5 + \alp_6}{ \lam_3 - \lam_5} = -2.
\end{eqnarray*}

Hence we have the following results 
\begin{eqnarray*}
&&\yvomn{-\frac{1}{5}(\otea{\gamma_1} + \otea{-\gamma_1})}{0} \cdot R \\
&&=-\eps(\alp_1 - \alp_6, -\lam_1 + \lam_6) \frac{1}{5} \otimes e^{\alp_1 - \alp_6 - \lam_1 + \lam_6} \\
&&\,\,\,\, + \eps(-\alp_1 + \alp_6, \lam_1 - \lam_3 + \lam_5 - \lam_6) \frac{1}{5} \otimes e^{-\alp_1 + \alp_6 + \lam_1 - \lam_3 + \lam_5 - \lam_6} \\
&&=-\frac{1}{5}  \otimes e^{\lam_1 - \lam_3 + \lam_5 - \lam_6} + \frac{1}{5}  \otimes e^{-\lam_1 + \lam_6},\\
\end{eqnarray*}
and
\begin{eqnarray*}
&&\yvomn{-\frac{1}{5}(\otea{\gamma_2} + \otea{-\gamma_2})}{0} \cdot R \\
&&=\eps(\alp_3 - \alp_5, \lam_1 - \lam_3 + \lam_5 - \lam_6) \frac{1}{5} \otimes e^{\alp_3 - \alp_5 + \lam_1 - \lam_3 + \lam_5 - \lam_6} \\
&&\,\,\,\, - \eps(-\alp_3 + \alp_5, \lam_3 - \lam_5) \frac{1}{5}\otimes e^{-\alp_3 + \alp_5 + \lam_3 - \lam_5} \\
&&= \frac{1}{5}\otimes e^{\lam_3 - \lam_5} -\frac{1}{5} \otimes e^{\lam_1 - \lam_3 + \lam_5 - \lam_6}.
\end{eqnarray*}
Lastly we have
\begin{eqnarray*}
&&\yvomn{\frac{1}{5}(\otea{\gamma_3} + \otea{-\gamma_3})}{0} \cdot R\\
&&=\eps(\alp_1 + \alp_3 - \alp_5 - \alp_6, -\lam_1 + \lam_6) \frac{1}{5} \otimes e^{\alp_1 + \alp_3 - \alp_5 - \alp_6 - \lam_1 + \lam_6} \\
&&\,\,\,\, + \eps(-\alp_1 - \alp_3 + \alp_5 + \alp_6, \lam_3 - \lam_5) \frac{1}{5} \otimes e^{-\alp_1 - \alp_3 + \alp_5 + \alp_6 + \lam_3 - \lam_5} \\
&&=\frac{1}{5} \otimes e^{\lam_3 - \lam_5}+ \frac{1}{5} \otimes e^{-\lam_1 + \lam_6}.
\end{eqnarray*}

Now putting all of these calculations together, we have
$$\yvomn{B}{0} \cdot R = \frac{2}{5}R$$.

So $\yvomn{\omg}{0} \cdot R = \frac{4}{15}R + \frac25 R = \frac{2}{3}R$.
\end{proof}

To determine the HWVs in the decomposition of $V^{\Lam_6}$, the automorphism $\hat{\tau}$ is used.  First, since $\tau$ exchanges $\lam_1$ to $\lam_6$, $V^{\Lam_6}$ has the same decomposition as $V^{\Lam_1}$.  Therefore we have
$$ V^{\Lam_6} = \mbox{Vir}\left(\frac{4}{5}, \frac{2}{3}\right) \otimes W^{\Omg_0} \bigoplus \mbox{Vir}\left(\frac{4}{5}, \frac{1}{15}\right) \otimes W^{\Omg_4}. $$

We can obtain the two highest weight vectors for these two summands using the invariance of $\omg$ under $\hat{\tau}$, and make use of $\hat{\tau}$ as a VOA automorphism. First, observe that since $\otea{\lam_1}$ is a HWV for Vir$\left(\frac{4}{5}, \frac{1}{15}\right)\otimes W^{\Omg_4}$ in $V^{\Lam_1}$ and $\tau(\lam_1) = \lam_6$, we have $\hat{\tau}(\otea{\lam_1}) = \otea{\lam_6}$. 

 \begin{lem}Let $v \in V_P$ be a vector such that $\hat{\tau} v = v$.  If $\yvomn{v}{n} \cdot w = kw$, then $\yvomn{v}{n} \cdot \hat{\tau} w = k \hat{\tau} w$.
\end{lem} 
\begin{proof}We have $\yvomn{v}{n} \cdot\hat{\tau} w= \yvomn{\hat{\tau} v}{n}\cdot \hat{\tau}w = \hat{\tau}( \yvomn{ v}{n}\cdot w) = \hat{\tau}(kw) = k\hat{\tau} w. $
\end{proof}

\begin{lem}$\otea{\lam_6}$ is a HWV for Vir$\left(\frac{4}{5}, \frac{1}{15}\right)\otimes W^{\Omg_4}$ in $V^{\Lam_6}$.
\end{lem} 
\begin{proof}The proof follows from Lemma 6.15 and the calculation in $V^{\Lam_1}.$
\end{proof}
\begin{lem}$\tau R = \otea{\lam_1 - \lam_6} + \otea{-\lam_3 + \lam_5} - \otea{-\lam_1 + \lam_3 - \lam_5 + \lam_6}$ is a HWV for \\Vir$\left(\frac{4}{5}, \frac{2}{3}\right) \otimes W^{\Omg_4}$ in $V^{\Lam_6}$.
\end{lem}
\begin{proof} The proof follows from Lemma 6.15 and the calculation in $V^{\Lam_1}.$
\end{proof}

\begin{thm} \label{ch6thm2}We have the following decompositions as direct sums of $Vir \otimes \Ff$-modules: 
\begin{eqnarray*}
V^{\Lam_1} &=& Vir\left(\frac{4}{5},\frac{2}{5}\right) \otimes W^{\Omg_0} \bigoplus Vir\left(\frac{4}{5},\frac{1}{15}\right)  \otimes W^{\Omg_4}, \\
V^{\Lam_6} &=& Vir\left(\frac{4}{5},\frac{2}{5}\right) \otimes W^{\Omg_0} \bigoplus Vir\left(\frac{4}{5},\frac{1}{15}\right)  \otimes W^{\Omg_4}
\end{eqnarray*}
where the highest weight vectors of the summands for $V^{\Lam_1}$ and $V^{\Lam_6}$ are, respectively, 
$$\begin{aligned}
&\otea{\lam_1} \mbox{ and } \otea{-\lam_1 + \lam_6} + \otea{\lam_3 - \lam_5} - \otea{\lam_1 - \lam_3 + \lam_5 - \lam_6},\\
&\mbox{ and } \\ 
&\otea{\lam_6} \mbox{ and } \otea{\lam_1 - \lam_6} + \otea{-\lam_3 + \lam_5} - \otea{-\lam_1 + \lam_3 - \lam_5 + \lam_6}.
\end{aligned}$$

\end{thm}

%
% 
% 
% Conclusions
% 
%
% don't put a chapter # on the conclusions
\chapter*{Conclusions}

% but include it on the contents
\addcontentsline{toc}{chapter}{Conclusions}

The main focus of this dissertation has been on the branching rules of affine Kac-Moody algebras, but the vertex operator theory played a key role in finding the HWVs in each $\fgt$-module $V^{\Lam_i}$ and in finding the conformal vectors.  We now make some observations about vertex operator theory associated to this decomposition.  The coset vertex algebra $\mathcal{C}$, defined as the commutant (see \cite{LL}, p.105) of $W^{\Omg_0}$ in $V^{\Lam_0}$, has Virasoro decomposition $\mathcal{C} = Vir\left(\frac{4}{5},0\right) \oplus Vir\left(\frac{4}{5},3\right)$.  In \cite{FM}(or \cite{Miy}) it is shown that if $\yvoz{v}\cdot w  \neq 0$ for some $v,w \in Vir\left(\frac{4}{5},3\right)$, then $\mathcal{C}$ is isomorphic to the Zamolodchikov $\mathcal{W}_3$-algebra.  Since $V^{\Lam_0}$ is a simple VOA, we have $\yvoz{v} \cdot w \neq 0$ for any non-zero vectors $v,w \in V^{\Lam_0}$, so $\yvoz{U} \cdot U \neq 0$ for the HWV $U \in Vir\left(\frac{4}{5},3\right)$ gives $\mathcal{C} \cong \mathcal{W}_3$.

The $\hat{\tau}$-fixed subspace of $V^{\Lam_0}$ forms a sub-VOA containing both $Vir\left(\frac{4}{5},0\right)$ and $W^{\Omg_0}$.  In fact, we can use the explicit HWVs in Theorem \ref{ch6thm1} to write the $\hat{\tau}$-eigenspace decomposition of $V^{\Lam_0}$.  The +1 $\hat{\tau}$-eigenspace is 
$$Vir\left(\frac{4}{5},0\right) \otimes W^{\Omg_0} \bigoplus Vir\left(\frac{4}{5},\frac{7}{5}\right) \oplus W^{\Omg_4} $$
and the -1 $\hat{\tau}$-eigenspace is 
$$Vir\left(\frac{4}{5},3\right) \otimes W^{\Omg_0} \bigoplus Vir\left(\frac{4}{5},\frac{2}{5}\right) \oplus W^{\Omg_4}. $$

We also note that the paper by \cite{KW} contains many applications of the coset Virasoro to branching rule decompositions with related results on character theory, but these authors do not work out the branching rule question we have considered in this dissertation.  The work of \cite{BG} has a table which contains many pairs of algebras, including our example, but they have no explicit results about HWVs or summands in the decomposition.

% add a new chapter without a chapter # for the references
%\addcontentsline{toc}{chapter}{Bibliography} 

% bibs need to be single-spaced

\fontsize{11}{11pt} \selectfont

% specifies the style for the bibliography and inputs the
% bibtex file references.bib for the bibliography
\bibliographystyle{alpha}
\bibliography{references}

\end{document}